\documentclass[11pt]{amsart}
\usepackage{amsopn}
\usepackage{amssymb, amscd}
\usepackage{multirow}
\usepackage{graphicx, graphics, epsfig}
\usepackage{faktor} 
\usepackage{enumerate}
\usepackage{tikz}
\usepackage{pgfplots}
\usepackage{hyperref}
\usepackage{color}
\usepackage{dynkin-diagrams}
\usepackage{epigraph}
\numberwithin{equation}{section}
\newcommand{\nc}{\newcommand}

\nc{\fg}{\mathfrak{f} } \nc{\vg}{\mathfrak{v} } \nc{\wg}{\mathfrak{w} }
\nc{\zg}{\mathfrak{z} } \nc{\ngo}{\mathfrak{n} } \nc{\kg}{\mathfrak{k} }
\nc{\mg}{\mathfrak{m} } \nc{\bg}{\mathfrak{b} } \nc{\ggo}{\mathfrak{g} } \nc{\eg}{\mathfrak{e} }
\nc{\ggob}{\overline{\mathfrak{g}} } \nc{\sog}{\mathfrak{so} }
\nc{\sug}{\mathfrak{su} } \nc{\spg}{\mathfrak{sp} } \nc{\slg}{\mathfrak{sl} }
\nc{\glg}{\mathfrak{gl} } \nc{\cg}{\mathfrak{c} } \nc{\rg}{\mathfrak{r} }
\nc{\hg}{\mathfrak{h} } \nc{\tg}{\mathfrak{t} } \nc{\ug}{\mathfrak{u} }
\nc{\dg}{\mathfrak{d} } \nc{\ag}{\mathfrak{a} } \nc{\pg}{\mathfrak{p} }
\nc{\sg}{\mathfrak{s} } \nc{\affg}{\mathfrak{aff} } \nc{\qg}{\mathfrak{q} } \nc{\lgo}{\mathfrak{l} } \nc{\holg}{\mathfrak{hol} } \nc{\stabg}{\mathfrak{stab} } \nc{\iautg}{\mathfrak{iaut} } \nc{\bautg}{\mathfrak{baut} } \nc{\hiautg}{\mathfrak{hiaut} }

\nc{\pca}{\mathcal{P}} \nc{\nca}{\mathcal{N}} \nc{\lca}{\mathcal{L}}
\nc{\oca}{\mathcal{O}} \nc{\mca}{\mathcal{M}} \nc{\tca}{\mathcal{T}}
\nc{\aca}{\mathcal{A}} \nc{\cca}{\mathcal{C}} \nc{\gca}{\mathcal{G}}
\nc{\sca}{\mathcal{S}} \nc{\hca}{\mathcal{H}} \nc{\bca}{\mathcal{B}}
\nc{\dca}{\mathcal{D}} \nc{\eca}{\mathcal{E}} \nc{\wca}{\mathcal{W}} \nc{\ica}{\mathcal{I}} \nc{\kca}{\mathcal{K}} \nc{\acca}{\mathcal{AC}}

\nc{\vp}{\varphi} \nc{\ddt}{\tfrac{d}{dt}} \nc{\dsdt}{\tfrac{d^2}{dt^2}} \nc{\dds}{\frac{d}{ds}}
\nc{\dpar}{\frac{\partial}{\partial t}} \nc{\im}{\mathrm{i}} \nc{\dpars}{\tfrac{\partial}{\partial t}} 

\nc{\SO}{\mathrm{SO}} \nc{\Spe}{\mathrm{Sp}} \nc{\Sl}{\mathrm{SL}}
\nc{\SU}{\mathrm{SU}} \nc{\Or}{\mathrm{O}} \nc{\U}{\mathrm{U}} \nc{\Gl}{\mathrm{GL}}
\nc{\Se}{\mathrm{S}} \nc{\Cl}{\mathrm{Cl}} \nc{\Spin}{\mathrm{Spin}}
\nc{\Pin}{\mathrm{Pin}} \nc{\G}{\mathrm{GL}_n(\RR)} \nc{\g}{\mathfrak{gl}_n(\RR)}
\nc{\Eg}{\mathrm{E}} \nc{\Fg}{\mathrm{F}} \nc{\Gg}{\mathrm{G}} \nc{\Gr}{\mathrm{Grass}}

\nc{\RR}{{\Bbb R}} \nc{\HH}{{\Bbb H}} \nc{\CC}{{\Bbb C}} \nc{\ZZ}{{\Bbb Z}}
\nc{\FF}{{\Bbb F}} \nc{\NN}{{\Bbb N}} \nc{\QQ}{{\Bbb Q}} \nc{\PP}{{\Bbb P}} \nc{\OO}{{\Bbb O}}

\nc{\vs}{\vspace{.3cm}} \nc{\vsp}{\vspace{1cm}} \nc{\ip}{\langle\cdot,\cdot\rangle}
\nc{\ipp}{(\cdot,\cdot)} \nc{\la}{\langle} \nc{\ra}{\rangle} \nc{\unm}{\tfrac{1}{2}}
\nc{\unc}{\tfrac{1}{4}} \nc{\und}{\frac{1}{16}} \nc{\no}{\vs\noindent}
\nc{\lam}{\rho^2(\RR^n)^*\otimes\RR^n} \nc{\tangz}{{\rm T}^{\rm Zar}}
\nc{\nor}{{\sf n}}  \nc{\mum}{/\!\!/} \nc{\kir}{/\!\!/\!\!/}
\nc{\Ri}{\tfrac{4\Ric_{\mu}}{||\mu||^2}} \nc{\ds}{\displaystyle}
\nc{\ben}{\begin{enumerate}} \nc{\een}{\end{enumerate}} \nc{\f}{\frac}
\nc{\lb}{[\cdot,\cdot]} \nc{\isn}{\tfrac{1}{||v||^2}}
\nc{\gkp}{(\ggo=\kg\oplus\pg,\ip)} \nc{\ukh}{(\ug=\kg\oplus\hg,\ip)}
\nc{\tgkp}{(\tilde{\ggo}=\kg\oplus\pg,\ip)}
\nc{\wt}{\widetilde}
\nc{\iop}{\mathtt{i}} \nc{\jop}{\mathtt{j}} 
\nc{\Hk}{H_{\kil}} \nc{\gk}{g_{\kil}}

\nc{\Hess}{\operatorname{Hess}} \nc{\ad}{\operatorname{ad}}
\nc{\Ad}{\operatorname{Ad}} \nc{\rank}{\operatorname{rk}}
\nc{\Irr}{\operatorname{Irr}} \nc{\End}{\operatorname{End}}
\nc{\Aut}{\operatorname{Aut}} \nc{\Inn}{\operatorname{Inn}}
\nc{\Der}{\operatorname{Der}} \nc{\Ker}{\operatorname{Ker}}
\nc{\Iso}{\operatorname{Isom}} \nc{\Diff}{\operatorname{Diff}}
\nc{\Lie}{\operatorname{L}} \nc{\tr}{\operatorname{tr}} \nc{\dif}{\operatorname{d}}
\nc{\sen}{\operatorname{sen}} \nc{\modu}{\operatorname{mod}}
\nc{\CRic}{\operatorname{PP}} \nc{\Cric}{\operatorname{P}} \nc{\Ricci}{\operatorname{Ric}}
\nc{\sym}{\operatorname{sym}} \nc{\herm}{\operatorname{herm}} \nc{\symac}{\operatorname{sym^{ac}}}
\nc{\symc}{\operatorname{sym^{c}}} \nc{\scalar}{\operatorname{Sc}}
\nc{\grad}{\operatorname{grad}} \nc{\ricci}{\operatorname{Rc}} \nc{\kil}{\operatorname{B}} \nc{\cas}{\operatorname{C}} \nc{\lic}{\operatorname{L}}
\nc{\Nor}{\operatorname{Norm}}  \nc{\ricc}{\operatorname{Rc^{c}}}
\nc{\Ricc}{\operatorname{Ric^{c}}} \nc{\ricac}{\operatorname{Rc^{ac}}}
\nc{\Ricac}{\operatorname{Ric^{ac}}} \nc{\Riem}{\operatorname{Rm}} \nc{\Sec}{\operatorname{Sec}}
\nc{\riccig}{\operatorname{ric^{\gamma}}} \nc{\mm}{\operatorname{m}} \nc{\Mm}{\operatorname{M}}
\nc{\Le}{\operatorname{L}} \nc{\tang}{\operatorname{T}}
\nc{\level}{\operatorname{level}} \nc{\rad}{\operatorname{r}}
\nc{\abel}{\operatorname{ab}} \nc{\CH}{\operatorname{CH}} \nc{\Cone}{{\mathcal C}} \nc{\CCone}{\operatorname{CC}} \nc{\CP}{{\mathcal P}}
\nc{\mcc}{\operatorname{mcc}} \nc{\Adj}{\operatorname{Adj}}
\nc{\Order}{\operatorname{O}}  \nc{\inj}{\operatorname{inj}} \nc{\proy}{\operatorname{pr}}
\nc{\vol}{\operatorname{vol}} \nc{\Diag}{\operatorname{Dg}} \nc{\Diagg}{\operatorname{Diag}}
\nc{\Spec}{\operatorname{Spec}} \nc{\Ima}{\operatorname{Im}} \nc{\Rea}{\operatorname{Re}}
\nc{\spann}{\operatorname{span}} \nc{\Aff}{\operatorname{Aff}} \nc{\E}{\operatorname{E}} \nc{\id}{\operatorname{id}} \nc{\dete}{\operatorname{det}} \nc{\Crit}{\operatorname{Crit}} \nc{\val}{\operatorname{val}} \nc{\Bihol}{\operatorname{Bihol}} 

\theoremstyle{plain}
\newtheorem{theorem}{Theorem}[section]
\newtheorem{proposition}[theorem]{Proposition}
\newtheorem{corollary}[theorem]{Corollary}
\newtheorem{lemma}[theorem]{Lemma}

\theoremstyle{definition}
\newtheorem{definition}[theorem]{Definition}

\newtheorem{conjecture}[theorem]{Conjecture}

\theoremstyle{remark}
\newtheorem{remark}[theorem]{Remark}
\newtheorem{example}[theorem]{Example}

\title{Compact homogeneous complex manifolds}

\author{Jorge Lauret}  

\address{FaMAF, Universidad Nacional de C\'ordoba and CIEM, CONICET (Argentina)}
\email{jorgelauret@unc.edu.ar} 

\thanks{This research was partially supported by a grant from Universidad Nacional de C\'ordoba (Argentina).  We would also like to acknowledge support from the ICTP through the Associates Programme and from the Simons Foundation through grant number 284558FY19}

\date{\today}

\begin{document}

\maketitle

\begin{abstract}
We study both the complex and Hermitian geometry of C-spaces, i.e., compact homogeneous spaces $M=G/K$ admitting a $G$-invariant complex structure.  Some results on Hodge, Bott-Chern and Aeppli cohomologies are given.  We also prove classification results in the pluriclosed, parallel Bismut torsion and locally conformally K\"ahler (or Vaisman) cases, respectively, as well as existence results for balanced and Calabi-Yau with torsion metrics.   
\end{abstract}

\tableofcontents

% 11/9/2026

\section{Introduction}\label{intro-sec}  

\epigraph{{\it Depu\'es del agua, m\'as agua, porque este charco no tiene borde.}}{Carlito. {\sc Caballeros de la Quema.}}

We study compact complex manifolds $M$ which are {\it homogeneous} in the following sense: there exists a compact Lie group of biholomorphisms acting transitively on $M$, which is not unique in general.  Each of such groups $G$ provides a presentation $M=G/K$ as a homogeneous space that admits a $G$-invariant complex structure, called a {\it C-space}.  The following characterization of a C-space $M=G/K$ was obtained by Wang \cite{Wng}, Samelson \cite{Sml} and Tits \cite{Tts} in the 50s (see also \cite{NiWll} for a recent approach): 
\begin{quote}
There exists a Lie subalgebra $\hg\subset\ggo$ such that $\hg$ is the centralizer of some abelian subalgebra of $\ggo$, $\kg\subset\hg$ and $[\hg,\hg]=[\kg,\kg]$,   
\end{quote}
where $\ggo$ and $\kg$ are the Lie algebras of $G$ and $K$, respectively.  Each of these subalgebras $\hg\subset\ggo$ determines a flag manifold $F=G/H$ and the following torus fibration over $F$, called a {\it Tits fibration}:  
\begin{equation}\label{fib-intro}
A=H/K\longrightarrow M=G/K \longrightarrow F=G/H, 
\end{equation}
with fiber the torus $A=H/K$.  For instance, any compact semisimple Lie group $M=G$ is a C-space which fibers over the full flag $F=G/T$ with fiber $T$, for any maximal torus $T\subset G$.   

The aim of this paper is to undertake a study of both the complex and Hermitian geometry of C-spaces.  The main original results obtained include: 
\begin{enumerate}[{\small $\bullet$}] 
\item formulas for many Betti, Hodge, Bott-Chern and Aeppli numbers, 

\item formulas for all the canonical connections and their Ricci forms, 

\item characterizations of some of the main distinguished Hermitian structures in the subject: balanced, pluriclosed, parallel Bismut torsion, Calabi-Yau with torsion and locally conformally K\"ahler,  

\item the classification of pluriclosed complex C-spaces and pluriclosed metrics, as well as the global stability of pluriclosed flow on compact Lie groups,  

\item a partial classification of Hermitian metrics with parallel Bismut torsion, 

\item an alternative proof of the classification of locally conformally K\"ahler (or Vaisman) C-spaces given in \cite{HsgKms1}.   
\end{enumerate}
All the formulas and characterizations are in Lie theoretical terms, being of particular relevance the set of complementary roots of the flag $F$ and its combinatorics.  A number of other applications to the existence problem for the above concepts and their interplay are also obtained, including plenty of examples.

\subsection{Structure}
Given a C-space $M=G/K$ and one of its Tits fibrations $A\rightarrow M\rightarrow F$ as in \eqref{fib-intro}, we consider the following key $Q$-orthogonal decompositions relative to a background bi-invariant inner product $Q$ on $\ggo$: 
\begin{equation}\label{reddec-intro}
\ggo= \kg\oplus\pg = 
\underbrace{[\hg,\hg]\oplus\rlap{$\overbrace{\phantom{\zg(\kg) \oplus\ag}}^{\zg(\hg)}$}\zg(\kg)}_\kg \oplus \underbrace{\ag\oplus\qg}_\pg,  \qquad \pg=\ag\oplus\qg\equiv T_oM,  
\end{equation}
where $[\ggo,\ggo]= (\hg\cap[\ggo,\ggo])\oplus\qg$ is the reductive decomposition of the flag $F$.  Recall that in the homogeneous setting, the tangent space at the origin $o\in M$, $T_oM\equiv\pg$, is in a sense all the tangent spaces of $M$ at the same time.  Note that $T_oA\equiv\ag$ and $T_oF\equiv\qg$.     

Summarizing, any C-space $M=G/K$ is determined by only the flag $F=G/H$ and a compactly embedded subspace 
\begin{equation}\label{slope-intro}
\zg(\kg)\subset\zg(\hg), \qquad\mbox{giving rise to}\qquad \zg(\hg)=\zg(\kg)\oplus\ag.  
\end{equation}

\begin{remark}\label{lh-rem-intro}
By including locally homogeneous spaces, the set of all C-spaces $M=G/K$ of a given dimension which fibers over a fixed flag manifold $F=G/H$ can therefore be endowed with the Grassmannian topology, giving rise to a natural notion of {\it generic} for a given condition on these homogeneous fibrations (see \cite{KrsWlfZar} and Remark \ref{lh-rem}).  
\end{remark}

The root structure of the flag $F=G_f/H_f$, where $G_f:=[G,G]$ and $H_f:=H\cap[G,G]$, will be essential in the study of the geometry of the C-space.  For any maximal torus $T\subset H_f\subset G_f$, we consider $\Delta\subset(\tg^c)^*$, the root system defined by the Cartan subalgebra $\tg^c:=\tg\otimes\CC$ of the complex semisimple Lie algebra $\ggo_f^c$, where $\ggo_f=[\ggo,\ggo]$ (see \S\ref{roots} and \S\ref{flag-sec} for overviews on roots and flag manifolds, respectively).  The $\kil_{\ggo_f}$-orthogonal reductive decomposition 
$\ggo_f=\hg_f\oplus\qg$ (i.e., $T_oF\equiv \qg$) of $F=G_f/H_f$, where $\kil_{\ggo_f}$ is the Killing form of $\ggo_f$, gives the decompositions  
\begin{equation}\label{hq-intro}
\hg_f^c=\tg^c\oplus\bigoplus_{\alpha\in\Delta_{\hg_f}}\ggo_\alpha, \qquad 
\qg^c=\bigoplus_{\alpha\in\Delta_\qg}\ggo_\alpha, \qquad \Delta=\Delta_{\hg_f}\sqcup\Delta_\qg,
\end{equation}
where $\ggo_\alpha:=\{ E\in\ggo_f^c:[H,E]=\alpha(H)E,\; H\in\tg^c\}$, $\Delta_{\hg_f}:=\{\alpha\in\Delta:\alpha|_{\zg(\hg_f)}=0\}$ is the root system of $[\hg_f,\hg_f]$ and $\Delta_\qg$ is the set of {\it complementary roots}.  There exist vectors $E_\alpha$, $\alpha\in\Delta$ such that 
$$
\ggo_\alpha=\CC E_\alpha, \quad \kil_{\ggo^c}(E_\alpha,E_{-\alpha})=1, \quad H_\alpha=[E_\alpha,E_{-\alpha}], 
\quad 
[E_\alpha,E_\beta]=N_{\alpha,\beta}E_{\alpha+\beta}, 
$$
$N_{\alpha,\beta}\in\RR$,  where $H_\alpha\in\tg^c$ is defined by $\alpha=\kil_{\ggo_f^c}(\cdot,H_\alpha)$.

\begin{remark}
For simplicity, we will assume in most of the paper that $\zg(\ggo)\subset\ag$.  This can always be obtained by considering a new presentation of the C-space for a smaller transitive group as follows:
\begin{equation}\label{decintro1}
M=G/K=G_f/K\times T^z, \qquad z:=\dim{\zg(\ggo)}, \qquad  K\subset H_f\subset G_f=[G,G], 
\end{equation}
where $G=G_f\times T^z$, $G_f/K$ is also a C-space (possibly odd dimensional, by using the subalgebra $\hg\cap[\ggo,\ggo]$) and $\ag=(\ag\cap[\ggo,\ggo])\oplus\zg(\ggo)$. 
\end{remark}

In addition to \eqref{decintro1}, there is a finer decomposition of a C-space as a manifold, given by
\begin{equation}\label{decintro2}
M=G/K=\widetilde{G}/\widetilde{H}\times \underline{G}/\underline{K}\times\overline{G}, 
\end{equation}
where $\widetilde{G}/\widetilde{H}$ is a flag manifold, $\underline{G}/\underline{K}= [\underline{G},\underline{G}]/\underline{K}\times T^z$ is a {\it strict} C-space (see Definition \ref{strict-def}) and $\overline{G}$ is a compact semisimple Lie group (see \S\ref{chcm-sec} for a more detailed treatment).   

We refer to \cite{Grn, HsgKms1, Pds, FinGrnVzz, ChrSkn, Brb, PdsZhn, BdlMrc, BrbPdc} for previous studies of C-spaces (beyond Lie groups) in the literature, most of them carried out in the finite fundamental group case (i.e., $G$ semisimple).

\subsection{Complex structures and cohomology}
It is not hard to show that any even dimensional C-space $M=G/K$ indeed admits $G$-invariant complex structures.  Given one of the Tits fibrations $A\rightarrow M\rightarrow F$ as in \eqref{fib-intro}, say, attached to a subalgebra $\hg\subset\ggo$, recall from \eqref{reddec-intro} and \eqref{hq-intro} the $Q$-orthogonal decompositions $\zg(\hg)=\zg(\kg)\oplus\ag$, 
$$
\ggo= \kg\oplus\pg, \qquad \pg= \ag\oplus\qg,   \qquad 
\qg^c=\sum_{\alpha\in\Delta_\qg}\CC E_\alpha.     
$$
We consider the $\Ad(K)$-invariant complex structures on $\pg\equiv T_0M$ (see \S\ref{homsp-sec} for an overview of invariant geometric structures on homogeneous spaces) given by 
\begin{equation}\label{J-intro}
J=(J_\ag,J_\qg)=\left[\begin{matrix} J_\ag&0\\0&J_\qg\end{matrix}\right], 
\end{equation}
where $J_\ag:\ag\rightarrow\ag$ is {\it any} linear map such that $J_\ag^2=-I$ (in particular, $\dim{\ag}$ is even) and $J_\qg:\qg\rightarrow\qg$ is the complex structure on the flag $F$ attached to some invariant ordering $\Delta_\qg^+$ of $\Delta_\qg$, i.e., $J_\qg E_\alpha = \pm\im E_\alpha$ for all $\alpha\in\pm\Delta_\qg^+$.  Since 
$$
\pg^{1,0}=\ag^{1,0}\oplus \bigoplus_{\alpha\in\Delta_\qg^+}\CC E_\alpha, 
$$
the integrability of $J$ follows from the fact that $(\Delta_\qg^++\Delta_\qg^+)\cap\Delta_\qg\subset\Delta_\qg^+$.  

The Tits fibration therefore becomes a holomorphic fibration: 
\begin{equation}\label{holfib-intro}
(A,J_\ag)\longrightarrow (M=G/K,J) \longrightarrow (F,J_\qg),   
\end{equation}
and according to \cite[Theorem 1.3]{NiWll}, every $G$-invariant complex structure on a C-space $M=G/K$ is of the form $J=(J_\ag,J_\qg)$ for some of its Tits fibrations.  It is easy to see that, up to biholomorphism, it is enough to consider only one of the Tits fibrations.  Note that this can not be assumed if one wishes to study hypercomplex structures (see Remark \ref{hyper-rem}).   

The well-known vector from Lie theory given by
\begin{equation}\label{koszul2-intro}
Z_{J_\qg}:= \sum_{\alpha\in\Delta_\qg^+} \im H_\alpha\in\zg(\hg)=\zg(\kg)\oplus\ag, 
\end{equation}
plays a very important role in the geometry of a complex C-space $(M=G/K,J)$.  We call $Z_{J_\qg}$ the {\it Koszul vector} of $J$.  For instance, it follows from \cite{Ksz} (see also \cite{AlkPrl,Grn}) that the first Chern class $c_1(M,J)= [\sigma]\in H^2(M)$ vanishes if and only if 
\begin{equation}\label{c10-intro} 
Z_{J_\qg}\in\ag, 
\end{equation}
where the closed $2$-form $\sigma\in\Omega^2(M)^G \equiv(\Lambda^2\pg^*)^K$ is defined by 
$$
\sigma(X,Y) := 2Q([X,Y],Z_{J_\qg}), \qquad\forall X,Y\in\pg. 
$$  
In particular, if $K$ is either semisimple or trivial, then $c_1(M,J)=0$ for any $G$-invariant complex structure $J$ on $M=G/K$.  However, the first Chern class is generically (in the sense of Remark \ref{lh-rem-intro}) nonzero for any invariant complex structure.    

In \S\ref{cohom-sec}, we study the Dolbeault, Bott-Chern and Aeppli cohomologies of a C-space $M=G/K=G_f/K\times T^z$ as in \eqref{decintro1} endowed with a complex structure $J=(J_\ag,J_\qg)$ as in \eqref{J-intro}.  We use the Tanr\'e model \cite[Proposition 8]{Tnr} for the cohomology of holomorphic fibrations: if $\pg=\ag\oplus\qg$ and $\ggo_f=\hg_f\oplus\qg$ are respectively the reductive decompositions of the total space $M=G/K$ and the base $F=G_f/H_f$, then there is an isomorphism of double complexes,    
$$
\left(\Lambda^{\cdot,\cdot},\partial,\overline{\partial}\right) \simeq 
\left(\Lambda^{\cdot,\cdot}_T:=\Lambda^{\cdot,\cdot}_\ag\otimes\Lambda^{\cdot,\cdot}_\qg,\partial,\overline{\partial}\right), 
$$
where $\ag=\ag^{1,0}\oplus\ag^{0,1}$ is determined by $J_\ag$, 
$$
\Lambda^{p,q}_\ag:=\Lambda^p (\ag^{1,0})^*\otimes\Lambda^q (\ag^{0,1})^*,  \qquad  \Lambda^{\cdot,\cdot}_\qg \simeq \left((\Omega^{\cdot,\cdot} F)^{G_f},\partial_\qg,\overline{\partial_\qg}\right), 
$$  
and 
$$
\left\{\begin{array}{l}
\partial|_{\Lambda^{1,0}_\ag}=0, \\ 
\partial|_{\Lambda^{0,1}_\ag}=d:\Lambda^{0,1}_\ag\rightarrow\Lambda^{1,1}_\qg,  
\end{array}\right. \qquad 
\left\{\begin{array}{l}
\overline{\partial}|_{\Lambda^{1,0}_\ag}=d:\Lambda^{1,0}_\ag\rightarrow\Lambda^{1,1}_\qg, \\  
\overline{\partial}|_{\Lambda^{0,1}_\ag}=0.  
\end{array}\right.
$$
We obtain explicit formulas for many low Hodge, Bott-Chern and Aeppli numbers, including $h^{p,q}$, $p+q\leq 3$ in \S\ref{D-sec}, $h^{p,q}_{BC}$, $p,q\leq2$ in \S\ref{BC-sec} and $h^{p,q}_A$, $p,q\leq1$ in \S\ref{A-sec}.  Excepting $h^{1,1}_A$, the only quantities involved in the above formulas which are sensitive to the complex structure are 
$$
\dim{\zg(\ggo)^c\cap\ag^{1,0}}, \qquad \dim{\ag_f^{1,0}}, \qquad \dim{\ag_f^c\cap\ag^{0,1}}, 
$$
where $\ag=\zg(\ggo)\oplus\ag_f$ and $\ag_f^{1,0}:=\{ A^{1,0}:=\unm(A-\im J_\ag A):A\in\ag_f\}$.  Note that these three numbers depend on $J_\ag$ only if $\zg(\ggo)\ne 0$.  The following applications follow:
\begin{enumerate}[{\small $\bullet$}] 
\item Using that $h^{1,0}=\dim{\zg(\ggo)^c\cap\ag^{1,0}}$ and $h^{0,1}=\unm\dim{\ag}$, we obtain that a non-K\"ahler complex C-space never satisfies the $\partial\overline{\partial}$-Lemma, as this would imply that $\zg(\ggo)=\ag$, i.e., $M=F\times T^z$ is K\"ahler.  This was proved in \cite[Theorem 1]{Pds} for $G$ semisimple.   

\item The first Bott-Chern class (see \cite{Brb,Ist}), defined by
$$
c_1^{BC}(M,J):=[\sigma]\in H_{BC}^{1,1}(M), \qquad\mbox{where} \quad c_1(M,J)=[\sigma]\in H^2(M),  
$$ 
is non-zero if and only if $Z_{J_\qg}\ne 0$ (see \eqref{koszul2-intro}).  This implies that a complex C-space can never have holomorphically trivial canonical bundle, unless it is a torus, and can never admit a Chern-Ricci flat metric, not even non-invariant ones.  
\end{enumerate}
Our main result on cohomology is a full description of $\Ker\partial\overline{\partial}|_{\Lambda_T^{1,1}}$, where $\Lambda_T^{1,1}=\Lambda_\ag^{1,1}\oplus\Lambda_\qg^{1,1}$.  If $\ggo_f=\ggo_1\oplus\dots\oplus\ggo_s$ and $\hg_f=\hg_1\oplus\dots\oplus\hg_s$ are the decompositions in simple ideals and $\ggo_i=\hg_i\oplus\qg_i$ the corresponding reductive decompositions of the flags $G_i/H_i$ (i.e., $F=G_1/H_1\times\dots\times G_s/H_s$), then 
\begin{equation}\label{decp-intro}
\pg=\ag\oplus\qg_1\oplus\dots\oplus\qg_s, \qquad \qg= \qg_1\oplus\dots\oplus\qg_s,  
\end{equation}
and $\Delta_\qg^+=\Delta_{\qg_1}^+\sqcup\dots\sqcup\Delta_{\qg_s}^+$.  The decomposition \eqref{decintro2} naturally determines the following: after reordering, we can assume that for $1\leq t\leq u\leq s$, 
$$
\widetilde{\ggo}=\ggo_1\oplus\dots\oplus\ggo_t, \quad \underline{\ggo}=\ggo_{t+1}\oplus\dots\oplus\ggo_u, \quad 
\overline{\ggo}=\ggo_{u+1}\oplus\dots\oplus\ggo_s,  \quad \mbox{so}\quad  
\ag=\underline{\ag}_f\oplus\zg(\ggo)\oplus\overline{\ag}, 
$$ 
where $\overline{\ag}=\tg_{u+1}\oplus\dots\oplus\tg_s$ is a maximal torus of $\overline{\ggo}$.

We identify any $\sigma_\ag\in\Lambda_\ag^{1,1}$ with the symmetric bilinear form $h=h(\sigma_\ag):\ag\times\ag\rightarrow\RR$ given by $h(\cdot,\cdot):=\sigma(\cdot,J_\ag \cdot)$, which is compatible with $J_\ag$.  On the other hand, note that the only possibly nonzero components of a form $\sigma_\qg\in\Lambda_\qg^{1,1}$ are $\sigma_\qg(E_\alpha,E_{-\alpha})=-\im x_\alpha$, $x_\alpha\in\RR$, for all $\alpha\in\Delta_\qg^+$.  

\begin{theorem}\label{h11A-intro}
$\sigma=\sigma_\ag+\sigma_\qg\in\Ker\partial\overline{\partial}|_{\Lambda_T^{1,1}}$ if and only if 
\begin{enumerate}[{\rm (i)}] 
\item $h|_{\underline{\ag}_f\times\underline{\ag}_f}=0$, where $h=h(\sigma_\ag)$.  

\item $h|_{\overline{\ag}\times\overline{\ag}}=-(z_{u+1}\kil_{\ggo_{u+1}}+\dots+z_s\kil_{\ggo_s})|_{\overline{\ag}\times\overline{\ag}}$, for some $z_{u+1},\dots,z_s\in\RR$. 

\item $x_{\alpha+\beta}=x_\alpha+x_\beta-z_i$, for all $\alpha,\beta,\alpha+\beta\in\Delta_{\qg_i}^+$ and $i=1,\dots,s$, where we set $z_1=\dots=z_u=0$.  
\end{enumerate}
\end{theorem}   

As a first application, we obtain that $h^{1,1}_A= \dim{\zg(\kg)} + \dim{S}$, where $S$ is the subspace of all symmetric bilinear forms $h:\ag\times\ag\rightarrow\RR$ such that $h$ is compatible with $J_\ag$ and satisfies conditions (i) and (ii).  Since $\dim{S}$ is precisely the number of irreducible factors of $(G,J)$ as a complex manifold if $M=G=\overline{G}$, this generalizes the formula for $h^{1,1}_A$ given in \cite{Brb} for semisimple Lie groups.  Three other important applications of Theorem \ref{h11A-intro} to the study of pluriclosed manifolds will be given below in \S\ref{pluri-intro-sec}.

\subsection{Hermitian geometry} 
We start in \S\ref{chhm-sec} the study of the geometry of $G$-invariant Hermitian structures on a C-space $M=G/K=G_f/K\times T^z$ as in \eqref{decintro1}.  We consider $\Ad(K)$-invariant Hermitian structures on $\pg\equiv T_oM$ given by 
\begin{equation}\label{Jg-intro} 
(J,g), \qquad J=(J_\ag,J_\qg), \qquad 
g=g_\ag+g_\qg, \qquad g_\qg=(x_\alpha)_{\alpha\in\Delta_\qg^+}, 
\end{equation}
where $J$ is as in \eqref{J-intro}, $g_\ag$ is any inner product on $\ag$ such that $g_\ag(J_\ag\cdot,J_\ag\cdot)=g_\ag(\cdot,\cdot)$ and 
$$
g_\qg=\sum_{\alpha\in\Delta_\qg^+} x_\alpha (-\kil_{\ggo_f})|_{\qg_\alpha}, \qquad x_\alpha>0, \qquad \qg_\alpha^c:=\CC E_\alpha\oplus\CC E_{-\alpha}.   
$$

\begin{remark}
Generically (see Remark \ref{lh-rem-intro}), these are all $G$-invariant Hermitian structures up to holomorphic isometry, but for some C-spaces there may exist $G$-invariant Hermitian metrics which do not descend to the base of any Tits fibration.  In any case, a usual symmetrization process produces metrics of the form $g=g_\ag+g_\qg$ by averaging over $H_f$ any $G$-invariant metric and some conditions such as pluriclosed and CYT are invariant under symmetrization (see \cite{FinGrn1} and Remark \ref{aver}). 
\end{remark}

The K\"ahler form $\omega=g(J\cdot,\cdot)$ is therefore given by $\omega=\omega_\ag+\omega_\qg$,
where $\omega_\ag=g_\ag(J_\ag\cdot,\cdot)$ and the only nonzero components of $\omega_\qg$ are $\omega_\qg(E_\alpha,E_{-\alpha})=-\im x_\alpha$ for all $\alpha\in\Delta_\qg^+$.  The vector 
\begin{equation}\label{koszul-g-intro}
Z_{J_\qg,g_\qg}:=\sum_{\alpha\in\Delta_\qg^+}\tfrac{1}{x_\alpha}  \im H_\alpha   
\in \zg(\hg_f)=\zg(\kg)\oplus\ag_f,  
\end{equation}
called {\it metric Koszul vector}, plays a crucial role in the geometry of the Hermitian C-space $(M=G/K,J,g)$ (see \cite[Section 3]{Pds}).  

We give in \S\ref{cancon-sec} a formula for all the Hermitian connections $\nabla^t=\tfrac{t+1}{2}\nabla^c+\tfrac{1-t}{2}\nabla^b$, $t\in\RR$, introduced by Gauduchon in \cite{Gdc}, where $\nabla^c:=\nabla^1$ is the {\it Chern} connection and $\nabla^b:=\nabla^{-1}$ is the {\it Bismut} (or {\it Strominger}) connection.  After computing some components of the Chern curvature $R^c$ (see Proposition \ref{Ccurv}), we prove that the Chern Ricci form is given by  
$$
\rho^c(X,Y) = Q([X,Y],Z_{J_\qg}) = - \kil_{\ggo_f}([X,Y],Z_{J_\qg}), \qquad\forall X,Y\in\pg. 
$$
In particular, $\rho^c=0$ if and only if $\zg(\hg_f)=0$, i.e., $M$ is a torus.  Remarkably, $\rho^c$ does not depend on the metric $g$, so the solutions to the Chern Ricci flow $\dpar\omega(t)=-\rho^c(\omega(t))$ on a C-space are all simply given by $\omega(t)=\omega(0)-t\rho^c(\omega(0))$.  

Formulas for all the remaining Ricci forms $\rho^t$ are also obtained (see \S\ref{CC-sec}), from which follows that if a complex C-space admits a metric such that $\rho^{t_0}=0$, then $t_0<1$ and for every $t<1$ there exists a metric such that $\rho^t=0$ (cf.\ \cite{BrdStn}).       

The formulas for the Bismut connection of a Hermitian C-space and its curvature are studied in \S\ref{BiC-sec} and the Bismut holonomy in \S\ref{BiH-sec}, where the following applications are obtained. 

\begin{theorem}\label{Bapp-intro}      
Let $(M=G/K,J,g)$ be a Hermitian C-space.   
\begin{enumerate}[{\rm (i)}] 
\item $(J,g)$ is Bismut flat if and only if $M=G$ and $g$ is {\it almost-bi-invariant}, i.e., $g|_{[\ggo,\ggo]}=g_b$ for some biinvariant metric $g_b$ on $[G,G]$ (cf.\ \cite{WngYngZhn} and see Remark \ref{Bflat-rem}).  

\item The decomposition $\pg=\ag\oplus\qg_1\oplus\dots\oplus\qg_s$ given in \eqref{decp-intro} is $\holg(\nabla^b)$-invariant and $\holg(\nabla^b)|_\ag\equiv 0$ (see \cite{Stc} for the case when $M=G$ semisimple).  Moreover, generically, each $\qg_i$ is $\holg(\nabla^b)$-irreducible if $M=G/K$ is of fibration type (see Definition \ref{fibtype-def}).  
\end{enumerate}
\end{theorem}

\subsection{Special metrics}
On compact complex manifolds which, due to topological obstructions, do not admit K\"ahler metrics, it is unlikely to find canonical Hermitian metrics.  It is not hard in general to define notions weakening the K\"ahler condition, but then the existence and abundance questions arise, as well as the need of examples, without mention the possible interplays between the different concepts.  

We now list the notions studied in this paper and give, for each of them, a few general references where a lot more of information, mathematical physics motivation and literature can be found.  All the references we are aware of for each condition in the compact homogeneous case are also given.   

\begin{enumerate}[{\small $\bullet$}] 
\item {\it Balanced}: $d_g^*\omega=0$ (or $d\omega^{n-1}$=0).  See \cite{Mch, Fin, Brb2}; and \cite{Pds, FinGrnVzz, Kwn} for the compact homogeneous case.  

\item {\it LCB}: $d\theta_L=0$, where $\theta_L=-d_g^*\omega(J\cdot)$ is the Lee form, i.e., {\it locally conformally balanced}.  See \cite{FinTms, Prd}.  

\item {\it Pluriclosed} (or {\it SKT}): $\partial\overline{\partial}\omega=0$.  See \cite{Bsm, Str, Brb2, Fin, FinGrn2, Brn}; and for C-spaces, \cite{Pds, FinGrnVzz, Brb, FinGrn1, SKT-LG, Kwn}.  

\item {\it BTP}: $\nabla^bT^b=0$, i.e., {\it parallel Bismut torsion}.  See \cite{ZhaZhn2} and \cite{PdsZhn, Fr}. 

\item {\it BAS}: $\nabla^bT^b=0$ and $\nabla^bR^b=0$, i.e., {\it Bismut Ambrose-Singer}.  See \cite{NiZhn2} and \cite{BrbPdc}.  

\item {\it CYT}: $\rho^b=0$, i.e., {\it Calabi-Yau with torsion}.  See \cite{GrnGrnPn, Brb2, FinGrn2, Brn}; and in the compact homogeneous context, \cite{Grn, Pds}.  

\item {\it LCK}: $d\omega=\theta\wedge\omega$ and $d\theta=0$, i.e., {\it locally conformally K\"ahler}.  See the recent book \cite{OrnVrb}; and \cite{HsgKms1,HsgKms2, AlkCrtHsgKms, Gn, AlkHsgKms} for C-spaces.  

\item {\it Vaisman}: LCK together with $\nabla^g\theta=0$.  See \cite{OrnVrb, Ist}.  
\end{enumerate}

In what follows, we consider a C-space $M=G/K=G_f/K\times T^z$ as in \eqref{decintro1} endowed with a $G$-invariant Hermitian structure $(J,g)$ as in \eqref{Jg-intro} and describe the main results obtained on each of the above kinds of special metrics.   

\subsubsection{Balanced (see \S\ref{bal-sec})} 
We first prove that $(J,g)$ is balanced if and only if 
\begin{equation}\label{bal-intro} 
Z_{J_\qg,g_\qg}\in\zg(\kg). 
\end{equation} 
In particular, $c_1(M,J)\ne 0$ and $\zg(\kg)\ne 0$ if $(M,J)$ admits a balanced metric.  Conversely, on any C-space $M=G/K$ such that $\zg(\kg)\ne 0$ and $G_f$ is simple of rank $\geq 2$, there exists a $2d^2$-parametric space of $G$-invariant complex structures admitting a balanced metric, where $\dim{\ag}=2d$ (see Theorem \ref{bal2}).

\subsubsection{Pluriclosed (see \S\ref{SKT-sec})}\label{pluri-intro-sec} 
It is worth first noting that if $(J,g)$ is pluriclosed, then $[\omega]\in H^{1,1}_A(M)$ and $[\partial\omega]\in H^{2,1}_{\overline{\partial}}(M)$; moreover, they are both nonzero classes in our setting.  As a second application of Theorem \ref{h11A-intro}, we obtain the following classification result.  

\begin{theorem}\label{pluri-clasif-intro}
A C-space $M=G/K$ admits a pluriclosed $G$-invariant Hermitian structure if and only if it is the product of a flag manifold and a compact Lie group (i.e., the strict factor in \eqref{decintro1} is a torus).  Moreover, the restriction of the structure on the flag is K\"ahler and the restriction on the compact Lie group is also pluriclosed. 
\end{theorem}

\begin{remark}
The above classification was obtained in \cite{FinGrnVzz} in the case when $G$ is semisimple by different methods.  However, we found a counterexample to \cite[Lemma 6.4]{FinGrnVzz} (see Remark \ref{h21-rem} and Example \ref{h21}), which is strongly used in the proof of \cite[Theorem 6.1]{FinGrnVzz}.      
\end{remark}

It follows from \eqref{bal-intro} and Theorem \ref{pluri-clasif-intro} that the Fino-Vezzoni conjecture holds in the compact homogeneous case, i.e., a non-K\"ahler C-space can never admit an invariant balanced metric and an invariant pluriclosed metric at the same time (see \cite{Kwn} for a recent proof of this in a broader homogeneous context).   

If the C-space is a compact Lie group, then 
$$
M=G=T^z\times\overline{G}, \qquad \overline{\ggo}=\ggo_1\oplus\dots\oplus\ggo_s,  \qquad \ag=\zg(\ggo)\oplus\overline{\ag},
$$ 
where $\overline{\ag}$ is a maximal torus of $\overline{\ggo}$.  For any left-invariant complex structure $J=J_\ag+J_\qg$, often called {\it Samelson} structures, we obtain, as a third application of Theorem \ref{h11A-intro}, that a left-invariant metric compatible with $J$ of the form $g=g_\ag+g_\qg$, $g_\qg=(x_\alpha)_{\alpha\in\Delta^+}$ (note that $g$ is in addition $\Ad(T)$-invariant, where $T$ is the maximal torus of $G$ with Lie algebra $\ag$) is pluriclosed if and only if 
\begin{enumerate}[{\rm (a)}] 
\item $g_\ag|_{\overline{\ag}\times\overline{\ag}}=-(z_{1}\kil_{\ggo_{1}}+\dots+z_s\kil_{\ggo_s})|_{\overline{\ag}\times\overline{\ag}}$, for some $z_{1},\dots,z_s>0$, 

\item and for each $i=1,\dots,s$, if $\Pi_{\qg_i}=\{\alpha_{i,1},\dots,\alpha_{i,m_i}\}\subset\Delta_{\qg_i}^+$ are the simple roots, $m_i:=\dim{\zg(\hg_i)}$, then
$$
x_\alpha=z_i+\sum_{j=1}^{m_i} n_{i,j}(x_{\alpha_{i,j}}-1), \qquad \forall \alpha=\sum_{j=1}^{m_i} n_{i,j}\alpha_{i,j} +\sum_{\beta\in\Pi_{\hg_f}}k_{i,\beta}\beta\in\Delta_{\qg_i}^+,   
$$  
$n_{i,j},k_{i,\beta}\in\NN_0$.
\end{enumerate}
This was proved in \cite{SKT-LG} in the semisimple case.  It follows from Theorem \ref{Bapp-intro}, (i) that a pluriclosed metric $g$ as above is Bismut flat if and only if $x_\alpha=z_i$ for all $\alpha\in\Delta_{\qg_i}^+$ and $i=1,\dots,s$.  The space of all pluriclosed metrics of the form $g=g_\ag+g_\qg$ on $M=G$ is therefore either empty (i.e., $S=\{ 0\}$) or it depends on $\rank(\overline{\ggo})+\dim{S}$ parameters.  Note that $g_\ag$ is not necessarily the restriction of a biinvariant metric of $G$ on the maximal torus $\ag$ of $\ggo$.     

As fourth application of Theorem \ref{h11A-intro} gives that any positive definite class in $H^{1,1}_A(M)$ contains a unique Bismut-flat left-invariant metric (see Theorem \ref{Bapp-intro}, (i)).  This implies global stability for the pluriclosed flow, by using the strong general convergence result \cite[Theorem 1.2]{GrcJrdStr} (see   \cite[Theorem 4.3]{Brb} for the case when $G$ is semisimple).  

\begin{theorem}\label{stab-intro}
Let $M=G$ be a compact Lie group endowed with a left-invariant complex structure.  Then for any pluriclosed metric $\omega_0$ (not necessarily left-invariant), the solution $\omega(t)$ to pluriclosed flow with initial data $\omega_0$ exists on $[0,\infty)$ and converges to a Bismut-flat metric $\omega_\infty$.
\end{theorem}

\subsubsection{BTP (see \S\ref{BTP-sec})}
The BTP condition turns out to be quite strong in the compact homogeneous context (see \cite{PdsZhn, Fr}).  However, it follows from part (i) of the theorem below that any complex C-space admits a BTP invariant metric.  After computing many components of $\nabla^b T^b$, we obtain the following partial classification.  

\begin{theorem}\label{BTP-thm-intro}
Let $(M=G/K,J)$ be a complex C-space fibering over a flag $F=G_1/H_1\times\dots\times G_s/H_s$ and assume that $G_i\ne \Eg_6, \Eg_7, \Eg_8$ for all $i=1,\dots,s$.  Then a metric $g=g_\ag+g_\qg$, $g_\qg=(x_\alpha)_{\alpha\in\Delta_\qg^+}$ is BTP if and only if 
\begin{enumerate}[{\rm (i)}]
\item either $g_\qg=g_b|_\qg$ for some bi-invariant metric $g_b$ on $[\ggo,\ggo]$ (i.e., $x_\alpha=x_\beta$ for all $\alpha,\beta\in\Delta_{\qg_i}$, $i=1,\dots,s$, see \eqref{decp-intro}),   

\item or $g_\qg$ is K\"ahler on the flag $F=G_f/H_f$ (i.e., $x_{\alpha+\beta}=x_\alpha+x_\beta$ for all $\alpha,\beta,\alpha+\beta\in\Delta_\qg^+$) and $\tfrac{1}{x_\alpha}A_\alpha= \tfrac{1}{x_\beta}A_\beta$ for all $\alpha,\beta,\alpha+\beta\in\Delta_\qg^+$, where $A_\alpha$ is the $Q$-orthogonal projection of $H_\alpha$ on $\ag^c$.      
\end{enumerate}
\end{theorem}

The hypothesis on the simple factors $G_i$ comes from the partial classification of BTP metrics on flag manifolds obtained in \cite{Fr}.  Using part (ii) of the above theorem and the computation of certain component of $\nabla^bR^b$, we have found counterexamples to the following conjecture proposed in \cite[Conjecture 1.6]{PdsZhn}: on any C-space $M=G/K$ with $G$ semisimple, BTP implies BAS.

\begin{remark}\label{BP-intro}
Recently, it was proved in \cite[Theorem B]{BrbPdc} that a metric on a complex C-space is BAS if and only if it is as in Theorem \ref{BTP-thm-intro}, (i).  
\end{remark}

\subsubsection{CYT (see \S\ref{CYT-sec})}
We first show that $(J,g)$ is CYT if and only if the following two conditions hold:
\begin{enumerate}[{\rm (a)}]
\item $Z_{J_\qg}\in\ag\;$  (i.e., $c_1(M,J)=0$).  

\item $g_\ag\left((Z_{J_\qg,g_\qg})_\ag,A\right)=Q(Z_{J_\qg},A)\;$ for all $A\in\ag_f$, where $\ag=\zg(\ggo)\oplus
\ag_f$.    
\end{enumerate}
In particular, if $c_1(M,J)=0$, then any compatible {\it almost-normal} metric $g=g_\ag+g_\qg$ (i.e., $g|_{\ag_f\oplus\qg}=g_b|_{\ag_f\oplus\qg}$ for some bi-invariant metric $g_b$ on $[G,G]$) is CYT.  

On the other hand, given a C-space $M=G/K$ such that $\dim{\ag_f}\geq 2$ and the set $\cca_0$ of $G$-invariant complex structures with vanishing first Chern class is nonempty, we prove that there is always a $2$-parametric closed subset $\cca_{\nexists} \subset\cca_0$ such that $(M,J)$ does not admit a CYT $G$-invariant metric for any $J\in\cca_{\nexists}$ (see Figure \ref{CYT-fig} in \S\ref{CYT-sec}).

\subsubsection{LCK (see \S\ref{LCK-sec})}\label{LCK-intro-sec}
We provide an alternative proof of the following structure result, which was essentially obtained in \cite[Theorem 1, pp.\ 692]{HsgKms1}  (see also \cite{HsgKms2, AlkCrtHsgKms, Gn, AlkHsgKms}), though stated in a different way.      

\begin{theorem}\label{LCK-intro}
Any Hermitian C-space which is LCK can be constructed according to the following recipe, whose only ingredients are a flag manifold $F=G_f/H_f$, a regular element $W\in\zg(\hg_f)$ (i.e., $\alpha(\im W)\ne 0$ for any $\alpha\in\Delta_\qg$) and real numbers $a,b,c,t$ such that $b\ne 0$, $c,t>0$: 
\begin{enumerate}[{\rm (i)}]
\item The C-space is $M=G/K=G_f/K\times S^1$, where $K$ is defined by  
$$
\zg(\hg_f)=\zg(\kg)\oplus\RR W, \quad \kg=[\hg_f,\hg_f]\oplus\zg(\kg), \quad \zg(\ggo)=\RR Z_0, \quad \ag=\RR Z_0\oplus\RR W, \quad \ag_f=\RR W,
$$
provided that $K\subset H_f$ is closed, and $\left\{Z_0,W\right\}$ is a $Q$-orthonormal basis of $\ag$. 

\item The complex structure is given by $J=(J_\ag,J_\qg)$, where $J_\qg$ is attached to the invariant ordering $\Delta_\qg^+:=\{\alpha\in\Delta_\qg:\alpha(\im W)>0\}$ and $J_\ag$ is defined by   
$$
[J_\ag]_{\left\{Z_0,W\right\}} = 
\left[\begin{matrix} a&-\tfrac{a^2+1}{b} \\ b& -a\end{matrix}\right].  
$$ 
\item As a compatible metric $g=g_\ag+g_\qg$, consider 
$$
g_\qg:=t g_W, \quad g_W=(x_\alpha)_{\alpha\in\Delta_\qg^+}, \quad x_\alpha:=\alpha(\im W)>0; \quad     
[g_{\ag}]_{\{ Z_0,W\}} = c\left[\begin{matrix} b^2&-ab \\ -ab& a^2+1\end{matrix}\right].
$$
\end{enumerate}
\end{theorem}

The following properties all easily follow from Theorem \ref{LCK-intro}: 
\begin{enumerate}[{\small $\bullet$}] 
\item $d\omega=\theta\wedge\omega$, where  $\theta(X):=-Q\left(X,bcZ_0\right)$ for all $X\in\pg$.   

\item The space of all LCK metrics on a given $(M=G/K,J)$ depends on $\dim{\zg(\hg_f)}+2$ parameters.  

\item $g_W$ is a K\"ahler metric on the complex flag manifold $(F,J_\qg)$, that is, the base of the holomorphic Tits fibration $T^2 \longrightarrow M \longrightarrow F$. 

\item The topology of the LCK manifold may be very sensitive to the choice of $W$, e.g., if $G_f/K$ is an Aloff-Wallach space (see Example \ref{AW}).  

\item Any LCK complex C-space has $b_1(M)=1$ (see \cite[Claim 42.11]{OrnVrb} for a geometric alternative proof) and $b_2(M)=\dim{\zg(\kg)}=\dim{\zg(\hg_f)}-1$.  

\item Any LCK Hermitian C-space is Vaisman (see \cite{GdcMrnOrn} and \cite[Theorem 42.7]{OrnVrb} for geometric alternative proofs). 

\item An LCK complex C-space can never be balanced, and it is LCB if and only if $a=0$.   

\item $c_1(M,J)=0$ if and only if the K\"ahler metric $g_W$ on $F$ is Einstein (i.e., $g_W$ is the unique K\"ahler-Einstein metric on $(F,J_\qg)$ up to scaling, or equivalently, $W=-tZ_{J_\qg}$ for some $t>0$).  
 
\item If $c_1(M,J)=0$, then the metric $g$ is CYT if and only if  
$
\tfrac{t}{c} = (a^2+1)|\Delta_\qg^+| |Z_{J_\qg}|.   
$        
\end{enumerate}

It follows from the formulas given in \S\ref{cohom-sec} that for any LCK C-space, 
$$
h^{1,0}=0, \quad h^{0,1}=1, \qquad b_1=1, 
$$
$$
h^{2,0}=0, \quad h^{1,1}=\dim{\zg(\kg)}, \quad h^{0,2}=0, \qquad b_2=\dim{\zg(\kg)}, 
$$
$$
h^{3,0}=0, \quad h^{2,1}\geq 1, 
\quad
h^{1,2}=\dim{\zg(\kg)}+1,  \quad h^{0,3}=0, \qquad b_3=b_3(G_f/K)+\dim{\zg(\kg)},
$$
$$
h^{1,1}_{BC}= \dim{\zg(\kg)}+1,
\qquad h^{2,1}_{BC}=0,\qquad 
h^{1,1}_A=\dim{\zg(\kg)}. 
$$

\subsection{Landscape}
Given a C-space $M=G/K$ and a Tits fibration $M\rightarrow F$ over a flag manifold $F=G_f/H_f$, let $\cca$ denote the space of all $G$-invariant complex structures making the fibration holomorphic (see \eqref{J-intro} and \eqref{holfib-intro}).  Any $G$-invariant complex structure on $M=G/K$ is biholomorphic to one in $\cca$.  We denote by $\cca_0\subset\cca$ the subspace of those structures having zero first Chern class, and by $\cca_{bal}$, $\cca_{cyt}$, etc those admitting a balanced metric, a CYT metric, etc, respectively.  

According to \eqref{c10-intro} and \eqref{bal-intro}, 
$$
\cca=\cca_0\sqcup\cca_{bal}\sqcup\cca_{?}, 
$$
and if $\{J_1,\dots,J_k\}$ is the set of all $G_f$-invariant complex structures on the flag manifold $F=G_f/H_f$, then each of the three above components is the disjoint union of subsets of the form 
$$
\cca_i:=\left\{ J=(J_\ag,J_i): J_\ag^2=-I\right\}\equiv \Gl_{2d}(\RR)/\Gl_d(\CC),  \qquad \dim{\cca_i}=2d^2, \quad\dim{\ag}=d.  
$$ 
If $\zg(\kg)\ne 0$, then $\cca_0=\emptyset$ generically among the set of all C-spaces fibering over a fixed flag $F$ (see Remark \ref{lh-rem-intro}) and the region $\cca_?$ is nonempty in general.  Note that $\cca_0=\cca$ if $\zg(\kg)=0$.  On the other hand, $\cca_{bal}=\emptyset$ if $\zg(\kg)=0$ and $\cca_{bal}\ne\emptyset$ if $\zg(\kg)\ne 0$ and $\rank(\ggo_i)\geq 2$ for all $i=1,\dots,s$.  

The subspace $\cca_{cyt}\subset\cca_0$ is much harder to unravel.  It is always nonempty if $\cca_0\ne\emptyset$ and it intersects all the components $\cca_i$ of $\cca_0$.  If $\dim{\ag_f}\geq 2$ then $\cca_{cyt}\subsetneq\cca_0$.  We can give a full description of $\cca_{cyt}$ when $G$ is semisimple and $\dim{\ag}=2$ (see Figure \ref{CYT-fig} in \S\ref{CYT-sec}).  

On an LCK C-space as in \S\ref{LCK-intro-sec},  
$$
\cca=\cca_{lck}\sqcup\cca_{bal}\sqcup\cca_{?}, 
$$
where $\cca_{lck}=\cca_i$ if the regular element $W$ determines $J_i$.  Moreover, $\cca_0=\cca_{lck}$ if in addition $W\in\RR_{<0}Z_{J_i}$ and $\cca_0=\emptyset$ otherwise.  

For a Lie group $M^{2n}=G$, the space $\cca=\cca_0$ consists of all left-invariant complex structures which are also $T$-right-invariant for some fixed maximal torus $T$ of $G$.  If $n\geq 3$, then  
$$
\cca=\cca_{cyt}\sqcup\cca_{??}, \qquad \cca_{skt}=\cca_{bf}\subsetneq\cca_{cyt} \subsetneq\cca,   
$$
where $\cca_{bf}\subset\cca$ is the subspace of those complex structures admitting a Bismut flat metric.

\vspace{1cm}\noindent 
{\it Acknowledgements}.  I thank Daniele Angella, Beatrice Brienza, Gueo Grantcharov, Ramiro Lafuente, Emilio Lauret, Facundo Montedoro, Fabio Podest\`a, Tommaso Sferruzza and Alejandro Tolcachier for very fruitful conversations.  I am also indebted to the participants of a course I gave in C\'ordoba during the first semester of 2026 on this paper for their support and several invaluable comments and corrections: Martiniano Faure, Gerson Guti\'errez Quispe, Facundo Montedoro, Giorgia Petracci, Agustina Ruiz Linares, Edwin Rodr\'\i guez Valencia, Bruno Stassi, Alejandro Tolcachier and Cynthia Will.

% 15/7/2026

\section{Flag manifolds}\label{flag-sec}

There is a vast literature on flag manifolds, we refer to \cite{BrlHrz, Tkc, Nsh, AlkPrl, BrdFrgRmr, PnlPds} for more detailed treatments on the material of this section.  

A {\it flag manifold} is a homogeneous space $F=G/H$, where $G$ is a compact and connected semisimple Lie group and $H$ is the centralizer of some torus in $G$.  Equivalently, $F=G/H$ is an $\Ad(G)$-orbit in the Lie algebra $\ggo$ of $G$ (in particular, $\dim{F}$ is always even, see \eqref{hq}).  A flag manifold $F=G/H$ is called {\it full flag} whenever $H$ is abelian, i.e., $H$ is a maximal torus of $G$.  See \S\ref{homsp-sec} for some preliminaries on homogeneous spaces.  

\begin{proposition} 
A compact homogeneous space $G/H$ with finite fundamental group admits a $G$-invariant symplectic structure if and only if it is a flag manifold.  
\end{proposition}

\begin{remark}
In particular, products of flag manifolds and even-dimensional tori are the only compact manifolds that can admit a homogeneous K\"ahler metric; indeed, they all do (see \S\ref{K} below).
\end{remark}

\begin{proof}
We can assume that $G$ is compact semisimple and $H$ is connected.  The space of all closed $G$-invariant $2$-forms on the homogeneous space $G/H$ can therefore be described as follows (see \S\ref{dR-sec}): if $\ggo=\hg\oplus\qg$ is the $\kil_\ggo$-orthogonal decomposition and $\qg_0:=\{ X\in\qg:[\hg,X]=0\}$, then 
$$
\Omega^2_c(G/H)^G=\{ \sigma_Z:Z\in\zg(\hg)\oplus\qg_0\}, \quad \mbox{where}\quad \sigma_Z:=\kil_\ggo([\cdot,\cdot],Z)\in(\Lambda^2\qg^*)^H.  
$$ 
Thus $G/H$ admits a $G$-invariant symplectic form if and only if there exists $Z\in\zg(\hg)\oplus\qg_0$ such that $\sigma_Z$ is non-degenerate, which implies that $Z\in\zg(\hg)$ and $\qg_0=0$ using the invertibility of the map $\ad{Z}|_\qg$.  In that case, if $Y$ is in the centralizer of $Z$ in $\ggo$, say $Y=Y_\hg+Y_\qg$, then $[Z,Y_\qg]=0$ and so $Y\in\hg$.  This implies that $H=C_G(Z)$, i.e., $G/H$ is the adjoint orbit of $Z$ in $\ggo$, and $H$ is also the centralizer of the torus in $G$ given by the closure of $\exp(\RR Z)$.     
\end{proof}

Let $F=G/H$ be a flag manifold.  We fix from now on a maximal torus $T\subset H\subset G$ of both $H$ and $G$, and consider $\Delta\subset(\tg^c)^*$, the root system defined by the Cartan subalgebra $\tg^c$ of the complex semisimple Lie algebra $\ggo^c:=\ggo\otimes\CC$ (see \S\ref{roots} for a review on roots).  If 
$$
\ggo=\hg\oplus\qg, \qquad T_oF\equiv \qg, 
$$ 
is the $\kil_\ggo$-orthogonal reductive decomposition of $F=G/H$, then  
\begin{equation}\label{hq}
\hg^c=\tg^c\oplus\bigoplus_{\alpha\in\Delta_\hg}\ggo_\alpha, \qquad 
\qg^c=\bigoplus_{\alpha\in\Delta_\qg}\ggo_\alpha, 
\end{equation}
where $\Delta=\Delta_\hg\sqcup\Delta_\qg$ (disjoint union),  
$$
\Delta_\hg:=\left\{\alpha\in\Delta:\alpha|_{\zg(\hg)}=0\right\}, 
$$
$\zg(\hg)\subset\tg$ denotes the center of $\hg$ and $\Delta_\qg$ is the set of {\it complementary roots}.  Note that $\tg=\zg(\hg)\oplus\tg_{[\hg,\hg]}$, where $\tg_{[\hg,\hg]}$ is the maximal torus of $[\hg,\hg]$, so $\Delta_\hg$ is the root system of the semisimple Lie algebra $[\hg,\hg]$.  In particular, $\dim{F}=\dim{\qg}=|\Delta_\qg|$ is always even.   

For $G$ simple, any flag $F=G/H$ can be represented as a painted Dynkin diagram, on which the black dots are the simple roots in $\Delta_\hg$ and the number of white dots is $\dim{\zg(\hg)}$.  

\begin{example}\label{su5}
Consider the flag 
$$
F^{16}=\SU(5)/\Se(\U(1)\times\U(2)\times\U(2)), \qquad \hg=\sug(2)\oplus\sug(2)\oplus\RR^2,
$$ 
with painted Dynkin diagram 
$$
\dynkin[scale=2] A{o*o*} 
$$
and as usual, $\im\tg=\{ (a_1,\dots,a_5)\in\RR^5:\sum a_i=0\}$.  Thus $\Delta=\{\alpha_{ij}:1\leq i,j\leq 5\}$, $\alpha_{ij}\equiv\im H_{\alpha_{ij}}=e_i-e_j$,  
$$
\zg(\hg)=\{(a,b,b,c,c):a+2b+2c=0\}, \quad \Delta_\hg=\{\pm\alpha_{23},\pm\alpha_{45}\}, \quad \Delta_\qg=\Delta\setminus\Delta_\hg. 
$$  
\end{example}

\subsection{Isotropy representation}\label{isot}
For any subset $P\subset\Delta_\qg$ such that $\Delta_\qg=P\sqcup -P$, we have that the isotropy representation $\qg$ of the flag $F=G/H$ decomposes as
$$
\qg=\bigoplus_{\alpha\in P}\qg_\alpha,  \qquad \mbox{where}\quad 
\qg_\alpha:=\RR e_\alpha+\RR f_\alpha = (\ggo_\alpha\oplus\ggo_{-\alpha})\cap\ggo, 
$$
that is, $\qg_\alpha^c=\ggo_\alpha\oplus\ggo_{-\alpha}$ (see \S\ref{roots}).  Note that $\qg_\alpha=\qg_{-\alpha}$ and 
\begin{equation}\label{adH}
\ad{H}|_{\qg_\alpha}=-\alpha(\im H)\left[\begin{matrix} 
0&-1\\ 1&0 
\end{matrix}\right], \qquad\forall H\in\tg.
\end{equation}
If $\alpha,\beta\in P$, $\alpha\ne\beta$, then $[\qg_\alpha,\qg_\alpha]=\RR \im H_\alpha$ and 
$$
[\qg_\alpha,\qg_\beta] = 
\left\{\begin{array}{ll}
\qg_{\alpha+\beta}\oplus\qg_{\alpha-\beta}, &\qquad \alpha+\beta\in\Delta, \quad \alpha-\beta\in\Delta, \\
\qg_{\alpha+\beta}, &\qquad \alpha+\beta\in\Delta, \quad \alpha-\beta\notin\Delta, \\
\qg_{\alpha-\beta}, &\qquad \alpha+\beta\notin\Delta, \quad \alpha-\beta\in\Delta, \\
0, &\qquad \alpha+\beta\notin\Delta, \quad \alpha-\beta\notin\Delta.  
\end{array}\right.
$$
We consider a set of {\it central roots} $\alpha_i|_{\zg(\hg)}$ (called $T$-roots in \cite{AlkPrl,AlkChr}) of the flag $F=G/H$, defined by 
\begin{equation}\label{cenrts}
\left\{\alpha|_{\zg(\hg)}:\alpha\in P\right\}=\left\{\alpha_1|_{\zg(\hg)},\dots,\alpha_r|_{\zg(\hg)}\right\}, \qquad\alpha_i\in P.  
\end{equation}
It follows that 
$$
\qg=\mg_1\oplus\dots\oplus\mg_r, \qquad 
\mg_i:=\bigoplus_{\alpha\in P:\alpha|_{\zg(\hg)}=\alpha_i|_{\zg(\hg)}} \qg_{\alpha},
$$
is a decomposition in inequivalent irreducible $\Ad(H)$-invariant subspaces, as it also is    
$$
\qg^c=\mg_1^c\oplus\dots\oplus\mg_r^c, \qquad 
\mg_i^c=\bigoplus_{\alpha\in\Delta_\qg:\alpha|_{\zg(\hg)}=\pm\alpha_i|_{\zg(\hg)}} \ggo_{\alpha}.
$$
Alternatively, if $\Pi$ is any basis of $\Delta$ and we consider $P=\Delta_\qg^+:=\Delta_\qg\cap\Delta^+$, where $\Delta^+$ are the corresponding positive roots, then each $\alpha_i\in \Delta_\qg^+$ can be written as 
$$
\alpha_i=\sum_{\alpha\in\Pi_\hg} k_\alpha\alpha + \sum_{\alpha\in\Pi_\qg} n_\alpha\alpha\in P, \qquad k_\alpha, n_\alpha\in\ZZ_{\geq 0}, \qquad  \Pi_\hg:=\Pi\cap\Delta_\hg, \quad \Pi_\qg:=\Pi\cap\Delta_\qg. 
$$ 
In this way, the irreducible summand $\qg_i$ is determined by
$$
\left\{\beta\in \Delta_\qg^+:\beta|_{\zg(\hg)}=\alpha_i|_{\zg(\hg)}\right\} 
= \left\{\beta=\sum_{\alpha\in\Pi} m_\alpha\alpha: m_\alpha=n_\alpha\;\forall \alpha\in\Pi_\qg\right\}.
$$ 
If we set 
\begin{equation}\label{Zalf}
Z_\alpha:=(\im H_\alpha)_{\zg(\hg)},  
\end{equation}
where the subscript $\zg(\hg)$ denotes orthogonal projection on $\zg(\hg)$ with respect to the inner product $-\kil_{\ggo}$, then 
$$
\{ Z_\alpha:\alpha\in\Delta^+\} = \{ Z_\alpha:\alpha\in\Delta_\qg^+\} = \{ Z_{\alpha_1},\dots,Z_{\alpha_r}\},
$$
and the smaller subset $\{ Z_\alpha:\alpha\in\Pi_\qg\}$ is known to be a basis of $\zg(\hg)$.  

%\begin{figure}\label{su5-fig}
%....
%\end{figure} 

\begin{example}\label{su5-2}
If for the flag in Example \ref{su5} we consider $\Delta^+:=\{\alpha_{ij}:i<j\}$, then $\Pi=\{\alpha_{12},\alpha_{23},\alpha_{34},\alpha_{45}\}$, $\Delta_\hg^+=\Pi_\hg=\{\alpha_{23},\alpha_{45}\}$, $\Delta_\qg^+=\Delta^+\setminus\Delta_\hg^+$ and $\Pi_\qg=\{\alpha_{12},\alpha_{34}\}$.  As central roots we can take
$$
\alpha_1=\alpha_{12}, \qquad \alpha_2=\alpha_{34}, \qquad \alpha_3=\alpha_{14}=\alpha_{12}+\alpha_{23}+\alpha_{34},
$$
and the sets $\left\{\beta\in \Delta_\qg^+:\beta|_{\zg(\hg)}=\alpha_i|_{\zg(\hg)}\right\} $ can easily be drawn in, so $\dim{\mg_1}=4$, $\dim{\mg_2}=8$, $\dim{\mg_3}=4$. 
\end{example}

\subsection{Riemannian metrics}\label{riem}
$G$-invariant metrics on the flag $F=G/H$ are parameterized by tuples of the form 
$$
(x_\alpha)_{\alpha\in\Delta_\qg}=(x_1,\dots,x_r), \qquad x_i:=x_{\alpha_i}, 
$$ 
such that $x_\alpha>0$, $x_{-\alpha}=x_\alpha$ and $x_\alpha=x_\beta$ for all $\alpha,\beta$ in the same irreducible component of $\qg$ (i.e., $\alpha|_{\zg(\hg)}=\pm\beta|_{\zg(\hg)}$).  Note that $x_{\alpha+\beta}=x_\alpha$ for all $\alpha\in\Delta_\qg$, $\beta\in\Delta_\hg$, $\alpha+\beta\in\Delta$.  

Given any subset $P\subset\Delta_\qg$ as in \S\ref{isot}, the metric $g$ associated with $(x_\alpha)_{\alpha\in\Delta_\qg}$ is defined by 
$$
g=\sum_{\alpha\in P} x_\alpha (-\kil_\ggo)|_{\qg_\alpha} 
= x_1(-\kil_\ggo)|_{\mg_1}+\dots+x_r(-\kil_\ggo)|_{\mg_r}, \qquad x_i:=x_{\alpha_i},
$$
where $\alpha_1,\dots,\alpha_r$ are the central roots, so the corresponding symmetric $\CC$-bilinear form on $\ggo^c$, also denoted by $g$, satisfies that 
$$
g(E_\alpha,E_{-\alpha})=-x_\alpha, \qquad\forall \alpha\in\Delta_\qg, 
$$ 
and $g(E_\alpha,E_\beta)=0$ whenever $\alpha+\beta\ne 0$ (see \S\ref{roots}).  Equivalently,    
$$
\left\{\tfrac{1}{\sqrt{x_\alpha}}e_\alpha, \; \tfrac{1}{\sqrt{x_\alpha}}f_\alpha : \alpha\in P\right\} 
$$
is a $g$-orthonormal basis of $\qg$.   
This metric will be denoted by $(x_\alpha)_{\alpha\in\Delta_\qg}$ or $(x_1,\dots,x_r)$.  

The standard metric $\gk$ (i.e., defined by $-\kil_\ggo|_\qg$) corresponds to $x_\alpha=1$ for any $\alpha\in\Delta_\qg$.  In particular, $\left\{e_\alpha, \; f_\alpha : \alpha\in P\right\}$ is a $\gk$-orthonormal basis of $\qg$.

\subsection{Almost-complex structures}\label{alm-comp}
$G$-invariant almost-complex structures on the flag $F=G/H$ are parameterized by tuples of the form 
$$
(\epsilon_\alpha)_{\alpha\in\Delta_\qg}=(\epsilon_1,\dots,\epsilon_r), \qquad \epsilon_i:=\epsilon_{\alpha_i},
$$ 
such that $\epsilon_\alpha=\pm 1$, $\epsilon_{-\alpha}=-\epsilon_\alpha$ and $\epsilon_\alpha=\epsilon_\beta$ for all $\alpha,\beta$ with $\alpha|_{\zg(\hg)}=\beta|_{\zg(\hg)}$ (see \cite{SnmCss})).  The corresponding map $J$ is defined by     
$$
Je_\alpha=\epsilon_\alpha f_\alpha, \qquad 
Jf_\alpha=-\epsilon_\alpha e_\alpha, \qquad\forall\alpha\in\Delta_\qg,
$$ 
so $\{ e_\alpha,Je_\alpha\}$ is a $\gk$-orthonormal basis of $\qg_\alpha$.  Equivalently, the corresponding $\CC$-linear map $J:\qg^c\rightarrow\qg^c$ is given by 
$$
JE_\alpha=\im\epsilon_\alpha E_\alpha, \qquad\forall \alpha\in\Delta_\qg.
$$   
The $\Ad(K)$-invariance of $J$, i.e., $\epsilon_\alpha=\epsilon_\beta$ whenever $\alpha|_{\zg(\hg)}=\beta|_{\zg(\hg)}$, can alternatively be written as 
$$
\left(\Delta_\hg + \{\alpha\in\Delta_\qg:\epsilon_\alpha=1\}\right)\cap\Delta\subset \{\alpha\in\Delta_\qg:\epsilon_\alpha=1\}.
$$ 
In this way, the matrix of $J|_{\mg_i}$ in terms of an ordered basis $\{ e_\alpha,Je_\alpha:\alpha|_{\zg(\hg)}=\alpha_i\}$ of $\mg_i$ is given by 
$$
[J|_{\mg_i}]=\epsilon_i
\left[\begin{smallmatrix} 
0&-1&&&\\ 
1&0&&&\\ 
&&\ddots&&\\ 
&&&0&-1\\ 
&&&1&0
\end{smallmatrix}\right].
$$
We note that any $G$-invariant metric $g=(x_\alpha)_{\alpha\in\Delta_\qg}$ is almost-hermitian with respect to any $G$-invariant almost-complex structure $J=(\epsilon_\alpha)_{\alpha\in\Delta_\qg}$.  The corresponding K\"ahler form $\omega=g(J\cdot,\cdot)$ is given by 
$$
\omega(e_\alpha,f_\alpha)=\epsilon_\alpha x_\alpha,  \qquad\forall\alpha\in P,
$$
and zero otherwise, or equivalently, the corresponding skew-symmetric $\CC$-bilinear form on $\qg^c$ is given by
$$
\omega(E_\alpha,E_{-\alpha})=-\im\epsilon_\alpha x_\alpha, \qquad\forall\alpha\in\Delta_\qg, 
$$
and zero otherwise.

\subsection{Complex structures}\label{comp}
It follows e.g.\ from \cite[12.4]{BrlHrz} (see also \cite{AlkPrl}) that an almost-complex structure $J=(\epsilon_\alpha)_{\alpha\in\Delta_\qg}$ on the flag $F=G/H$ is integrable if and only if the set 
$$
P_J:=\{\alpha\in\Delta_\qg:\epsilon_\alpha=1\} 
%=\{\alpha\in\Delta_\qg:\alpha|_{\zg(\hg)}=\alpha_i, \quad \epsilon_i=1\} 
$$ 
defines an {\it invariant ordering} in $\Delta_\qg$, i.e., 
$$
\Delta_\qg=P_J\sqcup -P_J, \qquad (\Delta_\hg+P_J)\cap\Delta\subset P_J, \qquad  (P_J+P_J)\cap\Delta\subset P_J. 
$$
The following conditions are all equivalent to $P_J$ being an invariant ordering:
\begin{enumerate}[a)]
\item \cite[13.7]{BrlHrz} There is an ordering $\Delta^+$ (see \S\ref{roots}) such that $P_J=\Delta_\qg\cap\Delta^+=:\Delta_\qg^+$.  

\item \cite[(3.41)]{BrdFrgRmr} $P_J=\{\alpha\in\Delta_\qg:\alpha|_C>0\}$, where $C$ is a chamber (or connected component) of 
\begin{equation}\label{ch}
\im\zg(\hg)\setminus\bigcup_{\alpha\in\Delta_\qg}\Ker\alpha\cap\im\zg(\hg) 
=\im\zg(\hg)\setminus\bigcup_{j=1}^r\Ker\alpha_j\cap\im\zg(\hg),  
\end{equation}
where $\alpha_1,\dots,\alpha_r$ are the central roots (see \eqref{cenrts}). 

\item \cite[Theorem 1]{Nsh} There is a basis $\Pi\subset\Delta$ such that $\Pi\cap\Delta_\hg$ is a basis of $\Delta_\hg$ and $P_J=\la\Pi\ra_{\NN_0}\setminus\la\Pi\cap\Delta_\hg\ra_{\NN_0}$, where $\la\cdot\ra_{\NN_0}$ denotes the set of all linear combinations with nonnegative integer coefficients.    

\item \cite[Theorem 2]{Nsh} Given any fixed basis $\Pi_\hg$ in $\Delta_\hg$, there exists a basis $\Pi\subset\Delta$ such that $\Pi\cap\Delta_\hg=\Pi_\hg$ and $P_J=\la\Pi\ra_{\NN_0}\setminus\la\Pi_\hg\ra_{\NN_0}$.  
\end{enumerate}

Let $\cca^G$ denote the set of all $G$-invariant complex structures on the flag manifold $F=G/H$, thus 
$$
\cca^G=\{ J_1,\dots,J_k\},
$$ 
where $k$ is the number of chambers in \eqref{ch}.  It is well known that if $m:=\dim{\zg(\hg)}$, then 
\begin{equation}\label{numberk}  
k \leq k(m,r) := \left\{\begin{array}{l} 
2^r, \qquad r\leq m, \\ 
2\left(\binom{r-1}{0}+\dots+\binom{r-1}{m-1}\right), \qquad r>m,  
\end{array}\right.
\end{equation} 
where equality holds if and only if the intersection of any tuple of $m$ different subspaces $\Ker\alpha_j\cap\im\zg(\hg)$ in \eqref{ch} is always zero.  In particular, $k(1,r)=2$ and $k(2,r)=2r$ for any $r$.  For instance, the full flag $F=\SU(4)/T^3$ has $k<k(3,6)=32$, since the general position condition does not hold.  

Two structures $J_1$ and $J_2$ in $\cca^G$ are said to be {\it biholomorphic} ($J_1\simeq J_2$ for short) if there exists a biholomorphism between the complex manifolds $(F,J_1)$ and $(F,J_2)$.  The following conditions are equivalent (see also \cite{AlkPrl}): 
\begin{enumerate}[(i)] 
\item $J_1$ and $J_2$ are biholomorphic.  

\item \cite[Theorem 4]{Nsh} There exists $\psi\in\Aut(G/H)\subset\Diff(F)$ (see \S\ref{gauge}) such that $\psi^*J_1=J_2$.  Note that if $\psi$ is in addition inner, i.e., $\psi=I_a$ for some $a\in N_G(H)$, then $J_1=\Ad(a)J_2\Ad(a)^{-1}$ as linear operators of $\qg$.  

\item \cite[13.8]{BrlHrz} There exists $\psi\in\Aut(\Delta)$ such that $\psi\Delta_\hg=\Delta_\hg$ and $\psi P_{J_1}=P_{J_2}$.  

\item There exists $\psi\in\Aut(\Delta)$ such that $\psi\Delta_\hg=\Delta_\hg$ and $\psi C_1=C_2$, where $C_i$ is the chamber of \eqref{ch} attached to $J_i$.  

\item \cite[Theorems 2,4]{Nsh} There exists $\psi\in\Aut(\Delta)$ such that $\psi(\Pi_1)=\Pi_2$ and $\psi(\Pi_1\cap\Delta_\hg)=\Pi_2\cap\Delta_\hg$, where $\Pi_i$ is the simple root system defined by $J_i$ in c).  

\item \cite[Theorem 4]{Nsh} There exists $\psi\in\Aut(\Delta)$ such that such that $\psi(\Pi_1)=\Pi_2$ and $\psi(\Pi_\hg)=\Pi_\hg$, where $\Pi_i$ is the simple root system defined by $J_i$ in d).  
\end{enumerate}

It follows from (iii) that $J\simeq -J$ for any $G$-invariant complex structure $J$; indeed, we can take $\psi=-I\in\Aut(\Delta)$.  

According to (vi), given any fixed basis $\Pi_\hg$ in $\Delta_\hg$, the moduli space $\cca^G/\simeq$ of all $G$-invariant complex structures on the flag $F=G/H$ up to biholomorphism is in bijection with the quotient
$$
\{\mbox{bases}\;\Pi\subset\Delta:\Pi\cap\Delta_\hg=\Pi_\hg\}/\Aut(\Delta,\Pi_\hg), 
$$
where $\Aut(\Delta,\Pi_\hg):=\{\psi\in\Aut(\Delta):\psi\Pi_\hg=\Pi_\hg\}$. 

\begin{example}\label{su5-3}
For the flag in Examples \ref{su5} and \ref{su5-2}, in terms of the orthonormal basis of $\zg(\hg)$ given by
$$
\left\{e_1:=\unm(0,1,1,-1,-1),\; e_2:=\tfrac{1}{2\sqrt{5}}(4,-1,-1,-1,-1)\right\},
$$
$\Ker\alpha_1=\RR(\sqrt{5},1)$, $\Ker\alpha_2=\RR (0,1)$ and $\Ker\alpha_3=\RR(\sqrt{5},-1)$.  Thus $F$ admits $6$ complex structures, but since the permutation $w:=(1,4,5,2,3)\in W$ satisfies that $w\cdot\Delta_\hg=\Delta_\hg$, $w\cdot e_1=-e_1$, $w\cdot e_2=e_2$, there are exactly $2$ up to biholomorphism.  
\end{example}

The following is a very useful result to count the number of $G$-invariant complex structures on a flag $F=G/H$ up to biholomorphism.  Consider a fixed basis $\Pi$ of $\Delta$ and the set $\sca$ of all subsets $\phi\Pi_\hg\subset\Pi$, where $\phi:\Pi_\hg\rightarrow\Pi$ is an injective map such that $\la\phi\alpha,\phi\beta\ra=\la\alpha,\beta\ra$ for all $\alpha,\beta\in\Pi_\hg$.  

\begin{theorem}\label{eq-N}\cite[Lemma 2, Theorems 7,8,9]{Nsh} 
On any flag $F=G/H$ such that $\ggo$ is simple and different from $\sog(2n)$, $n\geq 4$, the moduli space $\cca^G/\simeq$ is in bijection with $\sca/O$, where $O$ is the automorphism group of the Dynkin diagram of $\ggo$.   
\end{theorem}

For instance, there are two on the flag in Example \ref{su5-2} (cf.\ Example \ref{su5-3}):
\begin{center}
\dynkin[scale=2] A{*oo*} \quad , \quad
\dynkin[scale=2] A{o*o*}\quad$\simeq$\quad\dynkin[scale=2] A{*o*o} , 
\end{center}
five on $\Eg_6/\SU(3)\times\SU(2)\times T^3$:
\begin{center}
\dynkin[scale=2] E{***ooo} \quad$\simeq$\quad \dynkin[scale=2] E{o*oo**} \quad , \quad
\dynkin[scale=2] E{oo*o**} \quad , \quad 
\dynkin[scale=2] E{*ooo**} \quad , \quad 
\dynkin[scale=2] E{*oo**o} \quad , \quad
\dynkin[scale=2] E{**o*oo} ,
\end{center}
and only one on $\Fg_4/\Spe(3)\times T^1$:
\begin{center}
\dynkin[scale=2] F{o***} .
\end{center}

\subsection{Koszul form}\label{K-form} 
The {\it Koszul form} of a complex flag manifold $(F=G/H,J)$, where $J$ is associated to the invariant ordering $\Delta_\qg^+$, is defined by: 
\begin{equation}\label{koszul}
\rho_J:= \sum_{\alpha\in\Delta_\qg^+} \alpha \equiv \sum_{\alpha\in\Delta_\qg^+} H_\alpha =:-\im Z_J\in\im\zg(\hg)\subset\im\tg,
\end{equation}
where we use the inner product $\ip:=\kil_{\ggo^c}|_{\im\tg}$ for the identification (see \S\ref{roots}).  Note that $\rho_J=2(\rho_\ggo-\rho_{[\hg,\hg]})$, where $\rho_\ggo$ and $\rho_{[\hg,\hg]}$ are the half of the sum of positive roots for the semisimple Lie algebras $\ggo$ and $[\hg,\hg]$ with respect to $\Delta^+$ and $\Delta_\hg^+$, respectively (see \eqref{weyl}).  It follows from \S\ref{isot} that 
\begin{equation}\label{rhoi}
\rho_J=\rho_1+\dots+\rho_r, \qquad\mbox{where}\quad \rho_i:=\sum_{\alpha\in\Delta_\qg^+:\alpha|_{\zg(\hg)}=\alpha_i|_{\zg(\hg)}} \alpha,
\end{equation}
and $\rho_i\in\im\zg(\hg)$ for all $i$ (see \cite[Lemma 8]{Pds}).  Moreover, $\rho_J\equiv -\im Z_J$ belongs to the chamber $C$ attached to $J$ (see \eqref{ch}), that is, 
\begin{equation}\label{Jal} 
\beta(-\im Z_J) =\la\rho_J,\beta\ra=\sum_{\alpha\in\Delta_\qg^+}\la\beta,\alpha\ra>0, \qquad\forall \beta\in\Delta_\qg^+. 
\end{equation}

\subsection{K\"ahler metrics}\label{K} 
Any $G$-invariant metric $g$ on a flag $F=G/H$ is Hermitian with respect to any $G$-invariant complex structure $J$ (i.e., $g(J\cdot,J\cdot)=g(\cdot,\cdot)$).  

Given a complex structure $J$ (associated with an invariant ordering $\Delta_\qg^+$, or a chamber $C$ in \eqref{ch}), a metric 
$g=(x_\alpha)_{\alpha\in\Delta_\qg^+}=g(x_1,\dots,x_r)$ is K\"ahler on $(F,J)$ if and only if any of the following conditions holds (see e.g.\ \cite[pp.3]{PnlPds} or \cite{Tkc}):
\begin{enumerate}[a)]
\item $x_{\alpha+\beta}=x_\alpha+x_\beta$ whenever $\alpha+\beta\in\Delta_\qg^+$.  

\item $x_k=x_i+x_j$ whenever $\alpha_k|_{\zg(\hg)}=\alpha_i|_{\zg(\hg)}+\alpha_j|_{\zg(\hg)}$.  

\item There exists $Z\in C$ such that $x_\alpha=\alpha(Z)$ for all $\alpha\in\Delta_\qg^+$ (set $g=g_Z$).  
\end{enumerate}
In that case, the K\"ahler form $\omega=g(J\cdot,\cdot)$ is given by $\omega=\kil_{\ggo^c}([\cdot,\cdot],Z)$, or equivalently, the corresponding skew-symmetric $\CC$-bilinear form on $\qg^c$ is given by
$$
\omega(E_\alpha,E_{-\alpha})=-\im\alpha(Z)=-\im x_\alpha, \qquad\forall\alpha\in\Delta_\qg^+, 
$$
and zero otherwise.  

There is a unique $G$-invariant K\"ahler-Einstein metric (up to scaling) on each complex flag manifold $(F=G/H,J)$ given by 
$$
g_{-\im Z_J} = (\la\alpha,\rho_J\ra)_{\alpha\in\Delta_\qg^+},
$$
where $-\im Z_J\in C$ (see \eqref{Jal}) is the Koszul form defined in \eqref{koszul}.  Indeed, the Ricci tensor is given by $\ricci(g_Z)=g_{-\im Z_J}$ for any $Z\in C$. 

\begin{example}\label{su4}
The flag manifold 
$$
F^{10}=\SU(4)/\Se(\U(1)\times\U(1)\times\U(2)), \qquad \dynkin[scale=2] A{oo*} ,
$$ 
has $\im\tg=\{ (a_1,\dots,a_4)\in\RR^5:\sum a_i=0\}$, $\Delta=\{\alpha_{ij}:1\leq i,j\leq 4\}$, $2\rho=(3,1,-1,-3)$ and  
$$
\zg(\hg)=\{(a,b,c,c):a+b+2c=0\}, \qquad \Delta_\hg=\{\pm\alpha_{34}\}, \qquad \Delta_\qg=\Delta\setminus\Delta_\hg. 
$$  
For the usual choice of positive roots $\Delta^+=\{\alpha_{ij}:i<j\}$, we have that $\Pi=\{\alpha_{12},\alpha_{23},\alpha_{34}\}$, $\Delta_\hg^+=\{\alpha_{34}\}$ and 
$$
\Delta_\qg^+=\{\alpha_{12},\; \alpha_{23},\; \alpha_{13},\; \alpha_{24},\; \alpha_{14}\}, 
$$
which determines a complex structure $J_1$ with $\rho_{J_1}=(3,1,-2,-2)$ and K\"ahler-Einstein metric $g_1=(2,3,5)$, relative to the central roots
$$
\alpha_1=\alpha_{12}, \quad \alpha_2=\alpha_{23},\; \alpha_{24}, \quad\alpha_3= \alpha_{13},\;  \alpha_{14}.  
$$  
On the other hand, the invariant ordering given by  
$$
\{-\alpha_{12},\; \alpha_{14},\; \alpha_{24},\; \alpha_{13},\; \alpha_{23}\},
$$
determines another complex structure $J_2$ with $\rho_{J_2}=(1,3,-2,-2)$ and K\"ahler-Einstein metric $g_2=(2,5,3)$.  
\end{example}

\begin{example}\label{b21}  
Consider a flag manifold $F=G/H$ with second Betti number $b_2(F)=1$, i.e., $\dim{\zg(\hg)}=1$.  This is equivalent to have $\Pi_\hg=\Pi\setminus\{\gamma\}$ for some fixed $\gamma\in\Pi$.  Thus $G$ is simple and  
$$
\Delta_\qg^+|_{\zg(\hg)} = \{\alpha_1=\gamma, \alpha_2=2\gamma +\beta_2, \dots, \alpha_r=r\gamma +\beta_r\}, \qquad\beta_i\in\Delta_\hg, 
$$
giving rise to the decomposition $\qg=\mg_1\oplus\dots\oplus\mg_r$ of the isotropy representation in irreducible components, where $\alpha_r$ is the maximal root.  Note that $Z_{\alpha_j}=jZ_\gamma$ for all $j=1,\dots,r$ (see \eqref{Zalf}), where $Z_\gamma=(\im H_\gamma)_{\zg(\hg)}$.  

According to \cite[Table 1]{ChrSkn1}, $1\leq r\leq 6$, and $r=1$ if and only if $F$ is a Hermitian irreducible symmetric space ($6$ infinite families and $2$ sporadic spaces).  There are $3$ infinite families and $10$ sporadic spaces with $r=2$ and $13$ sporadic examples with $3\leq r\leq 6$.  

Since $\dim{\im\zg(\hg)}=1$, there are only two chambers in \eqref{ch}, so $F$ admits a unique $G$-invariant complex structure $J$ up to biholomorphism, which is attached to the ordering $\Delta_\qg^+$ given above.  Note that the vector $Z_J:=\im\rho_J\in\zg(\hg)$ is given by 
$$
Z_J=\unm(\dim{\mg_1}+2\dim{\mg_2}+\dots+r\dim{\mg_r})Z_\gamma, \qquad Z_\gamma=(\im H_\gamma)_{\zg(\hg)}.  
$$
The metric $g=(1,2,\dots,r)$ is the unique $G$-invariant K\"ahler metric on $F$ up to scaling, which is automatically Einstein.   
\end{example}

% 14/8/2026

\section{Homogeneous complex manifolds}\label{chcm-sec} 

Let $M$ be a compact complex manifold that admits a transitive action of a compact Lie group of biholomorphisms, or equivalently, a compact homogeneous space $M=G/K$ admitting a $G$-invariant complex structure.  Back in 1954, the following characterization of such homogeneous spaces was obtained by Hsien-Chung Wang \cite{Wng} in the finite fundamental group case and by Tits \cite{Tts} in 1962 in the general case (see also \cite{NiWll} for a recent alternative proof via a purely Lie theoretical approach).  

We consider an almost-effective homogeneous space $M=G/K$, where $G$ is compact and connected and $K$ is a connected closed subgroup of $G$, and denote by $\ggo$ and $\kg$ the corresponding Lie algebras (see \S\ref{homsp-sec}).    

\begin{definition}\label{Csp-def}
A homogeneous space $M=G/K$ as above is called a {\it C-space} if there exists a Lie subalgebra $\hg\subset\ggo$ such that 
\begin{enumerate}[{\rm (i)}] 
\item $\hg$ is the centralizer of some abelian subalgebra of $\ggo$ (equivalently, $\hg=C_\ggo(\zg(\hg))$, where $\zg(\hg)$ is the center of $\hg$), 
 
\item $\kg\subset\hg$, 

\item $[\hg,\hg]=[\kg,\kg]$.     
\end{enumerate}
\end{definition} 

In particular, the semisimple part of $K$ must coincide with the semisimple part of the corresponding Lie group $H$, which is the centralizer of some torus in $G$.  Note that $\ggo=[\ggo,\ggo]\oplus\zg(\ggo)$, where $\zg(\ggo)$ is the center of $\ggo$, and by almost-effectiveness, $\kg\cap\zg(\ggo)=0$.  

\begin{theorem}\label{C}\cite[Theorem 4.1]{NiWll}
A compact homogeneous space $M=G/K$, where $G$ and $K$ are connected, admits a $G$-invariant complex structure if and only if it is a C-space of even dimension.     
\end{theorem}

In spite `C-space' is actually the name given by Wang only in the finite fundamental group case, for simplicity, we decided to use the same name in the general compact case.  

Given a C-space $M=G/K$, each subalgebra $\hg\subset\ggo$ as in Definition \ref{Csp-def} determines a flag manifold $F$ and the following torus fibration over $F$, called a {\it Tits fibration}:  
\begin{equation}\label{fib}
A=H/K\longrightarrow M=G/K \longrightarrow F=G/H=G_f/H_f, 
\end{equation}
where the Lie algebras of $G_f$ and $H_f$ are respectively given by 
$$
\ggo_f:=[\ggo,\ggo], \qquad \hg_f:=\hg\cap[\ggo,\ggo];
$$
in particular, $\hg=\hg_f\oplus\zg(\ggo)$ and $\zg(\hg)=\zg(\hg_f)\oplus\zg(\ggo)$.  Thus $F=G_f/H_f$ is a flag manifold (indeed, $\hg_f=C_{\ggo_f}(\zg(\hg_f))$) and since $G=G_f\times T^z$, $K\subset H=H_f\times T^z$ and 
$$
[K,K]=[H,H]=[H_f,H_f],
$$ 
we obtain that $A=H/K$ is a torus with Lie algebra $\ag\simeq\zg(\hg)/\zg(\kg)$.  The extremal case when $M$ is a torus and $F$ is a point is also included.  Note that any compact semisimple Lie group $M=G$ is a C-space which fibers over the full flag $F=G/T$ with fiber $T$, for any maximal torus $T\subset G$.   

\begin{example}\label{su5}
Given the flag 
$$
F^{18}=G/H=\SU(5)/\Se(\U(1)\times\U(1)\times\U(1)\times\U(2)), \qquad \dynkin[scale=2] A{ooo*} ,
$$
each vector $Z\in\zg(\hg)$ defines a C-space $M^{20}=G/K=\SU(5)/K$, where $K$ is the connected Lie subgroup with Lie algebra $\kg:=[\hg,\hg]\oplus\RR Z$, provided $K$ is closed in $G$.  Note that $K\simeq \SU(2)\times S^1$ up to finite cover.  The fiber of the Tits fibration $M\rightarrow F$ is therefore the $2$-torus $H/K$ with Lie algebra $\ag$, where $\zg(\hg)=\RR Z\oplus\ag$.   
\end{example} 

We fix from now, as a background metric, any bi-invariant inner product $Q$ on $\ggo$ such that 
$$
Q|_{[\ggo,\ggo]}=-\kil_{[\ggo,\ggo]}, 
$$ 
where $\kil_{[\ggo,\ggo]}$ is the Killing form of the semisimple Lie algebra $[\ggo,\ggo]$.  Given two subspaces $\bg,\cg\subset\ggo$, let $\bg_\cg$ denote the $Q$-orthogonal projection of $\bg$ on $\cg$.  

The following $Q$-orthogonal decompositions play an important role in the study of the geometry of a C-space M=G/K: for any subalgebra $\hg$ as in Definition \ref{Csp-def}, we consider
\begin{equation}\label{reddec}
\ggo= \kg\oplus\pg = 
\underbrace{[\hg_f,\hg_f]\oplus\rlap{$\overbrace{\phantom{\zg(\kg) \oplus\ag}}^{\zg(\hg)}$}\zg(\kg)}_\kg \oplus \underbrace{\ag\oplus\qg}_\pg,  \qquad T_oM\equiv\pg=\ag\oplus\qg, \quad T_oA\equiv\ag, \quad T_oF\equiv\qg,
\end{equation}
where $\ggo_f= \hg_f\oplus\qg$ is the reductive decomposition of the flag $F=G_f/H_f$ and $\ag$ is the $Q$-orthogonal complement of $\qg$ in $\pg$, or equivalently, the $Q$-orthogonal complement of $\zg(\kg)$ in $\zg(\hg)$.  Note that $[\hg_f,\hg_f]=[\kg,\kg]$.  

Summarizing, any C-space $M=G/K$ is actually determined by the flag $F=G_f/H_f$ and a subspace 
\begin{equation}\label{slope}
\zg(\kg)\subset\zg(\hg)=\zg(\ggo)\oplus\zg(\hg_f), \qquad\mbox{giving rise to}\qquad \zg(\hg)=\zg(\kg)\oplus\ag,
\end{equation}
such that $\zg(\kg)\cap\zg(\ggo)=0$ and the corresponding connected Lie subgroup $Z(K)$ with Lie algebra $\zg(\kg)$ is closed in $Z(H)$ (or equivalently, $Z(K)$ is compact, i.e., the subspace $\zg(\kg)$ is compactly embedded).  

The closedness is equivalent to the existence of a basis of $\zg(\kg)$ contained in the lattice $\tfrac{1}{2\pi\im}\Ker(\exp)\subset\im\tg$ defined by the Lie exponential map $\exp:\tg\rightarrow T$, which is in between the dual root and dual weight lattices (see \S\ref{roots}).  In this way, $Z(K)$ is closed in $Z(H)$ and so $K$ is closed in $H$ and $G$.  Note that the choice of the subspace $\zg(\kg)$ determining the homogeneous space $M=G/K$ is equivalent to the choice of its $Q$-orthogonal complement $\ag$, and that 
$$
\dim{\ag}=\dim{\zg(\hg)}-\dim{\zg(\kg)}
$$ 
must be even in order to $M$ admit a $G$-invariant complex structure.  

We now give a refinement of the presentation of a C-space as a homogeneous space, providing a simpler setting.  

\begin{proposition}\label{refi} 
For any C-space $M$ fibering on a flag manifold $F=G_f/H_f$, there exists a torus $T$ such that, up to a finite cover, $M=G/K$, where 
$$
G:=G_f\times T, \qquad K\subset H_f, \qquad [\kg,\kg]=[\hg_f,\hg_f].
$$   
\end{proposition}

\begin{remark}
In particular, $G_f/K$ is also a C-space since $\hg_f$ satisfies all the conditions in Definition \ref{Csp-def}, $[\ggo,\ggo]$ is the Lie algebra of $G_f$, $\hg=\hg_f\oplus\zg(\ggo)$, $\zg(\ggo)$ is the Lie algebra of $T$ and, as a manifold,  
\begin{equation}\label{gprime}
M=G/K=G_f/K\times T^z, \qquad z:=\dim{\zg(\ggo)}.
\end{equation}  
However, we will see in \S\ref{comp-sec} that this is not in general a product of complex manifolds, even in the case when $\dim{G_f/K}$ and $\dim{\zg(\ggo)}$ are both even.   
\end{remark}

\begin{proof}
By Lemma \ref{gprime2}, if $M=G'/K'$, then, up to a finite cover, 
$$
M=G'_{ss}/K'\cap G'_{ss} \times T^{\dim{\zg(\ggo')\cap\pg}} = G_f/K'\cap G_f \times T^{\dim{\zg(\ggo')\cap\ag}},
$$
which gives the proposition if we set $K:=K'\cap G_f$ and consider the torus $T\subset G'$ with Lie algebra $\zg(\ggo')\cap\ag$.  
\end{proof}

For any C-space as in \eqref{gprime}, the Tits fibration is therefore given by   
\begin{equation}\label{fib2}
A=H_f/K\times T^z\longrightarrow M=G/K \longrightarrow F=G_f/H_f, 
\end{equation}
and the $Q$-orthogonal decompositions given in \eqref{reddec} can be rewritten as follows: 
\begin{equation}\label{reddec2}
\ggo= \kg\oplus\pg = 
\underbrace{[\hg_f,\hg_f]\oplus\rlap{$\overbrace{\phantom{\zg(\kg) \oplus\ag}}^{\zg(\hg_f)}$}\zg(\kg)}_\kg \oplus \underbrace{\ag_f\oplus\zg(\ggo)\oplus\qg}_\pg,  \qquad \pg=\ag\oplus\qg, \quad \ag=\ag_f\oplus\zg(\ggo), 
\end{equation}
where $\ggo_f=[\ggo,\ggo]= \hg_f\oplus\qg$ is the reductive decomposition of the flag $F=G_f/H_f$ and $\ag_f$ is the $Q$-orthogonal complement of $\zg(\kg)$ in $\zg(\hg_f)$.  Note that $\ag_f=\ag\cap\ggo_f$.    

Any C-space $M=G/K=G_f/K\times T^z$ is therefore determined by the flag $F=G_f/H_f$ (or by $\hg$ as in Definition \ref{Csp-def}), the number $\dim{\zg(\ggo)}$ and the choice of a compactly embedded subspace 
\begin{equation}\label{slope}
\zg(\kg)\subset\zg(\hg_f), \qquad\mbox{giving rise to}\qquad \zg(\hg_f)=\zg(\kg)\oplus\ag_f.  
\end{equation}
We note that it is equivalent to choose the $Q$-orthogonal complement $\ag_f$ of $\zg(\kg)$ in $\zg(\hg_f)$, and that 
$$
\dim{\ag}=\dim{\ag_f}+\dim{\zg(\ggo)} = \dim{\zg(\hg_f)}-\dim{\zg(\kg)}+\dim{\zg(\ggo)}
$$ 
must be even in order to $M$ admit a $G$-invariant complex structure.  

\begin{remark}\label{lh-rem}
If we also include locally homogeneous spaces, i.e., any subspace $\zg(\kg)\subset\zg(\hg_f)$ (not necessarily compactly embedded), then the set of all C-spaces $M=G_f/K$ of a given dimension which fibers over a fixed flag manifold $F=G_f/H_f$ is identified with the Grassmannian 
$$
\Gr_l(\RR^m) = \SO(m)/\Se(\Or(l)\times\Or(m-l)), \qquad m=\dim{\zg(\hg_f)}, \qquad l=\dim{\zg(\kg)},  
$$
which provides a natural topology on this subset of homogeneous torus bundles (see \cite{KrsWlfZar} for further information).  In particular, a rigorous definition of when a property of these fibrations is {\it generic} follows.  
\end{remark}

For the maximal abelian subalgebra $\tg\subset\ggo$, we have that
\begin{equation}\label{tdec}
\tg=\tg_f\oplus\zg(\ggo) 
=\rlap{$\overbrace{\phantom{\tg_{[\hg_f,\hg_f]}\oplus\zg(\kg)}}^{\tg_\kg}$} \tg_{[\hg_f,\hg_f]}\oplus \underbrace{\zg(\kg)\oplus\ag_f}_{\zg(\hg_f)} \oplus\zg(\ggo)
=\rlap{$\overbrace{\phantom{\tg_{[\hg_f,\hg_f]}\oplus\zg(\kg)}}^{\tg_f}$} \tg_{[\hg_f,\hg_f]}\oplus \zg(\hg_f)\oplus\zg(\ggo),
\end{equation}
where $\tg_f$, $\tg_{[\hg_f,\hg_f]}$ and $\tg_\kg$ are the maximal abelian subalgebras of $\hg_f$ (and $\ggo_f$), $[\hg_f,\hg_f]$ and $\kg$, respectively.   

\begin{remark}\label{ext}
A C-space $M=G/K=G_f/K\times T^z$ as in \eqref{gprime}, 
\begin{enumerate}[{\rm (i)}] 
\item is a flag manifold if and only if $\zg(\kg)=\zg(\hg_f)$ and $\zg(\ggo)=0$, so $\kg\ne 0$, $\ag=0$ and $M=G_f/H_f$ (indeed, if $G_f/K$ is flag, say $\kg=C_{\ggo_f}(Z)$ for some $Z\in\zg(\kg)$, then $\zg(\hg_f)\subset\kg$ and so $K=H_f$),

\item it is a Lie group if and only if $\kg=0$, i.e., $M=G=G_f\times T^z$, where $G_f$ is a semisimple Lie group, so $H_f$ is abelian ($[\hg_f,\hg_f]=[\kg,\kg]=0$) and $\ag_f=\zg(\hg_f)=\hg_f=\tg_f$ (see \eqref{tdec}), i.e., $F=G_f/T_f$ is full flag,

\item and the following conditions are equivalent:
\begin{enumerate}[{\rm (a)}] 
\item $M=G/K$ admits a $G$-invariant K\"ahler metric.  

\item $\ag_f=0$ (i.e., $\ag=\zg(\ggo)$, or $\zg(\kg)=\zg(\hg_f)$, or $\kg=\hg_f$).  

\item $M=G_f/H_f\times T^z$ and $\dim{\zg(\ggo)}$ is even.
\end{enumerate}
\end{enumerate}
\end{remark}

The flag manifold $G_f/H_f$ decomposes as 
\begin{equation}\label{decflags}
G_f/H_f=G_1/H_1\times\dots\times G_s/H_s, 
\end{equation} 
where each $G_i/H_i$ is a flag manifold, the $G_i$ are the simple factors of $G_f$ and $H_f=H_1\times\dots\times H_s$.  Thus $\Delta=\Delta_1\sqcup\dots\sqcup\Delta_s$ and $\Pi=\Pi_1\sqcup\dots\sqcup\Pi_s$, where $\Delta_i=\Delta_{\qg_i}\sqcup\Delta_{\hg_i}$ is the root system of $\ggo_i=\hg_i\oplus\qg_i$ relative to a maximal torus $\tg_i\subset\hg_i$, and so 
$$
\Delta_\qg=\Delta_{\qg_1}\sqcup\dots\sqcup\Delta_{\qg_s}, \qquad 
\Pi_\qg=\Pi_{\qg_1}\sqcup\dots\sqcup\Pi_{\qg_s}, \qquad \Delta_\hg=\Delta_{\hg_1}\sqcup\dots\sqcup\Delta_{\hg_s}.  
$$  
It follows from \eqref{slope} that  
$$
\zg(\kg)\subset\zg(\hg_f)=\zg(\hg_1)\oplus\dots\oplus\zg(\hg_s)=\zg(\kg)\oplus\ag_f, 
$$ 
and the possible embeddings of $\zg(\kg)$ on the direct sum can have diverse and complicated slopes.  

\begin{definition}\label{strict-def}
A C-space $M=G/K$ is called {\it strict} if either $M$ is a torus or $\kg_{\hg_i}\ne 0$ and $\hg_i$ is not contained in $\kg$ for all $i$, where $\hg_f=\hg_1\oplus\dots\oplus\hg_s$ as in \eqref{decflags}.  
\end{definition}

In other words, $M=G/K$ is strict when it is not the product of a flag manifold and a semisimple Lie group (see the proof below). 
   
\begin{proposition}\label{decC} 
As a manifold, any C-space as in \eqref{gprime} admits the following decomposition: 
\begin{equation*}
M=G/K=\widetilde{G}/\widetilde{H}\times \underline{G}/\underline{K}\times\overline{G}, 
\end{equation*}
where $\widetilde{G}/\widetilde{H}$ is a flag manifold, $\underline{G}/\underline{K}=\underline{G}_f/\underline{K}\times T^z$ is a strict C-space and $\overline{G}$ is a compact semisimple Lie group.  
\end{proposition}

\begin{remark}\label{decC-rem}  
$\Pi_1(M)=\Pi_1(\underline{G}/\underline{K})$ and $b_1(M)=b_1(\underline{G}/\underline{K})$, since $\widetilde{G}/\widetilde{H}$ and $\overline{G}$ have both finite fundamental group and zero first Betti number.  On the other hand, $b_2(M)=\dim{\zg(\widetilde{\hg}})+b_2(\underline{G}/\underline{K})$, as $b_2(\overline{G})=0$ (see \S\ref{dRapp-sec}).  
\end{remark}

\begin{proof}
After reordering, we can assume that for some $1\leq t\leq u\leq s$, 
\begin{align*}
&\kg\subset \hg_1\oplus\dots\oplus\hg_u,  \qquad \kg_{\hg_i}\ne 0, \quad\forall i=1,\dots,u,\\ 
&\hg_1\oplus\dots\oplus\hg_t=\kg\cap(\hg_1\oplus\dots\oplus\hg_s),  
\end{align*}
which implies that 
$$
\kg=\hg_1\oplus\dots\oplus\hg_t\oplus\underline{\kg}, \qquad\mbox{where}\quad\zg(\underline{\kg})\subset\zg(\hg_{t+1})\oplus\dots\oplus\zg(\hg_u), 
$$
$\zg(\underline{\kg})_{\zg(\hg_i)}\ne 0$ and none $\zg(\hg_i)$ is contained in $\zg(\underline{\kg})$ for all $i=t+1,\dots,u$.  Note that if we set 
$$
\underline{\hg}:=\hg_{t+1}\oplus\dots\oplus\hg_u
$$ 
then $\zg(\underline{\kg})\ne\zg(\underline{\hg})$ whenever $\underline{\hg}\ne 0$.  

In this way, $\hg_i$ is necessarily abelian (i.e., $G_i/H_i$ is full flag) for all $u+1\leq i\leq s$ and as a manifold, $M=G/K$ decomposes as stated in the proposition, where $\widetilde{G}/\widetilde{H}:=G_1/H_1\times\dots\times G_t/H_t$ is a flag manifold, $\overline{G}:=G_{u+1}\times\dots\times G_s$ is a semisimple Lie group and $\underline{G}/\underline{K}:=G_{t+1}\times\dots\times G_u/\underline{K}\times T^z$ is a strict C-space which fibers over the flag $G_{t+1}\times\dots\times G_u/\underline{H}$, where $\underline{\hg}$ is as defined above.   
\end{proof}

\begin{remark}\label{decC-rem2}  
It follows from the above proof that the Lie algebras of $\widetilde{G}$, $\underline{G}_f$ and $\overline{G}$ are respectively given by 
$$
\widetilde{\ggo}=\bigoplus_{\zg(\hg_i)\subset\zg(\kg)}\ggo_i,  \qquad 
\underline{\ggo}_f=\bigoplus\limits_{\substack{\zg(\hg_j)\not\subset\zg(\kg),\\ \kg_{\hg_j}\ne 0}}\ggo_j, \qquad  
\overline{\ggo}=\bigoplus\limits_{\kg_{\hg_k}=0}\ggo_k,  
$$
where $\ggo_f=\ggo_1\oplus\dots\oplus\ggo_s$ is the decomposition in simple ideals.  Note that $\zg(\hg_i)\subset\zg(\kg)$ if and only if $\hg_i\subset\kg$.  After reordering, we can assume that for $1\leq t\leq u\leq s$, 
$$
\widetilde{\ggo}=\ggo_1\oplus\dots\oplus\ggo_t, \qquad \underline{\ggo}=\ggo_{t+1}\oplus\dots\oplus\ggo_u, \qquad 
\overline{\ggo}=\ggo_{u+1}\oplus\dots\oplus\ggo_s.   
$$
In particular, $\ag\subset \zg(\hg_{t+1})\oplus\dots\oplus\zg(\hg_s)$ and for any $i=u+1,\dots,s$, the Lie algebra $\hg_i$ is necessarily abelian, i.e., $\zg(\hg_i)=\hg_i=\tg_i$ is a maximal abelian subalgebra of $\ggo_i$ and $G_i/H_i$ is full flag.  
\end{remark}

\subsection{Complex structures}\label{comp-sec} 
We now show that any even dimensional C-space indeed admits invariant complex structures (see Theorem \ref{C}).  Given a C-space $M=G/K$ and $\hg$ as in Definition \ref{Csp-def}, we recall from \eqref{reddec} the $Q$-orthogonal decompositions $\zg(\hg)=\zg(\kg)\oplus\ag$, 
$$
\ggo= \kg\oplus\pg, \qquad \pg= \ag\oplus\qg,   \qquad 
\qg^c=\sum_{\alpha\in\Delta_\qg}\CC E_\alpha,   
$$
where $\Delta_\qg$ is the set of complementary roots of the flag $F=G_f/H_f$ defined by $\hg$, i.e., $\Delta=\Delta_\qg\sqcup\Delta_{[\kg,\kg]}$, where $\Delta$ is the root system of the semisimple Lie algebra $\ggo_f$ relative to the maximal torus $\tg_f=\tg_{[\kg,\kg]}\oplus\zg(\hg_f)$ of $\ggo_f$ (see \S\ref{flag-sec}).  Note that $[\ag,\ag]=0$ and $[\ag,\qg]\subset\qg$.  We consider the $G$-invariant complex structures on $M=G/K$ given by 
\begin{equation}\label{J}
J=(J_\ag,J_\qg)=\left[\begin{matrix} J_\ag&0\\0&J_\qg\end{matrix}\right], 
\end{equation}
where $J_\ag:\ag\rightarrow\ag$ is any linear map such that $J_\ag^2=-I_\ag$ (in particular, $\dim{\ag}$ is even) and $J_\qg:\qg\rightarrow\qg$ is the complex structure on the flag $F=G_f/H_f$ attached to some invariant ordering $\Delta_\qg^+$ of $\Delta_\qg$, i.e., $J_\qg E_\alpha = \pm\im E_\alpha$ for all $\alpha\in\pm\Delta_\qg^+$ (see \S\ref{comp}).  Since 
$$
\pg^{1,0}=\ag^{1,0}\oplus \bigoplus_{\alpha\in\Delta_\qg^+}\CC E_\alpha, 
$$
the integrability of $J$ follows from the fact that $\Delta_\qg^+$ is an invariant ordering (see \S\ref{comp} and \S\ref{igs}).  

Note that the definition of $J=(J_\ag,J_\qg)$ strongly depends on the chosen $\hg$ and that the Tits fibration over the flag $F$ defined by $\hg$ given in \eqref{fib} becomes a holomorphic fibration: 
\begin{equation}\label{holfib}
(A,J_\ag)\longrightarrow (M=G/K,J) \longrightarrow (F,J_\qg),   
\end{equation}
that is, $J=(J_\ag,J_\qg)$ is projectable along some of the Tits fibrations attached to the C-space.  Conversely, as shown in \cite{NiWll}, $J=(J_\ag,J_\qg)$ determines $\hg$ as follows,
\begin{equation}\label{JaJq} 
\bautg(J)=\hg=\kg\oplus\ag, \qquad\forall J=(J_\ag,J_\qg),
\end{equation}
where $\bautg(J)$ is the Lie algebra of the group $\Aut(G/K)\cap\Bihol(M,J)$ of all biholomorphic automorphisms (see  \eqref{baut}).  

\begin{theorem}\label{C2}\cite{NiWll} Let $M=G/K$ be a compact homogeneous space, where $G$ and $K$ are connected, with $Q$-orthogonal reductive decomposition $\ggo=\kg\oplus\pg$.  
\begin{enumerate}[{\rm (i)}] 
\item \cite[Theorem 1.1]{NiWll} If $J:\pg\rightarrow\pg$ is a $G$-invariant complex structure on $M$, then the subalgebra $\hg=\bautg(J)$ satisfies all the conditions in Definition \ref{Csp-def}.  In particular, $M=G/K$ is a C-space of even dimension.  

\item \cite[Theorem 1.3]{NiWll} Every $G$-invariant complex structure on a C-space $M=G/K$ is of the form $J=(J_\ag,J_\qg)$ (see \eqref{J}) for some subalgebra $\hg$ as in Definition \ref{Csp-def}.  
\end{enumerate}
\end{theorem}

Let $\hca$ denote the set of all subalgebras $\hg\subset\ggo$ for which the three conditions in Definition \ref{Csp-def} hold.  According to Theorem \ref{C2}, (ii), the space $\cca^G$ of all $G$-invariant complex structures on a C-space $M=G/K$ is given by 
$$
\cca^G=\bigcup_{\hg\in\hca} \cca^G_{\hg}, \qquad 
$$
where $\cca^G_{\hg}$ denotes the space all complex structures defined as in \eqref{J} by using $\hg$.  

For a fixed $\hg\in\hca$, the set of all possible $J_\qg$ is finite, say given by $\{J_1,\dots,J_k\}$, where $k$ is the number of chambers in \eqref{ch} (see \eqref{numberk}), and if $\dim{\ag}=2d$, then the space of all $J_\ag$ is parametrized by the $2d^2$-dimensional manifold 
$$
\Gl_{2d}(\RR)/\Gl_d(\CC),
$$
which has two connected components.  This implies that     
\begin{equation}\label{Ccal}
\cca^G_\hg=\cca^G_1\sqcup\dots\sqcup\cca^G_k, \qquad \dim{\cca^G_i}=2d^2,  
\end{equation}
and each $\cca^G_i$ is diffeomorphic to $\Gl_{2d}(\RR)/\Gl_d(\CC)$, so $\cca^G_i$ consists of two connected components.  

Let $M=G/K$ be a C-space with $Q$-orthogonal reductive decomposition $\ggo=\kg\oplus\pg$.  We set $\pg_0:=C_\ggo(\kg)\cap\pg$, the trivial part of the isotropy representation, hence $N_\ggo(\kg)=\kg\oplus\pg_0$ and $\dim{\pg_0}=\dim{N_G(K)/K}$, and consider the set 
$$
\aca:=\left\{ \ag\subset\pg_0:[\ag,\ag]=0, \; \dim{\ag}=\rank(\ggo)-\rank(\kg)\right\}.  
$$   
Note that $\ag\in\aca$ if and only if $\ag$ is maximal abelian in $\pg_0$, if and only if $\tg_\kg\oplus\ag$ is a maximal torus of $\ggo$, where $\tg_\kg$ is any maximal torus of $\kg$.  We have seen that there is an injective map $\hca\rightarrow\aca$, $\hg\mapsto\ag$, where $\hg=\kg\oplus\ag$, or $\zg(\hg)=\zg(\kg)\oplus\ag$, as in \eqref{reddec} (see also \eqref{tdec}).  This map is also surjective, defining a bijection between $\hca$ and $\aca$.  Indeed, for any $\ag\in\aca$ we set $\hg:=\kg\oplus\ag$, which contains $\kg$, $[\hg,\hg]=[\kg,\kg]$ and since it also contains the maximal torus $\tg_\kg\oplus\ag$ of $\ggo$, we obtain that $\hg$ is the centralizer of a torus, concluding that $\hg\in\hca$.      

In the case when 
\begin{equation}\label{kreg2}
\dim{\pg_0}=\rank(\ggo)-\rank(\kg), 
\end{equation}
we obtain that $\aca$ is a singleton, so $\hca=\{\hg\}$, where $\hg=N_\ggo(\kg)$, and $\cca^G=\cca^G_\hg$, i.e., any $G$-invariant complex structure is one of $J=(J_\ag,J_\qg)$ as in \eqref{J} for this single $\hg$.   

\begin{proposition}\label{bihol-h}
For any $\hg,\hg'\in\hca$, there exists a $G$-equivariant diffeomorphism of $M=G/K$ such that $\cca^G_{\hg'}=f^*\cca^G_\hg$.     
\end{proposition}

\begin{remark}\label{bihol-rem}
In particular, any complex structure $J\in\cca^G_\hg$ is biholomorphic to some $J'\in\cca^G_{\hg'}$ and viceversa.  On the other hand, the classification up to biholomorphism among a single $\cca^G_\hg$ is only completely understood in the Lie group case (see \cite{Ptt}) and for flag manifolds (see \S\ref{comp}).  Partial results in the case of a C-space are given in \S\ref{ms} below.  
\end{remark}

\begin{proof}
Since $\tg_\kg\oplus\ag$ and $\tg_\kg\oplus\ag'$ are both maximal torus of $N_\ggo(\kg)$, there exists $x\in N_G(K)$ such that $\Ad(x)(\tg_\kg\oplus\ag)=\tg_\kg\oplus\ag'$.  If $f\in\Diff(M)$ is the $G$-equivariant diffeomorphism of the C-space $M=G/K$ defined by $\Ad(x)$, then it is easy to see that the $G$-invariant complex structure $J':=f^*J$ on $M$ is of the form \eqref{J} for $\ag'$ and the root system of $\ggo$ attached to $\tg_\kg\oplus\ag'$, that is, $J'\in\cca^G_{\hg'}$.  
\end{proof}

\begin{remark}\label{hyper-rem}
If one is only interested in the complex geometry of a C-space, then it is enough to fix a Tits fibration over a flag manifold to obtain all the invariant complex structures up to biholomorphism (see Remark \ref{bihol-rem}).  However, since all the possible $J_\qg$ commute for a fixed base (i.e., a fixed $\hg$, see \eqref{J}), in order to obtain invariant hypercomplex structures, it is necessary to let the base vary.  See \cite{DmtTsn, BdlMrc} for the classification of C-spaces admitting an invariant hypercomplex structure.    
\end{remark}

The mere object of the following terminology is to simplify the presentation.  

\begin{definition}\label{compC-def} 
A C-space $M=G/K$ endowed with a $G$-invariant complex structure is called a {\it complex C-space}.  
\end{definition}

We recall from \eqref{reddec2} the reductive decomposition of a C-space 
$$
M=G/K=G_f/K\times T^z
$$ 
as in \eqref{gprime}, where $\ag=\zg(\ggo)\oplus\ag_f$.  If, after reordering, $\overline{\ggo}=\ggo_{u+1}\oplus\dots\oplus\ggo_s$ as in Remark \ref{decC-rem2}, then 
\begin{equation}\label{deca}
\ag=\zg(\ggo)\oplus\underline{\ag}_f\oplus\tg_{u+1}\oplus\dots\oplus\tg_s, \quad 
\underline{\ag}=\zg(\ggo)\oplus\underline{\ag}_f, \quad
\overline{\ag}=\tg_{u+1}\oplus\dots\oplus\tg_s, \quad 
\ag_f=\underline{\ag}_f\oplus\overline{\ag},
\end{equation}
where $\tg_i$ is the maximal torus of $\ggo_i$, and $\qg=\widetilde{\qg}\oplus\underline{\qg}\oplus\overline{\qg}$, where
$$
\pg=\widetilde{\pg}\oplus\underline{\pg}\oplus\overline{\pg}, \qquad 
\widetilde{\pg}=\widetilde{\qg}=\qg_1\oplus\dots\oplus\qg_t, 
$$
$$
\underline{\pg}=\underline{\ag}\oplus\underline{\qg}, \quad 
\underline{\qg}=\qg_{t+1}\oplus\dots\oplus\qg_u, \qquad 
\overline{\pg}=\overline{\ag}\oplus\overline{\qg}, \quad 
\overline{\qg}=\qg_{u+1}\oplus\dots\oplus\qg_s.
$$
Note that the first product from the left in the decomposition as in Proposition \ref{decC} given by 
$$
M=\underbrace{\widetilde{G}/\widetilde{H}}_{{\rm flag}} \times \underbrace{\underline{G}_f/\underline{K}\times T^z}_{{\rm strict}}\times \underbrace{G_{u+1}\times\dots\times G_s}_{{\rm semisimple\; group}},
$$
is always a product of complex manifolds on any complex C-space.  The remaining products are not of complex manifolds in general. 

\begin{definition}\label{Cirr} 
A complex C-space $(M=G_f/K\times T^z,J)$ without flag factor is called {\it reducible} if there exists a $J_\ag$-invariant decomposition $\ag=\ag_1\oplus\ag_2$ such that 
$$
\ag_1=\zg(\ggo)_1\oplus\underline{\ag}_f\oplus\tg_{i_1}\oplus\dots\oplus\tg_{i_{s_1}}, \qquad 
\ag_2=\zg(\ggo)_2\oplus\tg_{j_1}\oplus\dots\oplus\tg_{j_{s_2}},
$$
$\zg(\ggo)=\zg(\ggo)_1\oplus\zg(\ggo)_2$ and $\{ u+1,\dots,s\}=\{i_1,\dots,i_{s_1}\}\sqcup \{j_1,\dots,j_{s_2}\}$.  Otherwise, we say that $(M=G/K,J)$ is {\it irreducible}.  
\end{definition}

In particular, any reducible C-space is a product of complex C-spaces as a complex manifold.  Moreover, 
any complex C-space without flag factor is the product, as a complex manifold, of irreducible complex C-spaces which are all compact Lie groups endowed with left-invariant complex structures, excepting at most one of them (i.e., the one containing $\underline{\ag}$ if nonzero).     

\begin{example}\label{LG} ({\it Lie groups}).  
We assume here that the C-space is a Lie group, i.e., 
$$
M=G=\overline{G}\times T^z, 
$$ 
where $\overline{G}$ is a semisimple Lie group (cf.\ Remark \ref{ext}, (ii) and Proposition \ref{decC}). Thus   
$$
\ggo_f=\overline{\ggo},  \qquad \ag_f=\overline{\tg}, \qquad \ag=\overline{\tg}\oplus\zg(\ggo), \qquad \qg=\overline{\qg}, 
$$
where $\overline{\tg}=\hg_f=\zg(\hg_f)$ is the maximal torus of $\overline{\ggo}$ and $\overline{\ggo}=\overline{\tg}\oplus\overline{\qg}$.  Note that the base of the fibration is the full flag $F=\overline{G}/\overline{T}$ and the fiber is the maximal torus $T=\overline{T}\times T^z$ of $G$.  It was proved by Pittie \cite{Ptt} that complex structures of the form $J=(J_\ag,J_\qg)$ as in \eqref{J} for some maximal torus $\ag$ of $\ggo$, which were previously found independently by Samelson \cite{Sml} and Wang \cite{Wng}, exhausts the set of left-invariant complex structures on $G$ (cf.\ Theorem \ref{C2}, (ii)).  Since the base is full flag, the biholomorphism class of the complex manifold $(G,J)$ does not depend on the choice of the maximal abelian subalgebra $\ag$ nor of $\Delta^+$, so one can fix $J_\qg$.  Two different $J_\ag$'s produce biholomorphic complex structures if and only if they belong to the same orbit of certain finite subgroup of $\Aut(G)$ (see \cite{Ptt} for further details).  This implies that the moduli space $\cca^G/\simeq$ of all left-invariant complex structures on $G$ up to biholomorphism also depends on $2d^2$ parameters (recall that $\dim{\ag}=2d$).  
\end{example}

\subsection{Moduli space}\label{ms}
On a C-space with $G$ semisimple, two complex structures $J,I\in\cca^G$ are said to be {\it equivalent} ($J\simeq I$ for short) if there exists $\psi\in\Aut(G/K)\subset\Diff(M)$ (see \S\ref{gauge}) such that $\psi^*J=I$ (if $\psi$ is in addition inner, i.e., $\psi=I_a$ for some $a\in N_G(K)$, then $J=\Ad(a)I\Ad(a)^{-1}$ as linear operators of $\pg$).  If $J\simeq I$, then the complex manifolds $(M,J)$ and $(M,I)$ are clearly biholomorphic, although the converse may fail to hold.  

We consider $\Aut(\Delta)\subset\glg(\tg)$, the automorphism group of the root system of $\ggo$, and the finite groups 
$$
A(M):=\{\psi\in\Aut(\Delta): \psi\zg(\hg)=\zg(\hg), \; \psi\ag=\ag\}, 
$$
$$ 
A(M)_i:=\{\psi\in A(M): \psi Ch(J_i)=Ch(J_i)\},  \qquad i=1,\dots,k,
$$
where $Ch(J_i)$ is the corresponding chamber in \eqref{ch} (see \eqref{Ccal}).  Note that $A(M)_i$ acts on $\cca^G_i$ for all $i=1,\dots,k$ (see \S\ref{comp}).  

Given two complex structures $J=(J_\ag,J_\qg)$ and $I=(I_\ag,I_\qg)$, if there exists $\psi\in A(M)$ such that $\psi Ch(J_\qg)=Ch(I_\qg)$ and $\psi^*J_\ag=I_\ag$, then $J\simeq I$.  In this way, it follows from (iv) in \S\ref{comp} and Proposition \ref{bihol-h} that the moduli space of $G$-invariant complex structures on a C-space $M=G/K$ up to equivalence is given by   
$$
\cca^G/\simeq \; =\cca^G_\hg/\simeq \; = \cca^G_{i_1}/A(M)_{i_1} \sqcup\dots\sqcup\cca^G_{i_l}/A(M)_{i_l}, 
$$
where
$$ 
\{ Ch(J_{i_1}),\dots,Ch(J_{i_l})\}= \{ Ch(J_1),\dots,Ch(J_k)\}/A(M).  
$$
Note that the moduli space $\cca^G/\simeq$ still depends on $2d^2$ parameters.

\subsection{First Chern class}\label{fCc-sec}
It follows from \cite{Ksz} (see also \cite[Corollary 4.13)]{AlkPrl} and \cite[Section 2]{Grn}) that the first Chern class of a complex C-space is given by 
$$
c_1(M,J)= [\sigma]\in H^2(M), 
$$ 
where the closed $2$-form $\sigma\in\Omega^2(M)^G \equiv(\Lambda^2\pg^*)^K$ is defined by  
$$
\sigma(X,Y) := \tr{J\ad{[X,Y]}|_\pg}-\tr{\ad{J[X,Y]_\pg}|_\pg}, \qquad\forall X,Y\in\pg.    
$$
We therefore obtain that $\sigma=2\im\rho_{J_\qg}([\cdot,\cdot])$, where $J=(J_\ag,J_\qg)$ and $\rho_{J_\qg}=\sum_{\alpha\in\Delta_\qg^+}\alpha$ is the Koszul form of the flag $(G_f/H_f,J_\qg)$ as in \eqref{koszul}.  Indeed, it follows from \eqref{adH} that $\tr{\ad{J[X,Y]_\pg}|_\pg}=0$, and 
$$
\tr{J\ad{[X,Y]}|_\pg} = 2\im\sum_{\alpha\in\Delta_\qg^+} \alpha([X,Y]) = 2\im\rho_{J_\qg}([X,Y]_{\zg(\hg_f)}).  
$$
In particular, $c_1(M,J)=0$ if and only if $\rho_{J_\qg}|_{\zg(\kg)}=0$.  

We define the {\it Koszul vector} of $(M=G/K,J)$ by 
\begin{equation}\label{koszul2}
Z_{J_\qg}:= \sum_{\alpha\in\Delta_\qg^+} \im H_\alpha\in\zg(\hg_f)=\zg(\kg)\oplus\ag_f, 
\end{equation}
that is, the real version of the Koszul form.  Note that 
$$
\im\rho_{J_\qg}(Z)=Q(Z,Z_{J_\qg}),  \qquad\forall Z\in\zg(\hg_f)\oplus\zg(\ggo),
$$
from which follows that the first Chern class of a complex C-space is always positive semi-definite, and it is positive definite only for flag manifolds.  Furthermore, the following characterization follows from \eqref{dR2} below.  

\begin{lemma}\label{c10}\cite[Theorem 2]{Grn}
A complex C-space has $c_1(M,J)=0$ if and only if 
$$
Z_{J_\qg}\in\ag_f. 
$$ 
\end{lemma}

It follows from \eqref{Ccal} that the space $\cca^G_0\subset\cca^G_\hg$ of all $G$-invariant complex structures on a C-space $M=G/K$ which are projectable on the flag $F=G_f/H_f$ defined by a given $\hg$ (see \eqref{holfib}) and have zero first Chern class is given by     
\begin{equation}\label{Cc10}
\cca^G_0=\cca^G_{i_1}\sqcup\dots\sqcup\cca^G_{i_l} = \bigsqcup_{Z_{J_i}\in\ag_f}\cca^G_i, \qquad \dim{\cca^G_{i_j}}=2d^2.  
\end{equation} 
Note that $\cca^G_0=\cca^G_\hg$ if $K$ is either semisimple or trivial (i.e., $\ag_f=\zg(\hg_f)$) and that $\cca^G_0=\emptyset$ in the K\"ahler case (i.e., $\ag_f=0$, see Remark \ref{ext}, (iii)).  In the case when $0\ne\zg(\kg),\ag_f$, one generically (see Remark \ref{lh-rem}) has that $\cca^G_0=\emptyset$ among the set of all C-spaces fibering over a fixed flag $F$, as this holds if and only if $Z_{J_i}\notin\ag_f$ for all $i=1,\dots,k$.      

\begin{example}\label{su5-5}
Consider a C-space $M^{20}=\SU(5)/K$ as in Example \ref{su5}.  We have that 
$$
\zg(\hg_f)=\{(a,b,c,d,d):a+b+c+2d=0\}, \qquad \zg(\hg_f)=\RR Z\oplus\ag, 
$$ 
$\dim{\ag}=2$, $\Delta_{\hg_f}=\{\pm\alpha_{45}\}$ and $\pg=\ag\oplus\qg$, where $\qg=\qg_1\oplus\dots\oplus\qg_6$.  The usual choice $\Delta_\qg^+=\Delta^+\setminus\{\alpha_{45}\}$ therefore determines the complex structure $J_\qg$ on $F=G_f/H_f$ with $Z_{J_\qg}=(4,2,0,-3,-3)$, and
thus $c_1(M,J)=0$, where $J=(J_\ag,J_\qg)$, if and only if $(4,2,0,-3,-3)\in\ag$.  On the other hand, the invariant ordering  
$$
\{\alpha_{23}, \alpha_{12}, -\alpha_{15}, \alpha_{13}, -\alpha_{25}, -\alpha_{14}, -\alpha_{35}, -\alpha_{24}, -\alpha_{34}\}
$$
determines another complex structure $J_2$ on $F=G_f/H_f$ with $Z_{J_2}=(0,-2,-4,3,3)$.  Note that the number $k$ of complex structures on the flag $F$ (see \eqref{numberk}) is not easy to compute, we only know that $k<32=k(3,6)$ since the general position condition does not hold.  
\end{example}

\begin{example}\label{b21-comp} ({\it Generalized Calabi-Eckmann manifolds})
Given two flag manifolds $F_1=G_1/H_1$ and $F_2=G_2/H_2$ with $b_2(F_1)=b_2(F_2)=1$ as in Example \ref{b21} ($F_2=S^1$ is also allowed), we consider the C-space $M=G/K$, where $G=G_f=G_1\times G_2$ and $\kg=[\hg_1,\hg_1]\oplus[\hg_2,\hg_2]$, that is, $\zg(\kg)=0$ and $\ag=\ag_f=\zg(\hg_f)=\zg(\hg_1)\oplus\zg(\hg_2)$, so 
$$
\pg=\ag\oplus\qg_1\oplus\qg_2, \qquad \dim{\ag}=2. 
$$
The complex structures $J=(J_\ag,J_\qg)$ all have $J_\qg=J_{\qg_1}+J_{\qg_2}$, where $J_{\qg_i}$ is the unique invariant complex structure on $F_i$, and 
$$
J_\ag = \left[\begin{matrix} 
a&-\tfrac{a^2+1}{b}\\ b&-a 
\end{matrix}\right], \qquad a,b\in\RR, \quad b\ne 0,  
$$
in terms of the basis $\{Z_{\gamma_1},Z_{\gamma_2}\}$ of $\ag$, $Z_{\gamma_i}=(\im H_{\gamma_i})_{\zg(\hg_i)}$ (see Example \ref{b21}).  We note that $M=G_1/[H_1,H_1]\times G_2/[H_2,H_2]$ as a manifold but this is never a product as a complex manifold.  Since $\ag_f=\zg(\hg_f)$ (i.e., $K$ is semisimple), we have that $c_1(M,J)=0$ for all $a,b$ by Lemma \ref{c10}, i.e., $\cca_0=\cca$.  The Koszul vector is given by 
\begin{align*}
Z_J=Z_{J_{\qg_1}}+Z_{J_{\qg_2}} =& \unm(\dim{\mg_{1,1}}+\dots+r_1\dim{\mg_{1,r_1}})Z_{\gamma_1} \\ 
& +\unm(\dim{\mg_{2,1}}+\dots+r_2\dim{\mg_{2,r_2}})Z_{\gamma_2},   
\end{align*}
where $\qg_i=\mg_{i,1}\oplus\dots\oplus\mg_{i,r_i}$ is the decomposition in $\Ad(H_i)$-irreducible components.  Note that if $F_2=S^1$, then $\ggo_2=\hg_2=\zg(\hg_2)=\RR$, $\qg_2=0$ and $Z_{\gamma_2}=0$. 
\end{example}

\begin{remark}
In the above example, when the two flags are 
$$
F_i=G_i/H_i=\SU(n_i+1)/\U(n_i)=\CC P^{n_i}, \qquad n_i\in\ZZ_{\geq 0}, \quad i=1,2, 
$$
the C-space $M=G/K=S^{2n_1+1}\times S^{2n_2+1}$ is the well-known {\it Calabi-Eckmann} manifold, which is called {\it Hopf} manifold if $n_2=0$ (i.e., $F_2=S^1$).  The case when $F_1$ and $F_2$ are both irreducible Hermitian symmetric spaces (i.e., $r_1=r_2=1$) was studied in \cite{Pds}.   
\end{remark}

More general products of $S^1$-bundles than the given in Example \ref{b21-comp} have been studied in \cite{Crr}.  We note that the name {\it generalized Calabi-Eckmann maniold} is also often used in the much more general context given in \cite[Theorem 2.2]{FinGrn2}.

We consider for future use the full open cone in $\zg(\hg_f)$ attached to a $G$-invariant complex structure defined by 
\begin{equation}\label{cone}
C(J_\qg):=\la Z_\alpha:\alpha\in\Delta_\qg^+\ra_{\RR_{>0}} 
=\la Z_{\alpha_1},\dots,Z_{\alpha_r}\ra_{\RR_{>0}} 
=\la Z_\alpha:\alpha\in\Pi_\qg\ra_{\RR_{>0}},
\end{equation}
where 
$$
Z_\alpha:=(\im H_\alpha)_{\zg(\hg_f)}, \qquad \alpha\in\Delta_\qg.
$$
The description on the right is usually the easiest way to compute $C(J_\qg)$ since $\{ Z_\alpha:\alpha\in\Pi_\qg\}$ is a basis of $\zg(\hg_f)$, i.e., $|\Pi_\qg|=\dim{\zg(\hg_f)}$.  Note that 
$$
Z_{J_\qg}= \sum_{\alpha\in\Delta_\qg^+} Z_\alpha\in C(J_\qg).  
$$  

\begin{remark}
We will show in \S\ref{bal-sec} that $(M=G/K,J)$ admits a $G$-invariant balanced Hermitian metric if and only if the cone $C(J_\qg)$ meets $\zg(\kg)$.  
\end{remark}

\subsection{Isotropy representation}\label{isot-sec} 
It is natural to ask for a description of the set of all invariant almost-complex structures on a given C-space.  To this aim, a deep understanding of both the isotypical and irreducible components of the isotropy representation is crucial.  

The isotropy representation $\pg$ of a C-space $M=G/K=G_f/K\times T^z$ as in \eqref{gprime} fibering over a flag manifold $F=G_f/H_f$ decomposes in $\Ad(K)$-invariant subspaces as
\begin{equation}\label{isotrep}
\pg=\ag\oplus\mg_1\oplus\dots\oplus\mg_r, \qquad \ag=\zg(\ggo)\oplus\ag_f, \qquad \qg=\mg_1\oplus\dots\oplus\mg_r, 
\end{equation}
where the subspaces $\mg_j$ are defined as in \S\ref{isot} for $F=G_f/H_f$, that is, $\alpha_1,\dots,\alpha_r$ is a set of central roots and 
$$
\mg_j=\bigoplus_{\alpha\in P_j} \qg_{\alpha}, \qquad P_j:=\left\{\alpha\in P:\alpha|_{\zg(\hg_f)}=\alpha_j|_{\zg(\hg_f)}\right\}, \qquad \Delta_\qg=P\sqcup -P.
$$ 
Thus $\dim{\mg_j}$ is even and $\dim{\mg_j}= 2$ if and only if $\mg_j=\qg_{\alpha_j}$ and $P_j=\{\alpha_j\}$.  Recall that the subspaces $\mg_j$ are actually $\Ad(H_f)$-invariant, $\Ad(H_f)$-irreducible and pairwise $\Ad(H_f)$-inequivalent.  Note that the action of $K$ on $\ag$ is trivial.  

\begin{proposition}\label{isotdec} 
\hspace{1cm} 
\begin{enumerate}[{\rm (i)}] 
\item If $\dim{\mg_j}\geq 4$, then $\mg_j$ is $\Ad(K)$-irreducible.  

\item Two subspaces $\mg_j,\mg_k$, $j\ne k$ are equivalent as $\Ad(K)$-representations if and only if $\mg_j=\qg_{\alpha_j}$, $\mg_k=\qg_{\alpha_k}$ and $\alpha_j|_{\zg(\kg)}=\alpha_k|_{\zg(\kg)}$.  

\item The trivial $\Ad(K)$-subrepresentation of $\pg$ is given by 
$$
\pg_0=\ag\oplus \bigoplus_{j\in I_0} \qg_{\alpha_j}, \qquad I_0:=\{1\leq j\leq r :  \mg_j=\qg_{\alpha_j}, \;\alpha_j|_{\zg(\kg)}=0\}.  
$$
\item The subspaces $\mg_1,\dots,\mg_r$ in \eqref{isotrep} are all $\Ad(K)$-irreducible (in particular, nontrivial) and pairwise inequivalent as $\Ad(K)$-representations if and only if the subspace $\zg(\kg)\subset\zg(\hg_f)$ and all the central roots $\alpha_j$ such that $\mg_j=\qg_{\alpha_j}$ satisfy the following:  
\begin{equation}\label{kreg}
\alpha_j|_{\zg(\kg)}\ne 0, \quad\mbox{and} \quad \alpha_j|_{\zg(\kg)}=\alpha_k|_{\zg(\kg)} \quad\mbox{if and only if}\quad \alpha_j=\alpha_k. 
\end{equation}
\end{enumerate}
\end{proposition}

\begin{remark}\label{ft-gen}
Condition \eqref{kreg} is generic, in the sense that it holds if the subspace $\zg(\kg)$ is not orthogonal to any of the finitely many vectors of $\zg(\hg_f)$ given by $Z_{\alpha_j}$, $Z_{\alpha_j}-Z_{\alpha_k}$.  We are considering here the Grassmannian topology given in Remark \ref{lh-rem}.  This was proved in \cite[Theorem B]{KrsWlfZar} in a more general context.  
\end{remark}

\begin{remark}
In particular, condition \eqref{kreg} holds if $\dim{\mg_j}\geq 4$ for all $j=1,\dots,r$.  Moreover, in the case when $\zg(\kg)=0$, condition \eqref{kreg} holds if and only if $\dim{\mg_j}\geq 4$ for all $j=1,\dots,r$.   
\end{remark}

\begin{proof}
We first prove part (i).  If $\dim{\mg_j}\geq 4$, then $\hg_f$ is not abelian and restricted to $[\hg_f,\hg_f]$, the representation $\mg_j$ is also irreducible, as any partial sum of $\qg_\alpha$ inside $\mg_j$ is always $\zg(\hg_f)$-invariant since $\alpha|_{\zg(\hg_f)}=\alpha_j$ for any such $\alpha$.  The $\Ad(K)$-irreducibility of $\mg_j$ therefore follows from $[\hg_f,\hg_f]=[\kg,\kg]\subset\kg$.  

Parts (ii) and (iii) follow from (i) and the fact that the action of $[\kg,\kg]$ is trivial on $\mg_j$ if $\dim{\mg_j}=2$, i.e., $\mg_j=\qg_{\alpha_j}$.  Finally, parts (ii) and (iii) implies (iv), concluding the proof.   
\end{proof} 

\begin{definition}\label{fibtype-def}
A C-space $M=G/K$ is called of {\it fibration type} if condition \eqref{kreg} holds.  
\end{definition} 

According to Proposition \ref{bihol-h}, the fibration type property is independent of the fibration $M\rightarrow F$ chosen, i.e., it is independent of the subalgebra $\hg$ in Definition \ref{Csp-def}, and it is generic by Remark \ref{ft-gen}.  

On a C-space $M=G/K$ of fibration type, any $G$-invariant almost-complex structure on $M$ leaves the decompositions of $\pg$ and $\qg$ in \eqref{isotrep} invariant (not the one for $\ag$ in general).  On the other hand, any $G$-invariant metric on $M=G/K$ is given by 
\begin{equation}\label{g}
g=g_\ag+g_\qg, 
\end{equation}
where $g_\ag$ is any inner product on $\ag$ and 
$$
g_\qg=(x_\alpha)_{\alpha\in\Delta_\qg}=(x_1,\dots,x_r) =x_1Q|_{\mg_1}+\dots+x_rQ|_{\mg_r},
$$ 
the inner product on $\qg$ given in \S\ref{riem} for the flag $F=G_f/H_f$.  In other words, any $G$-invariant metric is also projectable along some of the Tits fibrations in the fibration type case.  These metrics are actually $\Ad(H_f)$-invariant and are called {\it adapted} in \cite{Pds}.  Note that $(J,g)$ is compatible if and only if $(J_\ag,g_\ag)$ is compatible.  Furthermore, any $G$-invariant $2$-form on a C-space of fibration type is given by  
\begin{equation}\label{s}
\sigma=\sigma_\ag+\sigma_\qg,   
\end{equation} 
where $\sigma_\ag$ is any $2$-form on $\ag$ and 
$$
\sigma_\qg=(y_\alpha)_{\alpha\in\Delta_\qg}=(y_1,\dots,y_r),
$$ 
is the $2$-form on $\qg$ given by $\sigma_\qg(E_\alpha,E_{-\alpha})=y_\alpha\in\im\RR$ for any $\alpha\in\Delta_\qg$ and zero otherwise, for some numbers $y_\alpha$ such that $y_{-\alpha}=-y_\alpha$ and $y_\alpha=y_\beta$ whenever $\alpha|_{\zg(\hg)}=\beta|_{\zg(\hg)}$.  Equivalently, $\sigma_\qg(e_\alpha,f_\alpha)=\im y_\alpha$ for all $\alpha\in\Delta_\qg$ and zero otherwise.

\subsection{Extra $G$-invariant structures}\label{extra-sec}
We now consider the case when condition \eqref{kreg} does not hold.  After reordering, we can assume that for some $1\leq t\leq u\leq r$, 
\begin{equation}\label{ptilde}
\pg=\ag\oplus\mg_1\oplus\dots\oplus\mg_t \oplus\mg_{t+1}\oplus\dots\oplus\mg_u \oplus\mg_{u+1}\oplus\dots\oplus\mg_r,
\end{equation}  
where $\alpha_j|_{\zg(\kg)}=0$ if and only if $j\leq t$ and $\mg_j=\qg_{\alpha_j}$ if and only if $j\leq u$.  In particular, $\dim{\mg_j}\geq 4$ for all $j>u$, so if $\hg_f$ is abelian then $u=r$, and $t=u$ if $\zg(\kg)=0$.  According to Proposition \ref{isotdec}, 
$$
\pg_0=\ag\oplus\mg_1\oplus\dots\oplus\mg_t,
$$
and for $j,k>t$, $\mg_j$ is equivalent to $\mg_k$ if and only if $j,k\leq u$ and $\alpha_j|_{\zg(\kg)}=\alpha_k|_{\zg(\kg)}$.  We can therefore decompose in isotypical components as follows, 
\begin{equation}\label{mtilde}
\mg_{t+1}\oplus\dots\oplus\mg_u = \widetilde{\mg}_1\oplus\dots\oplus\widetilde{\mg}_v,
\end{equation}
where the $\widetilde{\mg}_i$ are the partial sums of $\mg_j$ with identical restriction $\alpha_j|_{\zg(\kg)}$.  

In this way, it follows from \eqref{ptilde} and \eqref{mtilde} that 
$$
\pg=\pg_0\oplus\pg_1\oplus\pg_2, \qquad  
\pg_1:=\widetilde{\mg}_1\oplus\dots\oplus\widetilde{\mg}_v, \qquad \pg_2:=\mg_{u+1}\oplus\dots\oplus\mg_r,
$$
hence the $G$-invariant almost complex structures are all of the following form: 
$$
\widetilde{J}=(J_0,J_{\pg_1},J_{\pg_2}), \qquad J_{\pg_1}=(J_1,\dots,J_v), 
$$
where 
\begin{enumerate}[{\small $\bullet$}] 
\item $J_0$ is any linear operator of $\pg_0$ such that $J_0^2=-I$, 

\item $J_i:\widetilde{\mg}_i\rightarrow\widetilde{\mg}_i$ satisfies that $J_i^2=-I$ and commutes with $I_i:\widetilde{\mg}_i\rightarrow\widetilde{\mg}_i$, defined by $I_ie_\alpha=f_\alpha$, $I_if_\alpha=-e_\alpha$ for any $\alpha\in P$ such that $\qg_\alpha\subset\widetilde{\mg}_i$, 

\item $J_{\pg_2}= (\epsilon_{u+1},\dots,\epsilon_{r})$ is as in \S\ref{alm-comp}, i.e., $J_{\pg_2}e_\alpha=\epsilon_j f_\alpha$ if $\qg_\alpha\subset\mg_j$, $j=u+1,\dots,r$.  
\end{enumerate}

Any $G$-invariant metric is given by:  
$$
g=g_{\pg_0}+g_1+\dots+g_v + g_{\pg_2}, \qquad g_{\pg_2}=(x_{u+1},\dots,x_r), 
$$
where $g_{\pg_0}$ is any inner product on $\pg_0$, $g_i$ is any inner product on $\widetilde{\mg}_i$ compatible with $I_i$ and $g_{\pg_2}$ is as in \S\ref{riem}.  A pair $(J,g)$ is therefore compatible if and only if the pairs $(J_{\pg_0},g_{\pg_0})$ and $(J_i,g_i)$, $i=1,\dots,v$ are so. We note that there are additional $G$-invariant metrics to the $\Ad(H_f)$-invariant ones given in \eqref{g} if and only if $t>1$, or if $u>t$ and at least one $\widetilde{\mg}_i$ has dimension $\geq 4$ (i.e., $\widetilde{\mg}_i$ is not $\Ad(K)$-irreducible).  

\begin{example}\label{su5-full}
A C-space $M^{22}=\SU(5)/K$, where $K$ is abelian and $\dim{K}=2$, fibers over the full flag 
$$
F^{20}=G_f/H_f=\SU(5)/\Se(\U(1)^5), \qquad  K\subset H_f=T^4=\Se(\U(1)^5), \qquad \dynkin[scale=2] A{oooo} ,
$$
and $\hg_f=\zg(\hg_f)=\{(a,b,c,d,e):a+b+c+d+e=0\}$.  If we consider 
$$
\kg:=\left\la (1,0,-\tfrac{1}{3},-\tfrac{1}{3},-\tfrac{1}{3}),(0,1,-\tfrac{1}{3},-\tfrac{1}{3},-\tfrac{1}{3})\right\ra_\RR, 
$$
then $\ag=\ag_f=\la Z_{\alpha_{34}}, Z_{\alpha_{45}}\ra_\RR$ and $\pg$ decomposes in isotypic components as
$$
\pg=\underbrace{\ag\oplus\qg_{34}\oplus\qg_{45}\oplus\qg_{35}}_{\pg_0} \oplus  
\underbrace{\qg_{12}}_{\widetilde{\mg}_1} \oplus 
\underbrace{\qg_{13}\oplus\qg_{14}\oplus\qg_{15}}_{\widetilde{\mg}_2}\oplus 
\underbrace{\qg_{23}\oplus\qg_{24}\oplus\qg_{25}}_{\widetilde{\mg}_3}.
$$
Note that $t=3$, $u=r=10$ and $v=3$.  The space of all $G$-invariant metrics therefore depends on $36+1+9+9=55$ parameters, while the space of metrics of the form $g=g_\ag+g_\qg$ as in \eqref{g} only depends on $3+10=13$ parameters.  
\end{example}

\begin{example}\label{su5-isot}
From the flag
$$
F^{10}=G_f/H_f=\SU(4)/\Se(\U(1)\times\U(2)\times\U(1)), \qquad \dynkin[scale=2] A{o*o} ,
$$
where $\zg(\hg_f)=\{(a,b,b,c):a+2b+c=0\}$ and $\Delta_{\hg_f}=\{\pm\alpha_{23}\}$, we can define the C-space $M^{12}=\SU(4)/\SU(2)$, i.e., $\zg(\kg)=0$.  Thus $\ag=\zg(\hg_f)=\la Z_{\alpha_{12}},Z_{\alpha_{34}}\ra_\RR$ and the decomposition in isotypical components is given by 
$$
\pg=\underbrace{\ag\oplus\qg_{14}}_{\pg_0} \oplus \underbrace{\qg_{12}\oplus\qg_{13}}_{\mg_2} \oplus \underbrace{\qg_{34}\oplus\qg_{24}}_{\mg_3}, 
$$
that is, $t=u=1$ and $r=3$.  In this case, the space of all $G$-invariant metrics depends on $10+1+1=12$ parameters and the space of projectable metrics on $3+3=6$ parameters. 
\end{example}

\begin{example}\label{su4-full-2}
Consider the C-space $M^{14}=\SU(4)/K$, $\dim{K}=1$, $\kg:=\RR(2,1,-1,-2)$, fibering over the full flag 
$$
F^{12}=G_f/H_f=\SU(4)/\Se(\U(1)^4), \qquad  H_f=T^3=\Se(\U(1)^4), \qquad \dynkin[scale=2] A{ooo} ,
$$
and set $\alpha:=\alpha_{12}$, $\beta:=\alpha_{23}$, $\gamma:=\alpha_{13}$.  We therefore have that
$$
\pg=\ag\oplus\underbrace{\qg_\alpha\oplus\qg_\gamma}_{\widetilde{\mg}_1} \oplus 
\underbrace{\qg_{\alpha+\beta}\oplus\qg_{\beta+\gamma}}_{\widetilde{\mg}_2}
\oplus\underbrace{\qg_{\beta}}_{\widetilde{\mg}_3}\oplus\underbrace{\qg_{\alpha+\beta+\gamma}}_{\widetilde{\mg}_4}, 
$$
where $\ag=\pg_0$ is the orthogonal complement of $\kg$ in $\tg$.  Any $G$-invariant almost-complex structure on $M$ is therefore given by $\widetilde{J}=(J_0,J_1,J_2,J_3,J_4)$, where $J_0:\ag\rightarrow\ag$, $J_0^2=-I$ and $J_i:\widetilde{\mg}_i\rightarrow\widetilde{\mg}_i$, $J_i^2=-I$, $[J_i,I_i]=0$ for all $i=1,\dots,4$.  Since $N_\ggo(\kg)=\tg$,  $\cca^G=\cca^G_{\tg}$.  

If we instead consider $\kg:=\RR(1,1,-1,-1)$, then $\ag=\ag_f=\la Z_\alpha, Z_\gamma\ra_\RR$ and 
$$
\pg=\underbrace{\ag\oplus\qg_\alpha\oplus\qg_\gamma}_{\pg_0} \oplus \underbrace{\qg_\beta\oplus\qg_{\alpha+\beta}\oplus\qg_{\beta+\gamma}\oplus\qg_{\alpha+\beta+\gamma}}_{\widetilde{\mg}_1}.  
$$
Thus any $G$-invariant almost-complex structure on $M$ is given by $\widetilde{J}=(J_0,J_1)$, where $J_0:\pg_0\rightarrow\pg_0$, $J_0^2=-I$, $J_1:\widetilde{\mg}_1\rightarrow\widetilde{\mg}_1$, $J_1^2=-I$ and $[J_1,I_1]=0$.  

If we define $J_1:=I_1$, then 
$$
\widetilde{\mg}_1^{1,0}=\ggo_\beta\oplus\ggo_{\alpha+\beta}\oplus\ggo_{\beta+\gamma}\oplus\ggo_{\alpha+\beta+\gamma}, 
$$ 
and hence $[\pg_0,\widetilde{\mg}_1^{1,0}]\subset\widetilde{\mg}_1^{1,0}$ by \eqref{eE} and $[\widetilde{\mg}_1^{1,0},\widetilde{\mg}_1^{1,0}]\subset\widetilde{\mg}_1^{1,0}$.  The integrability of $\widetilde{J}$ is therefore equivalent to $[\pg_0^{1,0}, \pg_0^{1,0}]_\pg\subset\pg_0^{1,0}$, which holds if and only if $J_0$ is a complex structure on the Lie algebra $\pg_0\simeq\sug(2)\oplus\sug(2)$.  
\end{example}

\begin{example}\label{su6-isot}
We consider the flag
$$
F^{28}=G_f/H_f=\SU(6)/\Se(\U(1)^2\times\U(2)\times\U(1)^2), \qquad \dynkin[scale=2] A{oo*oo} ,
$$
which has $\zg(\hg_f)=\{(a,b,c,c,d,e):a+b+2c+d+e=0\}$ and $\Delta_{\hg_f}=\{\pm\alpha_{34}\}$.  For the C-space $M^{30}=\SU(6)/K$, $\dim{K}=5$, where
$$
\zg(\kg):=\{(a,a,c,c,d,d):a+c+d=0\}, 
$$
we have that $\ag=\ag_f=\la Z_{\alpha_{12}}, Z_{\alpha_{56}}\ra_\RR$ and the decomposition in isotypical components is given by 
$$
\pg=\underbrace{\ag\oplus\qg_{12}\oplus\qg_{56}}_{\pg_0} \oplus 
\underbrace{\qg_{23}\oplus\qg_{24} \oplus\qg_{13}\oplus\qg_{14}}_{\widetilde{\mg}_1} \oplus 
\underbrace{\qg_{45}\oplus\qg_{35}\oplus\qg_{46}\oplus\qg_{36}}_{\widetilde{\mg}_2} \oplus  
\underbrace{\qg_{25}\oplus\qg_{15}\oplus\qg_{26}\oplus\qg_{16}}_{\widetilde{\mg}_3}. 
$$
Thus the $G$-invariant almost-complex structures on $M$ are all of the form $\widetilde{J}=(J_0,J_1,J_2,J_3)$, where $J_0:\pg_0\rightarrow\pg_0$, $J_0^2=-I$, $J_i:\widetilde{\mg}_i\rightarrow\widetilde{\mg}_i$, $J_i^2=-I$, $[J_i,I_i]=0$ for all $i=1,2,3$.  Concerning the integrability of $\widetilde{J}$, we note that if we set $J_i=I_i$, $i=1,2,3$, then, in much the same way as in Example \ref{su4-full-2}, we obtain that 
$\widetilde{J}$ is integrable if and only if $[\pg_0^{1,0}, \pg_0^{1,0}]_\pg\subset\pg_0^{1,0}$, which holds if and only if $J_0$ is a complex structure on the Lie algebra $\pg_0\simeq\sug(2)\oplus\sug(2)$.  
\end{example}

\begin{remark}
In the above two examples, one obtains $G$-invariant complex structures $\widetilde{J}$ which are not of the form $J=(J_\ag,J_\qg)$ for the flag chosen as base of the Tits fibration.  Nevertheless, they have this form for the flag defined by $\widetilde{\hg}=\bautg(\widetilde{J})$ (see Theorem \ref{C2}).  
\end{remark}

\subsection{de Rham cohomology}\label{dR-sec}
We use in this section \cite{H3}, see \S\ref{dRapp-sec} for an overview.  Recall that 
$$
(\Omega^pM)^G\equiv (\Lambda^p\pg^*)^K, 
$$ 
i.e., the space of all $G$-invariant $p$-forms on $M=G/K$ is naturally identified with the space of all $\Ad(K)$-invariant $p$-forms on the vector space $\pg\equiv T_oM$.  

Let $M=G/K=G_f/K\times T^z$ be a C-space as in \eqref{gprime} fibering over the flag $F=G_f/H_f$.

\subsubsection{First cohomology (see \cite[Section 2.1]{H3} and \S\ref{dRapp-sec})}\label{dR1} 
The space of $G$-invariant $1$-forms is given by  
$$
(\Omega^1M)^G=\{\theta_Z:Z\in\pg_0\}, \qquad \pg_0\supset\ag=\zg(\ggo)\oplus\ag_f, 
$$ 
where $\theta_Z:=Q(\cdot,Z)$ and $\pg_0$ is described in Proposition \ref{isotdec}, (iii).  The closed $G$-invariant $1$-forms are given by  
$$
(\Omega^1_cM)^G=\{\theta_Z:Z\in\zg(\ggo)\}\simeq\zg(\ggo), \qquad\mbox{so}\quad b_1(M)=\dim{\zg(\ggo)},
$$ 
that is, $b_1(G_f/K)=0$.

\subsubsection{Second cohomology (see \cite[Section 2.2]{H3} and \S\ref{dRapp-sec})}\label{dR2} 
If we define $\sigma_Z:=Q([\cdot,\cdot],Z)$, then the space of all closed $2$-forms is given by   
\begin{equation}\label{c2f}
(\Omega^2_cM)^G = \left\{ \sigma_Z+\gamma:Z\in\zg(\kg)\oplus\pg_0, \; \gamma\in(\Lambda^2\pg^*)^K, \; \gamma(\cdot, ([\ggo,\ggo]+\kg)_\pg)=0\right\},    
\end{equation}
so we can view 
$$
\gamma\in\Lambda^2\zg(\ggo)^*\simeq \Lambda^2\left(\ggo/([\ggo,\ggo]+\kg)\right)^*.
$$  
Since $\sigma_Z=-d\theta_Z$ for any $Z\in\pg_0$, we obtain that 
$$
H^2(M) \simeq \zg(\kg)\oplus \Lambda^2\zg(\ggo)^*, \qquad b_2(M)=\dim{\zg(\kg)} + 
\tbinom{\dim{\zg(\ggo)}}{2}.
$$
In particular, $b_2(G_f/K)=\dim{\zg(\kg)}$.

\subsubsection{Third cohomology (see \cite[Section 4]{H3} and \S\ref{dRapp-sec})}\label{dR3} 
By the K\"unneth formula, the third Betti number of a C-space $M=G/K=G_f/K\times T^z$ is given by 
\begin{align}
b_3(M) =& b_3(G_f/K) +b_2(G_f/K)b_1(T^z) + b_1(G_f/K)b_2(T^z) +b_3(T^z) \notag\\ 
=& b_3(G_f/K) + \dim{\zg(\kg)}\dim{\zg(\ggo)} + \tbinom{\dim{\zg(\ggo)}}{3}, \label{Kb3} 
\end{align}
so only $b_3(G_f/K)$ needs to be computed.  

We consider $\ggo_f=\ggo_1\oplus\dots\oplus\ggo_s$, the decomposition in simple factors, and $\Delta_\qg=\Delta_{\qg_1}\sqcup\dots\sqcup\Delta_{\qg_s}$ as in \eqref{decflags}.  Let $\vp_R$ be the closed $G_f$-invariant $3$-form on $G_f/K$ defined by any bi-invariant symmetric bilinear form $R$ on $\ggo_f$ such that $R|_{\kg\times\kg}=0$ and recall that $H^3(G_f/K)=\{[\vp_R]:R|_{\kg\times\kg}=0\}$ (see \eqref{phiR-def}).  On the C-space $G_f/K$, we have that $R(E_\alpha,E_\beta)=0$ whenever $\alpha+\beta\ne 0$.  We also denote by $R$ and $\vp_R$ the corresponding $\CC$-linear forms on $\pg^c$.  

\begin{lemma}\label{phiR}
If $R=z_1\kil_{\ggo_1}+\dots+z_s\kil_{\ggo_s}$, $R|_{\kg\times\kg}=0$, then the only possibly nonzero components of $\vp_R$ are given by 
$$
\vp_R(E_\alpha,E_\beta,E_\gamma)=N_{\alpha,\beta}z_i, \qquad\forall \alpha,\beta,\gamma\in\Delta_{\qg_i}, \quad \alpha+\beta+\gamma=0,  
$$
and
$$
\vp_R(A, E_\alpha,E_{-\alpha})
=-2z_iQ(A,A_\alpha)-R(A,A_\alpha), \qquad \forall A\in\ag_f^c, \quad \alpha\in\Delta_{\qg_i}, 
$$ 
where $A_\alpha:=(H_\alpha)_{\ag^c}$ and $i=1,\dots,s$.  
\end{lemma}

\begin{remark}
Every $3$-form $\vp_R$ is harmonic with respect to an $s$-parametric space of $G_f$-invariant metrics on $G_f/K$ including the standard metric $\gk$ (see \cite[Corollary 7.12]{H3}).  
\end{remark}

\begin{proof}
According to \eqref{phiR-def}, 
$$
\vp_R(E_\alpha,E_\beta,E_\gamma)=  R([E_\alpha,E_\beta],E_\gamma) = N_{\alpha,\beta}R(E_{\alpha+\beta},E_\gamma) = N_{\alpha,\beta}z_i \kil_{\ggo^c_i}(E_{\alpha+\beta},E_\gamma),  
$$
and 
$$
\vp_R(A, E_\alpha,E_{-\alpha})= 2\alpha(A)R(E_\alpha,E_{-\alpha}) -R(A_\alpha,A) =2z_i\kil_{\ggo^c}(A,A_\alpha)-R(A_\alpha,A),
$$ 
concluding the proof.  
\end{proof}

We now show that the positivity of the third Betti number is quite a strong condition on a C-space.    

\begin{proposition}\label{b3dR} 
If $\hg_i$ is abelian precisely for $i=1,\dots,t$ for some $t\leq s$ (see \eqref{decflags}), then $0\leq b_3(G_f/K)\leq t$, depending on the slope of $\zg(\kg)\subset\zg(\hg_1)\oplus\dots\oplus\zg(\hg_s)$.  More precisely, given any basis
$\{ Z^1,\dots,Z^m\}$ of $\zg(\kg)$, 
$$
b_3(G_f/K)=t-\dim{S_{\zg(\kg)}},
$$ 
where 
\begin{equation}\label{defS}
S_{\zg(\kg)}:=\spann_\RR\left\{ \left(\kil_{\ggo_1}(Z^j_1,Z^k_1),\dots,\kil_{\ggo_t}(Z^j_t,Z^k_t)\right):1\leq j\leq k\leq m\right\} \subset\RR^t.
\end{equation}
\end{proposition}

\begin{remark}
In particular, it is easy to check that $b_3(G_f/K)=s$ if and only if $\kg=0$.  Beyond Lie groups, in the maximal case, i.e., $b_3(G_f/K)=s-1$, the homogeneous spaces were called {\it aligned} in \cite{H3} and satisfy many nice structural conditions, which paved the way to the study of generalized Einstein metrics in \cite{BRF} and Einstein metrics in \cite{HHK,Es2,Rical}. 
\end{remark}

\begin{remark}
It follows that $b_3(G_f/K)=0$ if all the Lie algebras $\hg_i$ are non-abelian, i.e., no flag $G_i/H_i$ is full.  
\end{remark}

\begin{proof}
We use formula \eqref{b3-app}.  Since 
$$
\kg=\zg(\kg)\oplus\overline{\kg}_{t+1}\dots\oplus\overline{\kg}_s, \qquad \mbox{where}\quad 0\ne\overline{\kg}_i:=[\hg_i,\hg_i]\subset\ggo_i, 
$$
if $\kil_{\pi_i(\kg_j)}=a_{ij}\kil_{\ggo_i}|_{\pi_i(\kg_j)}$, where $\kg_j$ are the simple factors of $[\kg,\kg]$ (each one contained in some $\overline{\kg}_i$), then $a_{ij}\ne 0$ if and only if  $\kg_j\subset\overline{\kg}_i$.  For $R=z_1\kil_{\ggo_1}+\dots+z_s\kil_{\ggo_s}$, we therefore obtain that $R|_{\kg\times\kg}=0$ if and only if $z_{t+1}=\dots=z_s=0$ and $(z_1,\dots,z_t)$ is orthogonal to $\S_{\zg(\kg)}$, concluding the proof since $\dim{S_\kg}=\dim{S_{\zg(\kg)}}+s-t$, where $S_\kg$ is the subspace defined in \eqref{b3-app}.   
\end{proof}

\subsubsection{Fourth cohomology (see \cite{BswChtMty})}\label{dR4} 
The K\"unneth formula gives that 
\begin{align}
b_4(M) =& b_4(G_f/K) +b_3(G_f/K)b_1(T^z) + b_2(G_f/K)b_2(T^z) \notag\\ 
&+b_1(G_f/K)b_3(T^z) +b_4(T^z) \notag\\ 
=& b_4(G_f/K) + b_3(G_f/K)\dim{\zg(\ggo)} + \dim{\zg(\kg)}\tbinom{\dim{\zg(\ggo)}}{2} + \tbinom{\dim{\zg(\ggo)}}{4}, \label{Kb4} 
\end{align}
and it follows from \cite[Theorem 1.3]{BswChtMty} that 
$$
H^4(G_f/K)\simeq\{\mbox{bi-inv.\ sym.\ bil.\ forms on}\; \kg\} / \{ 
R|_{\kg\times\kg}: R \;\mbox{bi-inv.\ sym.\ bil.\ form on}\; \ggo_f\}.
$$
Thus
$$
b_4(G_f/K) = \tbinom{\dim{\zg(\kg)+1}}{2} +b_3(G_f/K)+v-s,
$$
where $[\kg,\kg]=\kg_1\oplus\dots\oplus\kg_v$ in simple factors.  

\begin{example}\label{BRF-bal}
From any two full flag manifolds $F_1=G_1/H$, $F_2=G_2/H$ with the same rank, we construct the C-space $M=G/K=G_1\times G_2/K$, where 
$$
\kg:=\Delta_{p,q}(\hg)= \{ (pZ,qZ):Z\in\hg_f\}\subset\hg_f\oplus\hg_f, \qquad  p,q\in\ZZ,
$$ 
hence 
$$
\ag=\{ (-qZ,pZ):Z\in\hg_f\}\simeq\hg_f,  \qquad \qg=\qg_1\oplus\qg_2.  
$$ 
According to \cite{H3}, this is an aligned homogeneous space with $c_1:=\tfrac{p^2+q^2}{p^2}$.  It is easy to check that $\dim{S_{\zg(\kg)}}=1$, so $b_3(M)=1$.  Indeed, if $R=\kil_{\ggo_1}-\tfrac{1}{c_1-1}\kil_{\ggo_2}$, then $\vp_R$ is a closed $3$-form on $M=G/K$ such that $H^3(M)=\RR[\vp_R]$.  
\end{example}

\subsection{Differential of $2$-forms}\label{form2-sec}
In this section, we give a formula for the differential of $G$-invariant $2$-forms.  Recall that $\Omega^2(M)^G\equiv(\Lambda^2\pg^*)^K$ (see \S\ref{igs}).  

We will work on the complexification $\pg^c=\ag^c\oplus\qg^c$.  Any tensor on $\pg$ is, as usual, $\CC$-linearly extended to a tensor on $\pg^c$.  It is easy to see that 
\begin{equation}\label{a1g}
\ag_f^c=\la A_\alpha:\alpha\in\Delta_\qg^+\ra_\CC =\la A_\alpha:\alpha\in\Pi_\qg\ra_\CC, \qquad  
\ag=\ag_f\oplus\zg(\ggo),  
\end{equation}
where
$$
A_\alpha:=(H_\alpha)_{\ag^c} \in\im \ag, \qquad\forall\alpha\in\Delta_\qg,
$$ 
is the $Q$-orthogonal projection of $H_\alpha$ on $\ag^c$.  We also have that $\{ E_\alpha:\alpha\in\Delta_\qg\}$ is a basis of $\qg^c$, and recall from \S\ref{flag-sec} and \S\ref{roots} that  
$$
[A,E_\alpha]=\alpha(A)E_\alpha, \quad [E_\alpha,E_\beta]=N_{\alpha,\beta}E_{\alpha+\beta}, \quad[E_\alpha,E_{-\alpha}]_{\pg^c}=A_\alpha, \quad\forall A\in\ag^c, \;\alpha,\beta\in\Delta_\qg,
$$
and $[E_\alpha,E_\beta]=0$ whenever $0\ne\alpha+\beta\notin\Delta_\qg$.  Note that we view $\alpha$ as defined in the whole vector space $\ag^c$ by extending $\alpha$ linearly from $\ag_f^c\subset\zg(\hg_f)^c\subset\tg_f^c$ and setting $\alpha|_{\zg(\ggo)^c}\equiv 0$.    

\begin{lemma}\label{do}
For any $G$-invariant $2$-form $\sigma=\sigma_\ag+\sigma_\qg$, $\sigma_\qg=(y_\alpha)_{\alpha\in\Delta_\qg}$ as in \eqref{s} (i.e., $\sigma_\qg(E_\alpha,E_{-\alpha})=y_\alpha$), the only possibly nonzero components of $d\sigma$ are given by 
$$
d\sigma(E_\alpha,E_\beta,E_\gamma)= N_{\alpha,\beta}(y_\alpha+y_\beta+y_\gamma), \qquad \alpha+\beta+\gamma=0, 
$$
$$
d\sigma(A,E_\alpha,E_{-\alpha}) =\sigma_\ag(A,A_\alpha), \qquad \forall A\in\ag^c, \;\alpha\in\Delta_\qg.    
$$ 
Furthermore, as a real form on $\pg$,  
$$
d\sigma(A,e_\alpha,f_\alpha) = \sigma_\ag(A,(\im H_\alpha)_\ag) 
\qquad\forall A\in\ag,\; \alpha\in\Delta_\qg^+, 
$$
where the subscript $\ag$ denotes $Q$-orthogonal projection on $\ag$ (i.e., $(\im H_\alpha)_\ag=\im A_\alpha$).  
\end{lemma}

\begin{proof}
For any $\alpha,\beta,\gamma\in\Delta_\qg$, 
\begin{align*}
d\sigma(E_\alpha,E_\beta,E_\gamma) &= 
-N_{\alpha,\beta}\sigma_\qg(E_{\alpha+\beta},E_\gamma) 
+ N_{\alpha,\gamma}\sigma_\qg(E_{\alpha+\gamma},E_\beta) 
- N_{\beta,\gamma}\sigma_\qg(E_{\beta+\gamma},E_\alpha) \\ 
&= N_{\alpha,\beta}y_\gamma 
- N_{\alpha,\gamma}y_\beta 
+ N_{\beta,\gamma}y_\alpha 
= N_{\alpha,\beta}(y_\alpha+y_\beta+y_\gamma), 
\end{align*}
if $\alpha+\beta+\gamma=0$ and zero otherwise.  

For any $A\in\ag^c$ and $\alpha\in\Delta_\qg$, 
\begin{align*}
d\sigma(A,E_\alpha,E_\beta) &= 
-\alpha(A)\sigma_\qg(E_\alpha,E_\beta) 
+ \beta(A)\sigma_\qg(E_\beta,E_\alpha) 
- \sigma([E_\alpha,E_\beta]_\pg,A) \\ 
&= \left\{\begin{array}{l} 
0, \qquad \alpha+\beta\ne 0, \\ 
\sigma_\ag(A,A_\alpha), \qquad \alpha+\beta= 0,
\end{array}\right.
\end{align*}
and for any $A\in\ag$ and $\alpha\in\Delta_\qg^+$, 
\begin{align*}
d\sigma(A,e_\alpha,f_\alpha) &= 
\unm d\sigma(A,E_\alpha-E_{-\alpha},\im(E_\alpha+E_{-\alpha})) 
=\im d\sigma(A,E_\alpha,E_{-\alpha}) 
= \sigma_\ag(A,(\im H_\alpha)_\ag), 
\end{align*}
concluding the proof.
\end{proof}

In particular, $d\sigma=0$ if and only if $\sigma_\ag=0$ and $d_F\sigma_\qg=0$, the former being equivalent to $y_\alpha+y_\beta+y_\gamma=0$ whenever $\alpha+\beta+\gamma=0$.  Note that, in that case, $(F,J_\qg,\sigma_\qg)$ is K\"ahler if $\sigma_\qg$ is positive definite, i.e., $\im y_\alpha>0$ for all $\alpha\in\Delta_\qg^+$ (see \S\ref{K}).  

For any $Z\in\zg(\hg_f)$, the closed $2$-form $\sigma_Z=Q([\cdot,\cdot],Z)$ has $\sigma_\ag=0$ and $\sigma_\qg=(\alpha(Z))_{\alpha\in\Delta_\qg}$, i.e., $y_\alpha=\alpha(Z)$.  If $\im Z$ belongs to some chamber of \eqref{ch}, then $g_Z=(\im\alpha(Z))_{\alpha\in\Delta_\qg}$ is a K\"ahler metric  on the flag $F=G_f/H_f$ (see \S\ref{K}) with Ricci form $\rho=\sigma_{\im\rho_{J_\qg}}$.

% 14/8/2026 

\section{Complex cohomologies}\label{cohom-sec}

We study in this section the Dolbeault, Bott-Chern and Aeppli cohomologies of a complex C-space $M=G/K$.  We refer to \cite{Ang,AngTrd} for more detailed treatments on these cohomologies.  The de Rham cohomology of the manifold $M$ was studied in \S\ref{dR-sec}.  Since $G$ is compact, it follows from \cite[Theorem 3.3]{Klm} that all these cohomologies can be computed within $G$-invariant forms (see also \cite[Theorem II]{Btt}, \cite{RmnSnk}, \cite{Grf}).  Recall from \S\ref{igs} that $(\Omega^kM)^G\equiv(\Lambda^k\pg^*)^K$.  

If $\Lambda^k:=(\Lambda^k(\pg^c)^*)^K\equiv(\Omega^kM)^G\otimes\CC$, then the decomposition 
$$
\pg^c=\pg^{1,0}\oplus\pg^{0,1}, 
$$ 
where $\pg^{1,0}$ and $\pg^{0,1}$ are the $\im$-eigenspace and the $-\im$-eigenspace of $J$, respectively, determines the spaces of $(p,q)$-forms on $\pg^c$, 
$$
\Lambda^{p,q}(\pg^c)^*:=\Lambda^p(\pg^{1,0})^*\otimes\Lambda^q(\pg^{0,1})^*,\qquad 1\leq p,q\leq n,
$$
which satisfy,   
$$
\Lambda^k = \bigoplus_{p+q=k}\Lambda^{p,q}, \qquad \mbox{where}\quad \Lambda^{p,q}:=(\Lambda^{p,q}(\pg^c)^*)^K.  
$$
Note that $\dim{\pg^c}=2n$ and $\dim{\pg^{1,0}}=\dim{\pg^{0,1}}=n$.  In this section, $\dim$ always means complex dimension if the vector space is complex.  In this way, $\Lambda^{p,q}$ is identified with the space $(\Omega^{p,q}M)^G$ of all $G$-invariant $(p,q)$-forms on $M$.  

The differential of forms is given by
$$
d:\Lambda^{p,q}\rightarrow \Lambda^{p+1,q}\oplus\Lambda^{p,q+1}, \quad d=\partial+\overline{\partial}, \quad
\partial:\Lambda^{p,q}\rightarrow \Lambda^{p+1,q}, \quad \overline{\partial}:\Lambda^{p,q}\rightarrow\Lambda^{p,q+1}, 
$$
and since $d^2=0$, one has that $\partial^2=0$, $\overline{\partial}^2=0$ and $\partial\overline{\partial}+\overline{\partial}\partial=0$.  This defines the double complex $\left(\Lambda^{\cdot,\cdot},\partial,\overline{\partial}\right)$ and the following types of cohomologies for a complex manifold $M$:
\begin{enumerate}[{\small $\bullet$}] 
\item {\it Dolbeault}: $\quad H_{\overline{\partial}}^{p,q}(M) :=\Ker\overline{\partial}/\Ima \overline{\partial}$, 
$\quad\Lambda^{p,q-1}\xlongrightarrow{\overline{\partial}}\Lambda^{p,q} \xlongrightarrow{\overline{\partial}}\Lambda^{p,q+1}$.  The dimensions 
$$
h^{p,q}:=\dim{H_{\overline{\partial}}^{p,q}(M)}
$$ 
are called {\it Hodge numbers}.   

\item {\it Conjugate Dolbeault}: $\quad H_{\partial}^{p,q}(M) :=\Ker\partial/\Ima \partial$, $\quad\Lambda^{p-1,q}\xlongrightarrow{\partial}\Lambda^{p,q} \xlongrightarrow{\partial}\Lambda^{p+1,q}$.  Note that $\dim{H_{\partial}^{p,q}}=h^{q,p}$ for all $p,q$.  

\item {\it Bott-Chern}: $\quad H_{BC}^{p,q}(M):=\Ker d/\Ima \partial\overline{\partial}$, $\quad\Lambda^{p-1,q-1}\xlongrightarrow{\partial\overline{\partial}}\Lambda^{p,q} \xlongrightarrow{d}\Lambda^{p+q+1}$.  The {\it Bott-Chern numbers} are defined by $h^{p,q}_{BC}:=\dim{H_{BC}^{p,q}(M)}$, $p,q\geq 1$ and since conjugation defines an isomorphism, $h^{p,q}_{BC}=h^{q,p}_{BC}$ for all $p,q$. 

\item {\it Aeppli}: $\quad H_{A}^{p,q}(M):=\Ker\partial\overline{\partial}/(\Ima \partial+\Ima \overline{\partial})$, 
$$
\left.\begin{array}{l}
\Lambda^{p-1,q}\xlongrightarrow{\partial} \\ 
\Lambda^{p,q-1}\xlongrightarrow{\overline{\partial}}
\end{array}\right\}\Lambda^{p,q}  \xlongrightarrow{\partial\overline{\partial}}\Lambda^{p+1,q+1}, \qquad p,q\geq 1.
$$   
The dimensions $h^{p,q}_{A}:=\dim{H_{A}^{p,q}(M)}$ are called {\it Aeppli numbers}. 
\end{enumerate}
A compact complex manifold $M$ is said to satisfy the $\partial\overline{\partial}$-{\it Lemma}
when 
$$
\Ker\partial\cap\Ker\overline{\partial}\cap\Ima d = \Ima\partial\overline{\partial},
$$
that is, every $d$-closed $(p,q)$-form is $d$-exact if and only if it is $\partial\overline{\partial}$-exact, or equivalently, the natural map $H_{BC}^{p,q}(M)\rightarrow H_{A}^{p,q}(M)$ induced by the identity is injective.  

The following interplay between Hodge and Betti numbers is known as the {\it Fr\"olicher inequality}: 
\begin{equation}\label{Fi}
\sum_{p+q=k}h^{p,q}\geq b_k = \dim{H^k(M,\RR)} = \dim{H^k(M,\CC)},   
\end{equation}
where equality holds if $M$ satisfies the $\partial\overline{\partial}$-Lemma.  A finer inequality (see \cite{AngTms} or \cite[Theorem 3]{AngTrd}) is given by 
\begin{equation}\label{Fi2}
\sum_{p+q=k}h^{p,q}_{BC}+h^{p,q}_{A}\geq 2b_k,   \qquad k\geq 2,  
\end{equation}
where equality holds for all $k$ if and only if  
\begin{equation}\label{Fi3}
h^{p,q}=h^{p,q}_{BC}=h^{p,q}_{A},   \qquad \forall p,q\in\NN_0,  
\end{equation}
if and only if $M$ satisfies the $\partial\overline{\partial}$-Lemma.   

For the Hodge numbers of a product $M=M_1\times M_2$ of compact complex manifolds we have {\it K\"unneth formula}: 
\begin{equation}\label{Kunn}
h^{p,q}(M)=\sum_{r+t=p,\, s+u=q} h^{r,s}(M_1) h^{t,u}(M_2).
\end{equation}
This formula does not hold in general for Bott-Chern or Aeppli numbers, although there are K\"unneth type formulas known in these cases (see \cite{Stl}).  

Back on a the case of a complex C-space, we consider a C-space as in \eqref{gprime}, 
$$
M^{2n}=G/K=G_f/K\times T^z, 
$$ 
and the Tits fibration over the flag $F=G_f/H_f$, which is holomorphic for any $G$-invariant complex structure, that is, for any $J=(J_\ag,J_\qg)$ as in \eqref{J}.  We will use the Tanr\'e model \cite[Proposition 8]{Tnr} for the Dolbeault cohomology of holomorphic fibrations, which also works for Bott-Chern and Aeppli cohomologies (see \cite[Theorem C]{Stl} and \cite[Section 3]{Brb}): if $\pg=\ag\oplus\qg$ and $\ggo_f=\hg_f\oplus\qg$ are respectively the reductive decompositions of the total space $M=G/K$ and the base $F=G_f/H_f$, then there is an isomorphism of double complexes,    
$$
\left(\Lambda^{\cdot,\cdot},\partial,\overline{\partial}\right) \simeq 
\left(\Lambda^{\cdot,\cdot}_T=\Lambda^{\cdot,\cdot}_\ag\otimes\Lambda^{\cdot,\cdot}_\qg,\partial,\overline{\partial}\right), 
$$
where 
$$
\Lambda^{p,q}_\ag:=\Lambda^p (\ag^{1,0})^*\otimes\Lambda^q (\ag^{0,1})^*, \qquad 
\Lambda^{\cdot,\cdot}_\qg \simeq \left((\Omega^{\cdot,\cdot} F)^{G_f},\partial_\qg,\overline{\partial_\qg}\right), \quad d_F=\partial_\qg+\overline{\partial_\qg},
$$
$\ag^c=\ag^{1,0}\oplus\ag^{0,1}$ is defined by $J_\ag$ and the corresponding operators, also denoted by $\partial,\overline{\partial}$, satisfy: 
$$
\left\{\begin{array}{l}
\partial|_{\Lambda^{1,0}_\ag}=0, \\ 
\partial|_{\Lambda^{0,1}_\ag}=d:\Lambda^{0,1}_\ag\rightarrow\Lambda^{1,1}_\qg, \\ 
\partial|_{\Lambda^{p,q}_\qg}=\partial_\qg:\Lambda^{p,q}_\qg\rightarrow\Lambda^{p+1,q}_\qg, 
\end{array}\right. \qquad 
\left\{\begin{array}{l}
\overline{\partial}|_{\Lambda^{1,0}_\ag}=d:\Lambda^{1,0}_\ag\rightarrow\Lambda^{1,1}_\qg, \\  
\overline{\partial}|_{\Lambda^{0,1}_\ag}=0, \\ 
\overline{\partial}|_{\Lambda^{p,q}_\qg}=\overline{\partial}_\qg:\Lambda^{p,q}_\qg\rightarrow\Lambda^{p,q+1}_\qg.
\end{array}\right.
$$
We have that 
$$
\Lambda_\ag^{1,0}=(\ag^{1,0})^*=\left\{\theta_A: A\in\ag^{1,0}\right\}, \qquad \Lambda_\ag^{0,1}=(\ag^{0,1})^*=\left\{\theta_A: A\in\ag^{0,1}\right\}, 
$$
where 
$$
\theta_A=Q(\cdot,A)\in\Lambda^1,
$$ 
is the $\CC$-linear version of the form introduced in \S\ref{dR-sec}.  On the other hand, since any $2$-form $\sigma_\qg = (y_\alpha)_{\alpha\in\Delta_\qg}$ as in \eqref{s} satisfies that $\sigma_\qg\in\Lambda_\qg^{1,1}$, we obtain that $d\Lambda^1\subset\Lambda^{1,1}_\qg$.  More precisely, 
$$
d\theta_A=-\sigma_A\in\Lambda_\qg^{1,1}, \qquad \forall A\in\ag^c=\zg(\ggo)^c\oplus\ag_f^c, 
$$ 
where
$$
\sigma_A:=Q([\cdot,\cdot],A) = (y_\alpha)_{\alpha\in\Delta_\qg}, \qquad y_\alpha:=Q(A_\alpha,A), \quad\forall\alpha\in\Delta_\qg. 
$$  
Note that $\theta_A=0$ if and only if $A=0$, and $\sigma_A=0$ if and only if $A\in\zg(\ggo)^c$.  
  
We also have that 
$$
\dim{\Lambda^{p,q}_\ag}=\tbinom{d}{p}\tbinom{d}{q}, \qquad \mbox{where} \quad \dim{\ag}=2d,  
$$
and since 
$$
\Lambda_\qg^1=0, \qquad \Lambda_\qg^2=\Lambda_\qg^{1,1} = \left\{\sigma_\qg = (y_\alpha)_{\alpha\in\Delta_\qg}\right\},  
$$
$\Lambda^{p,0}_\qg=0$ and $\Lambda^{0,p}_\qg=0$ for $p=1,2$, and so  
$$
\Lambda_T^{1,0}= \Lambda^{1,0}_\ag, \qquad \Lambda_T^{0,1}=\Lambda^{0,1}_\ag, 
$$
$$
\Lambda_T^{2,0}= \Lambda^{2,0}_\ag, \qquad \Lambda_T^{0,2}=\Lambda^{0,2}_\ag,  \qquad \Lambda_T^{1,1}=\Lambda^{1,1}_\ag\oplus\Lambda^{1,1}_\qg,
$$ 
$$
\Lambda_T^{2,1}= \Lambda^{2,1}_\ag\oplus (\Lambda^{1,0}_\ag\otimes\Lambda^{1,1}_\qg)\oplus \Lambda^{2,1}_\qg, 
\qquad 
\Lambda_T^{1,2}= \Lambda^{1,2}_\ag\oplus (\Lambda^{0,1}_\ag\otimes\Lambda^{1,1}_\qg)\oplus \Lambda^{1,2}_\qg.
$$ 

\begin{remark}
If $\ag\subsetneq\pg_0$ (see Proposition \ref{isotdec}), then 
$$
\Lambda^{1,0}_T\oplus\Lambda^{0,1}_T= \Lambda^{1,0}_\ag\oplus\Lambda^{0,1}_\ag \subsetneq\Lambda^1 =(\pg_0^c)^*,
$$
which represents one of the great advantages of using the Tanr\'e model.  
\end{remark}

Given any subspace $\bg\subset\ag$, we denote by $\bg^{1,0}$ the projection of $\bg$ on $\ag^{1,0}$, i.e., 
$$
\bg^{1,0}:=\{ A^{1,0}=\unm(A-\im J_\ag A):A\in\bg\}, \qquad \bg^{0,1}:=\{ A^{0,1}=\unm(A+\im J_\ag A):A\in\bg\}.  
$$ 
Note that $\overline{\bg^{0,1}}=\bg^{1,0}$ and $\bg^c\subset\bg^{1,0}\oplus\bg^{0,1}$, where equality holds if and only if $J_\ag\bg=\bg$, if and only if $\dim{\bg}=2\dim{\bg^{1,0}}$.  We also have that  
\begin{equation}\label{dimb}
\dim{\bg^c\cap\ag^{1,0}}\leq \dim{\bg^{1,0}}\leq \dim{\bg}\leq 2\dim{\bg^{1,0}}.  
\end{equation}
and analogously for $\bg^{0,1}$.  

Recall the $Q$-orthogonal decompositions 
$$
\pg=\ag_f\oplus\zg(\ggo)\oplus\qg, \qquad \ag=\ag_f\oplus\zg(\ggo), \qquad \zg(\hg_f)=\zg(\kg)\oplus\ag_f. 
$$
The following observations will be very useful in the computation of many low cohomology groups:    
\begin{enumerate}[{\rm (a)}] 
\item According to \S\ref{dR1}, $\Ker \overline{\partial}|_{\Lambda^{1,0}_\ag} = \{\theta_A:A\in\zg(\ggo)^c\cap\ag^{1,0}\}$, so 
$$
\dim{\zg(\ggo)}-d\leq \dim{\zg(\ggo)^c\cap\ag^{1,0}} =\dim{\Ker \overline{\partial}|_{\Lambda^{1,0}_\ag}}\leq d,
$$ 
depending on $J_\ag$, where $\dim{\ag}=2d$.  Indeed,  
$$
\dim{\zg(\ggo)^c} + \dim{\ag^{1,0}} - \dim{\zg(\ggo)^c\cap\ag^{1,0}} \leq \dim{\ag^c},  
$$ 
where equality holds if and only if $\zg(\ggo)^c +\ag^{1,0}=\ag^c$.  All the same holds for $\Ker\partial|_{\Lambda^{0,1}_\ag}$ and $\ag^{0,1}$.  

\item 
Since for any $\psi\in\Lambda^{p,q}_\ag$, 
$$
\overline{\partial}\psi(E_\alpha,E_{-\alpha},A_1,\dots,A_{p-1},B_1,\dots,B_q) = -\psi(A_\alpha^{1,0},A_1,\dots,A_{p-1},B_1,\dots,B_q), 
$$
for all $A_i\in\ag^{1,0}$, $B_j\in\ag^{0,1}$, $\alpha\in\Delta_\qg$, we obtain that $\overline{\partial}\psi=0$ if and only if $\iota_{\ag_f^{1,0}}\psi=0$ by \eqref{a1g} if $p\geq 1$, and analogously for $\partial\psi$ and $\ag_f^{0,1}$ if $q\geq 1$.  Thus  
$$
\dim{\Ker\overline{\partial}|_{\Lambda^{p,q}_\ag}} = \tbinom{d-\dim{\ag_f^{1,0}}}{p}\tbinom{d}{q}, \quad \forall p\geq 1, \qquad 
\dim{\Ker\partial|_{\Lambda^{p,q}_\ag}} = \tbinom{d}{p}\tbinom{d-\dim{\ag_f^{0,1}}}{q}, \quad\forall q\geq 1, 
$$ 
where $d-\dim{\ag_f^{1,0}} \geq \dim{\zg(\ggo)}-d$ by \eqref{dimb}.  Note that $\dim{\ag_f^{1,0}}=d$ if and only if $G$ is semisimple (i.e., $\ag_f=\ag$).  On the other hand, 
$$
\overline{\partial}|_{\Lambda^{0,p}_\ag} = 0, \qquad 
 \partial|_{\Lambda^{p,0}_\ag} = 0, \qquad\forall p\geq 1.  
$$  
\item Since  
$$
\Lambda^{1,1}_\qg = \{\sigma_\qg = (y_\alpha)_{\alpha\in\Delta_\qg}:y_\alpha\in\im\RR\}, \qquad \overline{\partial}\sigma_\qg=\overline{\partial_\qg}\sigma_\qg =(d\sigma_\qg)^{1,2}\in\Lambda^{1,2}_\qg, 
$$ 
we obtain that $\dim{\Lambda^{1,1}_\qg}=r$, where $r$ is the number of $\Ad(H_f)$-irreducible summands of $\qg$ (see \S\ref{isot} and \eqref{s}), and $\overline{\partial}\sigma_\qg=0$ if and only if $d\sigma_\qg=0$, if and only if $y_{\alpha+\beta}= y_\alpha+y_\beta$ for all $\alpha,\beta,\alpha+\beta\in\Delta_\qg^+$ by Lemma \ref{do} (analogously for $\partial$).  Using induction in the height of the roots and setting $y_\alpha=0$ for all $\alpha\in\Delta_{\hg_f}$, it can be shown that this is equivalent to 
$$
y_\gamma=\sum\limits_{\alpha\in\Pi_\qg}n_\alpha y_\alpha, \qquad \mbox{for any}\quad  
\gamma=\sum\limits_{\alpha\in\Pi_\qg}n_\alpha\alpha + \sum\limits_{\beta\in\Pi_{\hg_f}}k_\beta\beta\in\Delta_\qg^+, \quad n_\alpha,k_\beta\in\NN_0. 
$$
Thus 
$$
\Ker\overline{\partial}|_{\Lambda^{1,1}_\qg} = \Ker\partial|_{\Lambda^{1,1}_\qg} =\Ker d_F|_{\Lambda^{1,1}_\qg} =\Ker d|_{\Lambda^{1,1}_\qg},
$$
and so
$$
\dim{\Ker\overline{\partial}|_{\Lambda^{1,1}_\qg}} =\dim{\zg(\hg_f)} =|\Pi_\qg| = b_2(F).  
$$
Note that $|\Pi_\qg|\leq r\leq|\Delta_\qg^+|$. 

\item The Hodge numbers of the flag $F=G_f/H_f$ satisfy that $h^{p,q}(F)=0$ for all $p\ne q$, and since $(F,J_\qg)$ is K\"ahler, $h^{p,p}(F)=b_{2p}(F)$ for all $p$  (see \cite[14.10]{BrlHrz}).    
\end{enumerate}

\subsection{Dolbeault}\label{D-sec} 
Using the above observations (a)-(d), it is straightforward to check that the following formulas hold:
\begin{equation}\label{hn}
\begin{array}{c}
h^{1,0}=\dim{\zg(\ggo)^c\cap\ag^{1,0}}, \qquad h^{0,1}=d, \\ 
h^{1,1}= (d-\dim{\ag_f^{1,0}})d+ \dim{\zg(\hg_f)} -d +\dim{\zg(\ggo)^c\cap\ag^{1,0}}, \\
h^{2,0}=\tbinom{d-\dim{\ag_f^{1,0}}}{2}, \qquad h^{0,2}=\tbinom{d}{2},   
\end{array}
\end{equation}
where $\dim{\ag}=2d$.  In particular, if $G$ is semisimple (i.e., $\zg(\ggo)=0$, or $\ag=\ag_f$, or $b_1=0$), then 
$$
h^{1,0}=0, \quad h^{0,1}=d, \quad
h^{1,1}= \dim{\zg(\hg_f)}-d =\dim{\zg(\kg)}+d, \quad h^{2,0}=0, \quad h^{0,2}=\tbinom{d}{2}.    
$$
In accordance with \eqref{Fi}, it follows from (a) that 
$$
h^{1,0}+h^{0,1} = \dim{\zg(\ggo)^c\cap\ag^{1,0}} +d\geq \dim{\zg(\ggo)} =b_1, 
$$
where equality holds if and only if $\zg(\ggo)^c+\ag^{1,0}=\ag^c$.  

\begin{proposition}\label{eqs}
For any complex C-space $(M=G/K,J)$, the following conditions are equivalent:
\begin{enumerate}[{\rm (i)}] 
\item $h^{1,0}+h^{0,1}=b_1$.  

\item $\dim{\zg(\ggo)^c\cap\ag^{1,0}} = \dim{\zg(\ggo)} -d$.

\item $\zg(\ggo)+J_\ag\zg(\ggo)=\ag$.
\end{enumerate}
\end{proposition}

\begin{remark}
In particular, a non-K\"ahler complex C-space (i.e., $\ag_f\ne 0$, or $\zg(\ggo)\ne\ag$) never satisfies the $\partial\overline{\partial}$-Lemma (see \eqref{Fi}; this was proved in \cite[Theorem 1]{Pds} in the case when $G$ is semisimple).  Indeed, in addition to condition (ii) we would have that $\dim{\zg(\ggo)^c\cap\ag^{1,0}} = d$ by $h^{1,0}=h^{0,1}=h^{1,0}_{BC}$ (see \eqref{Fi3}), so $\dim{\zg(\ggo)}=2d$, i.e. $\zg(\ggo)=\ag$.
\end{remark}

\begin{proof}
We only need to prove the equivalence between parts (i) and (iii).  If $\bg:=\zg(\ggo)$ and $\bg^c+\ag^{1,0}=\ag^c$, then $\bg^c+\ag^{0,1}=\ag^c$ since $\overline{\bg^c}=\bg^c$, which implies that $\bg^c+J_\ag\bg^c=\ag^c$ as it contains both $\ag^{1,0}$ and $\ag^{0,1}$.  Conversely, if $\bg^c+J_\ag\bg^c=\ag^c$, then 
$$
\dim{\bg^c\cap\ag^{1,0}} = \dim{(\bg\cap J_\ag\bg)^{1,0}} = \unm\dim{\bg^c\cap J_\ag\bg^c} 
=\dim{\bg^c}-\dim{\ag^{1,0}},
$$
and hence $\bg^c+\ag^{1,0}=\ag^c$, concluding the proof.    
\end{proof}

\begin{proposition}\label{eqs2}
For any complex C-space $(M=G/K,J)$, the following conditions are equivalent:
\begin{enumerate}[{\rm (i)}] 
\item $h^{2,0}+h^{1,1}+h^{0,2} =b_2$.  

\item $\dim{\zg(\ggo)^c\cap\ag^{1,0}} = \dim{\zg(\ggo)} -d$ (i.e., $h^{1,0}+h^{0,1}=b_1$ by Proposition \ref{eqs}) and 
$$
\dim{\ag_f^{1,0}} = 2d-\dim{\zg(\ggo)}. 
$$  
\end{enumerate}
\end{proposition}

\begin{proof}
Using \S\ref{dR2} and that
$$
\dim{\zg(\hg_f)}+\dim{\zg(\ggo)}=\dim{\zg(\kg)}+2d, 
$$ 
we obtain from \eqref{hn} that the Fr\"olicher inequality $h^{2,0}+h^{1,1}+h^{0,2} \geq b_2$ is equivalent to 
\begin{align*}
&\tbinom{d-\dim{\ag_f^{1,0}}}{2} +(d-\dim{\ag_f^{1,0}})d +\dim{\zg(\ggo)^c\cap\ag^{1,0}}+ \tbinom{d}{2} \\ 
&\qquad \geq \dim{\zg(\ggo)} - d +\tbinom{\dim{\zg(\ggo)}}{2},
\end{align*}
or equivalently, 
\begin{align*}
\tbinom{2d-\dim{\ag_f^{1,0}}}{2}+\dim{\zg(\ggo)^c\cap\ag^{1,0}}
\geq  \dim{\zg(\ggo)} -d + \tbinom{\dim{\zg(\ggo)}}{2}.   
\end{align*}
Since $2d-\dim{\ag_f^{1,0}}\geq\dim{\zg(\ggo)}$ (see (b)), equality holds if and only if condition (ii) holds, concluding the proof.  
\end{proof}

\begin{proposition}\label{h3D}
For any complex C-space $(M=G/K,J)$, 
$$
h^{1,2} = d \left(\dim{\zg(\hg_f)} + \tbinom{d-\dim{\ag_f^{1,0}}}{2} - \dim{\ag_f^{1,0}}\right),
 \qquad h^{3,0}=\tbinom{d-\dim{\ag_f^{1,0}}}{3}, \qquad h^{0,3}=\tbinom{d}{3}, 
$$ 
and
$$
h^{2,1} =D - \tbinom{d}{2} +(d+1)\tbinom{d-\dim{\ag_f^{1,0}}}{2},
$$
where $D:= \dim{\Ker \overline{\partial}|_{(\Lambda^{1,0}_\ag\otimes\Lambda^{1,1}_\qg)\oplus\Lambda^{2,1}_\qg}}$ and $\dim{\ag}=2d$.  In particular, $h^{2,1}=D - \tbinom{d}{2}$, or $h^{1,2}=d(\dim{\zg(\hg_f)}-d)$, if and only if $G$ is semisimple.  
\end{proposition}

\begin{remark}
Since $\Ker \overline{\partial}|_{\Lambda^{1,0}_\ag}\otimes\Ker \overline{\partial}|_{\Lambda^{1,1}_\qg} \subset\Ker \overline{\partial}|_{\Lambda^{1,0}_\ag\otimes\Lambda^{1,1}_\qg}$, 
we obtain that 
\begin{equation}\label{Din}
D\geq  (\dim{\zg(\ggo)^c\cap\ag^{1,0}})\dim{\zg(\hg_f)},
\end{equation} 
by (a) and (c).  Although we know that $\Ker \overline{\partial}|_{\Lambda^{2,1}_\qg}=0$ due to the fact that  $h^{2,1}(F)=0$, the number $D$ is difficult to compute.  Indeed, both $\overline{\partial}(\Lambda^{1,0}_\ag\otimes\Lambda^{1,1}_\qg)$ and $\overline{\partial}(\Lambda^{2,1}_\qg)$ may contain nonzero subspaces of $\Lambda^{2,2}_\qg$.  Anyway, we note that $D>0$ as soon as $\dim{\zg(\ggo)}\geq d+1$ by (a) and \eqref{Din}, as in Example \ref{su3t3} below.  
\end{remark}

\begin{remark}\label{h21-rem}
According to \cite[Lemma 6.4]{FinGrnVzz}, $h^{2,1}=0$, i.e., $D=\tbinom{d}{2}$, if $G$ semisimple, $\kg\ne 0$ and $(M=G/K,J)$ is irreducible as a complex manifold.  We will give a counterexample to this assertion in \S\ref{h21-h11A} below.    
\end{remark}

\begin{proof}
Since
$$
\Lambda_T^{2,0}=\Lambda^{2,0}_\ag 
\xlongrightarrow{\overline{\partial}}
\Lambda_T^{2,1}=\Lambda^{2,1}_\ag\oplus(\Lambda^{1,0}_\ag\otimes\Lambda^{1,1}_\qg) \oplus\Lambda^{2,1}_\qg \xlongrightarrow{\overline{\partial}} \Lambda_T^{2,2}, 
$$
$\overline{\partial}\Lambda^{2,1}_\ag\subset\Lambda^{1,1}_\ag\otimes\Lambda^{1,1}_\qg$ and 
$$
\overline{\partial}\left((\Lambda^{1,0}_\ag\otimes\Lambda^{1,1}_\qg) \oplus\Lambda^{2,1}_\qg\right)\subset (\Lambda^{1,0}_\ag\otimes\Lambda^{1,2}_\qg)\oplus \Lambda^{2,2}_\qg, 
$$ 
the formula for $h^{2,1}$ follows from observation (b).  

On the other hand, we have that
$$
\Lambda_T^{1,1}=\Lambda^{1,1}_\ag \oplus\Lambda^{1,1}_\qg
\xlongrightarrow{\overline{\partial}}
\Lambda_T^{1,2}=\Lambda^{1,2}_\ag\oplus(\Lambda^{0,1}_\ag\otimes\Lambda^{1,1}_\qg) \oplus\Lambda^{1,2}_\qg \xlongrightarrow{\overline{\partial}} 
\Lambda_T^{1,3}, 
$$
and since $\overline{\partial}\Lambda^{1,2}_\ag\subset \Lambda^{0,2}_\ag\otimes\Lambda^{1,1}_\qg$, $\overline{\partial}(\Lambda^{0,1}_\ag\otimes\Lambda^{1,1}_\qg)\subset \Lambda^{0,1}_\ag\otimes\Lambda^{1,2}_\qg$ and $\overline{\partial}\Lambda^{1,2}_\qg \subset \Lambda^{1,3}_\qg$, we obtain the formula for $h^{1,2}$ from the following consequences of observations (a)-(d) and the fact that $\Ker \overline{\partial}|_{\Lambda^{0,1}_\ag\otimes\Lambda^{1,1}_\qg}  
=\Lambda^{0,1}_\ag\otimes\Ker \overline{\partial}|_{\Lambda^{1,1}_\qg}$: 
$$
\dim{\overline{\partial}\Lambda^{1,1}_\ag}=d\dim{\ag_f^{1,0}}, \quad \dim{\Ker \overline{\partial}|_{\Lambda^{0,1}_\ag\otimes\Lambda^{1,1}_\qg}}  
=d\dim{\zg(\hg_f)}, \quad
\overline{\partial}\Lambda^{1,1}_\qg=\Ker \overline{\partial}|_{\Lambda^{1,2}_\qg},
$$
concluding the proof.  
\end{proof}

\subsection{Bott-Chern}\label{BC-sec} 
Let $\Lambda_c^{p,q}$ denote the space of all closed $G$-invariant $(p,q)$-forms.  

\begin{proposition}\label{hBC}
Let $(M=G/K,J)$ be a complex C-space. 
\begin{enumerate}[{\rm (i)}] 
\item $H_{BC}^{1,1}(M)=\Lambda^{1,1}_c$ and 
$$
h^{1,1}_{BC}=\dim{\zg(\hg_f)}+ (d-\dim{\ag_f^{1,0}})^2.  
$$  
\item $h^{p,0}_{BC}=\binom{d}{p}$ for all $p\geq 1$. 

\item $H_{BC}^{2,1}(M)=\Lambda^{2,1}_c$. 

\item  If $b_3(M)=0$ (see Proposition \ref{b3dR}), then
\begin{align*}
h_{BC}^{2,1} = \tbinom{d}{2} + \tbinom{d-\dim{\ag_f^c\cap\ag^{0,1}}}{2} + d\left(2d-\dim{(\ag_f^c+\ag^{0,1})}\right).
\end{align*}
In particular, $h_{BC}^{2,1}\geq\tbinom{d}{2}$, where equality holds if and only if $G$ is semisimple.  

\item If $G$ is semisimple, then 
$$
h_{BC}^{2,1} = \tbinom{d}{2} + t - \dim{S_{\zg(\kg)\oplus\ag^{0,1}}},
$$
where $t$ is as in Proposition \ref{b3dR} and $S_{\zg(\kg)\oplus\ag^{0,1}}\subset\RR^t$ is defined as in \eqref{defS} but for a basis $\{ Z^1,\dots,Z^m\}$ of $\zg(\kg)\oplus\ag^{0,1}$.  
In particular, $h_{BC}^{2,1}\leq\tbinom{d}{2}+b_3(M)$.    
\end{enumerate}
\end{proposition} 

\begin{remark}\label{ddL2} 
It follows from \eqref{hn} that
$$
h^{1,1}=(d-\dim{\ag_f^{1,0}})d+ \dim{\zg(\hg_f)} -d +\dim{\zg(\ggo)^c\cap\ag^{1,0}} 
\geq h^{1,1}_{BC}, 
$$ 
where equality holds if and only if $\ag_f^{1,0}=0$
\end{remark}

\begin{remark}\label{fBCclass}
The first Bott-Chern class (see \cite{Brb,Ist})
$$
c_1^{BC}(M):=[\sigma_{Z_{J_\qg}}]\in H_{BC}^{1,1}(M)
$$ 
is therefore always non-zero (see \S\ref{fCc-sec}), which implies that a complex C-space has holomorphically trivial canonical bundle if and only if it is a torus.  Note that $c_1^{BC}(M)\geq 0$, and $c_1^{BC}(M)> 0$ only in the flag case.  
\end{remark}

\begin{remark}
In part (iv), $\dim{\ag} = 2d$, $\dim{\ag_f} = 2d-\dim{\zg(\ggo)}$ and  
$$
\dim{(\ag_f^c+\ag^{0,1})} = 3d-\dim{\zg(\ggo)}- \dim{\ag_f^c\cap\ag^{0,1}}. 
$$
\end{remark}

\begin{proof}
According to \eqref{c2f}, $\gamma\in\Lambda^2(\ag^c)^*$ is closed if and only if $\gamma(\ag_f^c,\cdot)=0$ (recall that $\ag=\ag_f\oplus\zg(\ggo)$), hence 
$$
\Ker d|_{\Lambda^{1,1}_\ag} =\{\gamma\in\Lambda^{1,1}_\ag:\gamma(\ag_f^{1,0},\ag^{0,1}) =\gamma(\ag^{1,0},\ag_f^{0,1}) =0\},
$$ 
and part (i) therefore follows from $\overline{\partial}\Lambda^0=0$ and (c).  Since $\partial\overline{\partial}|_{\Lambda^{p,0}}=0$ and $\partial\overline{\partial}|_{\Lambda^{0,p}}=0$ for all $p\geq 1$ by (b) (recall that $\partial\overline{\partial}=-\overline{\partial}\partial$), parts (ii) and (iii) follow. 

We now prove part (iv).  If $b_3(M)=0$, then $\Lambda^{2,1}_c=d\Lambda^2\cap\Lambda^{2,1}$, and for $\sigma=\sigma_\ag+\sigma_\qg\in\Lambda^2$, it follows from Lemma \ref{do} that $d\sigma\in\Lambda^{2,1}$ if and only if $d\sigma_\qg=0$ and 
$$
0=d\sigma(A,E_\alpha,E_{-\alpha})=\sigma_\ag(A,A_\alpha), \qquad\forall A\in\ag^{0,1},\; \alpha\in\Delta_\qg,
$$
i.e., $\sigma_\ag(\ag^{0,1},\ag_f^c)=0$. Thus  
$$
h^{2,1}_{BC}=\dim\{\sigma\in\Lambda^2(\ag^c)^*:\sigma(\ag^{0,1},\ag_f^c)=0\},
$$
and so part (iv) follows by using the decomposition 
$$
\ag^{0,1}+\ag_f^c=(\ag^{0,1}\ominus(\ag^{0,1}\cap\ag_f^c)) \oplus (\ag_f^c\ominus(\ag^{0,1}\cap\ag_f^c))\oplus (\ag^{0,1}\cap\ag_f^c).
$$
The last sentence follows from the fact that $\overline{\partial}\Lambda_\ag^{2,0}=d\Lambda_\ag^{2,0}\subset\Lambda_c^{2,1}$.  

In order to prove part (v), we recall from \S\ref{dR3} that 
$$
\Lambda^3_c=\left\{\vp_R+d\sigma: R|_{\kg\times\kg}=0, \; \sigma=\sigma_\ag+\sigma_\qg, \; \sigma_\qg=(y_\alpha)_{\alpha\in\Delta_\qg}\right\}, 
$$
where $R=z_1\kil_{\ggo_1}+\dots+z_s\kil_{\ggo_s}=Q(\widetilde{R}\cdot,\cdot)$ and $\widetilde{R}|_{\ggo_i}=-z_iI$ for all $i$.  We need to compute the dimension of $\Lambda^{2,1}_c=\Lambda^3_c\cap\Lambda^{2,1}$ (see (iii)).  It follows from Lemmas \ref{phiR} and \ref{do} that $(\vp_R+d\sigma)^{1,2}=0$ if and only if for all $i=1,\dots,s$, 
\begin{align} 
&y_{\alpha+\beta}-y_\alpha-y_\beta=z_i, \qquad \forall \alpha,\beta,\alpha+\beta\in\Delta_{\qg_i}^+, \label{BC1}\\ 
&\sigma_\ag(A,B)=2z_iQ(A,B)+R(A,B), \qquad \forall A\in\ag^{0,1}, \quad B\in\zg(\hg_i)^c_{\ag^c}, \label{BC2}  
\end{align}
where $\zg(\hg_f)=\zg(\hg_1)\oplus\dots\oplus\zg(\hg_s)=\zg(\kg)\oplus\ag$ as in \eqref{decflags} (note that $\ag=\ag_f$ since $G$ is semisimple).  According to \eqref{BC2}, for any 
$$
B=B_1+\dots+B_s\in\ag^c =\zg(\hg_1)^c_{\ag^c}+\dots+\zg(\hg_s)^c_{\ag^c}, \qquad B_i=(Z_i)_{\ag^c}, \quad Z_i\in\zg(\hg_i)^c, 
$$
and $A\in\ag^{0,1}$, 
\begin{align*}
\sigma_\ag(A,B)=& R(A,B) + 2\sum_{i=1}^s z_iQ(A,B_i) =R(A,B) + 2\sum_{i=1}^s Q(A,z_iZ_i) \\ 
=& R(A,B) - 2\sum_{i=1}^s Q(A,\widetilde{R}Z_i) =R(A,B) - 2\sum_{i=1}^s R(A,Z_i) \\ 
=& R(A,B) - 2R(A,Z) = - R(A,B) -2R(A,Z_{\zg(\kg)^c}), 
\end{align*}
where $Z:=Z_1+\dots+Z_s$ satisfies that $Z_{\ag^c}=B$.  This implies that $R(\ag^{0,1},\zg(\kg)^c)=0$ and so $\sigma_\ag(A,B)=-R(A,B)$ for all $A\in\ag^{0,1}$ and $B\in\ag^c$.  In particular, $R(\ag^{0,1},\ag^{0,1})=0$, that is, $R|_{(\zg(\kg)^c\oplus\ag^{0,1})\times(\zg(\kg)^c\oplus\ag^{0,1})}=0$, which is equivalent to $(z_1,\dots,z_t)$ being $Q$-orthogonal to the subspace $S_{\zg(\kg)^c\oplus\ag^{0,1}}$.  If we also consider \eqref{BC1}, then       
$$
h_{BC}^{2,1} = \dim{\Lambda^3_c\cap\Lambda^{2,1}} = \tbinom{d}{2} + t - \dim{S_{\zg(\kg)\oplus\ag^{0,1}}} + \dim{\zg(\hg_f)} - \dim{\zg(\hg_f)},
$$
since $(\vp_R+d\sigma)^{2,1}=0$ if and only if $y_{\alpha+\beta}-y_\alpha-y_\beta=-z_i$ for all $\alpha,\beta,\alpha+\beta\in\Delta_{\qg_i}^+$, if and only if $R=0$, $\sigma_\ag=0$ and $d\sigma_\qg=0$, concluding the proof. 
\end{proof}

\subsection{Aeppli}\label{A-sec}  
Let $(M=G/K,J)$ be a complex C-space.  It is easy to see that 
$$
h^{1,0}_A=h^{0,1}_A=d. 
$$  
In this section, we will give a formula for $h^{1,1}_A$.   

For any $\sigma=\sigma_\ag+\sigma_\qg\in\Lambda_T^{1,1}$, $\sigma_\qg =(y_\alpha)_{\alpha\in\Delta_\qg}$, we set $x_\alpha:=\im\epsilon_\alpha y_\alpha$, i.e., 
$$
\sigma_\qg(E_\alpha,E_{-\alpha})=-\im x_\alpha, \qquad\forall \alpha\in\Delta_\qg^+, 
$$ 
and define the symmetric bilinear form $h=h(\sigma_\ag):\ag\times\ag\rightarrow\RR$ by 
\begin{equation}\label{h11A-h} 
h(A,B):=-\im\left(\sigma_\ag(A^{1,0},B^{0,1}) + \sigma_\ag(B^{1,0},A^{0,1})\right),  \qquad\forall A,B\in\ag.
\end{equation}
Note that $h$ is compatible with $J_\ag$ (i.e., $h(J_\ag\cdot,J_\ag\cdot)=h(\cdot,\cdot)$) and $\sigma_\ag(A,B)=h(JA,B)$ for all $A,B\in\ag$.  Moreover, the map $\sigma_\ag\rightarrow h(\sigma_\ag)$ is an isomorhpism between $\Lambda^{1,1}_\ag$ and the vector space of all $J_\ag$-compatible symmetric bilinear forms on $\ag$.  

Recall from \eqref{deca} the decompositions 
$$
\ag=\underline{\ag}_f\oplus\zg(\ggo)\oplus\overline{\ag}, \qquad 
\overline{\ag}=\tg_{u+1}\oplus\dots\oplus\tg_s,
$$
and from \eqref{decflags},
$$
\Pi_\qg=\Pi_{\qg_1}\sqcup\dots\sqcup\Pi_{\qg_s}, \qquad \Delta_\qg=\Delta_{\qg_1}\sqcup\dots\sqcup\Delta_{\qg_s}.
$$ 
Note that $\Delta_{\qg_i}=\Delta_i$ and $\Pi_{\qg_i}=\Pi_i$ for all $i=u+1,\dots,s$.  

\begin{lemma}\label{h11A-lem}
$\sigma=\sigma_\ag+\sigma_\qg\in\Ker\partial\overline{\partial}|_{\Lambda_T^{1,1}}$ if and only if 
\begin{enumerate}[{\rm (i)}] 
\item $h|_{\underline{\ag}_f\times\underline{\ag}_f}=0$.  

\item $h|_{\overline{\ag}\times\overline{\ag}}=-(z_{u+1}\kil_{\ggo_{u+1}}+\dots+z_s\kil_{\ggo_s})|_{\overline{\ag}\times\overline{\ag}}$, for some $z_{u+1},\dots,z_s\in\RR$. 

\item $x_{\alpha+\beta}=x_\alpha+x_\beta-z_i$, for all $\alpha,\beta,\alpha+\beta\in\Delta_{\qg_i}^+$ and $i=1,\dots,s$, where we set $z_1=\dots=z_u=0$. 
\end{enumerate}
In particular, 
$$
\dim{\Ker\partial\overline{\partial}|_{\Lambda_T^{1,1}}}=\dim{\zg(\hg_f)}+\dim{S}, 
$$ 
where $S$ is the subspace of all symmetric bilinear forms $h:\ag\times\ag\rightarrow\RR$ such that $h$ is compatible with $J_\ag$ and satisfies conditions (i) and (ii).  
\end{lemma}

\begin{remark}\label{h11A-strict}
For any C-space $M=G/K$ such that $G$ is semisimple and there is no group factor (i.e., $\zg(\ggo)=0$ and $\overline{\ag}=0$), we have that $S=0$, so $\Ker\partial\overline{\partial}|_{\Lambda_T^{1,1}} =\Ker\overline{\partial}|_{\Lambda_T^{1,1}}$ (see (c)).  
\end{remark}

\begin{remark}\label{xalfa}
Part (iii) is equivalent to the following: for each fixed $i=1,\dots,s$, 
$$
x_\alpha=z_i+\sum_{j=1}^m n_j(x_j-1), \qquad \forall \alpha=\sum_{j=1}^m n_j\alpha_j +\sum_{\beta\in\Pi_{\hg_f}}k_\beta\beta\in\Delta_{\qg_i}^+, \quad n_j,k_\beta\in\NN_0,  
$$
where $m:=\dim{\zg(\hg_i)}$, $\Pi_{\qg_i}=\{\alpha_1,\dots,\alpha_m\}\subset\Delta_{\qg_i}^+$ and $x_j:=x_{\alpha_j}$ for $j=1,\dots,m$.  
\end{remark}

\begin{remark}
It follows from part (i) that if a complex C-space $(M=G/K,J)$ admits a $G$-invariant pluriclosed metric (see \S\ref{SKT-sec}) of the form $g=g_\ag+g_\qg$, then $M$ is the product of a flag manifold and a Lie group (see \eqref{g} for more details and cf.\ Remark \ref{h21-ce} below), i.e., its strict factor is a torus (see Proposition \ref{decC}).    
\end{remark}

\begin{proof}
Since
$$
\Lambda_T^{1,1}=\Lambda^{1,1}_\ag \oplus\Lambda^{1,1}_\qg
\xlongrightarrow{\overline{\partial}}
(\Lambda^{0,1}_\ag\otimes\Lambda^{1,1}_\qg) \oplus\Lambda^{1,2}_\qg 
\xlongrightarrow{\partial} 
(\Lambda^{0,1}_\ag\otimes\Lambda^{2,1}_\qg) \oplus\Lambda^{2,2}_\qg, 
$$
if $\sigma=\sigma_\ag+\sigma_\qg\in\Lambda^{1,1}_\ag\oplus\Lambda^{1,1}_\qg$, then 
$$
\partial\overline{\partial}\sigma_\ag= \partial(d\sigma_\ag)^{0,1+1,1} =(d(d\sigma_\ag)^{0,1+1,1})^{0,1+2,1} +(d(d\sigma_\ag)^{0,1+1,1})^{0,0+2,2}, 
$$
$$
\partial\overline{\partial}\sigma_\qg= \partial(d\sigma_\qg)^{0,0+1,2} =(d(d\sigma_\qg)^{0,0+1,2})^{0,0+2,2},  
$$
and it follows from Lemma \ref{do} that $(d(d\sigma_\ag)^{0,1+1,1})^{0,1+2,1}=0$.  Indeed, its only possibly nonzero component is, for any $A\in\ag^{0,1}$, $\alpha,\beta\in\Delta_\qg^+$: 
\begin{align*}
d(d\sigma_\ag)^{0,1+1,1}(A, E_\alpha,E_\beta,E_{-(\alpha+\beta)}) 
=& -N_{\alpha,\beta}d\sigma_\ag(E_{\alpha+\beta},A,E_{-(\alpha+\beta)}) \\
&+N_{\alpha,-(\alpha+\beta)}d\sigma_\ag(E_{-\beta},A,E_\beta) \\
&-N_{\beta,-(\alpha+\beta)}d\sigma_\ag(E_{-\alpha},A,E_\alpha) \\ 
=&N_{\alpha,\beta}\left(\sigma_\ag(A,A_{\alpha+\beta}^{1,0})-\sigma_\ag(A,A_{\beta}^{1,0}) -\sigma_\ag(A,A_{\alpha}^{1,0})\right) =0, 
\end{align*}
where we have used that $N_{\alpha,-(\alpha+\beta)}=-N_{\alpha,\beta}$ and $N_{\beta,-(\alpha+\beta)}=N_{\alpha,\beta}$ (see \S\ref{roots}).  

For $(d(d\sigma_\ag)^{0,1+1,1})^{0,0+2,2}$ and $\alpha,\beta\in\Delta_\qg^+$, $\alpha\ne\beta$, we have that 
\begin{align}
d(d\sigma_\ag)^{0,1+1,1}(E_\alpha,E_{-\alpha},E_\beta,E_{-\beta}) 
=& -d\sigma_\ag(A_\alpha^{0,1},E_\beta,E_{-\beta}) -d\sigma_\ag(A_\beta^{0,1},E_\alpha,E_{-\alpha}) \label{eqA1}\\ 
=&-\sigma_\ag(A_\alpha^{0,1},A_\beta^{1,0}) -\sigma_\ag(A_\beta^{0,1},A_\alpha^{1,0}). \notag
\end{align} 
Concerning $(d(d\sigma_\qg)^{0,0+1,2})^{0,0+2,2}$, if $\sigma_\qg=(y_\alpha)_{\alpha\in\Delta_\qg}$, then by Lemma \ref{do},  
\begin{enumerate}[{\small $\bullet$}] 
\item for $\alpha,\beta\in\Delta_\qg^+$, $\alpha-\beta\in\Delta_\qg^-$, 
\begin{align}
&d(d\sigma_\qg)^{0,0+1,2}(E_\alpha,E_{-\alpha},E_\beta,E_{-\beta}) \notag \\
=& d\sigma_\qg(N_{\alpha,\beta}E_{\alpha+\beta},E_{-\alpha},E_{-\beta}) 
 -d\sigma_\qg(N_{\alpha,-\beta}E_{\alpha-\beta},E_{-\alpha},E_\beta) \notag \\ 
=& N_{\alpha+\beta,-\alpha}N_{\alpha,\beta}(y_{\alpha+\beta}-y_{\alpha}-y_{\beta}) 
- N_{\alpha-\beta,-\alpha}N_{\alpha,-\beta}(y_{\alpha-\beta}-y_{\alpha}+y_{\beta}) \notag \\ 
=& -N_{\alpha,\beta}^2(y_{\alpha+\beta}-y_{\alpha}-y_{\beta}) 
+ N_{\alpha,-\beta}^2(y_{\alpha-\beta}-y_{\alpha}+y_{\beta}), \label{eqA2}
\end{align}
where $N_{\alpha-\beta,-\alpha}=-N_{\alpha,-\beta}$ (see \S\ref{roots}), 

\item for $\alpha,\beta\in\Delta_\qg^+$, $\alpha-\beta\in\Delta_\qg^+$, 
\begin{align}
&d(d\sigma_\qg)^{0,0+1,2}(E_\alpha,E_{-\alpha},E_\beta,E_{-\beta}) \notag \\
=& d\sigma_\qg(N_{\alpha,\beta}E_{\alpha+\beta},E_{-\alpha},E_{-\beta}) 
 -d\sigma_\qg(N_{-\alpha,\beta}E_{-\alpha+\beta},E_{\alpha},E_{-\beta}) \notag \\ 
=& N_{\alpha+\beta,-\alpha}N_{\alpha,\beta}(y_{\alpha+\beta}-y_{\alpha}-y_{\beta}) 
+ N_{-\alpha+\beta,\alpha}N_{\alpha,-\beta}(y_{-\alpha+\beta}+y_{\alpha}-y_{\beta})  \notag \\ 
=& -N_{\alpha,\beta}^2(y_{\alpha+\beta}-y_{\alpha}-y_{\beta}) 
- N_{\alpha,-\beta}^2(y_{\alpha-\beta}-y_{\alpha}+y_{\beta}), \label{eqA3} 
\end{align}
where $N_{-\alpha+\beta,\alpha}=N_{\alpha,-\beta}$ (see \S\ref{roots}), 

\item for $\alpha+\beta+\gamma+\delta=0$, where all pairs add nonzero and $\alpha,\beta\in\Delta^+$, $\gamma,\delta\in\Delta^-$, 
\begin{align}
& d(d\sigma_\qg)^{0,0+1,2}(E_\alpha,E_\beta,E_\gamma,E_\delta) \notag \\
=& -N_{\alpha,\beta}d\sigma_\qg(E_{\alpha+\beta},E_\gamma,E_\delta) \notag \\ 
&+ N_{\alpha,\gamma}d\sigma_\qg(E_{\alpha+\gamma},E_\beta,E_\delta) 
+ N_{\beta,\delta}d\sigma_\qg(E_{\beta+\delta},E_\alpha,E_\gamma) \notag \\ 
&- N_{\alpha,\delta}d\sigma_\qg(E_{\alpha+\delta},E_\beta,E_\gamma) 
- N_{\beta,\gamma}d\sigma_\qg(E_{\beta+\gamma},E_\alpha,E_\delta), \notag \\
=& -N_{\alpha,\beta} N_{\gamma,\delta} (y_{\alpha+\beta}-y_\gamma-y_\delta) \notag \\ 
&+N_{\alpha,\gamma} N_{\beta,\delta}(y_{\alpha+\gamma}+y_\beta-y_\delta)  
+N_{\beta,\delta}  N_{\alpha,\gamma} (y_{\beta+\delta}+y_\alpha-y_\gamma) \label{eqA4} \\
&-N_{\alpha,\delta}  N_{\beta,\gamma} (y_{\alpha+\delta}+y_\beta-y_\gamma) 
-N_{\beta,\gamma}  N_{\alpha,\delta}(y_{\beta+\gamma}+y_\alpha-y_\delta),   \notag
\end{align}
where at most one term on each line is nonzero.  
\end{enumerate}

It follows from \eqref{eqA1}, \eqref{eqA2} and \eqref{eqA3} that if $\partial\overline{\partial}\sigma=0$, then for all $\alpha,\beta\in\Delta_\qg^+$, $\alpha\ne\beta$,
\begin{equation}\label{h11A-1}
h(\im A_\alpha,\im A_\beta) = N_{\alpha,\beta}^2(x_{\alpha+\beta}-x_{\alpha}-x_{\beta}) 
+\epsilon_{\alpha-\beta}N_{\alpha,-\beta}^2(\epsilon_{\alpha-\beta}x_{\alpha-\beta}-x_{\alpha}+x_{\beta}), 
\end{equation}
where $x_\alpha:=\im\epsilon_\alpha y_\alpha$ and $h=h(\sigma_\ag)$ is as in \eqref{h11A-h}.  

Note that 
$$
\ag_f=\ag_1+\dots+\ag_s, \qquad\mbox{where}\quad \ag_i:=\la\im A_\alpha:\alpha\in\Pi_{\qg_i}\ra_\RR, 
$$ 
so $h(\ag_i,\ag_j)=0$ for all $i\ne j$ by \eqref{h11A-1}.  We define $\Pi_0:=\{\alpha\in\Pi_{\qg_i}:A_\alpha=0\}$ and consider, for a fixed $i\in\{1,\dots,s\}$, the decomposition 
$$
\Pi_{\qg_i}\setminus\Pi_0=\Phi_1\sqcup\dots\sqcup\Phi_v,
$$
in connected components as a graph (in particular, $v=1$ for any $i=u+1,\dots,s$ since $A_\alpha=H_\alpha$ and so $\Pi_0=\emptyset$).  It follows from \eqref{h11A-1} that  
\begin{equation}\label{h11A-2} 
h(\bg_j,\bg_k)=0, \quad\forall j\ne k, \qquad\mbox{where} \quad \bg_j:=\la \im A_\alpha:\alpha\in\Phi_j\ra_\RR, \quad j=1,\dots,v.    
\end{equation} 
Using that each $\Phi_j$ is a connected graph, in much the same way as in the proof of \cite[Proposition 3.6]{SKT-LG}, we obtain from \eqref{h11A-1} that there exist $z_j\in\RR$ such that   
\begin{equation}\label{h11A-3}
h(\im A_\alpha,\im A_\beta)=z_j\la\alpha,\beta\ra =-z_j\kil_{\ggo_i}(\im H_\alpha,\im H_\beta), 
\qquad\forall \alpha,\beta\in \la\Phi_j\ra_{\NN_0},  \quad j=1,\dots,v,  
\end{equation}
including the case when $\alpha=\beta$, the proof is actually identical.  

If $i\in\{ 1,\dots,u\}$, i.e., $\Pi_{\hg_i}\cup\Pi_0\ne\emptyset$, then for any $j=1,\dots,v$ there exist $\alpha\in\Phi_j$ and $\beta\in\Pi_{\hg_i}\cup\Pi_0$ such that $\alpha$ and $\beta$ are adjacent, i.e., $\la\alpha,\beta\ra< 0$, or equivalently, $\tfrac{4\la\alpha,\beta\ra^2}{\la\alpha,\alpha\ra\la\beta,\beta\ra}\in\{ 1,2,3\}$.  According to \eqref{h11A-1}, we have that 
\begin{align}
h(\im A_\alpha,\im A_\beta) =& N_{\alpha,\beta}^2(x_{\alpha+\beta}-x_{\alpha}-x_{\beta}), \label{h11A-11}\\
h(\im A_\alpha,\im A_\alpha) = h(\im A_{\alpha+\beta},\im A_\alpha) =& N_{\alpha+\beta,\alpha}^2(x_{2\alpha+\beta}-x_{\alpha+\beta}-x_{\alpha})  \label{h11A-12}  \\
&+N_{\alpha+\beta,-\alpha}^2(x_{\beta}-x_{\alpha+\beta}+x_{\alpha}), \notag  \\
2h(\im A_\alpha,\im A_\alpha) = h(\im A_{2\alpha+\beta},\im A_\alpha) =& N_{2\alpha+\beta,\alpha}^2(x_{3\alpha+\beta}-x_{2\alpha+\beta}-x_{\alpha}) \label{h11A-13} \\
&+N_{2\alpha+\beta,-\alpha}^2(x_{\alpha+\beta}-x_{2\alpha+\beta}+x_{\alpha}), \notag \\ 
3h(\im A_\alpha,\im A_\alpha) = h(\im A_{3\alpha+\beta},\im A_\alpha) =& N_{3\alpha+\beta,-\alpha}^2(x_{2\alpha+\beta}-x_{3\alpha+\beta}+x_{\alpha}),  \label{h11A-14}
\end{align}
where the terms containing $x_\beta$ in \eqref{h11A-11} and \eqref{h11A-12} must be replaced with $0$ if $\beta\in\Pi_{\hg_f}$.  It is easy to check that both instances $h(\im A_{\alpha},\im A_\alpha)>0$ and $h(\im A_{\alpha},\im A_\alpha)<0$ give rise to a contradiction in the three cases.  This implies that $z_j=0$ for all $j=1,\dots,v$ by \eqref{h11A-3}, that is, $h|_{\ag_i\times\ag_i}=0$ for all $i\in\{ 1,\dots,u\}$.  

Since
$$
\underline{\ag}_f=\ag_1+\dots+\ag_u, \qquad 
\overline{\ag}=\tg_{u+1}\oplus\dots\oplus\tg_s, 
$$   
we obtain that $h|_{\underline{\ag}_f\times\underline{\ag}_f}=0$ and by \eqref{h11A-3}, 
$$
h|_{\overline{\ag}\times\overline{\ag}}=-(z_{u+1}\kil_{\ggo_{u+1}}+\dots+z_s\kil_{\ggo_s})|_{\overline{\ag}\times\overline{\ag}}, \qquad\mbox{for some}\quad z_{u+1},\dots,z_s\in\RR.
$$
Thus conditions (i) and (ii) hold.  Using \eqref{h11A-1}, the proof of (iii) follows by induction on the height of the positive roots.  

Conversely, if these conditions hold for $h$ and the $x_\alpha$, then, by using \eqref{eqA1}, \eqref{eqA2}, \eqref{eqA3} and \eqref{eqA4}, it is easy to check that $\sigma=\sigma_\ag+\sigma_\qg\in\Ker\partial\overline{\partial}|_{\Lambda^{1,1}}$, where $\sigma_\ag$ is defined by $h$ and $\sigma_\qg:=(y_\alpha)_{\alpha\in\Delta_\qg}$, $y_\alpha:=-\im\epsilon_\alpha x_\alpha$.  

Finally, the formula $\dim{\Ker\partial\overline{\partial}|_{\Lambda_T^{1,1}}} = \dim{S}+\dim{\zg(\hg_f)}$ follows from the fact that the $x_\alpha$, $\alpha\in\Delta_\qg^+$ are determined by $\{ x_\beta:\beta\in\Pi_\qg\}$ as in observation (c) or Remark \ref{xalfa}, concluding the proof.  
\end{proof}

\begin{theorem}\label{h11A} \hspace{1cm}
\begin{enumerate}[{\rm (i)}]
\item For any complex C-space $(M=G/K,J)$, 
$$
h^{1,1}_A= \dim{\zg(\kg)} + \dim{S}, 
$$
where $S$ is the subspace defined in Lemma \ref{h11A-lem}.  

\item If $M=G$ is a semisimple Lie group and $J$ is irreducible (see Definition \ref{Cirr}), then 
$$
h^{1,1}_A =
\left\{\begin{array}{lcl}
1, &\quad& \mbox{if $J$ is compatible with a non-zero bi-invariant} \\ 
&\quad& \mbox{symmetric bilinear form on $\ggo$}, \\ 
0, &\quad& \mbox{otherwise}.
\end{array}\right.
$$
\end{enumerate}
\end{theorem}

\begin{remark}\label{h11A-strict2}
According to Remark \ref{h11A-strict}, $h^{1,1}_A=\dim{\zg(\kg)}$ for any complex C-space $M=G/K$ with $G$ semisimple and without group factor (i.e., $\zg(\ggo)=0$ and $\overline{\ag}=0$).  
\end{remark}

\begin{remark}\label{h11A-red} 
In part (ii), $h^{1,1}_A=1$ in the particular case when $J$ is compatible with a bi-invariant metric, a result that was obtained in \cite{Brb}.  The biinvariant form in (ii) is unique up to scaling and necessarily non-degenerate by irreducibility, i.e., $z_i\ne 0$ for all $i=u+1,\dots,s$ in Lemma \ref{h11A-lem}, (ii) ($u=0$ in this case).  We note that if $J$ is reducible, then $h^{1,1}_A$ is precisely the number of irreducible complex factors of $(G,J)$ (see Definition \ref{Cirr}), or equivalently, the dimension of the vector space of all $J$-compatible symmetric bilinear forms on $\ggo$.      
\end{remark}

\begin{remark} 
It is common to have $h^{1,1}_A\geq 2$ for irreducible complex C-spaces beyond the semisimple Lie group case $M=G$, even with $b_1=0$, see Examples \ref{A1} and \ref{A2}.    
\end{remark}

\begin{proof}
Since $\overline{\partial}\Lambda^{1,0}_\ag=d\Lambda^{1,0}_\ag$ and $\partial\Lambda^{0,1}_\ag=d\Lambda^{0,1}_\ag$, we have that 
$$
\Ima \partial+\Ima \overline{\partial}=d\Lambda^1 =\{\sigma_A:A\in\ag^c\}, 
\qquad
\dim{d\Lambda^1}= \dim{\ag_f}=2d-\dim{\zg(\ggo)}.    
$$
We therefore obtain from Lemma \ref{h11A-lem} that 
\begin{align*}
h^{1,1}_A=& \dim{\Ker\partial\overline{\partial}|_{\Lambda^{1,1}}} - \dim{(\Ima\partial+\Ima\overline{\partial})} \\ 
=& \dim{S}+\dim{\zg(\hg_f)} - \dim{\ag_f} =\dim{S}+\dim{\zg(\kg)},
\end{align*}
so part (i) follows.  

If $M=G$ and $G$ is semisimple as in part (ii), then $\ag=\overline{\ag}=\tg_1\oplus\dots\oplus\tg_s$, where $\ggo=\ggo_1\oplus\dots\oplus\ggo_s$ in simple factors.  Any $h\in S$ is therefore a biinvariant form as in Lemma \ref{h11A-lem}, (ii), which are necessarily non-degenerate if nonzero by irreducibility.  Moreover, given $h,h'\in S$, the operator $H:\overline{\ag}\rightarrow\overline{\ag}$ defined by $h=h'(H\cdot,\cdot)$ leaves the above decomposition invariant and satisfies that $H|_{\tg_i}=\tfrac{z_i}{z_i'}I$ for all $i=1,\dots,s$.  Since $J_\ag$ commutes with $H$, we obtain that $H$ is a multiple of the identity map by irreducibility.  Thus $\dim{S}\leq 1$, where equality holds if and only if $M=G$ admits a non-zero biinvariant symmetric form compatible with $J_\ag$.  
\end{proof}

If the $2d\times 2d$-matrices of $J_\ag$ and of a symmetric bilinear form $h$ in terms of the same basis are respectively given by 
\begin{equation}\label{JH}
J_\ag=\left[\begin{matrix} 0&-I\\ I&0 \end{matrix}\right], \qquad 
H=\left[\begin{matrix} A&B\\ -B&A \end{matrix}\right], \quad A^t=A, \quad B^t=-B,
\end{equation}
then $h$ is compatible with $J_\ag$, giving rise to a subspace of $(1,1)$-forms of dimension $d^2$.  Note that $B=0$ for $d=1$.   

\begin{example}\label{A1}
We consider a complex C-space $(M=G/K,J)$ without flag factor such that 
$$
\dim{\underline{\ag}_f}=\dim{\zg(\ggo)}+\dim{\overline{\ag}} = d,
$$ 
and $J_\ag$ is as in \eqref{JH} relative to the decomposition $\ag=\underline{\ag}_f\oplus(\zg(\ggo)\oplus\overline{\ag})$ and some $Q$-orthonormal basis.  In particular, $J$ is always irreducible. Since $h\in S$ if and only if $A=0$, we obtain that $\dim{S}=\binom{d}{2}$ and hence 
$$
h^{1,1}_A=\dim{\zg(\kg)}+\tbinom{d}{2}, 
$$
by Theorem \ref{h11A}, (i).  Note that $M$ is a strict C-space with $b_1=d$ if $\overline{\ag}=0$, and that $M$ has finite fundamental group (or $G$ semisimple, or $b_1=0$) if $\zg(\ggo)=0$.   
\end{example}

\begin{example}\label{A2}
If $(M=G=T^z\times\overline{G},J)$ is a Lie group such that 
$$
\dim{\zg(\ggo)}=\dim{\overline{\ag}} = d,
$$ 
and $J_\ag$ is as in \eqref{JH} relative to the decomposition $\ag=\zg(\ggo)\oplus\overline{\ag}$ and some $Q$-orthonormal basis, then $J$ is irreducible and $h\in S$ if and only if $A$ is the diagonal matrix with entries $z_i$ of multiplicity $\dim{\tg_i}$, $i=1,\dots,s$, where $\overline{\ggo}=\ggo_1\oplus\dots\oplus\ggo_s$ in simple factors (see Lemma \ref{h11A-lem}).  Thus $\dim{S}=s+\binom{d}{2}$ and by Theorem \ref{h11A}, (i),
$$
h^{1,1}_A=s+\tbinom{d}{2}. 
$$  
This provides a large family of pluriclosed left-invariant metrics on $M=G$ which, on the maximal torus $\ag$ of $\ggo$, are not the restriction of a biinvariant metric (see \S\ref{SKT-sec}).  
\end{example}

We now consider instead the following $2d\times 2d$-matrices of $J_\ag$ and of a compatible symmetric bilinear form $h$ in terms of the same basis, where $d$ is even: 
\begin{equation}\label{JH2}
J_\ag=\left[\begin{matrix} 
J_1&0\\ 
0&J_1 
\end{matrix}\right], \qquad 
H=\left[\begin{matrix} A&B\\ B^t&C \end{matrix}\right], \quad A^t=A, \quad C^t=C, 
\end{equation}
where $J_1$ is any skew-symmetric $d\times d$-matrix such that $J_1^2=-I$ and $AJ_1=J_1A$, $BJ_1=J_1B$ and $CJ_1=J_1C$.  This also gives a subspace of $(1,1)$-forms of dimension $d^2=\tfrac{d^2}{4}+\tfrac{d^2}{2}+\tfrac{d^2}{4}$.     

\begin{example}\label{A3}
Here we take $J_\ag$ as in \eqref{JH2} relative to the corresponding decompositions. Under the assumption in Example \ref{A1}, we obtain that $J$ is always reducible, $A=0$ and by Theorem \ref{h11A}, (i),
$$
h^{1,1}_A=\dim{\zg(\kg)}+\tfrac{d^2}{2} +t +\tbinom{\dim{\zg(\ggo)}+1}{2} + \dim{\zg(\ggo)}\dim{\overline{\ag}},
$$
where $t=h^{1,1}_A(\overline{G})$ is the number of irreducible complex factors of $(\overline{G},(J_1,J_{\overline{\qg}}))$ (see Remark \ref{h11A-red}).  Note that $\dim{\zg(\kg)}=h^{1,1}_A(\underline{G}/K)$, where the strict C-space $\underline{G}/K$ is endowed with the complex structure $(J_1,J_{\underline{\qg}}))$.  

On the other hand, in the Lie group case as in Example \ref{A2},  
$
h^{1,1}_A=\tfrac{3d^2}{4} +s.  
$
\end{example}

\begin{remark}\label{c1A}
If the {\it first Aeppli-Chern class} $c_1^{AC}(M):=[\sigma]\in H^{1,1}_A(M)$ of a complex C-space $M$ (see \cite{BrdStn}) vanishes, then $c_1(M)=0$ (see \S\ref{fCc-sec}).  Indeed, if $\sigma=\partial\theta_A+\overline{\partial}\theta_B$ for some $A\in\ag^{0,1}$ and $B\in\ag^{1,0}$, then $\sigma=d\theta_{A+B}$ and so $[\sigma]=0\in H^2_{dR}(M)$.  
\end{remark}

\subsection{Dolbeault and Aeppli interplay}\label{h21-h11A}
In this section, we provide a lower bound for the Hodge number $h^{2,1}$ of a complex C-space $M=G/K$ with $G$ semisimple, which according to Proposition \ref{h3D} is given by  
$$
h^{2,1}=D-\tbinom{d}{2}, \qquad 
D=\dim{\Ker{\overline{\partial}|_{(\Lambda^{1,0}_\ag\otimes\Lambda^{1,1}_\qg)\oplus \Lambda^{2,1}_\qg}}}. 
$$ 
The bound is obtained by using the map $\partial:\Lambda^{1,1}_T\rightarrow\Lambda^{2,1}_T$, which, as well known, defines a map in cohomology given by $H^{1,1}_A\rightarrow H^{2,1}_{\overline{\partial}}$, $[\sigma]\mapsto[\partial\sigma]$.  Indeed, if $[\sigma]=[\sigma']$ in $H^{1,1}_A$, then $\sigma-\sigma'=\partial(\theta_1)+\overline{\partial}(\theta_2)$ and so $\partial\sigma-\partial\sigma'=\partial \overline{\partial}(\theta_2)$, that is, $[\partial\sigma]=[\partial\sigma']$ in $H^{2,1}_{\overline{\partial}}$.  

\begin{proposition}\label{h21bound}
For any complex C-space $M=G/K$ with $G$ semisimple, 
$$
h^{2,1}\geq\dim{S}, 
$$
where $S$ is the subspace of all symmetric bilinear forms $h:\ag\times\ag\rightarrow\RR$ such that $h$ is compatible with $J_\ag$ and satisfies Lemma \ref{h11A-lem}, (i) and (ii). 
\end{proposition}

\begin{remark}\label{h21-LG}
In particular, for the Lie group $M=T^d\times\overline{G}$, where $\rank(\overline{G})=d$, as in Example \ref{A2}, we obtain that $h^{2,1}\geq s+\tbinom{d}{2}$.   
\end{remark}

\begin{proof}
Since
$$
\partial:\Lambda^{1,1}_T=\Lambda^{1,1}_\ag\oplus\Lambda^{1,1}_\qg \rightarrow 
(\Lambda^{1,0}_\ag\otimes\Lambda^{1,1}_\qg)\oplus \Lambda^{2,1}_\qg \subset \Lambda^{2,1}_T,
$$
$\partial\Ker{\partial\overline{\partial}|_{\Lambda^{1,1}_T}}\subset \Ker{\overline{\partial}|_{(\Lambda^{1,0}_\ag\otimes\Lambda^{1,1}_\qg)\oplus \Lambda^{2,1}_\qg}}$ and we obtain from Lemma \ref{h11A-lem} that 
$$
\dim{\partial\Ker{\partial\overline{\partial}|_{\Lambda^{1,1}_T}}} 
= \dim{\zg(\hg_f)}+\dim{S} - \dim{\Ker{\partial|_{\Ker{\partial\overline{\partial}|_{\Lambda^{1,1}_T}}}}}.   
$$ 
If $\sigma=\sigma_\ag+\sigma_\qg\in \Ker{\partial\overline{\partial}|_{\Lambda^{1,1}_T}}$ as described in Lemma \ref{h11A-lem}, then $\partial\sigma=\partial\sigma_\ag+\partial\sigma_\qg=0$ if and only if $\partial\sigma_\ag=0$ and $\partial\sigma_\qg=0$, if and only if $\sigma_\ag=0$ since $\ag=\ag_f$ (in particular, $z_{u+1}=\dots=z_s=0$) and $x_{\alpha+\beta}=x_\alpha+x_\beta$ for all $\alpha,\beta,\alpha+\beta\in\Delta_\qg^+$.  This implies that 
$
\dim{\Ker{\partial|_{\Ker{\partial\overline{\partial}|_{\Lambda^{1,1}_T}}}}} = \dim{\zg(\hg_f)},
$  
so $\dim{\partial\Ker{\partial\overline{\partial}|_{\Lambda^{1,1}_T}}} =\dim{S}$.  

On the other hand, if $[\partial\sigma]=0\in H^{2,1}_{\overline{\partial}}$, then there exists $\sigma'\in\Lambda_T^{2,0}=\Lambda_\ag^{2,0}$ such that $\partial\sigma=\overline{\partial}\sigma'$, that is, $\partial\sigma_\ag=\overline{\partial}\sigma'$ and $\partial\sigma_\qg=0$.  It follows from Lemma \ref{do} that
$$
\sigma_\ag(A,A_\alpha^{0,1}) = d\sigma_\ag(A,E_\alpha,E_{-\alpha}) 
= d\sigma'(A,E_\alpha,E_{-\alpha}) =\sigma'(A,A_\alpha^{1,0}), \qquad\forall A\in\ag^{1,0}, \alpha\in\Delta_\qg, 
$$
which implies that $\sigma_\ag(A,B^{0,1}) = \sigma'(A,B^{1,0})$ for all $B\in\ag^c$ since the $A_\alpha$ generates $\ag$.  Thus $\sigma_\ag=\sigma'=0$ and hence $\partial\sigma=0$.  

We therefore obtain that 
$$
h^{2,1}=\dim{H^{2,1}_{\overline{\partial}}}\geq \dim{\partial\Ker{\partial\overline{\partial}|_{\Lambda^{1,1}_T}}} =\dim{S},
$$
as was to be shown.  
\end{proof}

\begin{example}\label{h21}
For any complex C-space $(M=G/K,J)$ without flag factor such that $G$ is semisimple as in Example \ref{A1}, i.e., $\dim{\underline{\ag}}=\dim{\overline{\ag}} = d$ and $J_\ag$ is as in \eqref{JH}, we obtain from Proposition \ref{h21bound} that $h^{1,1}_A=\dim{\zg(\kg)}+\tbinom{d}{2}$ and 
$$
h^{2,1}\geq\tbinom{d}{2}.   
$$   
The simplest examples of this kind having $h^{2,1}>0$ are $M^{20}=\SU(4)/\SU(2)\times\SU(3)$, which has $d=2$ and $h^{1,1}_A=1$, and actually any C-space of the form $M=\underline{G}/\SU(2)\times\overline{G}$, where $\underline{G}$ and $\overline{G}$ are simple Lie groups of rank $3$ and $2$, respectively.  One can also take $M=\underline{G}/S^1\times\overline{G}$, for which $h^{1,1}_A=2$ and $h^{2,1}>0$.  
\end{example}

\begin{remark}\label{h21-ce}
Since all the above examples are irreducible as complex manifolds, they provide counterexamples to \cite[Lemma 6.4]{FinGrnVzz}.  This lemma was strongly used in the proof of \cite[Theorem 6.1]{FinGrnVzz}, that is, the classification of pluriclosed complex C-spaces.  In \S\ref{SKT-sec}, we give an alternative proof for such classification using the full description of $\Ker{\partial\overline{\partial}|_{\Lambda^{1,1}_T}}$ given in Lemma \ref{h11A-lem}.  
\end{remark}

\subsection{Examples}\label{Exc-sec} 
As an application of \S\ref{dR-sec}, formulas \eqref{hn}, Propositions \ref{eqs}, \ref{h3D}, \ref{hBC} and Theorem \ref{h11A}, we compute many low Betti, Hodge, Bott-Chern and Aeppli numbers among some classes of C-spaces.   

\begin{example}\label{su3t3}
For any complex C-space $(M=G/K,J)$ such that 
$$
\dim{\zg(\ggo)}=d+1, \quad\dim{\ag_f}=\dim{\ag_f^{1,0}}=d-1, \quad \dim{\zg(\ggo)^c\cap\ag^{1,0}}=1,  
$$ 
which is not K\"ahler for any $d\geq 2$, we obtain that 
$$
h^{1,0}=1, \quad h^{0,1}=d, \qquad b_1=d+1, 
$$
$$
h^{2,0}=0, \quad h^{1,1}=h^{1,1}_{BC}=\dim{\zg(\hg_f)}+1, \quad h^{0,2}=\tbinom{d}{2}, \qquad b_2=\dim{\zg(\kg)}+\tbinom{d+1}{2}.   
$$
In particular, $h^{1,0}+h^{0,1}=b_1$ and $h^{2,0}+h^{1,1}+h^{0,2}=b_2$, i.e., the Fr\"olicher inequality is sharp for $k=1,2$.  As an explicit example, we consider   
any invariant complex structure on the C-space $M^{8}=\SU(3)/\SU(2)\times T^3$, where $F^{4}=\SU(3)/\Se(\U(1)\times\U(2))$, $\dynkin[scale=2] A{o*}$, $\pg=\ag_f\oplus\zg(\ggo)\oplus\qg$, 
$$
\dim{\zg(\ggo)}=3, \quad \dim{\ag_f}=\dim{\zg(\hg_f)}=1, \quad \dim{\qg}=4, \quad d=2. 
$$
Thus
$$
h^{1,0}=1, \quad h^{0,1}=2, \qquad b_1=3, 
$$
$$
h^{2,0}=0, \quad h^{1,1}=h^{1,1}_{BC}=2, \quad h^{0,2}=1, \qquad b_2=3, 
$$
It is easy to check that $h^{1,1}_A=3$.  
\end{example}

\begin{example}\label{LG-c}
Consider the Lie group case given in Example \ref{LG}, i.e., $M=G=\overline{G}\times T^z$, where $\ggo=\overline{\ggo}\oplus\zg(\ggo)$, $\ag=\overline{\tg}\oplus\zg(\ggo)$ and $\ag_f=\overline{\tg}=\zg(\hg_f)$, the maximal torus of the semisimple Lie algebra $\overline{\ggo}$.  We obtain that  
$$
b_1=\dim{\zg(\ggo)}, \quad b_2=\tbinom{\dim{\zg(\ggo)}}{2}, \quad b_3=\tbinom{\dim{\zg(\ggo)}}{3}+s,
$$
where $s$ is the number of simple factors of $\overline{\ggo}$, 
$$
h^{1,0}=\dim{\zg(\ggo)^c\cap\ag^{1,0}}, \quad h^{0,1}=d, 
$$
$$
h^{2,0}=\tbinom{d-\dim{\overline{\tg}^{1,0}}}{2}, \quad h^{1,1}=(d-\dim{\overline{\tg}^{1,0}})d+\dim{\overline{\tg}}-d+\dim{\zg(\ggo)^c\cap\ag^{1,0}}, \quad h^{0,2}=\tbinom{d}{2}, 
$$
$$
h^{3,0}=\tbinom{d-\dim{\overline{\tg}^{1,0}}}{3}, \qquad h^{2,1}=D-\tbinom{d}{2} 
+(d+1)\tbinom{d-\dim{\overline{\tg}^{1,0}}}{2}, 
$$
$$
h^{1,2}= d (\dim{\overline{\tg}}-d) +\tbinom{d+1}{2}(d-\dim{\overline{\tg}^{1,0}}),  \qquad h^{0,3}=\tbinom{d}{3}, 
$$
$$
h^{1,1}_{BC}= \dim{\overline{\tg}}+ (d-\dim{\overline{\tg}^{1,0}})^2, \qquad h^{2,1}_{BC}=\dim{\Lambda^{2,1}_c}. 
$$
In particular, $h^{1,0}+h^{0,1}=b_1$ if and only if $h^{2,0}+h^{1,1}+h^{0,2}=b_2$, if and only if $\zg(\ggo)+J_\ag\zg(\ggo)=\zg(\ggo)\oplus\overline{\tg}$.  See Theorem \ref{h11A} and Example \ref{A2} for $h_A^{1,1}$ and Remark \ref{h21-LG} for $h^{2,1}$.  Note that the above formulas greatly simplify in the case when $M=G$ is a semisimple Lie group, i.e., $\zg(\ggo)=0$ and $\dim{\overline{\tg}^{1,0}}=d$.   
\end{example}

\begin{example}\label{b21-c}
For the generalized Calabi-Eckmann manifolds given in Example \ref{b21-comp}, with $F_2\ne S^1$, $d=1$, $\zg(\ggo)=\zg(\kg)=0$, $s=2$ and $\dim{\zg(\hg_f)}=2$, so we have that  
$$
h^{1,0}=0, \quad h^{0,1}=1, \qquad b_1=0, 
$$
$$
h^{2,0}=0, \quad h^{1,1}=1, \quad h^{0,2}=0, \qquad b_2=0, 
$$
$$
h^{3,0}=0, \quad h^{2,1}=D, 
\quad
h^{1,2}=1,  \quad h^{0,3}=0, \qquad b_3=0,
$$
$$
h^{1,1}_{BC}= 2, \quad h^{2,1}_{BC}=0; \qquad h^{1,1}_A=0.
$$
In particular, $h^{2,0}+h^{1,1}+h^{0,2}-b_2=1$ and $h^{3,0}+h^{2,1}+h^{1,2}+h^{0,3}-b_3=D+1$.  
\end{example}

\begin{example}\label{b21-c2}
If in the above example we consider $F_2=S^1$, then $d=1$, $\dim{\zg(\ggo)}=1$, $\zg(\kg)=0$ and $\dim{\zg(\hg_f)}=1$, so we have that  
$$
h^{1,0}=0, \quad h^{0,1}=1, \qquad b_1=1, 
$$
$$
h^{2,0}=0, \quad h^{1,1}=0, \quad h^{0,2}=0, \qquad b_2=0, 
$$
$$
h^{3,0}=0, \quad h^{2,1}=D, 
\quad
h^{1,2}=1,  \quad h^{0,3}=0, \qquad b_3=0,
$$
$$
h^{1,1}_{BC}= 1, \quad h^{2,1}_{BC}=0; \qquad h^{1,1}_A=0.
$$
\end{example}

\begin{example}\label{h21-cex}
Starting from the full flag manifold $F=\SU(3)/\Se(\U(1)^3)$, for each pair $p,q\in\NN$, we consider the C-space $M^{14}=G/K=\SU(3)\times\SU(3)/K$ endowed with any $G$-invariant complex structure $J=(J_\ag,J_\qg)$ as in \eqref{J}, where, if $\hg_1$ denotes the Lie algebra of the $2$-torus $\Se(\U(1)^3)$, then 
$$
\kg:= \{ (pZ,qZ):Z\in\hg_1\}\subset\hg_1\oplus\hg_1=\hg, \qquad 
\ag:=\{ (-qZ,pZ):Z\in\hg_1\}\simeq\hg_1.    
$$ 
According to Example \ref{BRF-bal}, $H^3(M)=\RR[\vp_R]$, where $R=\kil_{\ggo_1}-\tfrac{p^2}{q^2}\kil_{\ggo_2}$.  
Here $d=1$, $\zg(\ggo)=0$, $\dim{\zg(\kg)}=2$, $\dim{\zg(\hg_f)}=4$ and $s=2$.  We have that
$$
h^{1,0}=0, \quad h^{0,1}=1, \qquad b_1=0, 
$$
$$
h^{2,0}= 0, \quad h^{1,1}=3, \quad h^{0,2}= 0, \qquad b_2=2, 
$$
$$
h^{3,0}= 0, \quad h^{2,1}=D, 
\quad
h^{1,2}=3,  \quad h^{0,3}= 0, \qquad b_3=1,
$$
$$
h^{1,1}_{BC}=4; \qquad h^{1,1}_A=2. 
$$
\end{example}

% 28/8/2026

\section{Homogeneous Hermitian manifolds}\label{chhm-sec} 

In this section, we study the geometry of $G$-invariant Hermitian structures on a C-space.  

\subsection{Hermitian metrics}\label{met-sec}
Given a C-space as in \eqref{gprime}, 
\begin{quote}
$M=G/K=G_f/K\times T^z$, which fibers on the flag $F=G_f/H_f$ with fiber the torus $A=H_f/K\times T^z$,  
\end{quote}
we consider the reductive decomposition $\ggo=\kg\oplus\pg$, $\pg=\ag\oplus\qg$ as in \eqref{reddec2} and the $G$-invariant Hermitian structures given by 
\begin{equation}\label{Jg} 
(J,g), \quad J=(J_\ag,J_\qg), \quad J_\qg=(\epsilon_\alpha)_{\alpha\in\Delta_\qg} \leftrightarrow \Delta_\qg^+, \qquad 
g=g_\ag+g_\qg, \quad g_\qg=(x_\alpha)_{\alpha\in\Delta_\qg}, 
\end{equation}
where $J$ and $g$ are respectively defined in \eqref{J} and \eqref{g} and  
$\Delta_\qg^+$ is the invariant ordering defining $J_\qg$, i.e., $\Delta_\qg^+=\{\alpha\in\Delta_\qg:\epsilon_\alpha=1\}$ (see \S\ref{comp}).  The compatibility condition is simply $g_\ag(J_\ag\cdot,J_\ag\cdot)=g_\ag(\cdot,\cdot)$, giving rise to a large class of $G$-invariant Hermitian structures $(J,g)$ on a given $M=G/K$.  We note that the metric $g$ is projectable on the base $F$ and makes of the Tits fibration $A\rightarrow M\rightarrow F$ a Riemannian fibration.  

In particular, for any $G$-invariant metric $g=g_\ag+g_\qg$, there always exist infinitely many $G$-invariant complex structure $J=(J_\ag,J_\qg)$ with respect to which $g$ is Hermitian, unless $\ag=0$ (i.e., $M$ is a flag manifold).  

The following definition will make the presentation simpler.  

\begin{definition}\label{HC-def}
The Hermitian manifold $(M=G/K,J,g)$, where $(J,g)$ is as in \eqref{Jg}, is called a {\it Hermitian C-space}.  
\end{definition}

\begin{remark}\label{aver}
According to \S\ref{extra-sec}, there may exist $G$-invariant Hermitian metrics on a complex C-space $(M=G/K,J)$ which are not of the form $g=g_\ag+g_\qg$ as in \eqref{g} for any flag, i.e., which do not descend to the base of any Tits fibration.  Nevertheless, since $H_f\subset N_G(K)$, for any $G$-invariant metric $\hat{g}$ on a C-space $M=G/K=G_f/K\times T^z$, we can consider the {\it symmetrization} with respect to the right $H_f$-action given by 
$$
g(X,Y) := \int_{H_f} R_h^*\hat{g}(X,Y)\; d\nu(h), 
$$
where $R_h(aK):=ahK$ for all $h\in H_f$, $a\in G$ (see \ref{gauge}).  Thus $g$ is $G$-invariant and in addition $\Ad(H_f)$-invariant, which implies that it is of the form $g=g_\ag+g_\qg$ as in \eqref{g}.  Some properties are invariant under symmetrization, e.g., pluriclosed and CYT (see \cite{FinGrn1}). 
\end{remark}

The K\"ahler form $\omega=g(J\cdot,\cdot)\in\Lambda^{1,1}$ (see \S\ref{cohom-sec}) is given by 
\begin{equation}\label{o}
\omega=\omega_\ag+\omega_\qg, 
\end{equation}
where $\omega_\ag=g_\ag(J_\ag\cdot,\cdot)$, and the only nonzero components of $\omega_\qg$ are   
$$
\omega_\qg(E_\alpha,E_{-\alpha})=-\im\epsilon_\alpha x_\alpha,  \qquad\forall \alpha\in\Delta_\qg,   
$$
or equivalenlty, $\omega_\qg(e_\alpha,f_{\alpha})= x_\alpha$ for all $\alpha\in\Delta_\qg^+$.  

The following vector, called {\it metric Koszul vector}, plays a crucial role in the geometry of a Hermitian C-space $(M=G/K,J,g)$ (see \cite[Section 3]{Pds}): 
\begin{equation}\label{koszul-g}
Z_{J_\qg,g_\qg}:=\sum_{\alpha\in\Delta_\qg^+}\tfrac{1}{x_\alpha}  \im H_\alpha 
=\sum_{\alpha\in\Delta_\qg^+}\tfrac{1}{x_\alpha} Z_\alpha  \in \zg(\hg_f)=\zg(\kg)\oplus\ag_f.  
\end{equation}
Note that $Z_{J_\qg,g_\qg}\in C(J_\qg)$, the cone defined in \eqref{cone}; indeed, it follows from \eqref{rhoi} that 
$$
Z_{J_\qg,g_\qg}:= \tfrac{\dim{\mg_1}}{2x_1}Z_{\alpha_1}+\dots+\tfrac{\dim{\mg_r}}{2x_r}Z_{\alpha_r}, 
$$
where $\alpha_1,\dots,\alpha_r$ are the central roots.  Moreover, 
$$
C(J_\qg)=\{Z_{J_\qg,g_\qg}:g_\qg\;\mbox{is a $G_f$-invariant metric on $F=G_f/H_f$}\}.
$$
We have that $Z_{J_\qg,g_\qg}=Z_{J_\qg}$ (i.e., the Koszul and metric Koszul vectors coincide, see \eqref{koszul2}) for any metric of the form $g=g_\ag+\gk$, where $\gk=-\kil_{\ggo_f}|_{\qg}$ is the standard metric on the flag $F=G_f/H_f$ (i.e., $x_\alpha=1$ for all $\alpha$).  

\begin{example}\label{b21-herm}
We consider the generalized Calabi-Eckmann manifolds given in Example \ref{b21-comp}, where $\zg(\kg)=0$, $\ag=\ag_f=\zg(\hg_f)=\zg(\hg_1)\oplus\zg(\hg_2)$ is $2$-dimensional and 
$$
J_\ag = \left[\begin{matrix} 
a&-\tfrac{a^2+1}{b}\\ b&-a 
\end{matrix}\right], \qquad a,b\in\RR, \quad b\ne 0,  
$$
in terms of the basis $\{Z_{\gamma_1},Z_{\gamma_2}\}$ of $\ag$, $Z_{\gamma_i}=(\im H_{\gamma_i})_{\zg(\hg_i)}$ (see also Example \ref{b21}).  The $G$-invariant Hermitian metrics of the form $g=g_\ag+g_\qg$ are given by 
$$
g_\ag=c\left[\begin{matrix} 
b^2&-ab\\ -ab&a^2+1 
\end{matrix}\right], \quad c>0, \qquad g_\qg=(x_1,\dots,x_{r_1})+(y_1,\dots,y_{r_2}), \quad x_i,y_i>0.
$$
Here $g_\ag$ is written in terms of the basis $\{Z_{\gamma_1},Z_{\gamma_2}\}$ of $\ag$.  Their metric Koszul vectors are therefore given by  
$$
Z_{J_\qg,g_\qg}=\unm\left(\sum_{j=1}^{r_1}\tfrac{j\dim{\mg_{1,j}}}{x_j}\right)Z_{\gamma_1}+ \unm\left(\sum_{j=1}^{r_2}\tfrac{j\dim{\mg_{2,j}}}{y_j}\right)Z_{\gamma_2},   
$$
and hence $C(J_\qg)=\la Z_{\gamma_1},Z_{\gamma_2}\ra_{\RR_{>0}}\subset\zg(\hg_f)=\ag_f=\ag$.  
\end{example}

\subsection{K\"ahler form}\label{form-sec}
We compute in this section some useful formulas related to the K\"ahler form $\omega$ of a Hermitian C-space.  We will use the notation of \S\ref{form2-sec}.  

The following result is a direct consequence of Lemma \ref{do} and the fact that $y_\alpha=-\im\epsilon_\alpha x_\alpha$ for the K\"ahler form $\omega$.  

\begin{lemma}\label{dco}
For any Hermitian C-space, the only possibly nonzero components of $d\omega$ are given by 
$$
d\omega(E_\alpha,E_\beta,E_\gamma)= -\im N_{\alpha,\beta}(\epsilon_\alpha x_\alpha+\epsilon_\beta x_\beta+\epsilon_\gamma x_\gamma),  \qquad \alpha+\beta+\gamma=0, 
$$
$$
d\omega(A,E_\alpha,E_{-\alpha}) =\omega_\ag(A,A_\alpha)=g_\ag(J_\ag A,A_\alpha), \qquad \forall A\in\ag^c, \;\alpha\in\Delta_\qg,  
$$
and for the $3$-form $d^c\omega=-d\omega(J\cdot,J\cdot,J\cdot)$,  
$$
d^c\omega(E_\alpha,E_\beta,E_\gamma)= \epsilon_\alpha\epsilon_\beta\epsilon_\gamma N_{\alpha,\beta}(\epsilon_\alpha x_\alpha+\epsilon_\beta x_\beta+\epsilon_\gamma x_\gamma), \qquad \alpha+\beta+\gamma=0, 
$$
$$
d^c\omega(A,E_\alpha,E_{-\alpha}) =g_\ag(A,A_\alpha),  \qquad \forall A\in\ag^c,\;\alpha\in\Delta_\qg.  
$$
On the other hand, 
$$
d\omega(A,e_\alpha,f_\alpha) = \omega_\ag(A,\im A_\alpha) 
= g_\ag(J_\ag A, \im A_\alpha), \qquad\forall A\in\ag, \;\alpha\in\Delta_\qg^+.  
$$
\end{lemma}

In particular, $d^c\omega=0$ if and only if $d\omega=0$, if and only if $\ag_f=0$ (i.e., $(J,g)$ is K\"ahler and $(J_\qg,g_\qg)$ is K\"ahler on $F$, see Remark \ref{ext}, (iii)).  

Let $P_\ag:\ag\rightarrow\ag$ denote the positive definite $Q$-symmetric linear operator such that 
$$
g_\ag(X,Y)=Q(P_\ag X,Y), \qquad\forall X,Y\in\ag,
$$ 
and define
$$
\theta_Z:=g(\cdot,Z)\in(\Omega^1M)^G, \quad\forall Z\in\pg_0, \quad\mbox{which gives}\quad (\Omega^1M)^G = \{\theta_Z:Z\in\pg_0\}\simeq\pg_0. 
$$
Note that $\ag\subset\pg_0$ and we are using $g$ instead of $Q$ to define $\theta_Z$ (cf.\ \S\ref{dR-sec}).  If $Z\in\ag$, then $\theta_Z=g(\cdot,Z)=Q(\cdot,P_\ag Z)$ and so $d\theta_Z=0$ if and only if $P_\ag Z\in\zg(\ggo)$.  

Recall from \S\ref{Ht} the definition of $d_g^*$, the adjoint of $d$ with respect to $g$ (our convention is that $g(e_1\wedge e_2,e_1\wedge e_2)=1$ as in \cite{Gdc}).  

\begin{lemma}\label{dto} 
For any Hermitian C-space, 
$$
d_g^*\omega = -\theta_{(Z_{J_\qg,g_\qg})_\ag} =-g(\cdot,(Z_{J_\qg,g_\qg})_\ag), \qquad 
dd_g^*\omega= -\sigma_{P_\ag(Z_{J_\qg,g_\qg})_\ag} =g([\cdot,\cdot]_\pg,(Z_{J_\qg,g_\qg})_\ag), 
$$
where $g_\ag=Q(P_\ag\cdot,\cdot)$ and $(Z_{J_\qg,g_\qg})_\ag$ denotes the $Q$-orthogonal projection on $\ag$ of the metric Koszul vector $Z_{J_\qg,g_\qg}$.   
\end{lemma}

\begin{remark}
The {\it Lee form} $\theta_L$ of $(J,g)$, defined by 
\begin{equation}\label{lee}
\theta_L(X):=-d_g^*\omega(JX), \qquad\forall X\in\pg, 
\end{equation} 
or equivalently, as the unique $1$-form such that $d\omega^{n-1}=\theta_L\wedge\omega^{n-1}$ ($\dim{M}=2n$), is therefore given by $\theta_L=-\theta_{J_\ag(Z_{J_\qg,g_\qg})_\ag}=-g(\cdot,J_\ag(Z_{J_\qg,g_\qg})_\ag)$.  Since $(Z_{J_\qg,g_\qg})_\ag\in\ag_f$, the following conditions are equivalent: 
\begin{equation}\label{lee2}
\theta_L=0; \qquad d_g^*\omega=0; \qquad dd_g^*\omega=0; \qquad Z_{J_\qg,g_\qg}\in\zg(\kg).   
\end{equation} 
\end{remark}

\begin{proof}
Consider the $g$-orthonormal basis of $\qg$ given by 
$$
\{ X_i\}=\{\tfrac{1}{\sqrt{x_\alpha}}e_\alpha, \tfrac{1}{\sqrt{x_\alpha}}J_\qg e_\alpha:\alpha\in\Delta_\qg^+\}. 
$$  
For any $X\in\ag$, we have that 
\begin{align*}
d_g^*\omega(X) &= g(d_g^*\omega,\theta_X) = g(\omega,d\theta_X) 
=\unm\sum\omega(X_i,X_j)d\theta_X(X_i,X_j) \\ 
&= -\unm\sum\omega(X_i,X_j)\theta_X([X_i,X_j]_\pg) 
= -\sum_{\alpha\in\Delta_\qg^+} \tfrac{1}{x_\alpha^2}\omega(e_\alpha,J_\qg e_\alpha)\theta_X([e_\alpha,J_\qg e_\alpha]_\pg) \\ 
&= -\sum_{\alpha\in\Delta_\qg^+} \tfrac{1}{x_\alpha}\theta_X((\im H_\alpha)_\ag), 
= -\theta_X((Z_{J_\qg,g_\qg})_\ag) = -g_\ag((Z_{J_\qg,g_\qg})_\ag,X) \\ 
&= -Q(P_\ag(Z_{J_\qg,g_\qg})_\ag, X),
\end{align*}
concluding the proof.
\end{proof}

\subsection{Dimensions}\label{dim-sec} 
In order to give an idea of the richness of Hermitian geometry on a given C-space $M=G/K$, we now compute the dimensions of different spaces of $G$-invariant complex structures and Hermitian metrics.  We fix a Tits fibration $M\rightarrow F$.   

It follows from \S\ref{isot-sec} that if $\dim{\ag}=2d$ and $r$ is the number of irreducible factors for the isotropy representation of the corresponding flag $F=G_f/H_f$ (see \S\ref{isot}), then the space $\mca^G$ of all $G$-invariant metrics on a C-space $M=G/K$ descending on $F$ has 
$$
\dim{\mca^G}=r+d(2d+1),
$$ 
since $\dim{\Gl_{2d}(\RR)/\Or(2d)}=d(2d+1)$.  Let $\cca^G$ and $\hca^G$ denote, respectively, the space of all $G$-invariant complex and Hermitian structures on $M=G/K$ descending on $F$.  They are differentiable manifolds and we already saw in \eqref{Ccal} that 
$$
\dim{\cca^G}=2d^2,   
$$
and $\cca^G$ has twice as many connected components as the number $k$ of $G$-invariant complex structures on the flag $F=G_f/H_f$ (see \eqref{numberk}).   

For a fixed $J\in\cca^G$, the manifold $\hca^G_J$ of all Hermitian $G$-invariant metrics descending on $F$ on the complex manifold $(M=G/K,J)$ has  
$$
\dim{\hca^G_J}=r+d^2, 
$$
since $\dim{\Gl_{d}(\CC)/\U(d)}=d^2$.  Note that $\hca^G_J$ has codimension $d(d+1)$ in $\mca^G$.  It follows that 
$$
\dim{\hca^G}=r+3d^2. 
$$
On the other hand, given a metric $g\in\mca^G$, the manifold $\cca^G_g\subset\cca^G$ those which are compatible with $g$ has 
$$
\dim{\cca^G_g}=d(d-1), 
$$
since $\dim{\Or_{2d}(\RR)/\U(d)}=d(d-1)$, and so $\cca^G_g$ has codimension $d(d+1)$ in $\cca^G$.  

As seen just below Lemma \ref{c10}, the subset $\cca^G_0\subset\cca^G$ of complex structures with zero first Chern class (i.e., $Z_{J_\qg}\in\ag$) is either empty or consists of some of the connected components of $\cca^G$, so    
$$
\dim{\cca^G_0}=2d^2 \quad\mbox{if nonempty}.  
$$

\begin{example}\label{LG-2}
Concerning the Lie group case $M=G$ studied in Example \ref{LG}, we know that any left-invariant complex structure is one of $J=(J_\ag,J_\qg)$ as in \eqref{J}.  On the other hand, the metrics of the form $g=g_\ag+g_\qg$ as in \eqref{g} are far from exhaust all left-invariant metrics, they are actually those which are also $\Ad(T)$-invariant for some maximal torus $T\subset G$.  Indeed, they depend only on $d(2d+1)+|\Delta^+|$ parameters if $\dim{\ag}=2d$, much less than the $n(n+1)/2$ parameters one needs to run over all left-invariant metrics, where $n=2(d+|\Delta^+|)$.  We note that in the semisimple case, if $(M=G,J)$ is irreducible (see Definition \ref{Cirr}),  
then there exists at most one bi-invariant metric compatible with $J$ (up to scaling), which is not necessarily the Killing metric.  Indeed,  if $G=G_1\times G_2$, $G_1,G_2$ simple Lie groups, $\rank(G_1)=\rank(G_2)$ and relative to the decomposition $\tg=\tg_1\oplus\tg_2$,   
$$
J_\ag=\left[\begin{matrix} 0&-\tfrac{1}{b}I\\ bI&0 \end{matrix}\right], \qquad b\ne 0,
$$
then $J$ is irreducible and a bi-invariant metric $g_b=z_1(-\kil_{\ggo_1})+z_2(-\kil_{\ggo_2})$ is compatible with $J$ if and only if $\tfrac{z_1}{z_2}=b$.  On the other hand, if 
$$
J_\ag=\left[\begin{matrix} aI&-\tfrac{a^2+1}{b}I\\ bI&-aI \end{matrix}\right], \qquad a, b\ne 0,
$$
then there are no bi-invariant metrics compatible with $J$.  
\end{example}

% 28/8/26

\section{Canonical connections}\label{cancon-sec}

On a given Hermitian manifold $(M^{2n},J,g)$, there is a family $\nabla^t$, $t\in\RR$ of Hermitian connections (i.e., $g$ and $J$ parallel) introduced by Gauduchon in \cite{Gdc} and called {\it canonical connections}, which are defined by: 
\begin{equation}\label{cc-def}
g(\nabla^t_XY,Z)=g(\nabla^g_XY,Z)+\tfrac{t-1}{4}d^c\omega(X,Y,Z)+\tfrac{t+1}{4}d^c\omega(X,JY,JZ),
\end{equation}
for all vector fields $X,Y,Z$ on $M$, where $\nabla^g$ is the Levi-Civita connection of $(M,g)$, $\omega=g(J\cdot,\cdot)$ is the K\"ahler form and $d^c\omega:=-d\omega(J\cdot,J\cdot,J\cdot)$.  The torsion 
\begin{equation}\label{t-def}
T^t(X,Y):= \nabla^t_XY-\nabla^t_YX-[X,Y],
\end{equation}
vanishes precisely whenever the structure is K\"ahler, in which case $\nabla^t=\nabla^g$ for all $t\in\RR$.  Distinguished points in this line include the {\it Chern} connection $\nabla^c:=\nabla^1$ and the {\it Bismut} (or {\it Strominger}) connection $\nabla^b:=\nabla^{-1}$.  Note that $\nabla^t=\tfrac{t+1}{2}\nabla^c+\tfrac{1-t}{2}\nabla^b$ for all $t$.  

The (first) {\it Ricci} form associated with the connection $\nabla^t$ is defined by 
\begin{equation}\label{rf-def1}
\rho^t(X,Y):=-\unm\sum_{i=1}^{2n} g(R^t(X,Y)e_i,Je_i), 
\end{equation}
where $\{ e_i\}_{i=1}^{2n}$ is any $g$-orthonormal frame and 
\begin{equation}\label{ct-def}
R^t(X,Y):=\nabla^t_X\nabla^t_Y-\nabla^t_Y\nabla^t_X-\nabla^t_{[X,Y]}
\end{equation}
is the {\it curvature} of $\nabla^t$.  Equivalently, in terms of the $g$-unitary frame 
$$
\left\{E_i:=\tfrac{1}{\sqrt{2}}(e_i-\im Je_i)\right\}_{i=1}^n, \qquad JE_i=\im E_i, \qquad g(E_i,\overline{E_j})=\delta_{ij}, 
$$
one has that
\begin{equation}\label{rf-def2}
\rho^t(X,Y):=\im\sum_{i=1}^{n} g(R^t(X,Y)E_i,\overline{E_i}). 
\end{equation}
For any $t$, $\rho^t$ is a closed $2$-form representing the first Chern class $c_1(M,J)$ (up to scaling).  

It follows from \cite[(2.10)]{IvnPpd} that   
\begin{equation}\label{rhoBC}
\rho^b=\rho^c-dd_g^*\omega, 
\end{equation}
hence 
\begin{equation}\label{rhot}
\rho^t=\rho^c+\tfrac{t-1}{2}dd_g^*\omega, \qquad \rho^t-\rho^u=\tfrac{t-u}{2}dd_g^*\omega, \qquad\forall t,u\in\RR. 
\end{equation}
In particular, since we already have a formula for $dd_g^*\omega$ in the context of C-spaces (see Lemma \ref{dto}), it will be enough to compute the Chern Ricci form $\rho^c=\rho^1$.

\subsection{Levi-Civita connection}\label{LC-sec}
According to \S\ref{IC-sec}, the Nomizu operator $\Lambda^g$ of the Levi-Civita connection $\nabla^g$ of a $G$-invariant metric $g$ on a homogeneous space $M=G/K$ is given by
$$
\Lambda^g(X)Y = \unm[X,Y]_\pg+U(X,Y), \qquad \forall X,Y\in\pg, 
$$
where $U:\pg\times\pg\rightarrow\pg$ is the symmetric bilinear map defined by 
$$
2g(U(X,Y),Z):=g([Z,X]_\pg,Y)+g(X,[Z,Y]_\pg), \qquad \forall X,Y,Z\in\pg.
$$
Recall the notation given in \S\ref{form2-sec}.  For a C-space $M=G/K$, we consider the $\CC$-linear extensions of all tensors on $\pg^c=\ag^c\oplus\qg^c$ without further mention.    

\begin{lemma}\label{U} 
On a C-space $M=G/K$, if $g=g_\ag+g_\qg$ is as in \eqref{g}, then 
$$
U(A,B)=0, \qquad U(A,E_\alpha)=\left(\tfrac{\alpha(A)}{2}+\tfrac{g_\ag(A,A_\alpha)}{2x_\alpha}\right) E_\alpha, \qquad 
U(E_\alpha,E_{-\alpha})=0, 
$$
$$
U(E_\alpha,E_\beta)=\left\{\begin{array}{cl} 
\tfrac{N_{\alpha,\beta}(x_\beta-x_\alpha)}{2x_{\alpha+\beta}} E_{\alpha+\beta}, & \quad \alpha+\beta\in\Delta_\qg, \\  0, & \quad\alpha+\beta\notin\Delta_\qg,
\end{array}\right.
$$
for any $A,B\in\ag^c$ and $\alpha,\beta\in\Delta_\qg$, $\alpha+\beta\ne 0$.  
\end{lemma}

\begin{proof}
We have that
$$
2g(U(A,E_\alpha),E_{-\alpha}) = g(\alpha(A)E_{-\alpha},E_\alpha)+g(A,-A_\alpha) = -\alpha(A)x_\alpha-g(A,A_\alpha),
$$
$2g(U(E_\alpha,E_{-\alpha}),E_\beta)=0$, 
$$
2g(U(E_\alpha,E_{-\alpha}),A) = g(\alpha(A)E_{\alpha},E_{-\alpha})+g(E_{\alpha},-\alpha(A)E_{-\alpha})=0, 
$$
and
$$
2g(U(E_\alpha,E_\beta),E_{-(\alpha+\beta)}) = g(N_{-(\alpha+\beta),\alpha}E_{-\beta},E_\beta) + 
g(E_\alpha, N_{-(\alpha+\beta),\beta}E_{-\alpha}) = N_{\alpha,\beta}(x_\alpha-x_\beta),
$$
concluding the proof.  
\end{proof}

\begin{corollary}\label{LC}
On a C-space $M=G/K$, for any $g=g_\ag+g_\qg$ as in \eqref{g}, 
$$
\Lambda^g(A)B=0, \qquad \Lambda^g(A)E_\alpha=\left(\alpha(A)+\tfrac{g_\ag(A,A_\alpha)}{2x_\alpha}\right) E_\alpha, \qquad 
\Lambda^g(E_\alpha)A=\tfrac{g_\ag(A,A_\alpha)}{2x_\alpha} E_\alpha, 
$$
$$
\Lambda^g(E_\alpha)E_{-\alpha}=\unm A_\alpha, \qquad 
\Lambda^g(E_\alpha)E_\beta=
\left\{\begin{array}{cl} 
\tfrac{N_{\alpha,\beta}(x_{\alpha+\beta}+x_\beta-x_\alpha)}{2x_{\alpha+\beta}} E_{\alpha+\beta},  & \quad \alpha+\beta\in\Delta_\qg, \\  0, & \quad\alpha+\beta\notin\Delta_\qg,
\end{array}\right.
$$
for any $A,B\in\ag^c$ and $\alpha,\beta\in\Delta_\qg$, $\alpha+\beta\ne 0$.    
\end{corollary}

\subsection{Chern connection and curvature}\label{CC-sec} 
We consider a $G$-invariant Hermitian structure $(J,g)$ as in \eqref{Jg}, i.e., a Hermitian C-space (see Definition \ref{HC-def}).  It follows from \eqref{cc-def} that the Nomizu operator $\Lambda^c$ of the Chern connection $\nabla^c$ is given by 
$$
g(\Lambda^c(X)Y,Z)=g(\Lambda^g(X)Y,Z)-\unm d\omega(JX,Y,Z), \qquad\forall X,Y,Z\in\pg. 
$$
Recall the following useful facts about basic complex linear algebra: the closure is given by $\overline{X+\im Y}:=X-\im Y$ for all $X,Y\in\pg$, and one defines 
$$
X^{1,0}:=\unm(X-\im JX), \qquad X^{0,1}:=\unm(X+\im JX), \qquad\forall X\in\pg^c,
$$
from which follows that $\pg^{1,0}:=\{ X^{1,0}:X\in\pg\}$ is precisely the $\im$-eigenspace of $J$.  Note that $X^{0,1}=\overline{X^{1,0}}$ if and only if $X\in\pg$ and that
$$
g(\overline{X},\overline{Y})=\overline{g(X,Y)}, \quad\forall X,Y\in\pg^c, \qquad 
g(X^{1,0},X^{0,1})=\unm g(X,X), \quad\forall X\in\pg. 
$$

\begin{proposition}\label{NoCc} 
For any $A,B\in\ag^c$ and $\alpha\in\Delta_\qg$,
$$
\Lambda^c(A)B=0, \qquad \Lambda^c(A)E_\alpha=\alpha(A)E_\alpha, \qquad i.e., \quad \Lambda^c(A)=\ad{A}|_\pg,
$$
$$
\Lambda^c(E_\alpha)A=\tfrac{1}{x_\alpha} g_\ag\left(\unm(A-\im\epsilon_\alpha J_\ag A),A_\alpha\right)E_\alpha
=\left\{\begin{array}{ll} 
\tfrac{1}{x_\alpha} g_\ag(A^{1,0},A_\alpha)E_\alpha, &\quad \alpha\in\Delta_\qg^+,\\ 
\tfrac{1}{x_\alpha} g_\ag(A^{0,1},A_\alpha)E_\alpha, &\quad \alpha\in\Delta_\qg^-,
\end{array}\right.
$$
$$
\Lambda^c(E_\alpha)E_{-\alpha}=\unm(A_\alpha+\im\epsilon_\alpha J_\ag A_\alpha)
=\left\{\begin{array}{ll} 
A_\alpha^{0,1}, &\quad \alpha\in\Delta_\qg^+,\\ 
A_\alpha^{1,0}, &\quad \alpha\in\Delta_\qg^-.
\end{array}\right.
$$
If $\alpha+\beta\in\Delta_\qg$, then
$$
\Lambda^c(E_\alpha)E_\beta = 
\left\{\begin{array}{cl} 
\tfrac{N_{\alpha,\beta} 
x_\beta}{x_{\alpha+\beta}} E_{{\alpha+\beta}}, &\quad \alpha,\beta\in\Delta_\qg^\pm,\\ 
N_{\alpha,\beta}(E_{\alpha+\beta})^{0,1}, &\quad \alpha\in\Delta_\qg^+,\beta\in\Delta_\qg^-, \\
N_{\alpha,\beta}(E_{\alpha+\beta})^{1,0}, &\quad \alpha\in\Delta_\qg^-,\beta\in\Delta_\qg^+, 
\end{array}\right.
$$
and $\Lambda^c(E_\alpha)E_\beta=0$ if $\alpha+\beta\notin\Delta_\qg$.  
\end{proposition}

\begin{proof}
The only possibly nonzero components are:
\begin{align*}
g(\Lambda^c(A)E_\alpha,E_{-\alpha})=& g\left(\left(\alpha(A)+\tfrac{g(A,A_\alpha)}{2x_\alpha}\right)E_\alpha, E_{-\alpha}\right) 
-\unm d\omega(J_\ag A,E_\alpha,E_{-\alpha}) \\ 
=& \left(\alpha(A)+\tfrac{g(A,A_\alpha)}{2x_\alpha}\right)(-x_\alpha)
-\unm g_\ag(-A,A_\alpha) = \alpha(A)(-x_\alpha),   
\end{align*}
\begin{align*}
g(\Lambda^c(E_\alpha)A,E_{-\alpha})=& g\left(\tfrac{g(A,A_\alpha)}{2x_\alpha}E_\alpha, E_{-\alpha}\right) 
-\unm d\omega(J_\qg E_\alpha,A,E_{-\alpha}) \\ 
=& \tfrac{g(A,A_\alpha)}{2x_\alpha}(-x_\alpha)
-\unm \im\epsilon_\alpha (-g(J_\ag A,A_\alpha)) \\ 
=& \left(\tfrac{1}{x_\alpha} g\left(\unm(A-\im\epsilon_\alpha J_\ag A),A_\alpha\right)\right)(-x_\alpha)E_\alpha,   
\end{align*} 
\begin{align*}
g(\Lambda^c(E_\alpha)E_{-\alpha},A)=& g\left(\unm A_\alpha, A\right) 
-\unm d\omega(J_\qg E_\alpha,E_{-\alpha},A) 
= g\left(\unm A_\alpha, A\right)
-\unm \im\epsilon_\alpha g(J_\ag A,A_\alpha) \\ 
=&g\left(\unm(A_\alpha+\im\epsilon_\alpha J_\ag A_\alpha),A\right),   
\end{align*}
and
\begin{align*}
g(\Lambda^c(E_\alpha)E_\beta,E_{-(\alpha+\beta)}) 
=& g\left(\tfrac{N_{\alpha,\beta}(x_{\alpha+\beta}+x_\beta-x_\alpha)}{2x_{\alpha+\beta}} E_{\alpha+\beta},E_{-(\alpha+\beta)}\right) 
-\unm d\omega(J_\qg E_\alpha,E_\beta,E_{-(\alpha+\beta)})\\ 
=& \tfrac{N_{\alpha,\beta}(x_{\alpha+\beta}+x_\beta-x_\alpha)}{2x_{\alpha+\beta}} (-x_{\alpha+\beta}) 
-\unm \im\epsilon_\alpha d\omega(E_\alpha,E_\beta,E_{-(\alpha+\beta)})\\
=& \tfrac{N_{\alpha,\beta}(x_{\alpha+\beta}+x_\beta-x_\alpha)}{2x_{\alpha+\beta}} (-x_{\alpha+\beta}) 
-\unm \epsilon_\alpha N_{\alpha,\beta}(\epsilon_\alpha x_\alpha+\epsilon_\beta x_\beta+\epsilon_\gamma x_\gamma)\\ 
=&\tfrac{N_{\alpha,\beta} 
\left(x_{\alpha+\beta}+x_\beta-x_\alpha+\epsilon_\alpha (\epsilon_\alpha x_\alpha+\epsilon_\beta x_\beta-\epsilon_{\alpha+\beta} x_{\alpha+\beta})\right)}{2x_{\alpha+\beta}} (-x_{\alpha+\beta}) \\ 
=&\tfrac{N_{\alpha,\beta} 
\left((1-\epsilon_\alpha\epsilon_{\alpha+\beta})x_{\alpha+\beta}+(1+\epsilon_\alpha\epsilon_\beta)x_\beta\right)}{2x_{\alpha+\beta}} (-x_{\alpha+\beta}),    
\end{align*}
concluding the proof.
\end{proof}

According to \S\ref{IC-sec}, the Chern curvature is given by 
$$
R^c(X,Y)=\Lambda^c(X)\Lambda^c(Y)-\Lambda^c(Y)\Lambda^c(X)-\Lambda^c([X,Y]_\pg)-\ad{[X,Y]_\kg}|_\pg, \qquad\forall X,Y\in\pg.
$$

\begin{proposition}\label{Ccurv}
For any $A,B\in\ag^c$,
$$
R^c(E_\alpha,E_{-\alpha})A = -\tfrac{g_\ag(A^{0,1},A_\alpha)}{x_\alpha}A_\alpha^{0,1} + \tfrac{g_\ag(A^{1,0},A_\alpha)}{x_\alpha}A_\alpha^{1,0},  \qquad\forall\alpha\in\Delta_\qg^+,   
$$
$$
R^c(E_\alpha,E_{-\alpha})E_\alpha = \left(-\tfrac{g_\ag(A_\alpha^{1,0},A_\alpha)}{x_\alpha} - \la\alpha,\alpha\ra\right)E_\alpha,  \qquad\forall\alpha\in\Delta_\qg^+,
$$
$$
R^c(E_\alpha,E_{-\alpha})E_\beta = \left(\tfrac{N_{\alpha,-\beta}^2x_{\alpha-\beta}}{x_\beta} 
- \tfrac{N_{\alpha,\beta}^2x_\beta}{x_{\alpha+\beta}} 
- \la\alpha,\beta\ra\right) E_\beta,  \qquad\forall\alpha,\beta\in\Delta_\qg^+,
$$
where in the last line, the term containing $N_{\alpha,-\beta}^2$ appears only if $\alpha-\beta\in\Delta_\qg^-$ and the term containing $N_{\alpha,\beta}^2$ appears only if $\alpha+\beta\in\Delta_\qg$. 
\end{proposition}

\begin{proof}
For any $\alpha\in\Delta_\qg^+$, we have that 
\begin{align*}
R^c(E_\alpha,E_{-\alpha})A =& \Lambda^c(E_\alpha) \Lambda^c(E_{-\alpha})A - \Lambda^c(E_{-\alpha}) \Lambda^c(E_\alpha)A - \Lambda^c(A_\alpha)A -[(H_\alpha)_\kg,A]\\ 
=& -\tfrac{g(A^{0,1},A_\alpha)}{x_\alpha}\Lambda^c(E_\alpha) E_{-\alpha} - \tfrac{g(A^{1,0},A_\alpha)}{x_\alpha}\Lambda^c(E_{-\alpha}) E_\alpha \\ 
=& -\tfrac{g(A^{0,1},A_\alpha)}{x_\alpha}A_\alpha^{0,1} + \tfrac{g(A^{1,0},A_\alpha)}{x_\alpha}A_\alpha^{1,0},  
\end{align*}
and
\begin{align*}
R^c(E_\alpha,E_{-\alpha})E_\alpha =& -\Lambda^c(E_\alpha)A_\alpha^{1,0} 
- \Lambda^c(A_\alpha)E_\alpha - [(H_\alpha)_\kg,E_\alpha]\\  
=& -\tfrac{g(A_\alpha^{1,0},A_\alpha)}{x_\alpha}E_\alpha - [H_\alpha,E_\alpha].   
\end{align*}
If $\alpha,\beta\in\Delta_\qg^+$, then
\begin{align*}
R^c(E_\alpha,E_{-\alpha})E_\beta =& \Lambda^c(E_\alpha) \Lambda^c(E_{-\alpha})E_\beta - \Lambda^c(E_{-\alpha}) \Lambda^c(E_\alpha)E_\beta - \Lambda^c(A_\alpha)E_\beta - [(H_\alpha)_\kg,E_\beta]\\ 
=& \Lambda^c(E_\alpha) N_{-\alpha,\beta}(E_{-\alpha+\beta})^{1,0} - \Lambda^c(E_{-\alpha}) \tfrac{N_{\alpha,\beta}x_\beta}{x_{\alpha+\beta}}E_{\alpha+\beta} - [A_\alpha,E_\beta] - [(H_\alpha)_\kg,E_\beta]\\ 
=& N_{-\alpha,\beta}\tfrac{N_{\alpha,-\alpha+\beta}x_{-\alpha+\beta}}{x_\beta}E_\beta 
- \tfrac{N_{\alpha,\beta}x_\beta}{x_{\alpha+\beta}} N_{-\alpha,\alpha+\beta}(E_\beta)^{1,0} 
- [H_\alpha,E_\beta]\\ 
=& \tfrac{N_{\alpha,-\beta}^2x_{\alpha-\beta}}{x_\beta}E_\beta 
- \tfrac{N_{\alpha,\beta}^2x_\beta}{x_{\alpha+\beta}} (E_\beta)^{1,0} 
- \la\alpha,\beta\ra E_\beta,  
\end{align*}
concluding the proof.  
\end{proof}

We now compute the Chern Ricci form (see \eqref{rf-def1} and \eqref{rf-def2}).  

\begin{theorem}\label{CRF} 
The Chern Ricci form of any $G$-invariant Hermitian structure $(J,g)$ as in \eqref{Jg} on a C-space $M=G/K$ is given by  
$$
\rho^c(X,Y) = Q([X,Y],Z_{J_\qg}) = - \kil_{\ggo_f}([X,Y],Z_{J_\qg}), \qquad\forall X,Y\in\pg,  
$$
where $\kil_{\ggo_f}$ is the Killing form of $\ggo_f$ and $Z_{J_\qg}$ is the Koszul vector (see \eqref{koszul2}).  
\end{theorem}

\begin{remark}
The only possibly nonzero components of the Chern Ricci form are therefore   
$$
\rho^c(E_\alpha,E_{-\alpha}) = - \kil_{\ggo^c}(H_\alpha,Z_{J_\qg}) 
=-\im\sum_{\beta\in\Delta_\qg^+}\la\alpha,\beta\ra, \qquad\forall\alpha\in\Delta_\qg^+, 
$$
or equivalently, 
$$
\rho^c(e_\alpha,f_{\alpha}) = - \kil_\ggo(Z_\alpha,Z_{J_\qg})
=\sum_{\beta\in\Delta_\qg^+}\la\alpha,\beta\ra, \qquad\forall\alpha\in\Delta_\qg^+.
$$
In particular, $\rho^c=0$ if and only if $\zg(\hg_f)=0$ (see \eqref{Jal}), i.e., $M$ is a torus.  
\end{remark}

\begin{remark}
Since $\rho^c$ does not depend on the metric, the solutions to the Chern Ricci flow $\dpar\omega(t)=-\rho^c(\omega(t))$ on a C-space are all given by $\omega(t)=\omega(0)-t\rho^c$, where $\rho^c=\rho^c(\omega(t))$ for all $t$.  
\end{remark}

\begin{remark}\label{C1}
If $\pg_0$ is as in Proposition \ref{isotdec}, (iii), then 
$$
c_1(M,J)=\left\{\rho^c+d\theta_A:A\in\pg_0\right\} = \left\{\sigma_{Z_{J_\qg}+A}:A\in\pg_0\right\}, 
$$
where $\sigma_Z:=Q([\cdot,\cdot],Z)$ (see \S\ref{dR-sec}).  
\end{remark}

\begin{proof}
In order to use \eqref{rf-def2}, we consider   
$\{ \tfrac{1}{\sqrt{x_\alpha}}E_\alpha:\alpha\in\Delta_\qg^+\}$, a $g$-unitary frame of $\qg$ (note that $\overline{E_\alpha}=-E_{-\alpha}$), and any $g$-unitary frame of $\ag$, say $\{ A_1,\dots,A_d\}$.  According to \eqref{rf-def2},  
$$
\rho^c(E_\alpha,E_{-\alpha}) = \im\sum_{i=1}^d g(R^c(E_\alpha,E_{-\alpha})A_i,\overline{A_i}) 
-\im\sum_{\beta\in\Delta_\qg^+} \tfrac{1}{x_\beta}g(R^c(E_\alpha,E_{-\alpha})E_\beta,E_{-\beta}), 
$$
for any $\alpha\in\Delta_\qg^+$.  It follows from Proposition \ref{Ccurv} that
\begin{align*}
g(R^c(E_\alpha,E_{-\alpha})A,\overline{A}) =& g(-\tfrac{g(A^{0,1},A_\alpha)}{x_\alpha}A_\alpha^{0,1} + \tfrac{g(A^{1,0},A_\alpha)}{x_\alpha}A_\alpha^{1,0},\overline{A})\\ 
=& \tfrac{1}{x_\alpha}g(A,A_\alpha)g(\overline{A},A_\alpha), \qquad\forall A\in\pg^{1,0},
\end{align*}
and if $A_\alpha^{1,0}=\sum\limits_{i=1}^d a_iA_i$, then $A_\alpha^{0,1}=\sum\limits_i^d b_i\overline{A_i}$ and so $g(A_\alpha,A_\alpha)=2\sum\limits_i^d a_ib_i$.  Thus  
$$
\im\sum_{i=1}^d g(R^c(E_\alpha,E_{-\alpha})A_i,\overline{A_i}) =\im\sum_{i=1}^d \tfrac{1}{x_\alpha}g(A_i,A_\alpha)g(\overline{A_i},A_\alpha) =\im\sum_{i=1}^d \tfrac{1}{x_\alpha}b_ia_i =\tfrac{\im g(A_\alpha,A_\alpha)}{2x_\alpha}.
$$
On the other hand, by Proposition \ref{Ccurv}, for any $\alpha\in\Delta_\qg^+$,
\begin{align*}
&-\im\sum_{\beta\in\Delta_\qg^+} \tfrac{1}{x_\beta}g(R^c(E_\alpha,E_{-\alpha})E_\beta,E_{-\beta}) \\
=& -\im \left(\tfrac{g(A_\alpha^{1,0},A_\alpha)}{x_\alpha} + \la\alpha,\alpha\ra\right) 
+\im\sum_{\beta\in\Delta_\qg^+,\beta\ne\alpha} \tfrac{(N_{\alpha,-\beta})^2x_{\alpha-\beta}}{x_\beta} 
- \tfrac{(N_{\alpha,\beta})^2x_\beta}{x_{\alpha+\beta}} 
- \la\alpha,\beta\ra \\ 
=&  -\tfrac{\im g(A_\alpha^{1,0},A_\alpha^{0,1})}{x_\alpha} -\im \la\alpha,\alpha\ra 
- \im\sum_{\beta\in\Delta_\qg^+,\beta\ne\alpha}\la\alpha,\beta\ra 
=  -\tfrac{\im g(A_\alpha^{1,0},A_\alpha^{0,1})}{x_\alpha}  
-\kil_{\ggo^c}(H_\alpha,Z_{J_\qg}).
\end{align*}
All this implies that 
$$
\rho^c(E_\alpha,E_{-\alpha}) = \tfrac{\im g(A_\alpha,A_\alpha)}{2x_\alpha} 
-\tfrac{\im g(A_\alpha^{1,0},A_\alpha^{0,1})}{x_\alpha}  
-\kil_{\ggo^c}(H_\alpha,Z_{J_\qg}) 
= - \kil_{\ggo^c}([E_\alpha,E_{-\alpha}],Z_{J_\qg}).
$$
It is straightforward to check that all the other components $\rho^c(A,B)$, $\rho^c(A,E_\alpha)$, $\rho^c(E_\alpha,E_\beta)$, $\alpha+\beta\ne 0$ vanish, concluding the proof.  
\end{proof}

\subsection{Bismut connection and curvature}\label{BiC-sec}
The Nomizu operator $\Lambda^b$ of the Bismut connection $\nabla^b$ of a homogeneous space endowed with an invariant Hermitian structure is given by 
$$
\Lambda^b(X)Y=\Lambda^g(X)Y+\unm T^b(X,Y), \qquad 
g(T^b(X,Y),Z)=-d^c\omega(X,Y,Z),
$$  
and the torsion by 
$$
T^b(X,Y)=\Lambda^b(X)Y-\Lambda^b(Y)X-[X,Y]_\pg, 
$$
see \eqref{cc-def} and \eqref{t-def}.  

We consider a Hermitian C-space $(M=G/K,J,g)$. 

\begin{proposition}\label{Btor}
For any $A,B\in\ag^c$ and $\alpha,\beta\in\Delta_\qg$, $\alpha+\beta\ne 0$, 
$$
T^b(A,B)=0, \qquad T^b(A,E_\alpha)=\tfrac{g_\ag(A,A_\alpha)}{x_\alpha} E_\alpha, \qquad 
T^b(E_\alpha,E_{-\alpha})=-A_\alpha, 
$$
$$
T^b(E_\alpha,E_\beta)=\left\{\begin{array}{cl}
\tfrac{-\epsilon_\alpha\epsilon_\beta\epsilon_{\alpha+\beta} N_{\alpha,\beta}(\epsilon_\alpha x_\alpha+\epsilon_\beta x_\beta-\epsilon_{\alpha+\beta} x_{\alpha+\beta})}{x_{\alpha+\beta}} E_{{\alpha+\beta}}, &\quad \alpha+\beta\in\Delta_\qg, \\ 
0,&\quad \alpha+\beta\notin\Delta_\qg.
\end{array}\right.
$$
In particular, 
$$
T^b(E_\alpha,E_\beta)
=\tfrac{N_{\alpha,\beta}(x_{\alpha+\beta}-x_\alpha-x_\beta)}{x_{\alpha+\beta}} E_{{\alpha+\beta}}, 
\qquad\forall\alpha,\beta,\alpha+\beta\in\Delta_\qg^{\pm}.  
$$
\end{proposition}

\begin{proof}
The only possibly nonzero components are 
$$
g(T^b(A,E_\alpha),E_{-\alpha})=-d^c\omega(A,E_\alpha,E_{-\alpha}) = -g_\ag(A,A_\alpha),  
$$
$$
g(T^b(E_\alpha,E_{-\alpha}),A)=-d^c\omega(E_\alpha,E_{-\alpha},A) = -g_\ag(A,A_\alpha),  
$$
and
\begin{align*}
g(T^b(E_\alpha,E_\beta),E_{-(\alpha+\beta)})=& -d^c\omega(E_\alpha,E_\beta,E_{-(\alpha+\beta)}) \\
=& \epsilon_\alpha\epsilon_\beta\epsilon_{\alpha+\beta} N_{\alpha,\beta}(\epsilon_\alpha x_\alpha+\epsilon_\beta x_\beta-\epsilon_{\alpha+\beta} x_\gamma),  
\end{align*}
concluding the proof.  
\end{proof}

Using the above proposition together with Corollary \ref{LC}, we obtain the following.  

\begin{corollary}\label{Bcon}
For any $A,B\in\ag^c$ and $\alpha,\beta\in\Delta_\qg$, $\alpha+\beta\ne 0$, 
$$
\Lambda^b(A)B=0, \qquad \Lambda^b(A)E_\alpha=\left(\alpha(A)+\tfrac{g_\ag(A,A_\alpha)}{x_\alpha}\right) E_\alpha, \qquad 
\Lambda^b(E_\alpha)A=0=\Lambda^b(E_\alpha)E_{-\alpha},
$$
$$
\Lambda^b(E_\alpha)E_\beta = \left\{\begin{array}{cl}
\tfrac{N_{\alpha,\beta} 
\left(x_{\alpha+\beta}+x_\beta-x_\alpha-\epsilon_\alpha\epsilon_\beta\epsilon_{\alpha+\beta} (\epsilon_\alpha x_\alpha+\epsilon_\beta x_\beta-\epsilon_{\alpha+\beta} x_{\alpha+\beta})\right)}{2x_{\alpha+\beta}} E_{{\alpha+\beta}}, &\quad \alpha+\beta\in\Delta_\qg, \\ 
0,&\quad \alpha+\beta\notin\Delta_\qg.
\end{array}\right.
$$
In particular, 
\begin{equation}\label{Lb}
\Lambda^b(E_\alpha)E_\beta = \left\{\begin{array}{cl} 
\tfrac{N_{\alpha,\beta} 
\left(x_{\alpha+\beta}-x_\alpha\right)}{x_{\alpha+\beta}} E_{{\alpha+\beta}}, &
\quad\alpha,\beta,\alpha+\beta\in\Delta_\qg^\pm, \\ 
\tfrac{N_{\alpha,\beta} 
\left(x_{\beta}-x_\alpha\right)}{x_{\alpha+\beta}} E_{{\alpha+\beta}}, &
\quad \alpha\in\Delta_\qg^\pm,  \; \beta\in\Delta_\qg^\mp, \; \alpha+\beta\in\Delta_\qg^\mp, \\
0 &\quad \alpha\in\Delta_\qg^\pm, \; \beta\in\Delta_\qg^\mp, \; \alpha+\beta\in\Delta_\qg^\pm.
\end{array}\right.
\end{equation}
\end{corollary} 

According to \cite[Theorem 1]{WngYngZhn}, any simply connected compact Hermitian manifold $(M,J,g)$ which is {\it Bismut flat} (i.e., $R^b=0$) is holomorphically isometric to a semisimple Lie group endowed with a left-invariant $J$ and a compatible bi-invariant metric $g$.  

\begin{remark}\label{Bflat-rem}
In the Lie group case $M=G$ (see Example \ref{LG}), we have that $\ag=\zg(\ggo)\oplus\overline{\tg}$, where $\overline{\tg}$ is the maximal torus of $[\ggo,\ggo]=\overline{\tg}\oplus\qg$.  If there is a biinvariant metric $g_b$ on $[\ggo,\ggo]$ such that $g_\qg=g_b|_{\qg}$, i.e., $x_\alpha=x_\beta$ for all $\alpha,\beta\in\Delta_i$, $i=1,\dots,s$  (see \eqref{decflags}), then by Corollary \ref{Bcon},
\begin{equation}\label{Bf1}
\Lambda^b(E_\alpha)=0, \qquad\forall \alpha\in\Delta=\Delta_\qg.  
\end{equation}
On the other hand, if in addition $g|_{\overline{\tg}}=g_b|_{\overline{\tg}}$, then for all $\alpha\in\Delta$,
$$
\Lambda^b(H_\alpha)E_\beta= \left(\beta(H_\alpha)+\tfrac{1}{x_\beta}g_\ag(H_\alpha,H_\beta)\right)E_\beta 
= \left(-Q(H_\alpha,H_\beta)+\tfrac{1}{x_\beta}g_b(H_\alpha,H_\beta)\right)E_\beta =0, 
$$
that is, 
\begin{equation}\label{Bf2} 
\Lambda^b(A)E_\alpha=0, \qquad\forall A\in\overline{\tg}, \quad \alpha\in\Delta.
\end{equation}  
This implies that $(J,g)$ is Bismut flat (i.e., $R^b=0$) if the metric $g$ is {\it almost-bi-invariant}, i.e., $g|_{[\ggo,\ggo]}=g_b$ for some biinvariant metric $g_b$ on $[G,G]$.  Indeed, recall from \S\ref{IC-sec} that
$$
R^b(X,Y)=[\Lambda^b(X),\Lambda^b(Y)] -\Lambda^b([X,Y]), \qquad\forall X,Y\in\ggo,       
$$
so the only potential nonzero components are $R^b(A,E_\alpha)$ and $R^b(E_\alpha,E_\beta)$, which are easily seen to vanish by using \eqref{Bf1} and \eqref{Bf2}.  We conclude that $(J,g)$ is Bismut flat as soon as $g|_{[\ggo,\ggo]}=g_b$, regardless of whether $g(\zg(\ggo),[\ggo,\ggo])=0$ or not, that is, the metric $g$ is {\it not} necessarily biinvariant on the whole group $G$.  This does not contradict \cite[Theorem 1]{WngYngZhn}; indeed, it is easy to check that if we consider instead the Lie bracket $-T^b$ on $\ggo$, then $J$ is still left-invariant and $g$ is now bi-invariant relative to the corresponding new Lie group structure on $M=G$.  
\end{remark}

Recall from \S\ref{IC-sec} the formula for $R^b$ in the general case:   
$$
R^b(X,Y)=[\Lambda^b(X),\Lambda^b(Y)] -\Lambda^b([X,Y]) -\ad{[X,Y]_\kg}|_\pg, \qquad\forall X,Y\in\pg.        
$$

\begin{proposition}\label{Bcurv}
For any $\alpha,\beta\in\Delta_\qg^+$ and $A\in\ag^c$,
\begin{equation}\label{Bcurv1}
R^b(A,E_\alpha)E_\beta = g_\ag\left(A,\tfrac{1}{x_{\alpha+\beta}}A_{\alpha+\beta}-\tfrac{1}{x_{\beta}}A_{\beta}\right) \tfrac{N_{\alpha,\beta} \left(x_{\alpha+\beta}-x_\alpha\right)}{x_{\alpha+\beta}}E_{\alpha+\beta}, \qquad  \alpha+\beta\in\Delta_\qg, 
\end{equation}
\begin{equation}\label{Bcurv2}
R^b(E_\alpha,E_{-\alpha})E_\beta  = -\left(\la\alpha,\beta\ra+\tfrac{g_\ag(A_\alpha,A_\beta)}{x_\beta}\right)E_\beta, \qquad  \alpha\pm\beta\notin\Delta_\qg, 
\end{equation}
and 
\begin{equation}\label{Bcurv3}
R^b(E_\alpha,E_{-\alpha})E_\beta  =
\left(\tfrac{N_{-\alpha,\beta}^2 \left(x_{\beta}-x_\alpha\right)^2}{x_{-\alpha+\beta}x_\beta} 
-\tfrac{N_{\alpha,\beta}^2 \left(x_{\alpha+\beta}-x_\alpha\right)^2}{x_{\alpha+\beta}x_\beta} -\la\alpha,\beta\ra-\tfrac{g_\ag(A_\alpha,A_\beta)}{x_\beta}\right)E_\beta, 
\end{equation}
where in the last line, the term containing $N_{-\alpha,\beta}^2$ appears only if $-\alpha+\beta\in\Delta_\qg^+$ and the term containing $N_{\alpha,\beta}^2$ appears only if $\alpha+\beta\in\Delta_\qg^+$.  

Furthermore, if $\alpha+\beta\in\Delta_\qg$ and $\alpha-\beta\notin\Delta_\qg$, then 
\begin{equation}\label{Bcurv4}
R^b(E_{-\alpha},E_\beta)E_\alpha 
= -\tfrac{N_{\alpha,\beta}^2 \left(x_{\alpha+\beta}-x_\beta\right)\left(x_{\alpha+\beta}-x_\alpha\right)}{x_{\alpha+\beta} x_{\beta}} E_\beta. 
\end{equation}
\end{proposition} 

%\begin{remark}
%Using \eqref{Bcurv2}, one obtains that $A_\alpha\ne 0$ for any $\alpha\in\Delta_\qg$ if $R^b=0$.  
%\end{remark} 

\begin{proof}
It follows from Corollary \ref{Bcon} that   
\begin{align*}
R^b(A,E_\alpha)E_\beta =& \Lambda^b(A) \Lambda^b(E_\alpha)E_\beta - \Lambda^b(E_\alpha)\Lambda^b(A) E_\beta - \alpha(A)\Lambda^b(E_\alpha)E_\beta\\ 
=& \left(\alpha(A)+\beta(A)+\tfrac{g_\ag(A,A_{\alpha+\beta})}{x_{\alpha+\beta}}\right)\Lambda^b(E_\alpha)E_\beta \\ 
&- \left(\beta(A)+\tfrac{g_\ag(A,A_{\beta})}{x_{\beta}}\right)\Lambda^b(E_\alpha)E_\beta - \alpha(A)\Lambda^b(E_\alpha)E_\beta\\ 
=& g_\ag\left(A,\tfrac{1}{x_{\alpha+\beta}}A_{\alpha+\beta}-\tfrac{1}{x_{\beta}}A_{\beta}\right) \tfrac{N_{\alpha,\beta} \left(x_{\alpha+\beta}-x_\alpha\right)}{x_{\alpha+\beta}}E_{\alpha+\beta}.
\end{align*}
On the other hand, if $\alpha+\beta\notin\Delta_\qg$, then
$$
R^b(E_\alpha,E_{-\alpha})E_\beta = - \Lambda^b(A_\alpha)E_\beta - \beta((H_\alpha)_\kg)E_\beta 
=\left(-\la\alpha,\beta\ra-\tfrac{g_\ag(A_\alpha,A_\beta)}{x_\beta}\right)E_\beta,  
$$
and if $\alpha+\beta\in\Delta_\qg$ and $-\alpha+\beta\in\Delta_\qg^+$, then 
\begin{align*}
R^b(E_\alpha,E_{-\alpha})E_\beta 
=& \Lambda^b(E_\alpha) \Lambda^b(E_{-\alpha})E_\beta - \Lambda^b(E_{-\alpha}) \Lambda^b(E_\alpha)E_\beta - \Lambda^b(A_\alpha)E_\beta - \beta((H_\alpha)_\kg)E_\beta \\ 
=&\tfrac{N_{-\alpha,\beta} \left(x_{\beta}-x_\alpha\right)}{x_{-\alpha+\beta}} \Lambda^b(E_{\alpha}) E_{-\alpha+\beta} 
-\tfrac{N_{\alpha,\beta} \left(x_{\alpha+\beta}-x_\alpha\right)}{x_{\alpha+\beta}} \Lambda^b(E_{-\alpha}) E_{\alpha+\beta} \\
& -\left(\beta(A_\alpha)+\tfrac{g_\ag(A_\alpha,A_\beta)}{x_\beta}\right)E_\beta - \beta((H_\alpha)_\kg)E_\beta\\ 
=&  \tfrac{N_{\alpha,-\alpha+\beta} \left(x_{\beta}-x_\alpha\right)}{x_{\beta}}  \tfrac{N_{-\alpha,\beta} \left(x_{\beta}-x_\alpha\right)}{x_{-\alpha+\beta}}E_\beta 
-\tfrac{N_{\alpha,\beta} \left(x_{\alpha+\beta}-x_\alpha\right)}{x_{\alpha+\beta}} \tfrac{N_{-\alpha,\alpha+\beta} \left(x_{\alpha+\beta}-x_\alpha\right)}{x_{\beta}} E_{\beta} \\
&-\left(\beta(H_\alpha)+\tfrac{g_\ag(A_\alpha,A_\beta)}{x_\beta}\right)E_\beta  \\ 
=&\left(\tfrac{N_{-\alpha,\beta}^2 \left(x_{\beta}-x_\alpha\right)^2}{x_{-\alpha+\beta}x_\beta} 
-\tfrac{N_{\alpha,\beta}^2 \left(x_{\alpha+\beta}-x_\alpha\right)^2}{x_{\alpha+\beta}x_\beta} -\la\alpha,\beta\ra-\tfrac{g_\ag(A_\alpha,A_\beta)}{x_\beta}\right)E_\beta.   
\end{align*}
Note that if $-\alpha+\beta\notin\Delta_\qg^+$ then the term with $N_{-\alpha,\beta}^2$ does not appear.  

Finally, assume that $\alpha+\beta\in\Delta_\qg$ and $\alpha-\beta\notin\Delta_\qg$, hence
\begin{align*}
R^b(E_{-\alpha},E_\beta)E_\alpha 
=& \Lambda^b(E_{-\alpha}) \Lambda^b(E_\beta)E_\alpha - 0 - 0 - 0 
=\tfrac{N_{\beta,\alpha} \left(x_{\alpha+\beta}-x_\beta\right)}{x_{\alpha+\beta}} \Lambda^b(E_{-\alpha}) E_{\alpha+\beta} \\
=&  \tfrac{N_{\beta,\alpha} \left(x_{\alpha+\beta}-x_\beta\right)}{x_{\alpha+\beta}}  \tfrac{N_{-\alpha,\alpha+\beta} \left(x_{\alpha+\beta}-x_\alpha\right)}{x_{\beta}}E_\beta 
= -\tfrac{N_{\alpha,\beta}^2 \left(x_{\alpha+\beta}-x_\beta\right)\left(x_{\alpha+\beta}-x_\alpha\right)}{x_{\alpha+\beta} x_{\beta}} E_\beta, 
\end{align*}
concluding the proof.  
\end{proof} 

As an application of Proposition \ref{Bcurv}, we classify Bismut flat Hermitian structures in the compact homogeneous case. 

\begin{theorem}\label{Bflat}
A $G$-invariant Hermitian structure $(J,g)$ as in \eqref{Jg} on a C-space $M=G/K$ 
is Bismut flat if and only if $M=G$ and $g$ is almost-bi-invariant (see Remark \ref{Bflat-rem}).   
\end{theorem} 

\begin{proof}
Assume that $R^b=0$ and consider $\alpha\in\Pi_\qg$ and $\beta\in\Pi_{\hg_f}$ such that $\alpha+\beta\in\Delta_\qg$ (i.e. $\la\alpha,\beta\ra<0$).  If $2\alpha+\beta\notin\Delta_\qg$, then it follows from \eqref{Bcurv2} that
$$
- x_\alpha\la\alpha,\alpha\ra = g_\ag(A_\alpha,A_\alpha) 
= g_\ag(A_\alpha,A_{\alpha+\beta}) = -x_\alpha\la\alpha, \alpha+\beta\ra,
$$
so $\la\alpha,\beta\ra=0$, which is a contradiction.  Thus $2\alpha+\beta\in\Delta_\qg$, and if $3\alpha+\beta\in\Delta_\qg$, then since $4\alpha+\beta, 2\alpha+2\beta\notin\Delta_\qg$, by \eqref{Bcurv2}, 
$$
- 3x_\alpha\la\alpha,\alpha\ra = 3g_\ag(A_\alpha,A_\alpha) = g_\ag(A_\alpha,A_{3\alpha+2\beta}) 
 = -x_\alpha\la\alpha, 3\alpha+2\beta\ra,
$$
again a contradiction.  We therefore obtain that $\alpha+\beta, 2\alpha+\beta\in\Delta_\qg$ and $3\alpha+\beta\notin\Delta_\qg$.  According to \eqref{Bcurv1}, $x_{2\alpha+\beta}=2x_{\alpha+\beta}$ or $x_{2\alpha+\beta}=x_{\alpha}$, and by \eqref{Bcurv4}, $x_{2\alpha+\beta}=x_{\alpha+\beta}$ or $x_{2\alpha+\beta}=x_{\alpha}$, hence $x_{2\alpha+\beta}=x_{\alpha}$.  It now follows from \eqref{Bcurv3} that  
$$
- 2x_\alpha\la\alpha,\alpha\ra = 2g_\ag(A_\alpha,A_\alpha) = g_\ag(A_\alpha,A_{2\alpha+\beta}) 
 = -x_\alpha\la\alpha, 2\alpha+\beta\ra,
$$
a contradiction.  We conclude that $\Pi_{\hg_f}=\emptyset$, that is, $[\hg_f,\hg_f]=[\kg,\kg]=0$ and $\hg_f=\kg\oplus\ag_f$ is the maximal torus of $\ggo_f$ (see \eqref{tdec}).  

On the other hand, if $\alpha,\beta\in\Pi_\qg$ and $\alpha+\beta\in\Delta_\qg$, then it follows from \eqref{Bcurv4} (as $\alpha-\beta\notin\Delta_\qg$) that $x_{\alpha+\beta}=x_\alpha$ or $x_{\alpha+\beta}=x_\beta$, which implies that $g_\ag(A_\alpha,A_\beta) = -x_\alpha\la\alpha,\beta\ra$ and $x_\alpha=x_\beta$ by \eqref{Bcurv3}.  Since $g_\ag(A_\alpha,A_\beta)=0$ if $\alpha+\beta\notin\Delta_\qg$ (i.e. $\la\alpha,\beta\ra=0$) by \eqref{Bcurv2}, we obtain that for each $i=1,\dots,s$ there exists $z_i>0$ such that $x_\alpha=x_\beta=z_i$ for all $\alpha,\beta\in\Pi_{\qg_i}$ (see \eqref{decflags}) and 
$$
g_\ag(\im A_\alpha,\im A_\beta)=z_iQ(\im H_\alpha,\im H_\beta), \qquad\forall \alpha,\beta\in\Pi_{\qg_i}. 
$$  
Since $\{ \im H_\alpha:\alpha\in\Pi_{\qg_i}\}$ is linearly independent (indeed, $\{ Z_\alpha=(\im H_\alpha)_{\hg_i}:\alpha\in\Pi_{\qg_i}\}$ is a basis of $\hg_i$, see \eqref{cone}), we obtain that the set $\{ A_\alpha:\alpha\in\Pi_{\qg_i}\}$ is linearly independent for all $i=1,\dots,s$.  Thus $\dim{\ag_f}=|\Pi_\qg|=\dim{\hg_f}$ and so $\ag_f=\hg_f$ and $\zg(\kg)=0$; in particular, $A_\alpha=H_\alpha$ for any $\alpha\in\Delta$.  This implies that $\kg=0$, $M=G$ is a Lie group, $[\ggo,\ggo]=\ag_f\oplus\qg$, $\ag_f$ is the maximal torus of $[\ggo,\ggo]$ and $g_\ag|_{\ag_f}=g_b|_{\ag_f}$ for the biinvariant metric on $[\ggo,\ggo]$ given by $g_b:=z_1Q|_{\ggo_1}+\dots+z_sQ|_{\ggo_s}$.    

Using that by \eqref{Bcurv3}, for any $\alpha,\beta,\alpha+\beta\in\Delta_\qg^+$, 
$$
\tfrac{N_{-\alpha,\beta}^2 \left(x_{\beta}-x_\alpha\right)^2}{x_{-\alpha+\beta}x_\beta} 
-\tfrac{N_{\alpha,\beta}^2 \left(x_{\alpha+\beta}-x_\alpha\right)^2}{x_{\alpha+\beta}x_\beta} -\la\alpha,\beta\ra-\tfrac{g_\ag(A_\alpha,A_\beta)}{x_\beta}=0, 
$$
it follows by induction in the height of the roots that $x_\alpha=z_i$ for all $\alpha\in\Delta_i$, $i=1,\dots,s$, that is, $g_\qg=g_b|_\qg$.  Thus $g|_{\ag_f\oplus\qg}=g_b|_{\ag_f\oplus\qg}$.   

Conversely, if $M=G$ is a Lie group and $g$ is an almost-bi-invariant metric on $G$, say $g|_{\ggo_i}=z_i(-\kil_{\ggo_i})=z_iQ$, $i=1,\dots,s$, then by Corollary \ref{Bcon}, 
\begin{align*}
\Lambda^b(E_\alpha)E_\beta =& \unm N_{\alpha,\beta} z_i 
\left(1-\epsilon_\alpha\epsilon_\beta\epsilon_{\alpha+\beta} (\epsilon_\alpha +\epsilon_\beta -\epsilon_{\alpha+\beta} )\right) E_{\alpha+\beta} \\
=& \unm N_{\alpha,\beta} z_i
\left(1-\epsilon_\alpha\epsilon_\beta+\epsilon_{\alpha+\beta} (\epsilon_\alpha +\epsilon_\beta )\right) E_{\alpha+\beta} =0, \qquad\forall \alpha,\beta\in\Delta_{\qg_i},   
\end{align*}
then for any $\alpha\in\Delta_i$ and $A\in\ag_f^c=\la A_\beta=H_\beta:\beta\in\Delta\ra_{\CC}$, the maximal torus of $[\ggo^c,\ggo^c]$,  
\begin{align*}
\Lambda^b(A)E_\alpha= \left(\alpha(A)+\tfrac{1}{z_i}g_\ag(A,H_\alpha)\right) E_\alpha 
= \left(-Q(A,H_\alpha)+\tfrac{1}{z_i}g_\ag(A,H_\alpha)\right) E_\alpha = 0.   
\end{align*}
This implies that $R^b=0$ (see Remark \ref{Bflat-rem}), concluding the proof.  
\end{proof}

\subsection{Bismut holonomy}\label{BiH-sec}
According to \S\ref{IC-sec}, the holonomy algebra $\holg(\nabla^b)$ is the Lie algebra generated by the subspace $\ad{\kg}|_\pg+\Lambda^b(\pg)\subset\ug(n)$.  For a Hermitian C-space, it follows from Corollary \ref{Bcon} that $\holg(\nabla^b)$ vanishes on $\ag$ and each $\qg_i$ is $\holg(\nabla^b)$-invariant, i.e., the Bismut holonomy always reduces to
\begin{equation}\label{holred}
\holg(\nabla^b)\subset\left[\begin{matrix} 
0&&&\\ &\ug(m_1)&& \\ &&\ddots&\\ &&&\ug(m_s) 
\end{matrix}\right], \qquad\dim{\ag}=2d, \qquad\dim{\qg_i}=2m_i, 
\end{equation}
relative to the decomposition $\pg=\ag\oplus\qg_1\oplus\dots\oplus\qg_s$ (see \eqref{decflags}).  This was proved in \cite[Section 4]{Stc} for a compact semisimple Lie group $M=G$.  For an additional reduction to a subalgebra of $\sug(n)$, see Remark \ref{Bsun} below.  

It follows from Corollary \ref{Bcon} that 
\begin{align}
\Lambda(A)E_\alpha=& \left(\alpha(A)+\tfrac{1}{x_\alpha}g_\ag(A,A_\alpha)\right)E_\alpha 
=\left(-Q(A,A_\alpha)+\tfrac{1}{x_\alpha}Q(A,P_\ag A_\alpha)\right)E_\alpha \label{LA} \\
=& Q\left(A,(\tfrac{1}{x_\alpha}P_\ag-I) A_\alpha\right)E_\alpha, \qquad\forall A\in\ag,\quad \alpha\in\Delta_\qg.\notag
\end{align}

\begin{remark}\label{gQhol}
For the normal metric $g=Q|_\pg$ (i.e., $x_\alpha=1$ for all $\alpha\in\Delta_\qg$ and $P_\ag=I$), we have that $\Lambda(\pg)=0$ by \eqref{Lb} and \eqref{LA}, so the holonomy is simply given by $\holg(\nabla^b)=\ad{\kg}|_\pg$.  We note that it is sufficient to have $g_\qg=g_b|_\qg$ for some bi-invariant metric $g_b$ on $G$ (i.e., $x_\alpha=x_\beta$ for all $\alpha,\beta\in\Delta_{\qg_i}$, $i=1,\dots,s$) to obtain that $\Lambda(\qg)=0$, and since $\Lambda(\ag)$ always leaves the $\qg_\alpha$ invariant, the irreducible components of $\holg(\nabla^b)$ coincide with those of the isotropy representation $\ad{\kg}|_\pg$.   
\end{remark}

According to \eqref{holred}, the complexification $\holg(\nabla^b)^c:=\holg(\nabla^b)\otimes\CC$ of the holonomy algebra can be embedded as follows, 
$$
\holg(\nabla^b)^c\subset \glg_{2m_1}(\CC)\oplus\dots\oplus\glg_{2m_s}(\CC) \subset \glg_{2(n-d)}(\CC):=\glg(\qg^c), 
$$ 
and since $\qg^{1,0}$ and $\qg^{0,1}$ are both $\holg(\nabla^b)^c$-invariant and $\holg(\nabla^b)^c|_{\qg^{1,0}}\simeq\holg(\nabla^b)^c|_{\qg^{0,1}}$, the Bismut holonomy $\holg(\nabla^b)$ is fully determined by 
$$
\holg(\nabla^b)^c|_{\qg^{1,0}}\subset \glg_{m_1}(\CC)\oplus\dots\oplus\glg_{m_s}(\CC)\subset \glg_{n-d}(\CC):=\glg(\qg^{1,0}).
$$
It follows from \eqref{Lb} that relative to a basis $\{ E_\beta:\beta\in\Delta_\qg^+\}$ of $\qg^{1,0}$ ordered in such a way that the length of the roots is increasing, the matrix of $\Lambda^b(E_\alpha)$ is strict lower (resp. upper) triangular for any $\alpha\in\Delta_\qg^+$ (resp. $\alpha\in\Delta_\qg^-$).  

\begin{proposition}\label{holirr}
Let $(M=G/K,J,g)$ be a Hermitian C-space and assume that $M=G/K$ is a C-space of fibration type (see Definition \ref{fibtype-def}).  If for some $i\in\{1,\dots,s\}$, 
$x_\alpha\ne x_\gamma$ for any $\alpha=\sum\limits_{\gamma\in\Pi_i}n_\gamma\gamma\in\Delta_{\qg_i}^+$ of length $\geq 2$ and any $\gamma\in\Pi_i$ such that $n_\gamma>0$ and $\alpha|_{\zg(\hg_i)}\ne\gamma|_{\zg(\hg_i)}$ (see \eqref{decflags} and \S\ref{isot-sec}), 
then $\qg_i$ is $\holg(\nabla^b)$-irreducible.  
\end{proposition}

\begin{remark} 
In particular, all the $\qg_i$ are $\holg(\nabla^b)$-irreducible for any invariant Hermitian metric $g=(x_1,\dots,x_r)$ on a complex C-space of fibration type such that $x_i\ne x_j$ for all $i\ne j$.  
\end{remark}

\begin{proof}
Let $\alpha_{max}\in\Delta^+_i$ denote the maximal root (see \eqref{decflags}).  Thus $\alpha_{max}\in\Delta_{\qg_i}^+$ and for any $\beta\in\Delta_{\qg_i}^+$ there exist $\gamma_1,\dots,\gamma_t\in\Pi_i$ such that $\delta_1,\dots,\delta_t\in\Delta_{\qg_i}^+$, where
$$
\delta_1:=\alpha_{max}-\gamma_1, \quad \delta_2:=\alpha_{max}-\gamma_1-\gamma_2, \quad \dots,\quad 
\delta_t:=\alpha_{max}-\gamma_1-\dots-\gamma_t=\beta.  
$$
On each step, 
\begin{enumerate}[{\small $\bullet$}] 
\item either $\delta_{k+1}|_{\zg(\hg_f)}= \delta_{k}|_{\zg(\hg_f)}$ (i.e., $\gamma_{k+1}\in\Pi_{\hg_i}$) and $E_{\delta_{k+1}}$, $E_{\delta_{k}}$ therefore belong to the same $\mg_j$, hence to the same $\holg(\nabla^b)^c$-irreducible component, 

\item or $\delta_{k+1}|_{\zg(\hg_f)}\ne \delta_{k}|_{\zg(\hg_i)}$, which implies that $\gamma_{k+1}\in\Pi_{\qg_i}$ and $\delta_{k}|_{\zg(\hg_i)}\ne \gamma_{k+1}|_{\zg(\hg_i)}$, hence by \eqref{Lb}, 
$$
\Lambda(E_{-\gamma_{k+1}})E_{\delta_{k}} = \tfrac{N_{-\gamma_{k+1},\delta_{k}} (x_{\delta_{k}}-x_{\gamma_{k+1}})}{x_{\delta_{k+1}}} E_{\delta_{k+1}}, 
$$ 
where $x_{\delta_{k}}-x_{\gamma_{k+1}}\ne 0$ by the hypothesis.  Thus $E_{\delta_{k+1}}$ and $E_{\delta_{k}}$ belong to the same $\holg(\nabla^b)^c$-irreducible component.  
\end{enumerate}
We therefore obtain that $E_{\alpha_{max}}$ and $E_{\beta}$ belong to the same $\holg(\nabla^b)^c$-irreducible component for any $\beta\in\Delta_{\qg_i}^+$, that is, $\qg^{1,0}$ is $\holg(\nabla^b)^c$-irreducible, concluding the proof.  
\end{proof}

\begin{corollary}\label{holgK}
If $(M=G/K,J.g)$ is a Hermitian C-space of fibration type such that $G_f$ is simple and $g_\qg|_{\qg_i}$ is K\"ahler on the flag $F_i=G_i/H_i$ (i.e., $x_{\alpha+\beta}=x_\alpha+x_\beta$ for all $\alpha,\beta,\alpha+\beta\in\Delta_{\qg_i}^+$, see \S\ref{K}), then $\qg_i$ is $\holg(\nabla^b)$-irreducible.  
\end{corollary}

This in particular holds for any LCK Hermitian structure (see \S\ref{LCK-sec}).   

\begin{example}\label{b21-Bcon}
For generalized Calabi-Eckmann manifolds (see Examples \ref{b21-comp} and \ref{b21-herm}), $\holg(\nabla^b)$ leaves the $Q$-orthogonal decomposition $\pg=\ag\oplus\qg_1\oplus\qg_2$ invariant, providing the reduction 
$$
\holg(\nabla^b)\subset\left[\begin{matrix} 
0&0&0\\ 0&\ug(\tfrac{\dim{\qg_1}}{2})&0\\ 0&0&\ug(\tfrac{\dim{\qg_2}}{2})
\end{matrix}\right].    
$$
For a metric $g$ such that $g|_{\qg_1}=(x_1,\dots,x_{r_1})$ and $1\leq j<k\leq r_1$, we obtain from Corollary \ref{Bcon} that  
\begin{enumerate}[{\small $\bullet$}] 
\item either $x_{j+k}=x_j$ or $E_{k\gamma_i}$ and $E_{(j+k)\gamma_i}$ belong to the same $\holg(\nabla^b)^c$-irreducible component, 

\item either $x_k=x_j$ or $E_{k\gamma_i}$ and $E_{(k-j)\gamma_i}$ belong to the same $\holg(\nabla^b)^c$-irreducible component.
\end{enumerate} 
Using these properties, it is not hard to prove that $\qg_1$ is always $\holg(\nabla^b)$-irreducible, unless $x_1=\dots=x_{r_1}$ (see Remark \ref{gQhol}).  Analogously, the same holds for $\qg_2$.     
\end{example}

\subsection{Ricci forms}\label{CC-sec}
It follows from \eqref{cc-def} that, for any $t\in\RR$, the Nomizu operator $\Lambda^t$ of the canonical connection $\nabla^t$ is given by 
$$
\Lambda^t = \tfrac{1+t}{2}\Lambda^c+\tfrac{1-t}{2}\Lambda^b.
$$ 
A formula for the connection $\Lambda^t$ of any Hermitian C-space $(M=G/K,J,g)$ therefore follows from the formulas given for $\Lambda^c$ in Proposition \ref{NoCc} and $\Lambda^b$ in Corollary \ref{Bcon}.  Concerning the corresponding Ricci forms, it follows from \eqref{rhot} and Lemma \ref{dto} that 
\begin{equation}\label{rhot2}
\rho^t(X,Y)=  Q\left([X,Y],Z_{J_\qg} +\tfrac{t-1}{2} P_\ag(Z_{J_\qg,g_\qg})_\ag\right), \qquad\forall X,Y\in\pg. 
\end{equation}
Indeed, 
\begin{align*}
\rho^t=& \rho^c+\tfrac{t-1}{2}dd_g^*\omega = Q([\cdot,\cdot],Z_{J_\qg}) +\tfrac{t-1}{2} g([\cdot,\cdot]_\pg,(Z_{J_\qg,g_\qg})_\ag)\\  
=&  Q([\cdot,\cdot],Z_{J_\qg}) +\tfrac{t-1}{2} Q([\cdot,\cdot]_\pg,P_\ag(Z_{J_\qg,g_\qg})_\ag) \\
=&  Q([\cdot,\cdot],Z_{J_\qg}) +\tfrac{t-1}{2}  Q([\cdot,\cdot],P_\ag(Z_{J_\qg,g_\qg})_\ag).   
\end{align*}
In particular, it follows from Remark \ref{C1} that if $t\ne 1$, then for any $\sigma\in c_1(M,J)$, there exists a compatible metric $g=g_\ag+g_\qg$ such that $\rho^t(g)=\sigma$.  Also, a $G$-invariant Hermitian structure $(J,g)$ as in \eqref{Jg} has $\rho^t=0$ if and only if $Z_{J_\qg}\in\ag_f$ (i.e., $c_1(M,J)=0$) and 
\begin{equation}\label{rhot0}
\left(P_\ag(Z_{J_\qg,g_\qg})_\ag\right)_{\ag_f}=\tfrac{1-t}{2}Z_{J_\qg}. 
\end{equation} 

\begin{corollary}\label{nc-rft0}
Let $M=G/K$ be a complex C-space.  
\begin{enumerate}[{\rm (i)}] 
\item If a $G$-invariant Hermitian metric $g$ as in \eqref{g} has $\rho^t=0$, then $t<1$.   

\item There exists a $G$-invariant Hermitian metric $g$ as in \eqref{g} such that $\rho^{t_0}=0$ for some $t_0<1$ if and only if for every $t<1$, there is a metric such that $\rho^t=0$.  
\end{enumerate} 
\end{corollary}

\begin{remark}
We refer to \cite{BrdStn} for related results in the general context.  
\end{remark}

\begin{proof}
Since \eqref{rhot0} implies that
$$
\tfrac{1-t}{2} Q(Z_{J_\qg},Z_{J_\qg,g_\qg}) = Q(P_\ag(Z_{J_\qg,g_\qg})_\ag,Z_{J_\qg,g_\qg}) = g_\ag((Z_{J_\qg,g_\qg})_\ag,(Z_{J_\qg,g_\qg})_\ag)>0
$$
and $Q(Z_{J_\qg},Z_{J_\qg,g_\qg})>0$ by \eqref{Jal}, part (i) follows.  

For part (ii), we note that if $\rho^{t_0}(g)=0$ for $g=g_\ag+g_\qg$, then by \eqref{rhot0}, $\rho^{t}(g')=0$ for the metric given by $g':=\tfrac{1-t}{1-t_0}g_\ag+g_\qg$, concluding the proof.  
\end{proof} 

\begin{remark}\label{Bsun}
The vanishing of the Ricci form in the Bismut case $t=-1$ is called Calabi-Yau with torsion and will be studied in \S\ref{CYT-sec}.  Since $\tr{\Lambda^b(E_\alpha)}=0$ for all $\alpha\in\Delta_\qg$, we obtain that the Bismut holonomy reduces to $\sug(n)$ if and only if   
$$
-\im\sum_{\alpha\in\Delta_\qg^+} -Q(A,\im H_\alpha))+\tfrac{g_\ag(A,\im A_\alpha)}{x_\alpha} = \sum_{\alpha\in\Delta_\qg}\alpha(A)+\tfrac{g_\ag(A,A_\alpha)}{x_\alpha} =0, \qquad\forall A\in\ag.
$$
This is equivalent to $\left(P_\ag(Z_{J_\qg,g_\qg})_\ag\right)_{\ag_f}=Z_{J_\qg}$, that is, the Bismut Ricci form vanishes (see \eqref{rhot0}) or $(M^{2n}=G/K,J,g)$ is Calabi-Yau with torsion (see \S\ref{CYT-sec}).  
\end{remark}

% 9/9/2026

\section{Balanced}\label{bal-sec} 

\begin{definition}\label{bal-def} 
A Hermitian structure $(J,g)$ is called {\it balanced} if $d_g^*\omega=0$ (or equivalently, $d\omega^{n-1}=0$).  
\end{definition}

In other words, the Lee form vanishes (see \eqref{lee}).  Note that if $(J,g)$ is balanced, then all the Ricci forms of canonical connections coincide, i.e., $\rho^t=\rho^c$ for all $t\in\RR$ (see \eqref{rhot}).  The following characterizations directly follow from Lemma \ref{dto} and \eqref{lee2}.  

\begin{proposition}\label{bal}
Let $M=G/K=G_f/K\times T^z$ be a C-space.     
\begin{enumerate}[{\rm (i)}]
\item \cite[Theorem 9, (i)]{Pds} A $G$-invariant Hermitian structure $(J,g)$ as in \eqref{Jg} is balanced if and only if 
$$
Z_{J_\qg,g_\qg}\in\zg(\kg),
$$ 
where $Z_{J_\qg,g_\qg}$ is the metric Koszul vector (see \eqref{koszul-g}).  

\item A complex C-space $(M=G/K,J)$ admits a $G$-invariant balanced metric as in \eqref{g} if and only if 
$$
C(J_\qg)\cap\zg(\kg)\ne\emptyset,
$$
where $C(J_\qg)\subset\zg(\hg_f)$ is the cone defined in \eqref{cone}.   
\end{enumerate}
\end{proposition}

\begin{remark}\label{bal-rem} \hspace{1cm} 
\begin{enumerate}[{\small $\bullet$}] 
\item The conditions in the proposition do not depend on $(J_\ag,g_\ag)$.  Note that the balanced condition in part (i) is equivalent to $\sum\limits_{\alpha\in\Delta_\qg^+} \tfrac{1}{x_\alpha}\im A_\alpha=0$.  

\item If $(J,g)$ is balanced, then $c_1(M,J)\ne 0$.  Indeed, by \eqref{Jal},
$$
Q(Z_{J_\qg},Z_{J_\qg,g_\qg})= -\kil_{\ggo_f}(Z_{J_\qg},Z_{J_\qg,g_\qg})=\sum\limits_{\alpha\in\Delta_\qg^+} \tfrac{1}{x_\alpha}\im\alpha(Z_{J_\qg}) 
=\sum\limits_{\alpha\in\Delta_\qg^+} \tfrac{1}{x_\alpha}\la\alpha,\rho_{J_\qg}\ra>0,
$$  
thus $Z_{J_\qg}$ can never belong to $\ag$ if $Z_{J_\qg,g_\qg}\in\zg(\kg)$ (see Lemma \ref{c10}).  

\item The condition $\zg(\kg)\ne 0$ is necessary for the existence of a $G$-invariant balanced Hermitian structure on a C-space $M=G/K$.  In particular, $M$ can not be a Lie group and $K$ can not be semisimple.  Remarkably, condition $\zg(\kg)\ne 0$ turns out to be sufficient for the existence in several cases (see Theorem \ref{bal2} below). 

\item Any invariant Hermitian structure on a K\"ahler C-space (i.e., $\zg(\kg)=\zg(\hg_f)$) is balanced.  
\end{enumerate}
\end{remark}

According to Proposition \ref{bal}, (ii), the space $\cca^G_{bal}$ of all $G$-invariant complex structures on a C-space $M=G/K$ admitting a balanced compatible metric depends on $2d^2$ parameters if nonempty.  More precisely,   
\begin{equation}\label{Cbal}
\cca^G_{bal}=\cca^G_{i_1}\sqcup\dots\sqcup\cca^G_{i_m} =\bigsqcup_{C(J_i)\cap\zg(\kg)\ne\emptyset}\cca^G_i, \qquad \dim{\cca^G_{i_j}}=2d^2;  
\end{equation} 
cf.\ \eqref{Ccal} and \eqref{Cc10}.  On the other hand, the space $\hca^G_{bal,J}$ of all $G$-invariant balanced Hermitian structures (see \S\ref{dim-sec}) depends on $r'+d^2$ parameters if nonempty, where $0<r'\leq r$ is the dimension of 
$$
\{ g_\qg=(x_1,\dots,x_r):Z_{J_\qg,g_\qg}\in\zg(\kg)\}.  
$$  

\begin{example}\label{su5-22} 
The C-space $M^{20}=\SU(5)/K$ given in Example \ref{su5-5} has $\zg(\hg_f)=\RR Z_\kg\oplus\ag$, where $\kg=\RR Z_\kg$ and $\dim{\ag}=2$.  Any vector $(a,b,c,d,d)\in\zg(\hg_f)$, $d=-\tfrac{a+b+c}{2}$ will be identified from now on with $(a,b,c)$.  The usual invariant ordering $\Delta^+=\{\alpha_{ij}:i<j\}$ determines a complex structure $J_1$ with $\Pi_{\hg_f}=\alpha_{45}$, $\Pi_\qg=\{\alpha_{12},\alpha_{23},\alpha_{34}\}$, and $Z_{ij}:=Z_{\alpha_{ij}}$ given by 
$$
\begin{array}{c}
Z_{12}=(1,-1,0), \quad Z_{23}=(0,1,-1), \quad Z_{34}=Z_{35}=(0,0,1), \\ 
Z_{13}=(1,0,-1), \quad Z_{24}=Z_{25}=(0,1,0), \quad Z_{14}=Z_{15}=(1,0,0), 
\end{array}
$$
so $Z_{J_1}=(4,2,0)$ and 
$$
C(J_1)=\{(a,b,c)\in\zg(\hg):a,a+b,a+b+c>0\}.  
$$
By Proposition \ref{bal}, (ii) $(M=G/K,J)$, $J=(J_\ag,J_1)$, where $Z_\kg=(a,b,c)$, admits a balanced metric if and only if $a,a+b,a+b+c>0$ (or all $<0$).  If $g_\qg=(\tfrac{1}{x_1},\dots,\tfrac{1}{x_6})$, then 
$$
Z_{J_1,g_\qg}=(x_1+x_4+2x_6, -x_1+x_2+2x_5, -x_2+2x_3-x_4),
$$
so by Proposition \ref{bal}, (i), the metric $g=g_\ag+g_\qg$ is balanced if and only if 
$$
x_1+x_4+2x_6=\lambda a, \quad -x_1+x_2+2x_5=\lambda b, \quad -x_2+2x_3-x_4=\lambda c, 
$$
for some $\lambda>0$.  Thus $\dim{\hca^G_{bal,J_1}}=3+1=4$.  

On the other hand, the invariant ordering defined by 
$$
\Pi_{\hg_f}=\{-\alpha_{45}\}, \qquad \Pi_\qg=\{-\alpha_{23},-\alpha_{12},\alpha_{15}\}
$$ 
determines a complex structure $J_2$ with 
$$
\begin{array}{c}
Z_{-23}=(0,-1,1), \quad Z_{-12}=(-1,1,0), \quad Z_{15}=Z_{14}=(1,0,0), \\ 
Z_{-13}=(-1,0,1), \quad Z_{25}=Z_{24}=(0,1,0), \quad Z_{35}=Z_{34}=(0,0,1),
\end{array}
$$
giving rise to $Z_{J_2}=(0,2,4)$ and   
$$
C(J_2)=\{(a,b,c,d,d)\in\zg(\hg):c,b+c,a+b+c>0\}.  
$$
Thus $(M=G/K,J)$, $J=(J_\ag,J_2)$, where $Z_\kg=(a,b,c)$, admits a balanced metric if and only if $c,b+c,a+b+c>0$.  Note that on the family of homogeneous spaces defined by $Z_\kg=(1,b,1)$, $-1<b<\infty$, both complex structures $(J_\ag,J_1)$ and $(J_\ag,J_2)$ admit balanced metrics. Since 
$$
Z_{J_2,g_\qg}=(-x_2+2x_3-x_4, -x_1+x_2+2x_5, x_1+x_4+2x_6),
$$
the metric $g=g_\ag+g_\qg$ is balanced  if and only if 
$$
-x_2+2x_3-x_4=\lambda a, \quad -x_1+x_2+2x_5=\lambda b, \quad x_1+x_4+2x_6=\lambda c, 
$$
for some $\lambda>0$, so $\dim{\hca^G_{bal,J_1}}=3+1=4$.   
\end{example}

Recall from \eqref{decflags} the decomposition $\ggo_f=\ggo_1\oplus\dots\oplus\ggo_s$ in simple factors.  

\begin{theorem}\label{bal2}
A C-space $M=G/K$ such that $G_f$ is simple and $\rank(\ggo_f)\geq 2$ admits a balanced $G$-invariant Hermitian structure $(J,g)$ as in \eqref{Jg} if and only if $\zg(\kg)\ne 0$.     
\end{theorem}

\begin{proof}
If $J_1,\dots,J_{k}$ are all the $G$-invariant complex structures on the flag $F=G_f/H_f$, then we consider
$$
C(F):=\bigcup_{i=1}^{k}C(J_i)\subset\zg(\hg_f). 
$$
We will show that $C(F)=\zg(\hg_f)\setminus\{ 0\}$, from which the theorem follows using Proposition \ref{bal}.  For each $J_\qg$, which is attached to the invariant ordering $\Delta_\qg^+=\Delta_\qg\cap\Delta^+$, or equivalently, to the chamber $Ch(J_\qg)\subset\im\zg(\hg_f)\setminus\{ 0\}$ (see \eqref{ch}), we consider the corresponding chamber 
$$
Ch(\Delta^+)=\{ X\in\tg_f:-\kil_{\ggo_f}(X,\im H_\alpha)>0,\;\alpha\in\Pi\}\subset\tg_f, 
\quad\mbox{i.e.}, \quad Ch(J_\qg)=Ch(\Delta^+)\cap\zg(\hg_f), 
$$ 
and define the open convex cone
$$
C(\Delta^+):=\la \im H_\alpha:\alpha\in\Pi_\qg\ra_{\RR_{>0}} \subset\tg_f, \quad\mbox{i.e.}, \quad C(\Delta^+)\cap\zg(\hg_f) = C(J_\qg).  
$$
Since $\Pi$ is connected and $|\Pi|\geq 2$, we obtain that $\overline{Ch(\Delta^+)}\setminus\{ 0\}\subset C(\Delta^+)$.  Thus 
$$
\overline{Ch(J_\qg)}\setminus\{ 0\} = \overline{Ch(\Delta^+)}\cap\zg(\hg_f)\setminus\{ 0\} \subset C(\Delta^+)\cap\zg(\hg_f) = C(J_\qg).
$$
Note that the first and last equalities follow from the fact that $Z_\alpha=0$ if and only if $\alpha\in\Pi_{\hg_f}$.  Thus 
$$
\zg(\hg_f)\setminus\{ 0\}
= \bigcup_{i=1}^k \overline{Ch(J_i)}\setminus\{ 0\}\subset C(F),  
$$
which implies that $\zg(\hg_f)\setminus\{ 0\}=C(F)$.  
\end{proof}

The following stronger existence result follows from the above theorem and \eqref{Cbal}. 

\begin{corollary}\label{bal-cor}
On any C-space $M=G/K$ as in Theorem \ref{bal2} such that $\zg(\kg)\ne 0$ and $\dim{\ag}=2d$, there exists a $2d^2$-parametric space of $G$-invariant complex structures each one admitting an $(r'+d^2)$-parametric space of balanced metrics ($r'>0$).  
\end{corollary}

It follows from Remark \ref{bal-rem} that the spaces $\cca^G_0$ (see \eqref{Cc10}) and $\cca^G_{bal}$ (see \eqref{Cbal}) are disjoint, and one may ask how far is $\cca^G_0\sqcup\cca^G_{bal}$ from exhausting $\cca^G$.  

\begin{example}\label{c10-bal}
We consider the C-space defined by $Z_\kg=(0,1,1)$ in Example \ref{su5-22} endowed with $J=(J_\ag,J_1)$.  Since $Q(Z_{J_1},Z_\kg)\ne 0$, $c_1(M,J)\ne 0$ (see Lemma \ref{c10}), and since $\pm Z_\kg\notin C(J_1)$, $(M,J)$ does not admit balanced metrics by Proposition \ref{bal}, (ii).  This implies that $\cca^G_0\sqcup\cca^G_{bal}\ne\cca^G$.        
\end{example}

The following example shows that a balanced metric can also be generalized Einstein.  

\begin{example}\label{BRF-bal-2}
For the spaces given in Example \ref{BRF-bal}, it is proved in \cite{BRF} that if
$$
g=g_\ag+g_\qg, \qquad g_\ag=\tfrac{c_1}{c_1-1}\gk|_\ag, \qquad g_\qg=\gk|_{\qg_1}+\tfrac{1}{c_1-1}\gk|_{\qg_2},  
$$
then the generalized metric $(g,H_0)$ is {\it generalized Einstein} (or Bismut Ricci flat): $\Ricci(g)=\tfrac{1}{4} (H_0)_g^2$ and $d_g^*H_0=0$, where $H_0=\vp_R$ is the closed $3$-form defined in Example \ref{BRF-bal} (see \cite{GrcStr} for further information).  We note that $g$ is compatible with a $d(d-1)$-parametric set of complex structures if $\dim{\hg_f}=2d$.  Moreover, if $c_1=2$ (e.g., $G_1=G_2$ or $p=q$), then $g$ is always balanced for all these complex structures.  
\end{example}

\subsection{Locally conformally balanced (LCB)}\label{LCB-sec}
A Hermitian structure $(J,g)$ is called {\it locally conformally balanced} (LCB for short) whenever the Lee form is closed (see \cite{Prd}).  Recall that $(J,g)$ is balanced if and only if the Lee form vanishes (see Definition \ref{bal-def}).  

On a C-space $M=G/K=G_f/K\times T^z)$, where $\pg=\ag\oplus\qg$ and $\ag=\zg(\ggo)\oplus\ag_f$, since the Lee form of a Hermitian structure $(J,g)$ as in \eqref{Jg} is given by 
$$
\theta_L=-g(\cdot,J_\ag(Z_{J_\qg,g_\qg})_{\ag_f}), \qquad\mbox{where}\qquad  Z_{J_\qg,g_\qg}\in\zg(\hg_f)=\zg(\kg)\oplus\ag_f, 
$$
we obtain from \S\ref{dR1} that $(J,g)$ is LCB (i.e., $d\theta_L=0$) if and only if 
\begin{equation}\label{LCB}
P_\ag J_\ag(Z_{J_\qg,g_\qg})_{\ag_f}\in\zg(\ggo).  
\end{equation}
This implies that a complex C-space $(M=G/K,J)$ which is not balanced, i.e., $C(J_\qg)\cap\zg(\kg)=\emptyset$ (see Proposition \ref{bal}, (ii)), admits an LCB metric if and only if $\zg(\ggo)\ne 0$.   

\begin{remark}\label{CB-rem}
Since a $G$-invariant closed $1$-form on a compact homogeneous space $M=G/K$ is never exact, $(J,g)$ is never {\it conformally balanced} (see \cite{FinTms}).  
\end{remark}

% 14/8/26

\section{Pluriclosed (SKT)}\label{SKT-sec} 

The following condition generalizing K\"ahler (i.e., $d\omega=0$) stands out by also being linear in $\omega$.  

\begin{definition}\label{SKT-def} 
A Hermitian structure $(J,g)$ on a manifold $M$ is said to be {\it pluriclosed}, or {\it strong K\"ahler with torsion} (SKT for short), if $\partial\overline{\partial}\omega=0$, or equivalently, $dd^c\omega=0$.  
\end{definition}

In particular, $[\omega]\in H^{1,1}_A(M)$ and $[\partial\omega]\in H^{2,1}_{\overline{\partial}}(M)$ if $\omega$ is pluriclosed (see \S\ref{cohom-sec}).  Moreover, $[\omega]\ne 0\in H^{1,1}_A(M)$ since $\omega\ne\sigma_A$ for all $A\in\ag$ and $[\partial\omega]\ne 0\in H^{2,1}_{\overline{\partial}}(M)$ by Proposition \ref{h21bound}, (i) if $G$ is semisimple.  

We now show how the computation of $\Ker\partial\overline{\partial}|_{\Lambda^{1,1}}$ given in \S\ref{A-sec} provides the classification of C-spaces admitting a pluriclosed structure.    

\begin{theorem}\label{pluri-clasif}
A complex C-space admits a pluriclosed $G$-invariant metric if and only if it is the product of a flag manifold and a compact Lie group.  Moreover, the restriction of the metric on the flag is K\"ahler and the restriction of the Hermitian structure on the compact Lie group is pluriclosed, left-invariant and also invariant on the right for some maximal torus. 
\end{theorem}

\begin{proof}
Let $(M=G/K=G_f/K\times T^z,J)$ be a complex C-space that admits a pluriclosed $G$-invariant metric, which can be assumed to be of the form $g=g_\ag+g_\qg$ as in \eqref{g} by the symmetrization process given in Remark \ref{aver}.  The K\"ahler form is therefore given by $\omega=\omega_\ag+\omega_\qg$ (see \eqref{o}).  Since $\omega \in\Ker\partial\overline{\partial}|_{\Lambda^{1,1}}$ and $h(\omega_\ag)=g_\ag$ (see \eqref{h11A-h}), it follows from Lemma \ref{h11A-lem}, (i) that $\underline{\ag}_f=0$ (see \eqref{deca}).  This implies that 
$$
M=\widetilde{G}/\widetilde{H}\times T^z\times\overline{G}, \qquad G_f=\widetilde{G}\times\overline{G}, \qquad K=\widetilde{H},
$$
that is, $M$ is the product of a flag manifold and a compact Lie group (see Proposition \ref{decC}).  

On the other hand, we obtain from Lemma \ref{h11A-lem}, (iii) that the restriction of $g$ on the flag $\widetilde{G}/\widetilde{H}$ is a K\"ahler metric, and from Lemma \ref{h11A-lem}, (ii) and (iii) that 
the restriction of $(J,g)$ on the compact Lie group $T^z\times\overline{G}$ is a pluriclosed left-invariant Hermitian structure. 
\end{proof}

\subsection{Compact Lie groups}\label{skt-lg}
We focus in this section on the case when 
$$
M=G=T^z\times\overline{G}, \qquad \overline{\ggo}=[\ggo,\ggo]=\ggo_1\oplus\dots\oplus\ggo_s, \qquad z=\dim{\zg(\ggo)},
$$ 
i.e., the C-space $M$ is a compact Lie group.  We fix a left-invariant complex structure $J=J_\ag+J_\qg$ on $M$, where $\overline{\ggo}=\overline{\ag}\oplus\qg$, $\ag=\zg(\ggo)\oplus\overline{\ag}$ and $\overline{\ag}=\tg_1\oplus\dots\oplus\tg_s$ is a maximal torus of $\overline{\ggo}$ (see Examples \ref{LG}, \ref{LG-c} and \ref{LG-2}).   

According to Lemma \ref{h11A-lem}, (ii) and (iii), a left-invariant metric compatible with $J$ of the form $g=g_\ag+g_\qg$, $g_\qg=(x_\alpha)_{\alpha\in\Delta}$ (note that $g$ is in addition $\Ad(T)$-invariant, where $T$ is the maximal torus of $G$) is pluriclosed if and only if 
\begin{enumerate}[{\rm (a)}] 
\item $g_\ag|_{\overline{\ag}\times\overline{\ag}}=-(z_{1}\kil_{\ggo_{1}}+\dots+z_s\kil_{\ggo_s})|_{\overline{\ag}\times\overline{\ag}}$, for some $z_{1},\dots,z_s>0$, that is, relative to the decomposition $\ag=\zg(\ggo)\oplus\overline{\ag}$ and some $Q$-orthonormal basis,    
\begin{equation}\label{ga-skt}
g_\ag=\left[\begin{matrix} A&B\\ B^t&C \end{matrix}\right], \qquad A^t=A, \qquad  
C=\left[\begin{matrix} 
z_1I_{\tg_1}&&\\ 
&\ddots&\\ 
&& z_sI_{\tg_s}
\end{matrix}\right]. 
\end{equation}
Note that the matrices $A$ and $B$ must satisfy some closed and open conditions, respectively, in order to guarantee that $g_\ag$ is compatible with $J_\ag$ and positive definite.  

\item $x_{\alpha+\beta}=x_\alpha+x_\beta-z_i$, for all $\alpha,\beta,\alpha+\beta\in\Delta_{\qg_i}^+$ and $i=1,\dots,s$.  Equivalently, for each fixed $i=1,\dots,s$, 
\begin{equation}\label{gq-skt}
x_\alpha=z_i+\sum_{j=1}^m n_j(x_j-1), \qquad \forall \alpha=\sum_{j=1}^m n_j\alpha_j, \quad n_j\in\NN_0,  
\end{equation}
where $\Pi_i=\{\alpha_1,\dots,\alpha_m\}\subset\Delta_i^+$ and $x_j:=x_{\alpha_j}$ for $j=1,\dots,m$.  
\end{enumerate}
This was proved in \cite{SKT-LG} for compact semisimple Lie groups.  

\begin{remark}\label{xpos}
A sufficient condition to obtain that $x_\alpha>0$ for all $\alpha\in\Delta^+$ is given by 
$$
x_i\geq 1-\tfrac{1}{h(\alpha_{max})}, \qquad \forall i=1,\dots,n, 
$$
where $h(\alpha_{max})$ is the height of the maximal root $\alpha_{max}\in\Delta^+$.  
\end{remark}

\begin{remark}
A pluriclosed metric $g$ as above is Bismut flat if and only if $x_\alpha=z_i$ for all $\alpha\in\Delta_{\qg_i}^+$ and $i=1,\dots,s$ (see Theorem \ref{Bflat}).   
\end{remark}

Recall from Theorem \ref{h11A} that $h^{1,1}_A=\dim{S}$, where $S$ is the subspace of all symmetric bilinear forms $h:\ag\times\ag\rightarrow\RR$ such that $h$ is compatible with $J_\ag$ and $h|_{\overline{\ag}\times\overline{\ag}}$ coincides with the restriction of some biinvariant metric on $\overline{G}$.  

It follows from (a) and (b) that the space of all pluriclosed left-invariant metrics of the form $g=g_\ag+g_\qg$ on $M=G$ is either empty (i.e., $S=\{ 0\}$) or it depends on $\rank(\overline{\ggo})+h^{1,1}_A$ parameters.  Note that if $M=G=\overline{G}$ is semisimple, then $h^{1,1}_A$ is precisely the number of irreducible factors of $(G,J)$ as a complex manifold (see Definition \ref{Cirr} and Example \ref{LG-2}) and $g_\ag$ is necessarily the restriction of a biinvariant metric on $G$ (cf.\ \cite{SKT-LG}).  

The following examples show that in the case when $M=G$ is not semisimple, the pluriclosed geometry is much richer.   

\begin{example}\label{A3-skt}
We consider the left-invariant Hermitian structures $(J,g)$ given as follows:
$$
J_\ag=\left[\begin{matrix} J_1&0\\ 0&J_2 \end{matrix}\right], \qquad J_1:\zg(\ggo)\rightarrow\zg(\ggo), \qquad J_2:\overline{\ag}\rightarrow\overline{\ag}, 
$$
where $J_i$ is any $Q$-skew-symmetric map such that $J_i^2=-I$, $g_\ag$ is as in \eqref{ga-skt}, so the compatibility holds if and only if $AJ_1=J_1A$, $BJ_2=J_1B$ and $CJ_2=J_2C$, and $g_\qg=(x_\alpha)_{\alpha\in\Delta}$ is as in \eqref{gq-skt}.  This produces a family depending on  $s'+d_1^2+2d_1d_2+2d_2$ parameters of pluriclosed left-invariant metrics on $M=G$, where $s'$ is the number of $J_2$-invariant irreducible subspaces of $\overline{\ag}$ which are partial sums of the $\tg_i$ and $\dim{\zg(\ggo)}=2d_1$, $\dim{\overline{\ag}}=2d_2$.  One of these metrics is, on the maximal torus $\ag$ of $\ggo$, the restriction of a bi-invariant metric if and only if $B=0$.  Note that $J$ is reducible as soon as $\zg(\ggo)\ne 0$ or $s'\geq 2$.  
\end{example}

\begin{example}\label{A2-skt}
Assume that $\dim{\zg(\ggo)}=\dim{\overline{\ag}} = d$ as in Example \ref{A2} and set 
$$
J_\ag=\left[\begin{matrix} 0&-I\\ I&0 \end{matrix}\right], \qquad 
g_\ag=\left[\begin{matrix} C&B\\ -B&C \end{matrix}\right],  \qquad B^t=-B, 
$$
where $C$ is as in \eqref{ga-skt} and $g_\qg$ is as in \eqref{gq-skt}.  This provides a family of pluriclosed left-invariant metrics on $M=G$ depending on $s+\tbinom{d}{2}+2d$ parameters which, when $B\ne 0$, are not the restriction of any biinvariant metric on the maximal torus $\ag$ of $\ggo$.  Note that $J$ is always irreducible and that all these pluriclosed metrics are also LCB (see \S\ref{LCB-sec}).  
\end{example}

As a generalization of the above example, we can take 
\begin{equation}\label{A2-skt2}
J_\ag=\left[\begin{matrix} 0&-D^{-1}\\ D&0 \end{matrix}\right], \qquad 
g_\ag=\left[\begin{matrix} A&B\\ B^t&C \end{matrix}\right],  \qquad A=D^tCD, \quad (BD)^t=-BD, 
\end{equation}
which provides a much larger family of pluriclosed Hermitian structures.  Using \eqref{LCB}, it is easy to see that one of these pluriclosed metrics is LCB if and only if $BZ_{J_\qg,g_\qg}=0$.

\subsection{Pluriclosed flow}\label{PF-sec}
The {\it pluriclosed flow} for a one-parameter family of Hermitian metrics $\omega(t)=g(t)(J\cdot,\cdot)$ on a given complex manifold $(M,J)$ is defined by 
\begin{equation}\label{PF} 
\dpars\omega(t)=-(\rho^b)^{1,1}(t), \qquad \omega(0)=\omega_0,
\end{equation}
where $(\rho^b)^{1,1}(t)$ is the $(1,1)$-component of the Bismut Ricci form $\rho^b(t)$ of $\omega(t)$ (see \S\ref{CC-sec}).  

The pluriclosed flow in the case when $M=G$ is semisimple was studied in \cite[Section 5]{SKT-LG}, we study in what follows the general case of a compact Lie group $G$.  According to \eqref{rhot2}, the Bismut Ricci form $\rho^b$ of a left-invariant pluriclosed metric $g=g_\ag+g_\qg$ is given by 
$$
\rho^b(X,Y)=  Q\left([X,Y],Z_{J_\qg} - P_\ag Z_{J_\qg,g_\qg}\right), \qquad\forall X,Y\in\ggo, 
$$
where $Z_{J_\qg},Z_{J_\qg,g_\qg}\in\overline{\ag}$ are the Koszul (see \eqref{koszul-g}) and metric Koszul (see \eqref{koszul-g}) vectors, respectively.  Since the only nonzero components are $\rho^b(E_\beta,E_{-\beta})$ and 
$$
(P_\ag Z_{J_\qg,g_\qg})_{\overline{\ag}}= \sum_{i=1}^s\sum_{\alpha\in\Delta^+_i} \tfrac{z_i}{x_\alpha}\im H_\alpha,
$$
the pluriclosed flow evolution for $g(t)=g_\ag(t)+g_\qg(t)$, $g_\qg(t)=(x_\beta(t))_{\beta\in\Delta^+}$ is given by $g_\ag(t)\equiv g_\ag(0)$ and
\begin{equation}\label{PF} 
x_\beta'=-\sum_{\alpha\in\Delta^+_i} \left(1-\tfrac{z_i}{x_\alpha}\right) \la\alpha,\beta\ra, \qquad\forall\beta\in\Delta^+_i, \quad i=1,\dots,s.     
\end{equation}
These are precisely the equations obtained in \cite[Section 5]{SKT-LG} in the semisimple case, where it is proved that the solution is always immortal and $x_\beta(t)\to z_i$, as $t\to\infty$, for any $\beta\in\Delta^+_i$.  We therefore obtain the following from \cite[Theorem 5.1]{SKT-LG} and Theorem \ref{Bflat}.  

\begin{proposition}
The pluriclosed flow solution $g(t)$ starting at a left-invariant pluriclosed metric $g(0)=g_\ag+g_\qg$ as in (a)-(b) above, converges to the Bismut-flat metric $g_\infty$ given by $g_\infty|_{[\ggo,\ggo]}=g_b$ and $g_\infty|_{\ag}=g_a$, where $g_b:=-\left(z_1\kil_{\ggo_1}+\dots+z_s\kil_{\ggo_s}\right)$.  
\end{proposition}

In particular, Bismut-flat (or equivalently, almost-bi-invariant metrics) are globally stable for the pluriclosed flow among left-invariant pluriclosed metrics on $M=G$.  

In what follows, we prove the global stability of pluriclosed flow on the complex manifold $(M,J)$, where $M=G$ is any compact and connected Lie group, not necessarily semisimple, and $J$ is any left-invariant complex structure.  Remarkably, the flowing pluriclosed metrics are not assumed to be left-invariant.  The case when $G$ is semisimple was solved in \cite[Theorem 4.3]{Brb}.  In much the same way as in \cite{Brb}, our proof is based on the description of the $(1,1)$-degree Aeppli cohomology $H^{1,1}_A(M)$ given in \S\ref{A-sec}, the characterization of left-invariant Bismut-flat metrics given in Theorem \ref{Bflat} and the following strong general result.   

\begin{theorem}\label{tanque}\cite[Theorem 1.2]{GrcJrdStr}
Let $(M,J, \omega_{bf})$ be a compact Bismut-flat manifold.  Given $\omega_0$ a pluriclosed metric on $(M,J)$ so that $[\partial \omega_0] =[\partial \omega_{bf}] \in H^{2,1}_{\overline{\partial}}(M)$, the solution $\omega(t)$ to pluriclosed flow with initial data $\omega_0$ exists on $[0,\infty)$ and converges to a Bismut-flat metric $\omega_\infty$.
\end{theorem}

Let $S$ be the subspace of all symmetric bilinear forms $h:\ag\times\ag\rightarrow\RR$ such that $h$ is compatible with $J_\ag$ and $h|_{\overline{\ag}\times\overline{\ag}}$ coincides with the restriction of a biinvariant symmetric bilinear form on $\overline{G}$.  Each $h\in S$ defines a $(1,1)$-form $\sigma(h)=\sigma(h)_\ag+\sigma(h)_\qg$ given by 
$$
\sigma(h)_\ag(A,B):=\im h(A,B), \quad\forall A\in\ag^{1,0}, \; B\in\ag^{0,1},  \qquad \sigma(h)_\qg(E_\alpha,E_{-\alpha})=-\im z_i, \quad\forall\alpha\in\Delta_i^+,
$$
where $h|_{\overline{\ag}} =-(z_{1}\kil_{\ggo_{1}}+\dots+z_s\kil_{\ggo_s})|_{\overline{\ag}}$ for some $z_{1},\dots,z_s\in\RR$ (i.e., we set $x_\alpha=z_i$ for all $\alpha\in\Delta_i^+$ and $i=1,\dots,s$).  It follows from Lemma \ref{h11A-lem}, Theorem \ref{h11A} and their proofs that the map $h\mapsto [\sigma(h)]$ determines an isomorphism between $S$ and $H^{1,1}_A(M)$.  Moreover, if $h$ is positive definite, then $\sigma(h)$ is a Bismut-flat metric on $M=G$.  Indeed, $g_b:=-(z_{1}\kil_{\ggo_{1}}+\dots+z_s\kil_{\ggo_s})$ is a bi-invariant metric on $\overline{G}$ (i.e., $z_1,\dots,z_s>0$) and the metric $g=g_\ag+g_\qg$, where $g_\ag:=h$ and $g_\qg:=g_b|_\qg$ is almost-bi-invariant and hence it is Bismut-flat by Theorem \ref{Bflat}.  Note that $\omega=g(J\cdot,\cdot)=\sigma(h)$.   

The following result therefore follows.  

\begin{proposition}\label{Agen}
Let $M=G$ be a compact Lie group endowed with a left-invariant complex structure.  Then any positive definite class in $H^{1,1}_A(M)$ contains a unique Bismut-flat left-invariant metric.  
\end{proposition}

The global stability of Bismut-flat metrics under the pluriclosed flow on Bismut-flat manifolds is therefore guaranteed by the above proposition and Theorem \ref{tanque}.   

\begin{theorem}\label{stab}
Let $M=G$ be a compact Lie group endowed with a left-invariant complex structure.  Then for any pluriclosed metric $\omega_0$ (not necessarily left-invariant), the solution $\omega(t)$ to pluriclosed flow with initial data $\omega_0$ exists on $[0,\infty)$ and converges to a Bismut-flat metric $\omega_\infty$.
\end{theorem}

\begin{proof}
Since $[\omega_0]\in H^{1,1}_A(M)$ is positive definite, by Proposition \ref{Agen}, there exists a Bismut-flat left-invariant metric $\omega$ such that $[\omega_0]=[\omega]\in H^{1,1}_A(M)$, which implies that $[\partial\widetilde{\omega}] =[\partial \omega] \in H^{2,1}_{\overline{\partial}}$.  It follows from Theorem \ref{tanque} that the solution $\omega(t)$ to pluriclosed flow with $\omega(0)=\widetilde{\omega}$ exists on $[0,\infty)$ and $\omega(t)\to\omega_\infty$, as $t\to\infty$, where $\omega_\infty$ is a Bismut-flat metric.
\end{proof}

\subsection{A formula for $dd^c\omega$}
It is well known that $dd^c\omega=2\im\partial\overline{\partial}\omega$, which follows from the facts that $d=\partial+\overline{\partial}$ and $d^c=-\im(\partial-\overline{\partial})$.  For completeness, we now compute $dd^c\omega$ on any Hermitian C-space.  It is easy to check that the formula obtained below is in accordance with the formulas for $\partial\overline{\partial}\sigma$ for any $(1,1)$-form $\sigma$ computed in the proof of Lemma \ref{h11A-lem} (see \eqref{eqA1}-\eqref{eqA4}).    

\begin{lemma}\label{ddco}
For any $G$-invariant Hermitian structure $(J,g)$ as in \eqref{Jg} on a C-space $M=G/K$, the only possibly nonzero components of the $4$-form $dd^c\omega$ are given by 
\begin{align*}
dd^c\omega(E_\alpha,E_{-\alpha},E_\beta,E_{-\beta})
=& -2g_\ag(A_\alpha,A_\beta) 
- 2N_{\alpha,\beta}^2(x_{\alpha+\beta}-x_{\alpha}-x_{\beta}) \\
& - 2\epsilon_{\alpha-\beta}N_{\alpha,-\beta}^2(\epsilon_{\alpha-\beta}x_{\alpha-\beta}-x_{\alpha}+x_\beta), \qquad \forall \alpha,\beta\in\Delta_\qg^+, \alpha\ne\beta,
\end{align*}
and if $\alpha+\beta+\gamma+\delta=0$ and all pairs add nonzero, then for all $\alpha,\beta\in\Delta_\qg^+$, $\gamma,\delta\in\Delta_\qg^-$,   
\begin{align*}
dd^c\omega(E_\alpha,E_\beta,E_\gamma,E_\delta) 
=& N_{\alpha,\beta} N_{\gamma,\delta} (x_\alpha+x_\beta+x_\gamma+x_\delta-2x_{\alpha+\beta}) \\ 
&- \epsilon_{\alpha+\gamma}N_{\alpha,\gamma} N_{\beta,\delta}(-x_\alpha+x_\beta+x_\gamma-x_\delta+2\epsilon_{\alpha+\gamma}x_{\alpha+\gamma}) \\
&+ \epsilon_{\alpha+\delta}N_{\alpha,\delta}  N_{\beta,\gamma} (-x_\alpha+x_\beta-x_\gamma+x_\delta+2\epsilon_{\alpha+\delta}x_{\alpha+\delta}).
\end{align*}
\end{lemma}

\begin{remark}
These formulas are also in accordance with those given in \cite{AlkDvd} for full flag manifolds and in \cite{SKT-LG} for Lie groups.  
\end{remark}

\begin{remark}\label{LCSKT-rem}
As an application, we obtain that if $(J,g)$ is {\it twisted SKT}, i.e., $dd^c\omega=\theta\wedge d^c\omega$ for some closed $1$-form $\theta$ (see \cite{BfrFin}), then $(J,g)$ is necessarily SKT.  Indeed, $\theta=g(\cdot,Z)$ for some $Z\in\zg(\ggo)$ by \S\ref{dR1} and so $\theta(E_\alpha)=0$ for any $\alpha\in\Delta_\qg$.  
\end{remark}

\begin{proof}
We use the formula for $d^c\omega$ given in Lemma \ref{dco}.  If $\alpha+\beta+\gamma=0$, then for any $A\in\tg^c$, 
\begin{align*}
& dd^c\omega(A,E_\alpha,E_\beta,E_\gamma) \\ 
=& -\alpha(A)d^c\omega(E_\alpha,E_\beta,E_\gamma) 
+\beta(A)d^c\omega(E_\beta,E_\alpha,E_\gamma)
-\gamma(A)d^c\omega(E_\gamma,E_\alpha,E_\beta)\\ 
&-N_{\alpha,\beta}d^c\omega(E_{\alpha+\beta},A,E_\gamma) 
+ N_{\alpha,\gamma}d^c\omega(E_{\alpha+\gamma},A,E_\beta) 
- N_{\beta,\gamma}d^c\omega(E_{\beta+\gamma},A,E_\alpha) \\ 
=& -(\alpha+\beta+\gamma)(A)d^c\omega(E_\alpha,E_\beta,E_\gamma) \\  
&-N_{\alpha,\beta}d^c\omega(A,E_\gamma,E_{\alpha+\beta}) 
+ N_{\alpha,\gamma}d^c\omega(A,E_\beta,E_{\alpha+\gamma}) 
- N_{\beta,\gamma}d^c\omega(A,E_\alpha,E_{\beta+\gamma}) \\ 
=&-N_{\alpha,\beta}g_\ag(A,E_\gamma) 
+ N_{\alpha,\gamma}g_\ag(A,E_\beta) 
- N_{\beta,\gamma}g_\ag(A,E_\alpha) \\
=&-N_{\alpha,\beta}g_\ag(A,A_\alpha+A_\beta+A_\gamma) 
=-N_{\alpha,\beta}g_\ag(A,A_{\alpha+\beta+\gamma})=0. 
\end{align*}
For $A,B\in\tg^c$ and $\alpha+\beta=0$, 
\begin{align*}
dd^c\omega(A,B,E_\alpha,E_\beta)  
=& \alpha(A)d^c\omega(E_\alpha,B,E_\beta) 
-\beta(A)d^c\omega(E_\beta,B,E_\alpha)
-\alpha(B)d^c\omega(E_\alpha,A,E_\beta) \\ 
&+\beta(B)d^c\omega(E_\beta,A,E_\alpha)
-N_{\alpha,\beta}d^c\omega(E_{\alpha+\beta},A,B) \\ 
=&-(\alpha+\beta)(A)d^c\omega(B,E_\alpha,E_\beta) 
+(\alpha+\beta)(A)d^c\omega(A,E_\alpha,E_\beta) =0.
\end{align*}
If $\alpha+\beta+\gamma+\delta=0$ and all pairs add nonzero, then
\begin{align*}
& dd^c\omega(E_\alpha,E_\beta,E_\gamma,E_\delta) \\ 
=& -N_{\alpha,\beta}d^c\omega(E_{\alpha+\beta},E_\gamma,E_\delta) 
- N_{\gamma,\delta}d^c\omega(E_{\gamma+\delta},E_\alpha,E_\beta) \\ 
&+ N_{\alpha,\gamma}d^c\omega(E_{\alpha+\gamma},E_\beta,E_\delta) 
+ N_{\beta,\delta}d^c\omega(E_{\beta+\delta},E_\alpha,E_\gamma) \\
&- N_{\alpha,\delta}d^c\omega(E_{\alpha+\delta},E_\beta,E_\gamma) 
- N_{\beta,\gamma}d^c\omega(E_{\beta+\gamma},E_\alpha,E_\delta) \\
=& -N_{\alpha,\beta}\im \epsilon_{\alpha+\beta}\epsilon_\gamma\epsilon_\delta N_{\gamma,\delta} (y_{\alpha+\beta}+y_\gamma+y_\delta) 
- N_{\gamma,\delta}\im \epsilon_{\gamma+\delta}\epsilon_\alpha\epsilon_\beta N_{\alpha,\beta} (y_{\gamma+\delta}+y_\alpha+y_\beta) \\ 
&+ N_{\alpha,\gamma}\im \epsilon_{\alpha+\gamma}\epsilon_\beta\epsilon_\delta N_{\beta,\delta}(y_{\alpha+\gamma}+y_\beta+y_\delta)  
+ N_{\beta,\delta} \im\epsilon_{\beta+\delta}\epsilon_\alpha\epsilon_\gamma N_{\alpha,\gamma} (y_{\beta+\delta}+y_\alpha+y_\gamma) \\
&- N_{\alpha,\delta}\im \epsilon_{\alpha+\delta}\epsilon_\beta\epsilon_\gamma N_{\beta,\gamma} (y_{\alpha+\delta}+y_\beta+y_\gamma) 
- N_{\beta,\gamma}\im \epsilon_{\beta+\gamma}\epsilon_\alpha\epsilon_\delta N_{\alpha,\delta}(y_{\beta+\gamma}+y_\alpha+y_\delta). 
\end{align*}
In particular, for $\alpha,\beta,\gamma\in\Delta^+$, $\delta\in\Delta^-$, we obtain
\begin{align*}
& dd^c\omega(E_\alpha,E_\beta,E_\gamma,E_\delta) \\ 
=& N_{\alpha,\beta} N_{\gamma,\delta} (x_{\alpha+\beta}+x_\gamma-x_\delta) 
+N_{\gamma,\delta} N_{\alpha,\beta} (-x_{\gamma+\delta}+x_\alpha+x_\beta) \\ 
& - N_{\alpha,\gamma}  N_{\beta,\delta}(x_{\alpha+\gamma}+x_\beta-x_\delta)  
-N_{\beta,\delta}  N_{\alpha,\gamma} (-x_{\beta+\delta}+x_\alpha+x_\gamma) \\
&+N_{\alpha,\delta} N_{\beta,\gamma} (-x_{\alpha+\delta}+x_\beta+x_\gamma) 
+ N_{\beta,\gamma} N_{\alpha,\delta}(x_{\beta+\gamma}+x_\alpha-x_\delta). \\ 
=& (N_{\alpha,\beta} N_{\gamma,\delta} - N_{\alpha,\gamma}N_{\beta,\delta}+N_{\alpha,\delta} N_{\beta,\gamma})(x_\alpha+x_\beta+x_\gamma-x_\delta) =0.  
\end{align*}
Note that the above computation is not necessary if we use that $dd^c\omega=2\im\partial\overline{\partial}\omega$ is a $(2,2)$-form.  On the other hand, for $\alpha,\beta\in\Delta^+$, $\gamma,\delta\in\Delta^-$, we have that  
\begin{align*}
& dd^c\omega(E_\alpha,E_\beta,E_\gamma,E_\delta) \\ 
=& -N_{\alpha,\beta} N_{\gamma,\delta} (x_{\alpha+\beta}-x_\gamma-x_\delta) 
+ N_{\gamma,\delta} N_{\alpha,\beta} (-x_{\gamma+\delta}+x_\alpha+x_\beta) \\ 
&- \epsilon_{\alpha+\gamma}N_{\alpha,\gamma} N_{\beta,\delta}(\epsilon_{\alpha+\gamma}x_{\alpha+\gamma}+x_\beta-x_\delta)  
- \epsilon_{\beta+\delta}N_{\beta,\delta}  N_{\alpha,\gamma} (\epsilon_{\beta+\delta}x_{\beta+\delta}+x_\alpha-x_\gamma) \\
&+ \epsilon_{\alpha+\delta}N_{\alpha,\delta}  N_{\beta,\gamma} (\epsilon_{\alpha+\delta}x_{\alpha+\delta}+x_\beta-x_\gamma) 
+ \epsilon_{\beta+\gamma}N_{\beta,\gamma}  N_{\alpha,\delta}(\epsilon_{\beta+\gamma}x_{\beta+\gamma}+x_\alpha-x_\delta),   \\ 
=& N_{\alpha,\beta} N_{\gamma,\delta} (x_\alpha+x_\beta+x_\gamma+x_\delta-2x_{\alpha+\beta}) \\ 
&- \epsilon_{\alpha+\gamma}N_{\alpha,\gamma} N_{\beta,\delta}(-x_\alpha+x_\beta+x_\gamma-x_\delta+2\epsilon_{\alpha+\gamma}x_{\alpha+\gamma}) \\
&+ \epsilon_{\alpha+\delta}N_{\alpha,\delta}  N_{\beta,\gamma} (-x_\alpha+x_\beta-x_\gamma+x_\delta+2\epsilon_{\alpha+\delta}x_{\alpha+\delta}).  
\end{align*}
Finally, if $\alpha,\beta\in\Delta^+_\qg$, $\alpha\pm\beta\ne 0$ then
\begin{align*}
& dd^c\omega(E_\alpha,E_{-\alpha},E_\beta,E_{-\beta}) \\ 
=& -d^c\omega(A_\alpha,E_\beta,E_{-\beta}) 
+ N_{\alpha,\beta}d^c\omega(E_{\alpha+\beta},E_{-\alpha},E_{-\beta}) 
- N_{\alpha,-\beta}d^c\omega(E_{\alpha-\beta},E_{-\alpha},E_\beta) \\
& - N_{-\alpha,\beta}d^c\omega(E_{-\alpha+\beta},E_{\alpha},E_{-\beta})
+ N_{-\alpha,-\beta}d^c\omega(E_{-\alpha-\beta},E_\alpha,E_\beta)
-d^c\omega(A_\beta,E_\alpha,E_{-\alpha}) \\ 
=& -g_\ag(A_\alpha,A_\beta) 
+ N_{\alpha,\beta}\im\epsilon_{\alpha+\beta}\epsilon_{-\alpha}\epsilon_{-\beta} N_{\alpha+\beta,-\alpha}(y_{\alpha+\beta}+y_{-\alpha}+y_{-\beta}) \\
& - N_{\alpha,-\beta}\im\epsilon_{\alpha-\beta}\epsilon_{-\alpha}\epsilon_\beta N_{\alpha-\beta,-\alpha}(y_{\alpha-\beta}+y_{-\alpha}+y_\beta)\\ 
&- N_{-\alpha,\beta}\im\epsilon_{-\alpha+\beta}\epsilon_{\alpha}\epsilon_{-\beta} N_{-\alpha+\beta,\alpha}(y_{-\alpha+\beta}+y_{\alpha}+y_{-\beta})\\
&+ N_{-\alpha,-\beta}\im\epsilon_{-\alpha-\beta}\epsilon_\alpha\epsilon_\beta N_{-\alpha-\beta,\alpha}(y_{-\alpha-\beta}+y_\alpha+y_\beta)
-g_\ag(A_\beta,A_\alpha), \\
=& -2g_\ag(A_\alpha,A_\beta) 
+ 2N_{\alpha,\beta}\im\epsilon_{\alpha+\beta}\epsilon_{-\alpha}\epsilon_{-\beta} N_{\alpha+\beta,-\alpha}(y_{\alpha+\beta}+y_{-\alpha}+y_{-\beta}) \\
& - 2N_{\alpha,-\beta}\im\epsilon_{\alpha-\beta}\epsilon_{-\alpha}\epsilon_\beta N_{\alpha-\beta,-\alpha}(y_{\alpha-\beta}+y_{-\alpha}+y_\beta)\\
=& -2g_\ag(A_\alpha,A_\beta) 
- 2N_{\alpha,\beta}^2\im\epsilon_{\alpha+\beta}\epsilon_{-\alpha}\epsilon_{-\beta} (y_{\alpha+\beta}+y_{-\alpha}+y_{-\beta}) \\
& + 2N_{\alpha,-\beta}^2\im\epsilon_{\alpha-\beta}\epsilon_{-\alpha}\epsilon_\beta (y_{\alpha-\beta}+y_{-\alpha}+y_\beta),
\end{align*}
concluding the proof. 
\end{proof}

% 14/8/26

\section{Parallel Bismut torsion (BTP)}\label{BTP-sec}

\begin{definition}\label{BTP-def} 
A Hermitian structure $(J,g)$ is called {\it Bismut-torsion-parallel} (BTP for short) if it has parallel Bismut torsion, i.e., $\nabla^b T^b=0$.  
\end{definition}

In the context of C-spaces, formulas for $\nabla^b$ and $T^b$ were given in \S\ref{BiC-sec} (see also \S\ref{IC-sec}).  It follows from Proposition \ref{Btor} and Corollary \ref{Bcon} that for a $G$-invariant Hermitian structure $(J,g)$ as in \eqref{Jg} on a C-space $M=G/K$, the only possibly nonzero components of the Bismut connection and its torsion are given by 
\begin{equation}\label{BcT}
\left\{\begin{array}{l}
\Lambda^b(A,E_\alpha)=\left(\alpha(A)+\tfrac{g_\ag(A,A_\alpha)}{x_\alpha}\right) E_\alpha, \\ 
\Lambda^b(E_\alpha,E_\beta)= c_{\alpha,\beta}E_{\alpha+\beta},  
\end{array}\right. \qquad
\left\{\begin{array}{l}
T^b(A,E_\alpha)=\tfrac{g_\ag(A,A_\alpha)}{x_\alpha} E_\alpha, \\ 
T^b(E_\alpha,E_{-\alpha})=-A_\alpha, \\ 
T^b(E_\alpha,E_\beta)= t_{\alpha,\beta}E_{\alpha+\beta},
\end{array}\right.
\end{equation}
for all $\alpha,\beta\in\Delta_\qg$, $A\in\ag^c$, where $c_{\alpha,\beta}$ and $t_{\alpha,\beta}$ are some real numbers.  In particular, $(J,g)$ is BTP if and only if  
\begin{align}
(\Lambda^b(A)\cdot T^b)(E_{\alpha},E_{\pm\beta})=0, \qquad (\Lambda^b(E_{\pm\alpha})\cdot T^b)(A,E_{\pm\beta})=0, \label{BTP0} \\ (\Lambda^b(E_{\pm\gamma})\cdot T^b)(E_{\alpha},E_{\pm\beta})=0, \qquad 
\forall\alpha,\beta,\gamma\in\Delta_\qg^+, \quad A\in\ag^c.  \notag
\end{align}
Recall from \eqref{a1g} that $A_\alpha=(H_\alpha)_{\ag^c}$ for all $\alpha\in\Delta_\qg$. 

We first show that there are plenty of BTP metrics on any complex C-space.  

\begin{proposition}\label{BTP-exi}
Let $(M=G/K,J)$ be a complex C-space.  Then any metric of the form $g=g_\ag+g_\qg$, $g_\qg=(x_\alpha)_{\alpha\in\Delta_\qg}$ such that $g_\qg=g_b|_\qg$ for some bi-invariant metric $g_b$ on $[\ggo,\ggo]$ (i.e., $x_\alpha=x_\beta$ for all $\alpha,\beta\in\Delta_{\qg_i}$, $i=1,\dots,s$, see \eqref{decflags}) is BTP.   
\end{proposition}

\begin{proof}
Using \eqref{BcT}, we obtain that for all $A\in\ag^c$ and $\alpha,\beta,\alpha+\beta\in\Delta_\qg$, 
\begin{align}
(\Lambda^b(A)\cdot T^b)(E_\alpha,E_\beta) 
=& \Lambda^b(A)T^b(E_\alpha,E_\beta) -T^b(\Lambda^b(A)E_\alpha,E_\beta) -T^b(E_\alpha,\Lambda^b(A)E_\beta) \notag\\ 
=& 
\left(\tfrac{g_\ag(A,A_{\alpha+\beta})}{x_{\alpha+\beta}} - \tfrac{g_\ag(A,A_\alpha)}{x_\alpha} -\tfrac{g_\ag(A,A_\beta)}{x_\beta} \right) T^b(E_\alpha,E_\beta). \label{BTP1} 
\end{align}
Since $\Lambda(E_\alpha)=0$ for any $\alpha\in\Delta_\qg$ by \eqref{Lb} and \eqref{BTP1} vanishes, we obtain that $g$ is BTP (see \eqref{BTP0}).   
\end{proof}

The following technical necessary conditions for BTP will be very useful.    

\begin{proposition}\label{BTP} 
If a Hermitian C-space is BTP, then the following conditions hold for any $\alpha,\beta,\alpha+\beta\in\Delta_\qg^+$:
\begin{enumerate}[{\rm (i)}] 
\item $x_{\alpha+\beta}=x_\alpha$ or $\tfrac{1}{x_\beta}A_\beta = \tfrac{1}{x_{\alpha+\beta}}A_{\alpha+\beta}$.  

\item $x_{\alpha+\beta}=x_\alpha+x_\beta$ or 
$
\tfrac{1}{x_\alpha}A_\alpha+ \tfrac{1}{x_\beta}A_\beta=\tfrac{1}{x_{\alpha+\beta}} A_{\alpha+\beta}.
$   
\item If $\alpha-\beta\notin\Delta_\qg^-$, then either $x_{\alpha+\beta}=x_\alpha+x_\beta$ or $x_{\alpha+\beta}=x_\alpha$. 

\item If $\alpha-\beta\notin\Delta_\qg^+$, then either $x_{\alpha+\beta}=x_\alpha+x_\beta$ or $x_{\alpha+\beta}=x_\beta$. 

\item If $\gamma, \alpha+\beta+\gamma \in \Delta_\qg^+$ and $\alpha+\gamma, \beta+\gamma \notin \Delta_\qg$, then $x_{\alpha+\beta} =x_{\alpha} + x_\beta$ or $x_{\alpha+\beta+\gamma}=x_\gamma$.  

\item If $\gamma, \beta+\gamma, \alpha+\beta+\gamma \in \Delta_\qg^+$ and $\alpha+\gamma \notin \Delta_\qg$, then 
\begin{align*}
\tfrac{1}{x_{\beta+\gamma}} & N_{\beta,\gamma} N_{\alpha, \beta+\gamma} (x_{\beta+\gamma} -x_\beta - x_\gamma)(x_{\alpha+\beta+\gamma} - x_{\alpha})\\
=& \tfrac{1}{x_{\alpha+\beta}} N_{\alpha, \beta} N_{\alpha+\beta, \gamma}(x_{\alpha+\beta}- x_\alpha)(x_{\alpha+\beta+\gamma} - x_{\alpha+\beta} - x_\gamma).  
\end{align*}
\end{enumerate}
\end{proposition}

\begin{remark}
The above conditions do not depend on $g_\ag$, only $g_\qg=(x_\alpha)_{\alpha\in\Delta_\qg^+}$ is involved.   
\end{remark}

\begin{remark}
It follows from parts (iii) and (iv) that if $\alpha-\beta\notin\Delta_\qg$ (e.g., $\alpha,\beta\in\Pi_\qg$), then either $x_{\alpha+\beta}=x_\alpha+x_\beta$ or $x_{\alpha+\beta}=x_\alpha=x_\beta$.  
\end{remark}

\begin{proof}
Using \eqref{BcT}, we obtain that for all $A\in\ag^c$,
\begin{align}
(\Lambda^b(E_\alpha)\cdot T^b)(A,E_\beta)  
=& \Lambda^b(E_\alpha)T^b(A,E_\beta) -T^b(\Lambda^b(E_\alpha)A,E_\beta) -T^b(A,\Lambda^b(E_\alpha)E_\beta) \notag\\ 
=& c_{\alpha,\beta} \left(\tfrac{g_\ag(A,A_\beta)}{x_\beta} - \tfrac{g_\ag(A,A_{\alpha+\beta})}{x_{\alpha+\beta}}\right) E_{{\alpha+\beta}}, \qquad\forall \alpha,\beta,\alpha+\beta\in\Delta_\qg, A\in\ag^c, \label{BTP2}
\end{align}
which implies part (i).  On the other hand, part (ii) follows from \eqref{BTP1}.  

We also have that
\begin{align*}
0=& (\Lambda^b(E_{-\alpha})\cdot T^b)(E_\alpha,E_\beta) \\ 
=& \Lambda^b(E_{-\alpha})T^b(E_\alpha,E_\beta) -T^b(\Lambda^b(E_{-\alpha})E_\alpha,E_\beta) -T^b(E_\alpha,\Lambda^b(E_{-\alpha})E_\beta)\\ 
=&\Lambda^b(E_{-\alpha})\tfrac{N_{\alpha,\beta}(x_{\alpha+\beta}-x_\alpha-x_\beta)}{x_{\alpha+\beta}} E_{{\alpha+\beta}} -T^b(E_\alpha,\Lambda^b(E_{-\alpha})E_\beta)\\  
=&\tfrac{N_{\alpha,\beta}(x_{\alpha+\beta}-x_\alpha-x_\beta)}{x_{\alpha+\beta}} 
\tfrac{N_{-\alpha,\alpha+\beta}(x_{\alpha+\beta}-x_\alpha)}{x_{\beta}} E_{{\beta}} -T^b(E_\alpha,\Lambda^b(E_{-\alpha})E_\beta).  
\end{align*}
If $\alpha-\beta\notin\Delta_\qg^-$, then $\Lambda^b(E_{-\alpha})E_\beta=0$ and so part (iii) follows, and if 
$\alpha-\beta\notin\Delta_\qg^+$, then $\Lambda^b(E_{-\beta})E_\alpha=0$ and part (iv) also follows by using that $\Lambda^b(E_{-\beta})T^b(E_\alpha,E_\beta)=0$.  

Finally, we prove parts (v) and (vi): for all $\alpha, \beta, \gamma\in \Delta_\qg^+$, 
\begin{align*}
0 =& (\Lambda^b(E_{\gamma})T^b)(E_{\alpha},E_{\beta}) \\ 
=& \Lambda^b(E_{\gamma})T^b(E_{\alpha}, E_{\beta}) - T^b(\Lambda^b(E_{\gamma})E_{\alpha}, E_{\beta}) -T^b(E_{\alpha}, \Lambda^b(E_{\gamma})E_{\beta}) \\
=& \tfrac{N_{\alpha, \beta}(x_{\alpha+\beta}-x_\alpha-x_\beta)}{x_{\alpha+\beta}}\Lambda^b(E_\gamma)E_{\alpha+\beta} 
- \tfrac{N_{\gamma, \alpha}(x_{\gamma+\alpha} - x_\gamma)}{x_{\gamma+\alpha}}T^b(E_{\gamma+\alpha}, E_{\beta}) \\ 
& - \tfrac{N_{\gamma,\beta}(x_{\gamma+\beta} - x_\gamma)}{x_{\gamma+\beta}}T^b(E_{\alpha}, E_{\gamma+\beta}) \\
=&\tfrac{N_{\alpha,\beta} N_{\gamma, \alpha+\beta} (x_{\alpha+\beta} -x_\alpha - x_\beta)(x_{\gamma+\alpha+\beta} - x_{\gamma})}{x_{\alpha+\beta} x_{\gamma+\alpha+\beta}} E_{\gamma + \alpha + \beta} \\
&- \tfrac{N_{\gamma, \alpha} N_{\gamma+\alpha, \beta}(x_{\gamma+\alpha}- x_\gamma)(x_{\gamma+\alpha+\beta} - x_{\gamma+\alpha} - x_\beta)}{x_{\gamma+\alpha}x_{\gamma+\alpha+\beta}}  E_{\gamma+\alpha+\beta} \\ 
&- \tfrac{N_{\gamma, \beta} N_{\alpha, \gamma+\beta}(x_{\gamma+\beta}- x_\gamma)(x_{\gamma+\alpha+\beta} - x_{\alpha} - x_{\gamma+\beta})}{x_{\gamma+\beta}x_{\gamma+\alpha+\beta}}  E_{\gamma+\alpha+\beta},
\end{align*}
which concludes the proof.  
\end{proof}

In the case of a flag manifold $M=G/H$, the BTP condition is almost completely understood.  

\begin{theorem}\cite{Fr}\label{MF}
Any BTP metric on a flag manifold $M=G/H$ with $G$ simple and $G\ne \Eg_6, \Eg_7, \Eg_8$, is either the Killing metric up to scaling (i.e., $x_\alpha=x_\beta$ for all $\alpha,\beta\in\Delta_{\qg}$) or it is K\"ahler (i.e., $x_{\alpha+\beta}=x_\alpha+x_\beta$ for all $\alpha,\beta,\alpha+\beta\in\Delta_\qg^+$). 
\end{theorem}

This was first shown in \cite{PdsZhn} for the flags with two isotropy summands and for $M=\SU(n+1)/T^n$.  It is also proved in \cite{PdsZhn} that the Killing metric is the only left and $\Ad(T)$-invariant BTP metric up to scaling on any compact simple Lie group.  

The rigidity behavior provided by Theorem \ref{MF} is a motivation to consider, on any complex C-space,  metrics which descends to a K\"ahler metric on the base of the Tits fibration.   

\begin{proposition}\label{BTP-K}
Let $(M=G/K,J)$ be a complex C-space endowed with a Hermitian metric of the form $g=g_\ag+g_\qg$ such that $g_\qg=(x_\alpha)_{\alpha\in\Delta_\qg}$ is K\"ahler on the flag $F=G_f/H_f$ (i.e., $x_{\alpha+\beta}=x_\alpha+x_\beta$ for all $\alpha,\beta,\alpha+\beta\in\Delta_\qg^+$).  Then the following conditions are equivalent: 
\begin{enumerate}[{\rm (i)}] 
\item $g$ is BTP. 

\item $\tfrac{1}{x_\alpha}A_\alpha= \tfrac{1}{x_\beta}A_\beta$ for all $\alpha,\beta,\alpha+\beta\in\Delta_\qg^+$.  

\item For each $j=1,\dots,s$, there exist $W_j\in\ag_f$ and $\lambda_j>0$ such that 
$$
\la\im A_\alpha:\alpha\in\Pi_{\qg_j}\ra_\RR =\RR W_j \qquad\mbox{and}\qquad x_\alpha=\lambda_j\alpha(\im W_j), \quad\forall \alpha\in\Delta_{\qg_j}^+,     
$$
i.e., $g|_{\qg_j}=\lambda_j g_{\im W_j}|_{\qg_j}$ (see \S\ref{K}).   
\end{enumerate}
\end{proposition}

\begin{remark}
$\ag_f=\ag_1+\dots+\ag_s$, where $\ag_j:=\la\im A_\alpha:\alpha\in\Pi_{\qg_j}\ra_\RR$.  
\end{remark}

\begin{proof}
If $g$ is BTP, then condition (ii) follows from Proposition \ref{BTP}, (i).  Conversely, if part (ii) holds, since $T^b(E_\alpha,E_\beta)=0$ for any $\alpha,\beta\in\Delta_\qg$, then we obtain from \eqref{BTP0} and \eqref{BTP2} that $g$ is BTP.  The equivalence between parts (ii) and (iii) is straightforward.  
\end{proof}

The proof of Theorem \ref{MF} given in \cite{Fr} only uses parts (iii), (iv) and (vi) of Proposition \ref{BTP}, to first show that for any $\alpha,\beta\in\Delta_\qg^+$ such that $\alpha+\beta\in\Delta_\qg^+$, either $x_{\alpha+\beta}=x_\alpha=x_\beta$ or $x_{\alpha+\beta}=x_\alpha+x_\beta$ (see \cite[Lemma 4.7]{Fr}).  Secondly, a case by case analysis gives that the above local property globally holds, in the sense that one of the alternatives indeed holds for all $\alpha,\beta\in\Delta_\qg^+$ such that $\alpha+\beta\in\Delta_\qg^+$. 

We therefore obtain the following from \cite{Fr}.  

\begin{theorem}\label{BTP-thm}
Let $M=G/K$ be a complex C-space fibering over a flag $F=G_1/H_1\times\dots\times G_s/H_s$ and assume that $G_i\ne \Eg_6, \Eg_7, \Eg_8$ for all $i=1,\dots,s$.  A metric $g=g_\ag+g_\qg$ as in \eqref{g} is BTP if and only if $g$ is either as in Proposition \ref{BTP-exi} or as in Proposition \ref{BTP-K}. 
\end{theorem}

\subsection{Bismut Ambrose-Singer (BAS)}\label{BAS-sec}
As a stronger condition, it is natural to consider the following.  

\begin{definition}\label{BAS-def} 
A Hermitian structure $(J,g)$ is called {\it Bismut Ambrose-Singer} (BAS for short) if $\nabla^b T^b=0$ and $\nabla^b R^b=0$.  
\end{definition}

\begin{proposition}\label{BAS}
Any metric as in Proposition \ref{BTP-exi} is BAS. 
\end{proposition}

\begin{proof}
According to Corollary \ref{Bcon} and Proposition \ref{Bcurv}, the only possibly nonzero components of the Bismut connection and its curvature are
$$
\Lambda^b(A)E_\alpha=\left(\alpha(A)+g_\ag(A,A_\alpha)\right) E_\alpha, 
$$
$$
R^b(E_\alpha,E_{-\alpha})E_\beta=-(\beta(H_\alpha)+g_\ag(A_\alpha,A_\beta))E_\beta,
$$
which implies that 
$$
\Lambda(A)\cdot T^b = 0, \qquad 
(\Lambda(A)\cdot R^b)(E_\alpha,E_{-\alpha})E_\beta = 0,
$$
concluding the proof.  
\end{proof}

The results in \cite{PdsZhn} led the authors to propose the following.  

\begin{conjecture}\label{conjBTP}\cite[Conjecture 1.6]{PdsZhn}
On any C-space $M=G/K$ with $G$ semisimple, BTP implies BAS.
\end{conjecture}

\begin{example}\label{b21-BTP}
Let $M=G/K$ be the C-space given in Example \ref{b21-herm}.  We consider $G$-invariant Hermitian metrics $g=g_\ag+g_\qg$ such that the base of the fibration $(F,g_\qg)$ is K\"ahler, i.e.,  
$$
g_\qg=(x,2x,\dots,r_1x)+(y,2y,\dots,r_2y), \qquad x,y>0.
$$
According to Proposition \ref{BTP-K}, $(J,g)$ is BTP for all $a,b,c,x,y$.  Indeed, 
$$
\tfrac{1}{x_\alpha}A_\alpha = \tfrac{1}{x}Z_{\gamma_1}, \quad \forall\alpha\in\Delta_{\qg_1}^+, \quad 
\tfrac{1}{x_\alpha}A_\alpha = \tfrac{1}{y}Z_{\gamma_2}, \quad \forall\alpha\in\Delta_{\qg_2}^+.
$$
\end{example}

We now show that the above example is not in general BAS, which provides many counterexamples to Conjecture \ref{conjBTP}. 

\begin{lemma}\label{notBAS}
If $\alpha,\beta,\alpha+\beta\in\Delta_\qg^+$, $\alpha-\beta,2\alpha+\beta\notin\Delta_\qg$, then 
\begin{align*}
(\Lambda^b(E_\beta)\cdot R^b)(E_\alpha,E_{-\alpha})E_{-\alpha}  
=& c\left(\tfrac{1}{x_{\alpha+\beta}}g_\ag(A_\alpha,A_{\beta})-\left(\tfrac{1}{x_\alpha}-\tfrac{1}{x_{\alpha+\beta}}\right)g_\ag(A_\alpha,A_\alpha) 
 \right. \\
&\hspace{.4cm} \left. +\tfrac{N_{\alpha,\beta}^2(x_{\alpha+\beta}-x_\beta)}{x_{\alpha+\beta}}
 - \tfrac{N_{\alpha,\beta}^2(x_{\alpha+\beta}-x_\alpha)^2}{x_{\beta}x_{\alpha+\beta}}  
+\la\alpha,\beta\ra\right) E_{\alpha+\beta},
\end{align*}
where $c:=-\tfrac{N_{\alpha,\beta}(x_{\alpha+\beta}-x_\beta)}{x_{\alpha+\beta}}$.  
\end{lemma}

\begin{remark}\label{notBAS2} 
In particular, if in addition $x_\alpha=x_\beta$, $x_{\alpha+\beta}=2x_\alpha$, $A_\alpha=A_\beta$, then 
$$
(\Lambda^b(E_\beta)\cdot R^b)(E_\alpha,E_{-\alpha})E_{-\alpha} 
= -\unm N_{\alpha,\beta}\la\alpha,\beta\ra E_{\alpha+\beta}.  
$$
\end{remark}

\begin{proof}
Since $\Lambda^b(E_\beta)E_\alpha =cE_{\alpha+\beta}$, it follows from Proposition \ref{Bcurv} that 
\begin{align*}
&(\Lambda^b(E_\beta)\cdot R^b)(E_\alpha,E_{-\alpha})E_\alpha \\
=& -\left(\la\alpha,\alpha\ra+\tfrac{g_\ag(A_\alpha,A_\alpha)}{x_\alpha}\right)\Lambda^b(E_\beta)E_\alpha
- R^b(\Lambda^b(E_\beta)E_\alpha,E_{-\alpha})E_\alpha 
- R^b(E_\alpha,E_{-\alpha})\Lambda^b(E_\beta)E_\alpha \\
=& -\left(\la\alpha,\alpha\ra+\tfrac{g_\ag(A_\alpha,A_\alpha)}{x_\alpha}\right)cE_{\alpha+\beta}
- cR^b(E_{\alpha+\beta},E_{-\alpha})E_\alpha 
- cR^b(E_\alpha,E_{-\alpha})E_{\alpha+\beta} \\
=& -\left(\la\alpha,\alpha\ra+\tfrac{g_\ag(A_\alpha,A_\alpha)}{x_\alpha}\right)cE_{\alpha+\beta}
- c\left(-N_{\alpha+\beta,-\alpha}\Lambda^b(E_\beta)E_\alpha\right) \\
& - c\left(\tfrac{N_{-\alpha,\alpha+\beta}^2(x_{\alpha+\beta}-x_\alpha)^2}{x_{\beta}x_{\alpha+\beta}}  
-\la\alpha,\alpha+\beta\ra-\tfrac{g_\ag(A_\alpha,A_{\alpha+\beta})}{x_{\alpha+\beta}}\right)E_{\alpha+\beta} \\ 
=& -\left(\la\alpha,\alpha\ra+\tfrac{g_\ag(A_\alpha,A_\alpha)}{x_\alpha}\right)cE_{\alpha+\beta}
- c^2N_{\alpha,\beta}E_{\alpha+\beta} \\
& - c\left(\tfrac{N_{\alpha,\beta}^2(x_{\alpha+\beta}-x_\alpha)^2}{x_{\beta}x_{\alpha+\beta}}  
-\la\alpha,\alpha+\beta\ra-\tfrac{g_\ag(A_\alpha,A_{\alpha+\beta})}{x_{\alpha+\beta}}\right) E_{\alpha+\beta}\\
=& -c\left(\la\alpha,\alpha\ra+\tfrac{g_\ag(A_\alpha,A_\alpha)}{x_\alpha}
+ cN_{\alpha,\beta} 
 + \tfrac{N_{\alpha,\beta}^2(x_{\alpha+\beta}-x_\alpha)^2}{x_{\beta}x_{\alpha+\beta}}  
-\la\alpha,\alpha+\beta\ra-\tfrac{g_\ag(A_\alpha,A_{\alpha+\beta})}{x_{\alpha+\beta}}\right) E_{\alpha+\beta}\\
=& -c\left(\left(\tfrac{1}{x_\alpha}-\tfrac{1}{x_{\alpha+\beta}}\right)g_\ag(A_\alpha,A_\alpha) 
-\tfrac{1}{x_{\alpha+\beta}}g_\ag(A_\alpha,A_{\beta}) \right. \\
&\hspace{1cm} \left. -\tfrac{N_{\alpha,\beta}^2(x_{\alpha+\beta}-x_\beta)}{x_{\alpha+\beta}}
 + \tfrac{N_{\alpha,\beta}^2(x_{\alpha+\beta}-x_\alpha)^2}{x_{\beta}x_{\alpha+\beta}}  
-\la\alpha,\beta\ra\right) E_{\alpha+\beta},
\end{align*}
concluding the proof.  
\end{proof}

\begin{example}\label{b21-BTP-2}
We consider a C-space $M=G/K$ as in Example \ref{b21-BTP} such that 
$$
F_1=\Gg_2/\U(2), \qquad \dynkin[scale=2] G{o*}, \qquad 
\Delta_{\qg_1}^+=\{\gamma_1, \gamma_1+\delta, \gamma_1+2\delta, \gamma_1+3\delta, 2\gamma_1+3\delta\}. 
$$
Note that $\alpha:=\gamma_1$ and $\beta:=\gamma_1+3\delta$ satisfy the hypothesis of Lemma \ref{notBAS}.  If $g$ is as in Remark \ref{notBAS2}, then the BTP Hermitian structure $(J,g)$ is never BAS since  
$$
\la\alpha,\beta\ra = \la\gamma_1,\gamma_1+3\delta\ra = \la\gamma_1,\gamma_1\ra +3\la\gamma_1,\delta\ra =-\unm  \la\gamma_1,\gamma_1\ra <0.
$$
The simplest examples of this kind are 
$$
M^{12}=\Gg_2/\SU(2)\times S^1, \qquad M^{14}=\Gg_2/\SU(2)\times \SU(2).
$$
\end{example}

\begin{remark}\label{BP}
It has recently been proved in \cite[Theorem B]{BrbPdc} that a metric on a complex C-space is BAS if and only if it is as in Proposition \ref{BTP-exi}.  This implies that all the metrics in Proposition \ref{BTP-K} provide counterexamples to Conjecture \ref{conjBTP} if $\rank(\ggo_i)\geq 2$ for some $i\in\{ 1,\dots,s\}$.   
\end{remark}

% 28/8/2026

\section{Calabi-Yau with torsion (CYT)}\label{CYT-sec} 

The Bismut Ricci form $\rho^b$ of a Hermitian manifold is defined in \eqref{rf-def1}.   

\begin{definition}\label{BTP-def} 
A Hermitian structure $(J,g)$ is called {\it Calabi-Yau with torsion} (CYT for short) if $\rho^b=0$. 
\end{definition}
  
On a C-space $M=G/K=G_f/K\times T^z$, according to \eqref{rhot2}, the Bismut Ricci form $\rho^b$ of a $G$-invariant Hermitian structure $(J,g)$ as in \eqref{Jg} is given by 
 \begin{equation}\label{rhob}
 \rho^b(X,Y)=  Q\left([X,Y],Z_{J_\qg} - P_\ag(Z_{J_\qg,g_\qg})_\ag\right), \qquad\forall X,Y\in\pg, 
 \end{equation}
 where $Z_{J_\qg,g_\qg}\in\zg(\hg_f)=\zg(\kg)\oplus\ag_f$ is the metric Koszul vector (see \eqref{koszul-g}).  The following characterization therefore follows.  Recall that $g_\ag=Q(P_\ag\cdot,\cdot)$.  

\begin{proposition}\label{CYT}
$(M=G/K,J,g)$ is CYT if and only if $Z_{J_\qg}\in\ag_f$ (i.e., $c_1(M,J)=0$) and $\left(P_\ag(Z_{J_\qg,g_\qg})_\ag\right)_{\ag_f}=Z_{J_\qg}$, or equivalently, 
\begin{equation}\label{CYT3}
g_\ag\left((Z_{J_\qg,g_\qg})_\ag,A\right)=Q(Z_{J_\qg},A), \qquad\forall A\in\ag_f.  
\end{equation}
\end{proposition}

\begin{remark}\label{CYT-rem}
The CYT condition only depends on the restriction of the metric $g$ on $\ag_f\oplus\qg$, its behavior on $\zg(\ggo)\times\zg(\ggo)$ and $\zg(\ggo)\times\ag_f$ is not involved.   
\end{remark}

In particular, if $c_1(M,J)=0$, then the normal metric $g_Q:=Q|_\pg$ is CYT provided it is compatible with $J$.  Indeed, $P_\ag=I$ and $Z_{J_\qg,g_\qg}=Z_{J_\qg}$ in that case, so \eqref{CYT3} holds.  This was proved in \cite{Grn} in the case when $G$ is semisimple.  Remarkably, we now show that if $G$ is not semisimple and $\zg(\kg)=0$, then the CYT condition automatically holds for many other metrics which are not necessarily normal.   

\begin{definition}\label{an-def}
A metric $g=g_\ag+g_\qg$ as in \eqref{g} on a C-space $M=G/K$ is called {\it almost-normal} if there exists a bi-invariant metric $g_b$ on $[G,G]$ such that $g|_{\ag_f\oplus\qg}=g_b|_{\ag_f\oplus\qg}$.  
\end{definition}

Note that an almost-normal metric is normal if and only if $g(\zg(\ggo),\ag_f)=0$.  

\begin{corollary}\label{CYT-cor}
Let $(M=G/K,J)$ be a complex C-space. 
\begin{enumerate}[{\rm (i)}]
\item If $\zg(\kg)=0$, then any almost-normal metric $g$ compatible with $J$ is CYT.  

\item If $\zg(\kg)=0$ and $\dim{\ag_f}=1$, then any metric $g=g_\ag+g_\qg$ as in \eqref{g} is CYT.  
\end{enumerate}
\end{corollary}

\begin{remark}\label{b21-cyt}
Any C-space such that $\zg(\kg)=0$ and $\dim{\ag_f}=1$ is of the form $M=G/[H,H]\times T^{2m+1}$, where $G/H$ is a flag manifold with $\dim{\zg(\hg)}=1$ (see Example \ref{b21}).  
\end{remark}

\begin{proof}
Using the notation given in \eqref{decflags}, we have that 
$$
Z_{J_\qg}=Z_1+\dots+Z_s\in\ag_f=\zg(\hg_f)=\zg(\hg_1)\oplus\dots\oplus\zg(\hg_s), 
$$
and if $g_b=-(z_1\kil_{\ggo_1}+\dots+z_s\kil_{\ggo_s})$, then $g_\ag|_{\zg(\hg_i)}=g_b|_{\zg(\hg_i)}=z_iQ|_{\zg(\hg_i)}$ and
$$
(Z_{J_\qg,g_\qg})_\ag = Z_{J_\qg,g_\qg}=\tfrac{1}{z_1}Z_1+\dots+\tfrac{1}{z_s}Z_s.  
$$ 
This implies that condition \eqref{CYT3} holds, as it does for all $A\in\zg(\hg_i)$, so part (i) follows.  

Part (ii) follows from oart (i) and the fact any metric $g=g_\ag+g_\qg$ as in \eqref{g} is almost-normal, concluding the proof.  
\end{proof}

We assume from now on that $c_1(M,J)=0$, i.e., $Z_{J_\qg}\in\ag_f$.  The following technical preliminaries are needed to study the existence of CYT metrics on complex C-spaces.  

Consider the $J_\ag$-invariant $2$-dimensional subspace 
$$
\ag_J:=\la Z_{J_\qg},J_\ag Z_{J_\qg}\ra_\RR =\la Z_{J_\qg},Z_2\ra_\RR\subset\ag, \qquad Q(Z_{J_\qg},Z_2)=0, 
\quad Q(Z_2,Z_2)=Q(Z_{J_\qg},Z_{J_\qg}),
$$
and the vector 
\begin{equation}\label{Wdef}
W_J:= Z_{J_\qg} +\tfrac{Q(Z_{J_\qg},J_\ag Z_{J_\qg})}{Q(Z_{J_\qg},Z_{J_\qg})} J_\ag Z_{J_\qg}\in\ag_J.
\end{equation}
If $J_\ag Z_{J_\qg}=aZ_{J_\qg}+bZ_2$, $b\ne 0$, i.e., in terms of the basis $\{ Z_{J_\qg},Z_2\}$, 
\begin{equation}\label{Jaaj}
[J_\ag|_{\ag_J}] = 
\left[\begin{matrix} 
a&-\tfrac{a^2+1}{b}\\ b&-a 
\end{matrix}\right], \qquad \mbox{then} \quad W_J=(a^2+1)Z_{J_\qg}+abZ_2.  
\end{equation}
We note that if we normalize by $g_\ag(Z_{J_\qg},Z_{J_\qg})=1$, then the compatibility between $J_\ag$ and the metric $g_\ag$ implies that in terms of the basis $\{ Z_{J_\qg},Z_2\}$, 
\begin{equation}\label{gaaj}
[g_\ag|_{\ag_J}] = 
\left[\begin{matrix} 
1&-\tfrac{a}{b}\\ -\tfrac{a}{b}&\tfrac{a^2+1}{b^2} 
\end{matrix}\right].  
\end{equation}
We are now ready to state and prove an obstruction for the existence of CYT. 
 
\begin{theorem}\label{CYT2}
Let $(M=G/K,J)$ be a complex C-space such that $Z_{J_\qg}\in\ag_f$ (i.e. $c_1(M,J)=0$) and $J_\ag Z_{J_\qg}\in\ag_f$ (in particular, $\dim{\ag_f}\geq 2$).  If $(M=G/K,J)$ admits a $G$-invariant CYT metric, then
$$
W_J\in C(J_\qg)_{\ag_J}, 
$$ 
where $C(J_\qg)_{\ag_J}$ is the $Q$-orthogonal projection on $\ag_J$ of the cone $C(J_\qg)$ defined in \eqref{cone}.  Furthermore, the converse assertion holds if $G$ is semisimple and $\dim{\ag}=2$.  
\end{theorem}

\begin{remark}
If $G$ is semisimple, then $J_\ag Z_{J_\qg}\in\ag_f$ always holds as $\ag_f=\ag$.  We do not know whether this is actually an existence characterization for CYT in the case when $G$ is semisimple, that is, whether the converse assertion holds for $G$ semisimple and $\ag$ of arbitrary dimension.   
\end{remark}

\begin{proof}
According to the symmetrization process given in Remark \ref{aver}, the CYT $G$-invariant metric can be assumed to be of the form $g=g_\ag+g_\qg$ as in \eqref{g}, so if $g_\ag(Z_{J_\qg},Z_{J_\qg})=1$, then it follows from \eqref{CYT3} and \eqref{Wdef} that 
\begin{align*}
(Z_{J_\qg,g_\qg})_{\ag_J} =& g_\ag((Z_{J_\qg,g_\qg})_{\ag_J},Z_{J_\qg})Z_{J_\qg} + g_\ag((Z_{J_\qg,g_\qg})_{\ag_J},J_\ag Z_{J_\qg})J_\ag Z_{J_\qg} \\
=& Q(Z_{J_\qg},Z_{J_\qg})Z_{J_\qg} + Q(Z_{J_\qg},J_\ag Z_{J_\qg})J_\ag Z_{J_\qg} =Q(Z_{J_\qg},Z_{J_\qg})W_J, 
\end{align*}
which implies that $W_J\in C(J_\qg)_{\ag_J}$.  

Conversely, assume that $W_J\in C(J_\qg)_{\ag_J}$.  By \eqref{Jaaj}, there exists a metric $g_\qg$ such that 
$$
(Z_{J_\qg,g_\qg})_{\ag_J}=Q(Z_{J_\qg},Z_{J_\qg})\left((a^2+1)Z_{J_\qg}+abZ_2\right), 
$$
which implies that 
$$
g_\ag((Z_{J_\qg,g_\qg})_{\ag_J},Z_{J_\qg}) =Q(Z_{J_\qg},Z_{J_\qg})\left(a^2+1-\tfrac{a}{b}ab\right) =Q(Z_{J_\qg},Z_{J_\qg}), 
$$
$$
g_\ag((Z_{J_\qg,g_\qg})_{\ag_J},Z_2) =Q(Z_{J_\qg},Z_{J_\qg})\left((a^2+1)(-\tfrac{a}{b})+ab\tfrac{a^2+1}{b^2}\right) = 0 =Q(Z_{J_\qg},Z_2).  
$$
Thus \eqref{CYT3} holds for $Z_{J_\qg}$ and $Z_2$, so it holds for any $A\in\ag_J$ and any metric $g_\ag$ by \eqref{gaaj}.  Since $\ag_J=\ag$ if $\dim{\ag}=2$, condition $W_J\in C(J_\qg)_{\ag_J}$ is also sufficient for \eqref{CYT3} to hold, concluding the proof.    
\end{proof}

The existence region can therefore be described as follows (see Figure \ref{CYT-fig}).  

\begin{corollary}\label{CYT4}
Under the hypothesis of Theorem \ref{CYT2}, there exist $u<0<v$ such that for any complex structure $J=(J_\ag,J_\qg)$ with $c_1(M,J)=0$ and $J_\ag Z_{J_\qg}=aZ_{J_\qg}+bZ_2$, $b\ne 0$ admitting a CYT metric as in \eqref{g}, the following inequalities hold:
$$
u<\tfrac{ab}{a^2+1}<v.  
$$
If $G$ is semisimple, $\dim{\ag}=2$ and these inequalities hold, then the corresponding complex C-space $(M=G/K,J)$ admits a CYT G-invariant metric as in \eqref{g}.  
\end{corollary}

\begin{remark}\label{Ccyt}
In particular, the subset $\cca^G_{cyt}\subset\cca^G_0$ (see \eqref{Cc10}) of all $G$-invariant complex structures on a C-space $M=G/K$ admitting a CYT metric is always proper if $\dim{\ag_f}\geq 2$ and it is open if $G$ is semisimple and $\dim{\ag}=2$.  We also obtain that the existence of CYT metrics on a given complex C-space is not stable under small deformations.  Recall that $\cca^G_{cyt}$ is always nonempty; indeed, $\cca^G_{g_Q}\subset\cca^G_{cyt}$, where $\cca^G_{g_Q}$ is the $d(d-1)$-dimensional space of all $G$-invariant complex structures compatible with the normal metric $g_Q$.  Moreover, if $\zg(\kg)=0$, then $\cca^G_{g}\subset\cca^G_{cyt}$ for any almost-normal metric $g$ by Corollary \ref{CYT-cor}, (i).   
\end{remark}

\begin{proof}
By Theorem \ref{CYT2}, if $J$ admits a compatible CYT metric, then 
$
Z_{J_\qg}+\tfrac{ab}{a^2+1}Z_2 \in C(J_\qg)_{\ag_J}, 
$
so the numbers $u,v$ are defined by 
$$
(Z_{J_\qg}+\RR Z_2)\cap C(J_\qg)_{\ag_J}=\{ Z_{J_\qg}+x Z_2: u<x<v\}, 
$$ 
concluding the proof.  
\end{proof}

\begin{figure}
\begin{tikzpicture}[scale=.8]
\draw[->,thick] (-4, 0) -- (4, 0) node[right] {$a$};
\draw[->,thick] (0, -3) -- (0, 3) node[above] {$b$};
\draw[ultra thin,color=gray] (-4.5,-3.5) grid (4.5,3.5); 

\draw (0,0.5) node[left] {{\tiny $2v$}};
\draw (0,0.5) node {-};
\draw (0,-4/3) node[right] {{\tiny $2u$}};
\draw (0,-4/3) node {-};
\draw (0,0) node[below left] {{\tiny $0$}};

\draw (1,0) node[below] {{\tiny $1$}};
\draw (1,0) node {{\tiny l}};
\draw (-1,0) node[below] {{\tiny $-1$}};
\draw (-1,0) node {{\tiny l}};
\draw (0,-0.5) node[right] {{\tiny $-2v$}};
\draw (0,-0.5) node {-};
\draw (0,4/3) node[left] {{\tiny $-2u$}};
\draw (0,4/3) node {-};

\draw[very thick, color= red, fill=orange, opacity=0.8, domain=0.075:4]  plot(\x, {1/4*(\x+1/\x)});
\draw[very thick, color=red, fill=orange,  opacity=0.8, domain=0.2:4]  plot(\x, {2/3*(-\x-1/\x)});
\draw[very thick, color=red, fill=orange,  opacity=0.8, domain=-0.075:-4]  plot(\x, {1/4*(\x+1/\x)});
\draw[very thick, color=red, fill=orange,  opacity=0.8, domain=-0.2:-4]  plot(\x, {2/3*(-\x-1/\x)});
    
\draw (1.5,1.5) node {$\nexists$};
\draw (2.5,-0.5) node {$\exists$};
\end{tikzpicture}
\caption{Existence region for CYT}\label{CYT-fig}
\end{figure}
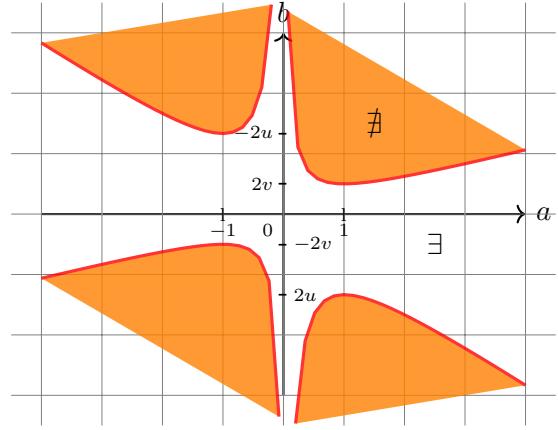

\begin{example}\label{b21-CYT}
Consider the BTP Hermitian structures given in Example \ref{b21-BTP} ($F_2\ne S^1$), which satisfy that $\dim{\ag}=2$.  Since $\ag=\zg(\hg_f)$, we have that $c_1(M,J)=0$ for all $a,b$; furthermore, the Koszul and metric Koszul vectors are respectively given by 
$$
Z_{J_\qg}=d_1Z_{\gamma_1}+d_2Z_{\gamma_2}, \qquad Z_{J_\qg,g_\qg}=\tfrac{d_1}{x}Z_{\gamma_1}+\tfrac{d_2}{y}Z_{\gamma_2},  
$$
for certain positive numbers $d_1,d_2$ (see Examples \ref{b21-comp} and \ref{b21-herm}).  
It follows from \eqref{CYT3} that $(J,g)$ is CYT if and only if 
$$
g_\ag((\tfrac{d_1}{x},\tfrac{d_2}{y}),(1,0))= d_1Q(Z_{\gamma_1},Z_{\gamma_1}), \qquad 
g_\ag((\tfrac{d_1}{x},\tfrac{d_2}{y}),(0,1))= d_2Q(Z_{\gamma_2},Z_{\gamma_2}), 
$$
which is easily seen to be equivalent to 
$$
\tfrac{1}{x}=\tfrac{(a^2+1)d_1+abd_2}{b^2cd_1Q(Z_{\gamma_1},Z_{\gamma_1})}, \qquad \tfrac{1}{y}=\tfrac{abd_1+b^2d_2}{b^2cd_2Q(Z_{\gamma_2},Z_{\gamma_2})}.  
$$
Thus $(M=G/K,J)$ admits a CYT metric of this type precisely when the two terms on the right are positive. The existence region obtained is different from the one given in Figure \ref{CYT-fig} since we are using different basis of $\ag$ to write the matrix of $J_\ag$.  
\end{example}

Hermitian structures which are pluriclosed and CYT are called {\it Bismut Hermitian Einstein} (BHE for short) and are very rare beyond the Bismut flat context.   

\begin{proposition}\label{skt-cyt}
For a pluriclosed Hermitian structure $(J,g)$ on a compact Lie group $M=G$ as in \S\ref{skt-lg}, the following conditions are equivalent: 
\begin{enumerate}[{\rm (i)}] 
\item $(J,g)$ is CYT.  

\item There exists a bi-invariant metric $g_b$ on $[G,G]$ such that $g|_{[\ggo,\ggo]}=g_b$.  

\item $(J,g)$ is Bismut flat. 
\end{enumerate}
\end{proposition}

\begin{proof}
For a pluriclosed metric $g$ as in \S\ref{skt-lg}, condition (ii) is equivalent to $x_\alpha=z_i$ for all $\alpha\in\Delta_i$, so the equivalence between (i) and (ii) follows from \cite[Proposition 4.1]{SKT-LG}.  On the other hand, the equivalence between (ii) and (iii) is the content of Theorem \ref{Bflat}.   
\end{proof}

The above result was obtained in \cite{SKT-LG} for $M=G$ semisimple.

% 14/8/26

\section{Locally conformally K\"ahler (LCK)}\label{LCK-sec}

A non-K\"ahler complex manifold may admit Hermitian metrics which are very close to be K\"ahler in the following sense.   

\begin{definition}\label{LCK-def} 
A Hermitian structure $(J,g)$ is called {\it locally conformally K\"ahler} (LCK for short) if  $d\omega=\theta\wedge\omega$ for some $1$-form $\theta$ such that $d\theta=0$.  If in addition, $\nabla^g\theta=0$, then $(J,g)$ is said to be {\it Vaisman}.   
\end{definition}

We refer to the recent book \cite{OrnVrb} and the references therein for a deep study of LCK manifolds.  The $1$-form $\theta$ in the above definition is precisely $\tfrac{1}{n-1}\theta_L$, where $\theta_L$ is the Lee form of $(J,g)$ (see \eqref{lee}).  In particular, an LCK structure is K\"ahler if and only if $\theta=0$.  

We will show below that any compact homogeneous Hermitian manifold which is LCK can be constructed in the following way (see Theorem \ref{LCK}).  

Any flag manifold $F=G_f/H_f$ and a {\it regular} element $W\in\zg(\hg_f)$ (i.e., $\alpha(\im W)\ne 0$ for any $\alpha\in\Delta_\qg$) determine an LCK Hermitian C-space as follows: 

\begin{enumerate}[{\small $\bullet$}]
\item The C-space is given by 
\begin{equation}\label{LCK-M}
M=G/K=G_f/K\times S^1, 
\end{equation}
where $G_f/K$ is the C-space defined by 
$$
\zg(\hg_f)=\zg(\kg)\oplus\RR W, \qquad \kg=[\hg_f,\hg_f]\oplus\zg(\kg), 
$$
provided that $K\subset H_f$ is closed (see \eqref{slope}).  In particular, $G=G_f\times S^1$ and $\zg(\ggo)=\RR Z_0$ for some $Z_0\in\zg(\ggo)$.  We can assume that $Q(Z_0,Z_0)=1$ and $Q(W,W)=-\kil_{\ggo_f}(W,W)=1$ for simplicity, and according to \eqref{reddec} and \eqref{reddec2},
$$
\ggo=\ggo_f\oplus\zg(\ggo), \qquad \pg=\ag\oplus\qg, \qquad\ag=\RR Z_0\oplus\RR W, \qquad \ag_f=\RR W, 
$$
where $\ggo_f=\hg_f\oplus\qg$.  Note that $\left\{Z_0,W\right\}$ is a $Q$-orthonormal basis of $\ag$. 

\item Consider the complex structure $J=(J_\ag,J_\qg)$ (see \eqref{J}) such that $J_\qg$ is attached to $\Delta_\qg^+:=\{\alpha\in\Delta_\qg:\alpha(\im W)>0\}$ and $J_\ag$ is given by $J_\ag Z_0:=aZ_0+bW$, i.e.,  
\begin{equation}\label{LCK-J}
[J_\ag]_{\left\{Z_0,W\right\}} = 
\left[\begin{matrix} a&-\tfrac{a^2+1}{b} \\ b& -a\end{matrix}\right], \qquad b\ne 0.  
\end{equation} 

\item As a compatible metric $g=g_\ag+g_\qg$ (see \eqref{g}), consider for any $t,c>0$,
\begin{equation}\label{LCK-g} 
\left\{
\begin{array}{l}
g_\qg:=t g_W, \quad g_W=(x_\alpha)_{\alpha\in\Delta_\qg^+}, \qquad x_\alpha:=\alpha(\im W)>0 
\qquad\mbox{(see \S\ref{K})}, \\ \\    

[g_{\ag}]_{\{ Z_0,W\}} = c\left[\begin{matrix} b^2&-ab \\ -ab& a^2+1\end{matrix}\right]. 
\end{array}\right.
\end{equation}
\end{enumerate}

It follows from Lemma \ref{dco} that the Hermitian structure $(J,g)$ (see \eqref{Jg}) satisfies that $d\omega(E_\alpha,E_\beta,E_\gamma)=0$ for all $\alpha,\beta,\gamma\in\Delta_\qg$, and by \eqref{LCK-g}, 
\begin{align*}
d\omega(Z_0,e_\alpha,f_\alpha) =& g_\ag(J_\ag Z_0, \im A_\alpha) = g_\ag(aZ_0+bW, -x_\alpha W) = -bc x_\alpha =-bc \omega(e_\alpha,f_\alpha),\\ 
d\omega(W,e_\alpha,f_\alpha) =& g_\ag(J_\ag W, \im A_\alpha) = g_\ag(-\tfrac{a^2+1}{b}Z_0-aW, -x_\alpha W) = 0.    
\end{align*}
This implies that $d\omega=\theta\wedge\omega$, where $\theta\in\Omega^1(M)^G$ is given by
$$
\theta(X):=-Q\left(X,bcZ_0\right), \qquad\forall X\in\pg,   
$$
and since $Z_0\in\zg(\ggo)$, $d\theta=0$ (see \S\ref{dR1} or \S\ref{form-sec}), so the Hermitian structure $(J,g)$ is LCK.  

The simplest examples of the construction comes from $F=\SU(n+1)/\U(n)=\CC P^n$, producing Hopf manifolds $M=S^{2n+1}\times S^1$.  

\begin{remark}\label{LCK-rem} 
The following observations on the above construction are in order:
\begin{enumerate}[{\rm (i)}] 
\item The C-space $M=G/K$ is determined by $F$ and $\RR W$, $J_\qg$ by the corresponding chamber $Ch(W)\subset\zg(\hg_f)\setminus\{ 0\}$ (see \eqref{ch}), $J_\ag$ by $W$ and $a,b\in\RR$, $b\ne 0$, $g_\qg$ by $W$ and $t>0$ and $g_\ag$ by $J_\ag$ and $c>0$.  Note that the above construction identically works for a regular element $W$ of any length by taking the basis $\{ Z_0,\tfrac{1}{|W|}W\}$ and setting $t:=|W|$.   

\item According to \S\ref{K}, $g_W$ is a K\"ahler metric on the complex flag manifold $(F,J_\qg)$.  

\item We have a $2$-torus holomorphic fibration over a K\"ahler manifold given by 
$$
T^2=H_f/K\times S^1 \longrightarrow M=G/K \longrightarrow F=G_f/H_f. 
$$
\item As well known, $G_f/K$ is a Sasakian manifold, and there is also an $S^1$ fibration over the flag $F$ given by
$$
S^1=H_f/K \longrightarrow G_f/K.  \longrightarrow F=G_f/H_f.
$$
\item If $M=G/K$ is of fibration type (see Definition \ref{fibtype-def}) and $G_f$ is simple, then $\qg$ is $\holg(\nabla^b)$-irreducible (see Corollary \ref{holgK}).  
\end{enumerate}
\end{remark}

\begin{example}\label{AW}
For the flag $F=\SU(3)/T^2$,  $\dynkin[scale=2] A{oo}$, if we consider $W_{p,q}\in\tg$, a regular element orthogonal to 
$$
(p,q,-(p+q))\in\tg=\hg_f=\zg(\hg_f), \qquad p,q\in\NN \quad \mbox{coprime numbers},
$$ 
then $G_f/S^1_{p,q}$ is an Aloff-Wallach space, where $S^1_{p,q}\subset\SU(3)$ is the Lie subgroup with Lie algebra $\RR (p,q,-(p+q))$.  The topology of the LCK manifold 
$$
M^8_{p,q}=G/K=G_f/S^1_{p,q}\times S^1
$$ 
is therefore very sensitive to the rational slope $p/q$.  Indeed, $H^4(G_f/S^1_{p,q},\ZZ)=\ZZ_{p^2+pq+q^2}$ and there are infinitely many homotopy classes among them (see \cite{Krg}).  Moreover, there are different diffeomorphic classes on the same topological space (see \cite{KrcStl}).  
\end{example}

The following classification was obtained in \cite{HsgKms1} (see also \cite{HsgKms2, AlkCrtHsgKms, Gn, AlkHsgKms}).     

\begin{theorem}\label{HK}\cite[Theorem 1, pp.\ 692]{HsgKms1} 
Any compact homogeneous LCK manifold $M$ is, up to biholomorphism, isomorphic to a holomorphic principal fiber bundle over a flag manifold $F$ with fiber a $1$-dimensional complex torus.  To be more precise, M can be written as a homogeneous space form $G/H$, where $G$ is a compact connected Lie group of holomorphic automorphisms on $M$ which is of the form $G=S^1\times S$, where $S$ is a compact simply connected semisimple Lie group, including the connected
component $H_0$ of $H$ which is a closed subgroup of $S$.  $M=G/H$ can be expressed as
$M=S^1\times_\Gamma S/H_0$, where $\Gamma=H/H_0$ is a finite abelian group acting holomorphically on the fiber $S^1=G/S$ of the fibration $G/H_0\rightarrow F$ on the right.
\end{theorem}

In the context of C-spaces, the above theorem is essentially equivalent to the following structural result.  

\begin{theorem}\label{LCK}
Any Hermitian C-space which is LCK is of the form given by the construction \eqref{LCK-M}-\eqref{LCK-g} for some flag manifold $F=G_f/H_f$, a regular unitary element $W\in\zg(\hg_f)$ and real numbers $a,b,c,t$, where $b\ne 0$ and $c,t>0$.  
\end{theorem}

\begin{remark}\label{LCK-rem2} \hspace{1cm}
\begin{enumerate}[{\rm (i)}] 
\item In particular, any LCK complex C-space has $b_1(M)=1$ by \S\ref{dR1} (see \cite[Claim 42.11]{OrnVrb} for a geometric alternative proof).  On the other hand, $b_2(M)=\dim{\zg(\kg)}=\dim{\zg(\hg_f)}-1$ (see \S\ref{dR2}).  

\item Any LCK Hermitian structure as in \eqref{Jg} is therefore {\it Vaisman} (i.e., $\nabla^g\theta=0$).  Indeed, 
$$
(\nabla^g_X\theta)Y = -\theta(\nabla^g_XY) = Q\left(\nabla^g_XY,bcZ_0\right) =0, \qquad\forall X,Y\in\pg,
$$
since $\nabla^g_XY\in\RR W\oplus\qg$ (see  Corollary \ref{LC}).  We refer to \cite{GdcMrnOrn} and \cite[Theorem 42.7]{OrnVrb} for geometric alternative proofs. 

\item An LCK Hermitian structure as in \eqref{Jg} is always BTP by Proposition \ref{BTP-K}; indeed, 
$$
A_\alpha=Q(H_\alpha,W)W = -\alpha(W)W = \im\alpha(\im W)W = \im x_\alpha W, \qquad\forall\alpha\in\Delta_\qg^+.  
$$  
This was proved to hold for any Vaisman Hermitian manifold in \cite[Corollary 3.9]{AndVll}.  

\item An LCK complex C-space admits an invariant pluriclosed metric if and only if it is the Hopf surface $M=\SU(2)\times S^1$.  

\item An LCK complex C-space can never be balanced (see Proposition \ref{lck-prop2} below).  
\end{enumerate}
\end{remark}

\begin{proof}
Let $(J,g)$ be an LCK $G$-invariant structure as in \eqref{Jg} on a C-space $M=G/K=G_f/K\times T^{\dim{\zg(\ggo)}}$ which fibers on the flag $F=G_f/H_f$.  Thus there exists $Z\in\zg(\ggo)$ such that 
\begin{equation}\label{LCK-1}
d\omega=\theta_Z\wedge\omega, \qquad\mbox{where} \quad \theta_Z:=Q(\cdot,Z), 
\end{equation}
is a $G$-invariant closed $1$-form (see \S\ref{dR-sec} or \S\ref{form-sec}).  This implies that $d\omega(E_\alpha,E_\beta,E_\gamma)=0$ for all $\alpha,\beta,\gamma\in\Delta_\qg$, so it follows from the formula for $d\omega$ given in Lemma \ref{dco} that the metric $g_\qg$ is K\"ahler on $F=G_f/H_f$, that is, there exists $W\in\zg(\hg_f)$ such that $g_\qg=g_W$, i.e., $x_\alpha=\alpha(\im W)>0$ for all $\alpha\in\Delta_\qg^+$ (see \S\ref{K}).  

On the other hand, since the non-degenerate $2$-form $\omega_\ag$ restricted to the $g$-orthogonal complement of $Z$ in $\ag$ vanishes by \eqref{LCK-1}, we obtain that $\dim{\ag}=2$, that is, 
$$
\ag=\RR Z\oplus\RR J_\ag Z, \qquad \zg(\ggo)=\RR Z, \qquad\dim{\ag_f}=1.   
$$ 
Finally, using \eqref{LCK-1} and Lemma \ref{dco}, for any $\alpha\in\Delta_\qg^+$, 
\begin{align*}
Q(P_\ag J_\ag Z, Z_\alpha) =& g_\ag(J_\ag Z, Z_\alpha) = g_\ag(J_\ag Z, \im A_\alpha) =d\omega(Z,e_\alpha,f_\alpha) = \theta_Z(Z)\omega(e_\alpha,f_\alpha) \\ 
=& Q(Z,Z)x_\alpha = -Q(Z,Z)Q(W,Z_\alpha) = Q(-Q(Z,Z)W,Z_\alpha), 
\end{align*}     
which implies that $(P_\ag J_\ag Z)_{\ag_f}=-Q(Z,Z)W$ since $\{ Z_\alpha:\alpha\in\Delta_\qg^+\}$ generates $\zg(\hg_f)$.  Thus $W\in\ag_f$ and so $\ag=\RR Z\oplus\RR W$ and $\ag_f=\RR W$.  We therefore obtain that $(M,J,g)$ corresponds to the construction given in \eqref{LCK-c} for $Z_0:=-Q(W,W) Z$ (i.e., $J_\ag Z_0=W$) and $c:=g_\ag(Z_0,Z_0)$, concluding the proof.  
\end{proof}

We note that given a flag $F=G_f/H_f$, each regular $W\in\zg(\hg_f)$ determines a C-space $M=G/K$ admitting an LCK structure.  

\begin{proposition}\label{eqdif}
For two C-spaces $M$ and $M'$ defined by $F=G_f/H_f$ and $\im W,\im W'\in Ch(J_\qg)$, where $Ch(J_\qg)\subset\im\zg(\hg_f)\setminus\{ 0\}$ is a chamber as in \eqref{ch}, the following conditions are equivalent: 
\begin{enumerate}[{\rm (i)}]
\item $M$ and $M'$ are equivariantly diffeomorphic (see \S\ref{gauge}).  

\item $W=t W'$ for some $t>0$.  

\item $M=M'$. 
\end{enumerate}
\end{proposition}

\begin{proof}
The homogeneous spaces $M$ and $M'$ defined by $W$ and $W'$ are equivariantly diffeomorphic if and only if 
there exist $\vp\in N_{G_f}(H_f)$ such that $\vp W=tW'$ for some $t>0$.  Since these vectors belong to the same chamber, this is equivalent to $W=tW'$ (see \S\ref{comp}), that is, $\kg=\kg'$ and so $M=M'$.  
\end{proof}

\begin{example}
Arguing as in \ref{AW}, if we consider $F=\SU(2)/S^1\times\SU(2)/S^1$,  $\dynkin[scale=2] A{o}\times \dynkin[scale=2] A{o}$, and $W\perp pZ_1+qZ_2\in\tg$, then $G_f/S^1_{p,q}=S^3\times S^2$ as a manifold for all $p,q\in\NN$.  We therefore obtain a single LCK manifold up to diffeomorphism, 
$$
M^6_{p,q}:=G_f/S^1_{p,q}\times S^1= S^3\times S^2\times S^1, \qquad\forall p,q\in\NN,   
$$   
but as homogeneous spaces, it follows from Proposition \ref{eqdif} that $M^6_{p,q}$ is equivariantly diffeomorphic to $M^6_{p',q'}$ if and only if $p/q=p'/q'$.
\end{example}

The following intriguing characterization follows from \S\ref{K} and Lemma \ref{c10}. 

\begin{proposition}\label{lck-prop1}
Let $(M=G_f/K\times S^1,J)$ be a complex C-space admitting an LCK metric defined as in \eqref{LCK-g}.  Then the following conditions are equivalent: 
\begin{enumerate}[{\rm (i)}] 
\item $c_1(M,J)=0$, i.e., the first Chern class vanishes.    

\item  The K\"ahler metric $g_W$ on $F$ is Einstein (i.e., $g_W$ is the unique K\"ahler-Einstein metric on $(F,J_\qg)$ up to scaling, see \S\ref{K}).    

\item $W=-tZ_{J_\qg}$ for some $t>0$.  
\end{enumerate}
\end{proposition}

We do not know whether the above characterization of LCK (or Vaisman, see Remark \ref{LCK-rem2}, (ii)) manifolds with zero first Chern class holds in some sense beyond the homogeneous case.  

\begin{proposition}\label{lck-prop2}
Let $(M=G_f/K\times S^1,J,g)$ be an LCK Hermitian C-space defined as in \eqref{LCK-M}-\eqref{LCK-g} by $F=G_f/H_f$, a regular $W\in\zg(\hg_f)$ and $a,b,c,t\in\RR$.  Then the  following holds: 
\begin{enumerate}[{\rm (i)}]
\item $(J,g)$ is never balanced and it is LCB if and only if $a=0$.   
 
\item $(J,g)$ is always BTP and it is BAS if and only if $M=\SU(2)\times S^1$, the Hopf surface.  
  
\item If $c_1(M,J)=0$, then the metric $g$ is CYT if and only if  
$$
\tfrac{t}{c} = (a^2+1)|\Delta_\qg^+| |Z_{J_\qg}|,
$$  
where $|\Delta_\qg^+|$ denotes the number of positive complementary roots of the flag $F=G_f/H_f$.  In particular, $(M=G_f/K\times S^1,J)$ is always a CYT manifold.     
\end{enumerate}
\end{proposition}

\begin{proof}
Using that $Q(\im H_\alpha,W)=-x_\alpha$ for any $\alpha\in\Delta_\qg^+$, it is easy to see that the projection on $\ag$ of the metric Koszul vector (see \eqref{koszul-g}) is given by 
$$
(Z_{J_\qg,g_\qg})_\ag = (Z_{J_\qg,g_\qg})_{\ag_f} = Q(Z_{J_\qg,g_\qg},W)W = \sum_{\alpha\in\Delta_\qg^+}\tfrac{1}{t\alpha(\im W)} Q(\im H_\alpha,W)
=-\tfrac{|\Delta_\qg^+|}{t} W.  
$$
In particular, $(J,g)$ can never be balanced by Proposition \ref{bal}, (i) and the LCB condition follows from \eqref{LCB}.  On the other hand, the CYT characterization follows from Proposition \ref{CYT} and the BTP condition from Remark \ref{LCK-rem2}, (iii).  The BAS characterization follows from \cite{BrbPdc}.  
\end{proof}

\begin{example}\label{LCK-c}
For any C-space admitting an LCK structure, which are classified in Theorem \ref{LCK}, $d=1$, $\dim{\zg(\ggo)}=1$, $\dim{\zg(\hg_f)}=\dim{\zg(\kg)}+1$ and $\dim{S}=1$, so we obtain
$$
h^{1,0}=0, \quad h^{0,1}=1, \qquad b_1=1, 
$$
$$
h^{2,0}=0, \quad h^{1,1}=\dim{\zg(\kg)}, \quad h^{0,2}=0, \qquad b_2=\dim{\zg(\kg)}, 
$$
$$
h^{3,0}=0, \quad h^{2,1}=D\geq 1, 
\quad
h^{1,2}=\dim{\zg(\kg)}+1,  \quad h^{0,3}=0, \qquad b_3=b_3(G_f/K)+\dim{\zg(\kg)},
$$
$$
h^{1,1}_{BC}= \dim{\zg(\kg)}+1,
\qquad h^{2,1}_{BC}=0,\qquad 
h^{1,1}_A=\dim{\zg(\kg)}. 
$$
\end{example}

% 14/8/26

\section{Geography of complex structures}\label{geo-sec}

Given a C-space $M=G/K$ and a Tits fibration $M\rightarrow F$ over a flag manifold $F$, let $\cca$ denote the space of all $G$-invariant complex structures as in \eqref{J}, i.e., those such that the Tits fibration is holomorphic.  It follows from Theorem \ref{C2} and Proposition \ref{bihol-h} that any $G$-invariant complex structure on $M=G/K$ is biholomorphic to one in $\cca$.  Recall from \eqref{Cc10}, \eqref{Cbal} and Remark \ref{Ccyt} the subspaces $\cca_0, \cca_{bal}, \cca_{cyt}\subset\cca$ of structures having zero first Chern class, admitting a balanced metric and admitting a CYT metric, respectively.  

According to \eqref{Ccal}, if $\{J_1,\dots,J_k\}$ is the set of all $G_f$-invariant complex structures on the flag manifold $F=G_f/H_f$, then $\cca=\cca_1\sqcup\dots\sqcup\cca_{k/2}$, where 
$$
\cca_i:=\left\{ J=(J_\ag,\pm J_i): J_\ag^2=-I\right\},  \qquad \dim{\cca_i}=2d^2, 
$$ 
has four connected components.  After reordering, we obtain from \eqref{Cc10} and \eqref{Cbal} that 
$$
\cca=\underbrace{\cca_1\sqcup\dots\sqcup\cca_l}_{\cca_0} \sqcup \underbrace{\cca_{l+1}\sqcup\dots\sqcup\cca_{l+m}}_{\cca_{bal}} \sqcup\; \cca_?, 
$$
where $\cca_?$ is nonempty in general if $\zg(\kg)\ne 0$ (see Example \ref{c10-bal}).  If $\zg(\kg)=0$ (i.e., $K$ is either semisimple or trivial), then $\cca_0=\cca$.  However, in the case when $0\ne\zg(\kg)$, $\cca_0=\emptyset$ generically among the set of all C-spaces fibering over a fixed flag $F$.  On the other hand, $\cca_{bal}=\emptyset$ if $\zg(\kg)=0$ and $\cca_{bal}\ne\emptyset$ if $\zg(\kg)\ne 0$ and $\rank(\ggo_i)\geq 2$ for all $i=1,\dots,s$ (see Theorem \ref{bal2}).  On the other hand, $\cca_{lcb}=\cca$ if $\zg(\ggo)\ne 0$ and $\cca_{lcb}=\cca_{bal}$ if $\zg(\ggo)=0$.   

We also consider the subspace
$$
\cca_{an}:=\left\{ J\in\cca: J\;\mbox{is compatible with some almost-normal $g$}\right\},  
$$
which decomposes as $\cca_{an}=\cca_{1,an}\sqcup\dots\sqcup\cca_{k,an}$, where 
$$ 
\cca_{i,an}:=\left\{ J=(J_\ag,J_i): J_\ag\;\mbox{is compatible with $g|_{\ag}$ for some almost-normal $g$}\right\}.  
$$  
Note that $\cca_{g_Q}\subset\cca_{an}$.

The subspace $\cca_{cyt}\subset\cca_0$ decomposes as
$$
\cca_{cyt}=\cca_{1,cyt}\sqcup\dots\sqcup\cca_{l,cyt}, \qquad \cca_{i,cyt}:=\cca_{cyt}\cap\cca_i, \quad\forall i=1,\dots,l,   
$$
and satisfy the following properties: 
\begin{enumerate}[{\small $\bullet$}]
\item $\cca_{i,cyt}\ne\emptyset$ for all $i$ (see Remark \ref{Ccyt}).  If $\dim{\zg(\kg)}=0$, then $\cca_{i,an}\subset\cca_{i,cyt}$ for all $i$ (see Corollary \ref{CYT-cor}, (i)).  

\item $\cca_{cyt}=\cca_0=\cca$ if $\zg(\kg)=0$ and $\dim{\ag_f}=1$ (see Corollary \ref{CYT-cor}, (ii)).  

\item $\cca_{i,cyt}\subsetneq\cca_i$ for all $i$ if $\dim{\ag_f}\geq 2$ (see Remark \ref{Ccyt}). 

\item $\cca_{i,cyt}$ is a proper and nonempty open subset of $\cca_i$ for all $i$ if $G$ is semisimple and $\dim{\ag}=2$ (see Corollary \ref{CYT4} and Figure \ref{CYT-fig}).  
\end{enumerate}

On an LCK C-space as in \S\ref{LCK-sec},  
$$
\cca=\cca_{lck}\sqcup\cca_{bal}\sqcup\cca_{?}, 
$$
where $\cca_{lck}=\cca_1$ if the corresponding regular element satisfies that $\im W\in Ch(J_1)$.  Moreover, $\cca_0=\cca_{lck}$ if $W\in\RR_{<0}Z_{J_1}$ and $\cca_0=\emptyset$ otherwise (see Proposition \ref{lck-prop1}).  The subspace $\cca_{bal}$ consists of the union of all $\cca_i$ such that $W$ is orthogonal to some element in the cone $C(J_i)$ (see Proposition \ref{bal}, (ii)).   

For a Lie group $M=G$, the space $\cca=\cca_0$ consists of all left-invariant complex structures which are also $T$-right-invariant for some fixed maximal torus $T$ of $G$.  We have that 
$$
\cca=\cca_{cyt}\sqcup\cca_?, \qquad \cca_{skt}=\cca_{ab}=\cca_{bf}\subsetneq\cca_{cyt} \subsetneq\cca,
$$
where $\cca_{skt}$ and $\cca_{bf}$ are, respectively, the spaces of all $J\in\cca$ admitting a pluriclosed left-invariant metric and admitting a Bismut flat left-invariant metric, and  
$$
\cca_{ab}:=\left\{ J\in\cca: J\;\mbox{is compatible with some almost-biinvariant $g$}\right\}.    
$$
As a unique exception, $\cca_{cyt}=\cca$ for the Hopf surface $M=\SU(2)\times S^1$.

% 18/7/2026

\section{Appendix A: Roots}\label{roots}

We give in this section some preliminaries and basic properties of root systems of complex semisimple Lie algebras, see e.g.\ \cite[III.\S4,5, X.\S3]{Hlg} or \cite[II.\S5, App C]{Knp} for more complete treatments.  

Let $G$ be a connected and compact semisimple Lie group.  We fix a maximal torus $T\subset G$ with Lie algebra $\tg$, so $\tg$ is a maximal abelian subalgebra of $\ggo$ and its complexification $\tg^c$ is a Cartan subalgebra of the complex semisimple Lie algebra $\ggo^c:=\ggo\otimes\CC=\ggo+\im\ggo$.  The adjoint action of $\tg^c$ on $\ggo^c$ determines a root space decomposition given by
$$
\ggo^c=\tg^c\oplus\bigoplus_{\alpha\in\Delta}\ggo_\alpha, 
$$
where $\Delta\subset(\tg^c)^*$ is called the root system attached to $G$, i.e., for each root $\alpha\in\Delta$, $[H,E]=\alpha(H)E$ for all $H\in\tg^c$ and $E\in\ggo_\alpha$.  This implies that $\alpha\in\Delta$ if and only if $-\alpha\in\Delta$, and 
$$
[\ggo_\alpha,\ggo_\beta]\left\{
\begin{array}{lcl} 
=\ggo_{\alpha+\beta}, && \alpha+\beta\in\Delta, \\ 
\subset\tg^c, && \alpha+\beta=0, \\ 
=0, && 0\ne \alpha+\beta\not\in\Delta. 
\end{array}\right.
$$  
The Killing form $\kil_{\ggo^c}$ of $\ggo^c$, being non-degenerate on $\tg^c$, determines an identification $\tg^c\equiv (\tg^c)^*$ and in particular we obtain a vector $H_\alpha\in\tg^c$ defined by $\alpha=\kil_{\ggo^c}(\cdot,H_\alpha)$ for each $\alpha\in(\tg^c)^*$.  For any {\it basis} $\Pi$ of $\Delta$ (i.e., a basis of $(\tg^c)^*$ such that every $\beta\in\Delta$ can be written as $\beta=\sum\limits_{\alpha\in\Pi} n_\alpha\alpha$, where the $n_\alpha$ are integers of the same sign, possibly zero), the set  
$
\{H_\alpha:\alpha\in\Pi\},   
$  
is a basis of $\tg^c$.  There exist vectors $E_\alpha$, $\alpha\in\Delta$ such that (see \cite[Chapter VI, Theorem 6.6]{Knp}) 
$$
\ggo_\alpha=\CC E_\alpha, \qquad \kil_{\ggo^c}(E_\alpha,E_{-\alpha})=1, \qquad H_\alpha=[E_\alpha,E_{-\alpha}], 
$$
$$
[E_\alpha,E_\beta]=N_{\alpha,\beta}E_{\alpha+\beta}, \qquad N_{-\alpha,-\beta}=-N_{\alpha,\beta},
$$
and so the structural constants $N_{\alpha,\beta}$ also satisfy the following properties: 
\begin{enumerate}[{\rm (i)}]
\item $N_{\beta,\alpha}=-N_{\alpha,\beta}$ for all $\alpha,\beta\in\Delta$. 

\item $N_{\alpha,\beta}=N_{\beta,\gamma}=N_{\gamma,\alpha}$ if $\alpha+\beta+\gamma=0$.  

\item If $\alpha+\beta+\gamma+\delta=0$ and the sum of any pair is nonzero, then  
$$
N_{\alpha,\beta}N_{\gamma,\delta}+ N_{\beta,\gamma}N_{\alpha,\delta}+ N_{\gamma,\alpha}N_{\beta,\delta}=0.
$$  
\end{enumerate}
Note that $\kil_{\ggo^c}(E_\alpha,E_\alpha)=0$ and $\overline{E_\alpha}=-E_{-\alpha}$.  The set  
$$
\{H_\alpha:\alpha\in\Pi\} \cup \{E_\alpha:\alpha\in\Delta\}
$$ 
is a basis of $\ggo^c$ called the {\it Chevalley} basis.  On the other hand, $\ip:=\kil_{\ggo^c}|_{\im\tg}$ is positive definite and also defines an inner product on $(\im\tg)^*$ by $\la\alpha,\beta\ra:=\la H_\alpha,H_\beta\ra$.  Since $\alpha:\im\tg\rightarrow\RR$ (i.e., $H_\alpha\in\im\tg$) for any $\alpha\in\Delta$, we can view $\Delta\subset(\im\tg)^*$ and we have that $\alpha|_{\im\tg}=\la\cdot, H_\alpha\ra$.  In particular, 
$$
\la\alpha,\beta\ra = \kil_{\ggo^c}(H_\alpha,H_\beta) = -\kil_\ggo(\im H_\alpha,\im H_\beta), \qquad\forall\alpha,\beta\in(\im\tg)^*,  
$$ 
where $\kil_\ggo$ is the Killing form of $\ggo$.  

In this way, $((\im\tg)^*,\Delta,\ip)$ is a {\it reduced root system}: 
\begin{enumerate}[{\small $\bullet$}]
\item $\Delta$ is a finite subset of nonzero elements in $(\im\tg)^*$ that generates the real vector space $(\im\tg)^*$.   

\item The only multiples of any $\alpha\in\Delta$ which are in $\Delta$ are $\pm\alpha$.  

\item $a_{\alpha,\beta}:=\tfrac{2\la\alpha,\beta\ra}{\la\alpha,\alpha\ra}\in\ZZ$ for all $\alpha,\beta\in\Delta$.  

\item $\beta-\tfrac{2\la\alpha,\beta\ra}{\la\alpha,\alpha\ra}\alpha\in\Delta$ for all $\alpha,\beta\in\Delta$.  
\end{enumerate}
Since 
$$
a_{\alpha,\beta}a_{\beta,\alpha}=\tfrac{4\la\alpha,\beta\ra^2}{\la\alpha,\alpha\ra\la\beta,\beta\ra}=0,1,2,3, \quad\mbox{i.e.}, \quad \tfrac{\la\alpha,\beta\ra}{|\alpha| |\beta|}=0,-\unm,-\tfrac{1}{\sqrt{2}},-\tfrac{\sqrt{3}}{2}, \qquad\forall \alpha,\beta\in\Pi, 
$$
the only possible angles between two simple roots $\alpha$ and $\beta$ are $0,\tfrac{2\pi}{3},\tfrac{3\pi}{4},\tfrac{5\pi}{6}$.   

The structural constants $N_{\alpha,\beta}$ can be computed as follows.  Given $\alpha,\beta\in\Delta$, there exist $p,q\in\ZZ$ such that $p\leq 0\leq q$, $p+q=-a_{\alpha,\beta}$, $\beta+n\alpha\in\Delta$ if and only if $p\leq n\leq q$ and   
$$
N_{\alpha,\beta}^2=\unm q(1-p)\la\alpha,\alpha\ra.   
$$
In particular, 
$$
N_{-\alpha,\alpha+\beta}=N_{\alpha,\beta}, \qquad N_{\alpha,-\alpha+\beta}=N_{-\alpha,\beta}, \qquad (N_{\alpha,-\beta})^2=(N_{\alpha,\beta})^2+\la\alpha,\beta\ra.  
$$ 
For any basis $\Pi$ of $\Delta$, the set $\{\im H_\alpha:\alpha\in\Pi\}$ is a basis of $\tg$ and 
$$
\{\im H_\alpha:\alpha\in\Pi\} \cup \{e_\alpha,f_\alpha :\alpha\in\Delta^+\}
$$
is a basis of $\ggo$ called the {\it Cartan-Weyl} basis, where 
$$
e_\alpha:=\tfrac{1}{\sqrt{2}}(E_\alpha-E_{-\alpha}), \qquad f_\alpha:=\tfrac{\im}{\sqrt{2}}(E_\alpha+E_{-\alpha}),
$$
and $\Delta^+$ is the set of all non-negative integer combinations of $\Pi$ which are in $\Delta$.  Thus $\Delta^+$ is an {\it ordering} of $\Delta$ (i.e., $\Delta=\Delta^+\cup -\Delta^+$ (disjoint union) and $(\Delta^++\Delta^+)\cap\Delta\subset\Delta^+$), and conversely, any basis $\Pi$ is the subset of {\it simple} roots (i.e., not the sum of two roots in $\Delta^+$) of some ordering $\Delta^+$.  We note that $\{ e_\alpha,f_\alpha:\alpha\in\Delta^+\}$ is a $-\kil_\ggo$-orthonormal subset, $[e_\alpha,f_\alpha]=\im H_\alpha$, $(\RR e_\alpha+\RR f_\alpha)^c=\ggo_\alpha\oplus\ggo_{-\alpha}$ and
$$
E_\alpha=\tfrac{1}{\sqrt{2}}(e_\alpha-\im f_\alpha), \qquad E_{-\alpha}=-\tfrac{1}{\sqrt{2}}(e_\alpha+\im f_\alpha). 
$$  
The following identities hold for any $\alpha,\beta\in\Delta^+$ and $\beta-\alpha\in\Delta^{\pm}$: 
$$
[e_\alpha,e_\beta]=\tfrac{N_{\alpha,\beta}}{\sqrt{2}}e_{\alpha+\beta} \pm\tfrac{N_{\alpha,-\beta}}{\sqrt{2}}e_{\pm(\beta-\alpha)}, \qquad 
[f_\alpha,f_\beta]=-\tfrac{N_{\alpha,\beta}}{\sqrt{2}}e_{\alpha+\beta} \mp\tfrac{N_{\alpha,-\beta}}{\sqrt{2}}e_{\pm(\beta-\alpha)}, 
$$
$$
[e_\alpha,f_\beta]=\tfrac{N_{\alpha,\beta}}{\sqrt{2}}f_{\alpha+\beta} \pm\tfrac{N_{\alpha,-\beta}}{\sqrt{2}}f_{\pm(\beta-\alpha)}, 
$$
which implies for $\beta-\alpha\in\Delta^+$ that
\begin{equation}\label{eE}
[e_\alpha,E_\beta]=\tfrac{N_{\alpha,\beta}}{\sqrt{2}}E_{\alpha+\beta} +\tfrac{N_{\alpha,-\beta}}{\sqrt{2}}E_{\beta-\alpha}, \qquad 
[f_\alpha,E_\beta]=\tfrac{\im N_{\alpha,\beta}}{\sqrt{2}}E_{\alpha+\beta} +\tfrac{\im N_{\alpha,-\beta}}{\sqrt{2}}E_{\beta-\alpha}.
\end{equation}

The {\it Weyl group} $W:=N_G(T)$ acts simply and transitively on the set of chambers (or connected components) of 
\begin{equation}\label{ch2}
\im\tg\setminus\bigcup_{\alpha\in\Delta}\Ker\alpha \cap {\im\tg},  
\end{equation}
and for each ordering $\Delta^+$ there exists a unique chamber $C$ such that 
$\Delta^+=\{\alpha\in\Delta:\alpha|_{C}>0\}$.  Furthermore, the group $\Aut(\Delta)$ of all invertible linear maps $\psi:(\im\tg)^*\rightarrow(\im\tg)^*$ such that $\psi\Delta=\Delta$ ($\psi$ is automatically an isometry) is the semidirect product $O\ltimes W$, where $O$ is the group of all automorphisms of the Dynkin diagram of $\ggo$ (i.e., outer automorphisms, nontrivial only for the types $A,D,E_6$).  $\Aut(\Delta)$ is isomorphic to the subgroup of $\Aut(\ggo)$ stabilizing $\tg$.     

The element $\rho\in(\im\tg)^*$ defined as half the sum of positive roots,  
\begin{equation}\label{weyl}
\rho = \unm\sum_{\alpha\in\Delta^+} \alpha \equiv \unm\sum_{\alpha\in\Delta^+} H_\alpha\in\im\tg,
\end{equation}
appears in many important formulas.  $\rho$ is also the sum of all {\it fundamental weights} $\omega_1,\dots,\omega_{\dim{\tg}}$, defined by $\la\omega_i,\tfrac{2}{\la\alpha_j,\alpha_j\ra}\alpha_j\ra= \delta_{ij}$ for all $j$, where $\Pi=\{\alpha_1,\dots,\alpha_{\dim{\tg}}\}$, and it is the center of a sphere containing all the simple roots:  $|\alpha-\rho|=|\rho|$ for all $\alpha\in\Pi$ (i.e., $\alpha-2\rho\perp\alpha$).   

The following are natural lattices (i.e., cocompact discrete subgroups) of $(\im\tg)^*$ and $\im\tg$ (see \cite[Chapter 6]{Spn} for further information): 
\begin{enumerate}[{\small $\bullet$}]
\item {\it root lattice}: 
$\quad R:=\la\alpha:\alpha\in\Delta\ra_\ZZ \subset(\im\tg)^*$, 

\item {\it weight lattice}: 
$\quad P:=\left\{\lambda\in(\im\tg)^*:\lambda\left(\tfrac{2\im }{\la\alpha,\alpha\ra}H_\alpha\right)\in\ZZ, \quad \forall\alpha\in\Delta\right\}\subset(\im\tg)^*$, 

\item 
{\it dual root lattice}: 
$\quad R^*:=\left\la\tfrac{2}{\la\alpha,\alpha\ra}H_\alpha:\alpha\in\Delta\right\ra_\ZZ \subset\im\tg$, 

\item 
{\it dual weight lattice}: 
$\quad P^*:=\{ H\in\im\tg:\alpha(H)\in\ZZ, \quad\forall \alpha\in\Delta\}\subset\im\tg$.  
\end{enumerate}
It turns out that the lattice $\Ker(\exp):=\{H\in\tg:\exp(H)=e\}\subset\tg$ satisfies that
$$
R^*\subset\tfrac{1}{2\pi\im}\Ker(\exp)\subset P^*,  \qquad \tfrac{1}{2\pi\im}\Ker\exp/R^*\simeq\Pi_1(G), \qquad P^*/\tfrac{1}{2\pi\im}\Ker(\exp)\simeq Z(G),
$$ 
so $R^*=\tfrac{1}{2\pi\im}\Ker(\exp)$ if and only if $G$ is simply connected and  $\tfrac{1}{2\pi\im}\Ker(\exp)=P^*$ if and only if the center $Z(G)$ of $G$ is trivial.

% 18/7/26

\section{Appendix B: Homogeneous spaces}\label{homsp-sec} 

Let $M$ be a connected differentiable manifold (not necessarily compact) and assume that $M$ is {\it homogeneous}, in the sense that there is a connected Lie group $G$ acting transitively on $M$.  Each of these transitive groups provides a presentation $M=G/K$ of $M$ as a homogeneous space, where $K\subset G$ is the isotropy subgroup at some origin point $o\in M$.  The action will always be assumed to be {\it almost-effective}, i.e., only a discrete subgroup of $G$ acts trivially.

We also assume that there is a {\it reductive decomposition} of $M=G/K$ (e.g., $K$ compact),
$$
\ggo=\kg\oplus\pg, \qquad T_oM\equiv\pg,
$$ 
i.e., $\Ad(K)\pg\subset\pg$, where $\ggo$ and $\kg$ are respectively the Lie algebras of $G$ and $K$.  This provides the following usual identification: each $X\in\ggo$ defines a vector field $X^*\in\chi(M)$ by $X^*_p:=\ddt|_0\exp{tX}\cdot p$ for all $p\in M$, and we identify $\pg\equiv T_oM$, $X\leftrightarrow X^*_o$, which is an isomorphism since $X^*_o=0$ if and only if $X\in\kg$.  Note that $[X^*,Y^*]=-[X,Y]^*$.  

In the case of a Lie group $M=G$, for each $X\in\pg=\ggo$, $X^*$ is precisely the right-invariant vector field on $G$ such that $X^*_e=X$.  

The isotropy representation $K\circlearrowleft T_oM$ is equivalent to the adjoint representation $\Ad(K) \circlearrowleft\pg$.       
Note that if $K$ is connected, then $G/K$ is almost-effective if and only if 
$$
\{Z\in\kg:\ad{Z}|_\pg=0\}=0. 
$$

\subsection{Invariant geometric structures}\label{igs} 
The spaces 
$$
\mca^G, \qquad \acca^G, \qquad \Omega^p(M)^G,
$$
of all $G$-invariant Riemannian metrics, almost complex structures and $p$-forms on $M=G/K$ are respectively identified with the set of all $\Ad(K)$-invariant inner products on $\pg$, the set of all $\Ad(K)$-invariant linear maps $J:\pg\rightarrow\pg$ such that $J^2=-I$ and the vector space $(\Lambda^p\pg^*)^K$ of all $\Ad(K)$-invariant $p$-forms on $\pg$.  Note that these spaces are all finite dimensional and could be empty.  For instance, $\mca^G\ne\emptyset$ if and only if $\overline{\Ad(K)}$ is compact (e.g., when $K$ is compact) and $\acca^G\ne\emptyset$ if and only if the number of irreducible factors on each isotypical component of real type of the isotropy representation is even (e.g., $S^6=\Gg_2/\SU(3)$ admits a $\Gg_2$-invariant almost-complex structure but $S^6=\SO(7)/\SO(6)$ does not admit an $\SO(7)$-invariant almost-complex structure).  Note that $|\acca^G|<\infty$ if and only if the isotropy representation is the sum of inequivalent irreducible representations of complex type.      

If $G$ is compact, then a metric $g\in\mca^G$ is called {\it normal} when it is determined by $Q|_{\pg\times\pg}$ for some bi-invariant metric $Q$ on $\ggo$, and if in addition $G$ is semisimple and $Q=-\kil_{\ggo}$, where $\kil_{\ggo}(X,Y):=\tr{\ad{X}\ad{Y}}$ for all $X,Y\in\ggo$ is the Killing form of $\ggo$, then $g$ is called {\it standard} and denoted by $g_{\kil}$.  

An almost-complex structure $J\in\acca^G$ is called {\it complex} whenever it is integrable, i.e., the Nijenhuis tensor vanishes:
$$
[X,Y]_\pg + J[JX,Y]_\pg + J[X,JY]_\pg - [JX,JY]_\pg=0, \qquad\forall X,Y\in\pg,   
$$ 
where the subscript $\pg$ denotes projection on $\pg$ relative to $\ggo=\kg\oplus\pg$.  Equivalently, the $\im$-eigenspace $\pg^{1,0}$ of the corresponding $\CC$-linear map, also denoted by $J$, satisfies that $[\pg^{1,0},\pg^{1,0}]_\pg\subset\pg^{1,0}$.  In that case, $(M,J)$ is a complex manifold, and a {\it compatible} pair $(J,g)$ (i.e., $g(JX,JY)=g(X,Y)$ for all $X,Y\in\pg$) is called a {\it Hermitian structure}.   
 
%The space of all $G$-invariant complex structures on $M$ is denoted by $\cca^G$ and that of Hermitian structures by $\hca^G$. 

\subsection{de Rham cohomology}\label{dRapp-sec} 
In this section, we follow the lines of \cite[Section 2]{H3}, where a more detailed treatment is given (see also \cite{BswChtMty}).  We assume here that $M$ and $G$ are compact, so that the real de Rham cohomology of $M=G/K$ can be computed within $G$-invariant forms.  Assume also that $K$ is connected, otherwise, we need to add invariance by the finite group $K/K_0$ everywhere.  We fix a bi-invariant metric $Q$ on $\ggo$ as a background metric and consider the $Q$-orthogonal reductive decomposition $\ggo=\kg\oplus\pg$.  

The differential $d:=d_M$ of $G$-invariant forms on the manifold $M$ is given by $d:\Lambda^p\pg^*\rightarrow\Lambda^{p+1}\pg^*$, 
$$
d\alpha(X_1,\dots,X_{p+1}) := \sum_{i<j}(-1)^{i+j}\alpha([X_i,X_j]_\pg,X_1,\dots,\hat{X_i},\dots,\hat{X_j}\dots,X_{p+1}),
$$
giving rise to the $p$th de Rham cohomology group $H^p(G/K)=\Ker d/\Ima d$ and the $p$th {\it Betti number} $b_p(G/K):=\dim{H^p(G/K)}$.  Alternatively, the isomorphism 
\begin{equation}\label{deltagk}
(\Lambda^p\pg^*)^K\longrightarrow \Lambda^p(\ggo,K):=\left\{ \beta\in(\Lambda^p\ggo^*)^K:\iota_\kg\beta=0\right\}, \qquad \alpha\mapsto \hat{\alpha}:=\pi^*\alpha,
\end{equation}
where $\pi:G\rightarrow G/K$ is the usual projection map and $\iota_Z\beta:=\beta(Z,\cdot,\dots,\cdot)$, can be used to compute $H^p(G/K)$ upstairs within left-invariant forms on the Lie group $G$;  indeed, $H^p(G/K)=\Ker \hat{d}/\Ima \hat{d}$, where $\hat{d}:\Lambda^p(\ggo,K)\rightarrow\Lambda^{p+1}(\ggo,K)$ is the differential of forms on the Lie group $G$.  

Any $1$-form on $\pg$ is given by $\theta_X:=Q(\cdot,X)$ for some $X\in\pg$, hence the the space of $G$-invariant $1$-forms on $M$ is given by 
$$
(\Omega^1M)^G=\{\theta_X:X\in\pg_0\}\simeq\pg_0:=\{Y\in\pg:[\kg,Y]=0\}, 
$$ 
and those which are in addition closed by 
$$
(\Omega^1_cM)^G=\{\theta_X:X\in\zg(\ggo)\cap\pg\}\simeq\zg(\ggo)\cap\pg\subset\pg_0, \qquad\mbox{so}\quad b_1(M)=\dim{\zg(\ggo)\cap\pg},
$$
where $\zg(\ggo)$ is the center of $\ggo$.  

We now compute $b_2(M)$, since in \cite[Section 2]{H3}, only the case when $G$ is semisimple was worked out.  Any closed $2$-form $\hat{\sigma}$ on $G$ is of the form 
$$
\hat{\sigma}= \hat{\sigma}_X +\hat{\gamma}, 
$$
where $\hat{\sigma}_X:= Q([\cdot,\cdot],X)$ for some $X\in\ggo$ and the $2$-form $\hat{\gamma}$ satisfies that $\hat{\gamma}([\ggo,\ggo],\cdot)=0$.  Note that we can view $\hat{\gamma}\in\Lambda^2\zg(\ggo)^*$ and that $\hat{\sigma}_X\in\Lambda^2\zg(\ggo)^*$ if and only if $X\in\zg(\ggo)$, if and only if $\hat{\sigma}_X=0$.  It is not hard to see that $\hat{\sigma}\in\Lambda^2(\ggo,K)$ if and only if $[\kg,X]=0$, i.e., $X\in\zg(\kg)\oplus\pg_0$, and $\hat{\gamma}([\ggo,\ggo]+\kg,\cdot)=0$.  Equivalently, $\hat{\gamma}\in\Lambda^2\zg(\ggo)^*$ satisfies that $\hat{\gamma}(\kg_{\zg(\ggo)},\cdot)=0$, where $\kg_{\zg(\ggo)}$ is the $Q$-orthogonal projection of $\kg$ on $\zg(\ggo)$, that is, 
$$
\hat{\gamma}\in \Lambda^2\left(\zg(\ggo)/\kg_{\zg(\ggo)}\right)^*\simeq \Lambda^2\left(\ggo/([\ggo,\ggo]+\kg)\right)^* . 
$$   
Using the $Q$-orthogonal decomposition $\zg(\ggo)=\kg_{\zg(\ggo)}\oplus(\zg(\ggo)\cap\pg)$, we obtain that 
$$
(\Omega^2_cM)^G = \left\{ \sigma_X+\gamma:X\in\zg(\kg)\oplus\pg_0, \; \gamma\in\Lambda^2\left(\zg(\ggo)\cap\pg\right)^*\right\}, \qquad  \sigma_X:=Q([\cdot,\cdot],X).  
$$
We note that $\zg(\ggo)_\pg\simeq\zg(\ggo)$ by almost-effectiveness (i.e., $\zg(\ggo)\cap\kg=0$).  Since $\sigma_X=-d\theta_X$ for any $X\in\pg_0$, we obtain that 
$$
H^2(M) \simeq \zg(\kg)\oplus \Lambda^2\left(\zg(\ggo)\cap\pg\right)^*, \qquad b_2(M)=\dim{\zg(\kg)} + 
\tbinom{\dim{\zg(\ggo)\cap\pg}}{2}.
$$
We are using here that $\zg(\kg)\cap\zg(\ggo)=0$ to obtain the direct sum.  

\begin{lemma}\label{oni}\label{gprime2}
%The Lie subgroup $G'\subset G$ with Lie algebra $\ggo':=[\ggo,\ggo]\oplus(\zg(\ggo)\cap\pg)$ acts transitively on $M=G/K$.  
Up to a finite cover, 
$$
M=G/K =G_{ss}/K\cap G_{ss} \times T^{\dim{\zg(\ggo)\cap\pg}},
$$ 
where $G_{ss}, T^{\dim{\zg(\ggo)\cap\pg}}\subset G$ are the Lie subgroups with Lie algebras $[\ggo,\ggo]$ and $\zg(\ggo)\cap\pg$, respectively.  
\end{lemma} 

\begin{remark}
Since $b_1(G_{ss}/K\cap G_{ss})=0$ and $\dim{\zg(\kg)\cap[\ggo,\ggo]}=\dim{\zg(\kg)}$ by almost-effectiveness, it follows from the K\"unneth formula that 
$$
b_2(M) = b_2(G_{ss}/K\cap G_{ss}) + b_2(T^{\dim{\zg(\ggo)\cap\pg}}) =\dim{\zg(\kg)} + 
\tbinom{\dim{\zg(\ggo)\cap\pg}}{2},
$$
as shown above.
\end{remark}

\begin{proof}
Using the orthogonal decomposition $\zg(\ggo)=\kg_{\zg(\ggo)}\oplus (\zg(\ggo)\cap\pg)$, we obtain that 
$$
\kg\subset\ggo_1:=[\ggo,\ggo]\oplus\kg_{\zg(\ggo)}, \qquad \ggo=\ggo_1\oplus\zg(\ggo)\cap\pg. 
$$  
It follows from \cite[Chapter 1, Corollary 5, pp. 86]{Ons} that $G_{ss}$ acts transitively on $G_1/K$, yielding $G_1/K=G_{ss}/K\cap G_{ss}$ (alternatively, we can apply \cite[Chapter 1, Proposition 9, pp.\, 94]{Ons} and the fact that $b_1(G_1/K)=0$ and so $G_1/K$ has finite fundamental group).  On the other hand, since, up to a finite cover, $G_1=G_{ss}\times T^{\dim{\zg(\ggo)\cap\pg}}$, we have that $M=G/K=G_1/K\times T^{\dim{\zg(\ggo)\cap\pg}}$, up to a finite cover, concluding the proof.   
\end{proof}

%In particular, $b_2(G/K)=\dim{\zg(\kg)}$ if and only if 
%$$
%\dim{\zg(\ggo)}-\dim{\kg_{\zg(\ggo)}}= \dim{\ggo}-\dim{([\ggo,\ggo]+\kg)}\leq 1,   
%$$   
%e.g., when $G$ is semisimple.  

The third cohomology is computed in \cite[Section 4]{H3} for $G$ semisimple and formulas for $b_3(G/K)$ and $b_4(G/K)$ are given in \cite{BswChtMty} in the general compact case.  Any bi-invariant symmetric bilinear form $R$ on $\ggo$ such that $R|_{\kg\times\kg}=0$ naturally defines a closed $G$-invariant $3$-form $\vp_R$ on $M$ by 
\begin{align}
\vp_R(X,Y,Z) :=& 4R([X,Y],Z) - R([X,Y]_\pg,Z) + R([X,Z]_\pg,Y) - R([Y,Z]_\pg,X) \label{phiR-def} \\ 
=& R([X,Y],Z) + R([X,Y]_\kg,Z) - R([X,Z]_\kg,Y) + R([Y,Z]_\kg,X), \notag
\end{align}
for all $X,Y,Z\in\pg$.  Note that $\vp_Q=Q([\cdot,\cdot],\cdot)\in\Lambda^3\pg^*$, but $\vp_Q$ is not closed in general since $Q|_{\kg\times\kg}\ne 0$ if $\kg\ne 0$.  Moreover, 
$$
H^3(M)=\{[\vp_R]:R|_{\kg\times\kg}=0\}  \qquad \mbox{and}\qquad
b_3(M)=\dim\{R\in\sym^2(\ggo)^G:R|_{\kg\times\kg}=0\}.
$$ 
In particular, $b_3(M)=0$ as soon as $G$ is simple and $\kg\ne 0$, and $b_3(M)\leq s$, where equality holds if and only if $\kg=0$, i.e., $M=G$.  Here $s$ is the number of simple factors of $G$.   

Consider the decompositions $\ggo=\ggo_1\oplus\dots\oplus\ggo_s$ and $\kg=\kg_0\oplus\kg_1\oplus\dots\oplus\kg_v$ in simple factors, where $\kg_0:=\zg(\kg)$.  If $\kil_{\pi_i(\kg_j)}=a_{ij}\kil_{\ggo_i}|_{\pi_i(\kg_j)}$, $a_{ij}\in\RR$, where $\pi_i:\ggo\rightarrow\ggo_i$ is the usual projection, then it is proved in \cite[Proposition 4.3]{H3} that $R=z_1\kil_{\ggo_1}+\dots+z_s\kil_{\ggo_s}$ satisfies that $R|_{\kg\times\kg}=0$ if and only if the vector $(z_1,\dots,z_s)\in\RR^s$ is orthogonal to the $v$ vectors 
$$
(\tfrac{1}{a_{11}},\dots,\tfrac{1}{a_{s1}}), \quad\dots\quad, (\tfrac{1}{a_{1v}},\dots,\tfrac{1}{a_{sv}}),
$$
where we set $\tfrac{1}{a_{ij}}:=0$ if $a_{ij}=0$ (i.e., $\pi_i(\kg_j)=0$), and to the $\binom{m+1}{2}$ vectors
$$
(\kil_{\ggo_1}(Z^j_1,Z^k_1),\dots,\kil_{\ggo_s}(Z^j_t,Z^k_t)), \qquad 1\leq j\leq k\leq m,
$$
where $\{ Z^1,\dots,Z^m\}$ is a basis of $\zg(\kg)$.  In this way, 
\begin{equation}\label{b3-app}
b_3(M)=s-\dim{S_\kg},
\end{equation} 
where $S_\kg$ is the subspace of $\RR^s$ generated by the above $v+\binom{m+1}{2}$ vectors.

\subsection{Hodge theory}\label{Ht}
Assume that $M=G/K$ is compact.  Any $g\in\mca^G$ determines an inner product on each $\Lambda^p\pg^*$ given by
$$
g(\alpha,\beta) := \tfrac{1}{p!}
%& \sum_{i_1<\dots<i_k}\alpha(X_{i_1},\dots,X_{i_k})\beta(X_{i_1},\dots,X_{i_k}) \\ =
\sum_{i_1,\dots,i_p}\alpha(X_{i_1},\dots,X_{i_p})\beta(X_{i_1},\dots,X_{i_p}), 
$$
where $\{ X_i\}$ is any $g$-orthonormal basis of $\pg$.  Note that $\{ X^{i_1}\wedge\dots\wedge X^{i_p}\}$ is therefore a $g$-orthonormal basis of $\Lambda^p\pg^*$, where $\{ X^i\}$ is the basis of $\pg^*$ $g$-dual to $\{ X_i\}$.  If 
$$
d_g^*:(\Lambda^{p+1}\pg^*)^K\longrightarrow(\Lambda^p\pg^*)^K
$$ 
is the adjoint of $d$ with respect to $g$ (i.e., $g(d_g^*\cdot,\cdot)=g(\cdot,d\cdot)$), then a $p$-form $\alpha$ is closed and {\it coclosed} (i.e., $d_g^*\alpha=0$) if and only if $\alpha$ is in the kernel of the Hodge Laplacian 
$$
\Delta_g:=dd_g^*+d_g^*d:(\Lambda^p\pg^*)^K\longrightarrow(\Lambda^p\pg^*)^K,
$$ 
and it is called $g$-{\it harmonic} in that case.  Since 
$$
(\Lambda^p\pg^*)^K = \rlap{$\underbrace{\phantom{\Ima d \oplus \Ker\Delta_g}}_{\Ker d}$} \Ima d \oplus\overbrace{\Ker\Delta_g \oplus \Ima d_g^*}^{\Ker d_g^*},
$$ 
we have that $H^p(G/K) \simeq \Ker \Delta_g$, that is, each class has a unique $g$-harmonic representative.  

We refer to \cite{H3} for a study of invariant harmonic $p$-forms on compact homogeneous spaces for $p=1,2,3$.

\subsection{Invariant connections, parallelism and holonomy}\label{IC-sec}
Any $G$-invariant connection $\nabla$ on $M=G/K$ is determined by its {\it Nomizu operator} (see \cite[Chapter X]{KbyNmz}), 
$$
\Lambda:\pg\rightarrow\glg(\pg), \qquad \Lambda(X)Y:=(\nabla_{X^*}Y^*)_o+[X,Y]_\pg, 
$$ 
or equivalently, 
$$
\Lambda(X)Y:=(\nabla_{X^*}Y^*-[X^*,Y^*])_o, \qquad\forall X,Y\in\pg.
$$
See for example \S\ref{LC-sec} for the case of the Levi-Civita connection $\nabla^g$ of a $G$-invariant metric $g$ on $M=G/K$.  In this case, $X^*$ is a Killing vector field for any $X\in\pg$, but the vector field $\nabla^g_{X^*}Y^*$ is not  Killing in general.  

%\begin{example}
%The Nomizu operator $\Lambda^g$ of the Levi-Civita connection $\nabla^g$ of any $g\in\mca^G$ is given by
%$$
%\Lambda^g(X)Y = \unm[X,Y]_\pg+U(X,Y), \qquad \forall X,Y\in\pg, 
%$$
%where $U:\pg\times\pg\rightarrow\pg$ is the symmetric bilinear map defined by 
%$$
%2g(U(X,Y),Z):=g([Z,X]_\pg,Y)+g(X,[Z,Y]_\pg), \qquad \forall X,Y,Z\in\pg.
%$$
%\end{example}

The torsion $T(X,Y):=\nabla_XY-\nabla_YX-[X,Y]$ and the curvature $R(X,Y):=\nabla_X\nabla_Y-\nabla_Y\nabla_X-\nabla_{[X,Y]}$ of the connection $\nabla$ are respectively given by 
$$
T(X,Y):=T(X^*,Y^*)_o=\Lambda(X)Y-\Lambda(Y)X-[X,Y]_\pg, \qquad\forall X,Y\in\pg, 
$$
$$
R(X,Y):=R(X^*,Y^*)_o=[\Lambda(X),\Lambda(Y)]-\Lambda([X,Y]_\pg)-\ad{[X,Y]_\kg}|_\pg, \qquad\forall X,Y\in\pg.
$$
It is well known that for any $G$-invariant  tensor or form $\psi$ (see e.g.\ \cite[Section 2.3]{AgrHfmLwn}), 
$$
(\nabla_{X^*}\psi)_o = \Lambda(X)\cdot\psi_o, \qquad\forall X\in\pg,
$$
where $\cdot$ denotes the usual action on tensors.  For instance, such action is given as follows in the following particular cases:
$$
(\Lambda(X)\cdot g)(Y,Z):=-g(\Lambda(X)Y,Z)-g(Y,\Lambda(X)Z), \quad \Lambda(X)\cdot J:= [\Lambda(X),J], 
$$
$$
(\Lambda(X)\cdot T)(Y,Z):=\Lambda(X)T(Y,Z)-T(\Lambda(X)Y,Z)-T(Y,\Lambda(X)Z), 
$$
\begin{align*}
(\Lambda(X)\cdot R)(Y,Z)W:=& \Lambda(X)R(Y,Z)W-R(\Lambda(X)Y,Z)W \\ 
&-R(Y,\Lambda(X)Z)W -R(Y,Z)\Lambda(X)W.   
\end{align*}
According to \cite[Chapter X,\S 4]{KbyNmz}, the {\it holonomy algebra} $\holg(\nabla)$ of a $G$-invariant connection $\nabla$ on a compact homogeneous space $M=G/K$ is the smallest Lie subalgebra of $\glg(\pg)$ containing 
$$
\ad{\kg}|_\pg=\{\ad{Z}|_\pg:Z\in\kg\} \qquad\mbox{and} \qquad \Lambda(\pg)=\{\Lambda(X):X\in\pg\}.  
$$  
Note that $\ad{\kg}|_\pg\simeq\kg$ by almost-effectiveness.  

Given a $G$-invariant connection $\nabla$ on $M=G/K$ and a metric $g\in\mca^G$, we have that $\nabla g=0$ if and only if $\Lambda(X)^t=-\Lambda(X)$ for all $X\in\pg$, where the transpose is taken with respect to the metric $g$.  On the other hand, for a $G$-invariant complex structure $J\in\cca^G$, $\nabla J=0$ if and only if $[\Lambda(X),J]=0$ for all $X\in\pg$.  In other words, the connection $\nabla$ is {\it Hermitian} (i.e., $\nabla g=0$ and $\nabla J=0$) with respect to a compatible pair $(J,g)$ if and only if $\Lambda(X)$ is skew-Hermitian for all $X\in\pg$, i.e., $\Lambda(X)$ belongs to  
$$
\ug(n):=\ug(\pg,J,g)=\{ U\in\glg(\pg): U^t=-U,\; [U,J]=0\},  
$$
where $\dim{M}=\dim{\pg}=2n$.  In particular, 
$
\holg(\nabla)\subset\ug(n)
$ 
for any Hermitian connection.

\subsection{Automorphisms}\label{gauge}
The Lie group $\Aut(G/K)\subset\Diff(M)$ of all Lie group automorphisms of $G$ taking $K$ onto $K$ acts by pullback on $G$-invariant tensors of any kind, playing the role of the natural `gauge group' in the $G$-invariant setting, in the sense that two $G$-invariant geometric structures are considered {\it equivalent} if and only if they belong to the same $\Aut(G/K)$-orbit.  Note that equivalent implies isometric or biholomorphic, but the converse may not hold.    

A distinguished subgroup of $\Aut(G/K)$ is given by the normalizer $N_G(K)$, which acts on $M$ by $n\cdot(a\cdot o)=I_n(a\cdot o):=nan^{-1}\cdot o$ and on $T_oM\equiv \ggo/\kg$ by $n\cdot X:=\Ad(n)X$.  Alternatively, the Lie group $N:=N_G(K)/K$ acts on $M$ by {\it $G$-equivariant diffeomorphisms} (i.e., $\psi(a\cdot p)=a\cdot \psi(p)$ for all $a\in G$, $p\in M$) in the following way: $n\cdot (a\cdot o)=R_n(a\cdot o):=an\cdot o$ (note that $R_n^*g=I_{n^{-1}}^*g$ for any $n\in N$).  In particular, pullback by a $G$-equivariant diffeomorphism produces an equivalent $G$-invariant structure.  

In general, the moduli spaces $\mca^G/\Aut(G/K)$ and $\cca^G/\Aut(G/K)$ are really hard to describe or understand.  

In the case when $G$ is compact, the Lie groups $\Aut(G/K)$ and $N_G(K)$ have the same Lie algebra, given by $N_\ggo(\kg)=\kg\oplus\pg_0$.  For any $g\in\mca^G$, we consider the Lie algebra of the group $\Aut(G/K)\cap\Iso(M,g)$ of all isometric automorphisms, where $\Iso(M,g)$ is the isometry group, given by 
\begin{equation}\label{iaut}
\iautg(g):=\left\{ Z\in N_\ggo(\kg):g(\ad{Z}|_\pg\cdot,\cdot)=-g(\cdot,\ad{Z}|_\pg\cdot)\right\}, 
\end{equation}  
and for $J\in\cca^G$, the Lie algebra of the group $\Aut(G/K)\cap\Bihol(M,J)$ of all biholomorphic automorphisms, where $\Bihol(M,J)$ is the group of all biholomorphisms of the complex manifold $(M,J)$, which is given by  
\begin{equation}\label{baut}
\bautg(J):=\left\{ Z\in N_\ggo(\kg):[\ad{Z}|_\pg,J]=0\right\}.  
\end{equation}  
We note that $\kg$ is always contained in both $\iautg(g)$ and $\bautg(J)$ and that their Lie algebra isomorphism classes are invariant for the action of $\Aut(G/K)$ on $\mca^G$ and $\cca^G$, respectively.  

If $(J,g)$ is compatible, then $\hiautg(J,g):=\iautg(g)\cap\bautg(J)$ is the Lie algebra of the intersection of $\Aut(G/K)$ with the group of all holomorphic isometries of the Hermitian manifold $(M,J,g)$.

\end{document}